\DeclareSymbolFont{AMSb}{U}{msb}{m}{n}
\documentclass[11pt,a4paper,leqno,noamsfonts]{amsart}

\usepackage{graphicx}
\usepackage{subcaption}

\usepackage[english]{babel}
\usepackage[autostyle]{csquotes}

   \makeatletter
   \renewcommand\@biblabel[1]{#1.}
   \makeatother
   \usepackage{multirow}

\usepackage[english]{babel}
\usepackage[dvipsnames]{xcolor}
\usepackage{graphicx,pifont,soul,physics} 
\usepackage[utopia]{mathdesign}
\usepackage[a4paper,top=3cm,bottom=3cm,
           left=3cm,right=3cm,marginparwidth=60pt]{geometry}

\usepackage{fancyhdr}
\usepackage{float}

\usepackage{matlab-prettifier}

\usepackage{mathtools}

\usepackage{beramono}
\usepackage[T1]{fontenc}

\usepackage{color}
\usepackage{rotating, graphicx}

\definecolor{identifiercolor}{rgb}{.4,.6,.56}
\definecolor{stringcolor}{gray}{0.5}
\definecolor{inactivecolor}{rgb}{0.15,0.15,0.5}

\usepackage{braket,caption,comment,mathtools,stmaryrd}
\usepackage[usestackEOL]{stackengine}
\usepackage{multirow,booktabs,microtype,relsize}

\usepackage[colorlinks,bookmarks]{hyperref} %
      \hypersetup{colorlinks,%
           citecolor=britishracinggreen,%
            filecolor=black,%
            linkcolor=PineGreen,%
            urlcolor=cornellred}
      \numberwithin{equation}{section}
\usepackage[capitalise]{cleveref}
\usepackage[colorinlistoftodos]{todonotes}
\definecolor{antiquewhite}{rgb}{0.98, 0.92, 0.84}
\definecolor{buff}{rgb}{0.94, 0.86, 0.51}
\definecolor{palecopper}{rgb}{0.85, 0.54, 0.4}
\definecolor{fluorescentyellow}{rgb}{0.8, 1.0, 0.0}
\definecolor{bole}{rgb}{0.47, 0.27, 0.23}

\usepackage{amsmath,float}
\usetikzlibrary{decorations.markings}

\usepackage[vcentermath]{youngtab}
\input xy
\xyoption{all}

\usepackage{pgfkeys}
\usepackage{pgfopts}
\usepackage{ytableau}

\definecolor{cornellred}{rgb}{0.7, 0.11, 0.11}
\definecolor{britishracinggreen}{rgb}{0.0, 0.26, 0.15}
\definecolor{cobalt}{rgb}{0.0, 0.28, 0.67}
\DeclareSymbolFont{usualmathcal}{OMS}{cmsy}{m}{n}
\DeclareSymbolFontAlphabet{\mathcal}{usualmathcal}

\newcommand\li{\operatorname{li}}
\newcommand\fod{\mathfrak{d}}
\newcommand\fog{\mathfrak{g}}
\newcommand\fom{\mathfrak{m}}
\newcommand\foD{\mathfrak{D}}
\newcommand\Int{\operatorname{Int}}
\newcommand\Arg{\operatorname{Arg}}

\newcommand{\BC}{{\mathbb{C}}}

\newcommand{\BP}{{\mathbb{P}}}
\newcommand{\BQ}{{\mathbb{Q}}}
\newcommand{\BR}{{\mathbb{R}}}

\newcommand{\BZ}{{\mathbb{Z}}}

\newcommand{\CA}{{\mathcal A}}

\newcommand{\CC}{{\mathcal C}}

\newcommand{\CE}{{\mathcal E}}
\newcommand{\CF}{{\mathcal F}}
\newcommand{\CG}{{\mathcal G}}
\newcommand{\CH}{{\mathcal H}}
\newcommand{\CI}{{\mathcal I}}

\newcommand{\CL}{{\mathcal L}}
\newcommand{\CM}{{\mathcal M}}

\newcommand{\CO}{{\mathcal O}}
\newcommand{\CP}{{\mathcal P}}
\newcommand{\CQ}{{\mathcal Q}}
\newcommand{\CR}{{\mathcal R}}
\newcommand{\CS}{{\mathcal S}}
\newcommand{\CT}{{\mathcal T}}

\newcommand{\ch}{{\mathrm{ch}}}

\newcommand{\bdd}{{\mathrm{bdd}}}
\newcommand{\unbdd}{{\mathrm{unbdd}}}

\DeclareMathOperator{\Hilb}{Hilb}

\DeclareMathOperator{\Coh}{Coh}

\DeclareMathOperator{\D}{D}

\DeclareFontFamily{OT1}{rsfs}{}
\DeclareFontShape{OT1}{rsfs}{n}{it}{<-> rsfs10}{}
\DeclareMathAlphabet{\curly}{OT1}{rsfs}{n}{it}
\renewcommand\hom{\mathscr{H}\kern-0.3em\mathit{om}}

\newcommand\Ext{\operatorname{Ext}}
\newcommand\Hom{\operatorname{Hom}}

\DeclareMathOperator{\lHom}{\mathscr{H}\kern-0.3em\mathit{om}}
\DeclareMathOperator{\RRlHom}{\mathbf{R}\kern-0.025em\mathscr{H}\kern-0.3em\mathit{om}}

\DeclareMathOperator{\lExt}{{\mathscr{E}\kern-0.2em\mathit{xt}}}

\DeclareMathOperator{\pr}{pr}

\newcommand{\RR}{\mathbf R}



\usepackage[all]{xy}
\usepackage{tikz}
\usepackage{tikz-cd}
\usepackage{adjustbox}
\usepackage{rotating}
\usepackage{comment}

\usetikzlibrary{matrix,shapes,intersections,arrows,decorations.pathmorphing}
\tikzset{commutative diagrams/arrow style=math font}
\tikzset{commutative diagrams/.cd,
mysymbol/.style={start anchor=center,end anchor=center,draw=none}}

\tikzset{
shift up/.style={
to path={([yshift=#1]\tikztostart.east) -- ([yshift=#1]\tikztotarget.west) \tikztonodes}
}
}

\theoremstyle{definition}

\newtheorem*{lemma*}{Lemma}
\newtheorem*{theorem*}{Theorem}
\newtheorem*{example*}{Example}
\newtheorem*{fact*}{Fact}
\newtheorem*{notation*}{Notation}
\newtheorem*{definition*}{Definition}
\newtheorem*{prop*}{Proposition}
\newtheorem*{remark*}{Remark}
\newtheorem*{construction*}{Construction}
\newtheorem*{corollary*}{Corollary}

\newtheorem*{conventions*}{Conventions}

\newtheorem{definition}{Definition}[section]

\newtheorem{example}[definition]{Example}

\newtheorem{notation}[definition]{Notation}
\newtheorem{remark}[definition]{Remark}
\newtheorem{construction}[definition]{Construction}

\newtheoremstyle{thm} 
        {3mm}
        {3mm}
        {\slshape}
        {0mm}
        {\bfseries}
        {.}
        {1mm}
        {}
\theoremstyle{thm}
\newtheorem{theorem}[definition]{Theorem}
\newtheorem{corollary}[definition]{Corollary}
\newtheorem{lemma}[definition]{Lemma}

\newtheorem{prop}[definition]{Proposition}
\newtheorem{thm}{Theorem}

\newtheoremstyle{ex} 
        {3mm}
        {3mm}
        {}
        {0mm}
        {\scshape}
        {.}
        {1mm}
        {}
\theoremstyle{ex}

\newtheoremstyle{sol} 
        {3mm}
        {3mm}
        {}
        {0mm}
        {\scshape}
        {.}
        {1mm}
        {}
\theoremstyle{sol}

\usepackage{tikz}
\usepackage{xparse}

\newcount\tableauRow
\newcount\tableauCol

\newenvironment{Tableau}[1]{%
  \tikzpicture[scale=0.5,draw/.append style={thick,black},
                      baseline=(current bounding box.center)]
    \tableauRow=-1.5
    \foreach \Row in {#1} {
       \tableauCol=0.5
       \foreach\k in \Row {
         \draw[thin](\the\tableauCol,\the\tableauRow)rectangle++(1,1);
         \draw[thin](\the\tableauCol,\the\tableauRow)+(0.5,0.5)node{$\k$};
         \global\advance\tableauCol by 1
       }
       \global\advance\tableauRow by -1
    }
}{\endtikzpicture}

\newtheorem*{Acknowledgments*}{Acknowledgments}

\DeclareMathAlphabet\BCal{OMS}{cmsy}{b}{n}

\address{Centre for Mathematical Sciences, University of Cambridge, Wilberforce Road, CB3 0WA, Cambridge, United Kingdom}

\title[Geometry of the stability scattering diagram
]{Geometry of the stability scattering diagram for $\mathbb{P}^2$ and applications}

\author{Mark Gross}
\email{mgross@dpmms.cam.ac.uk}

\author{Fatemeh Rezaee}
\email{fr414@cam.ac.uk}

\date{}

\begin{document}
\maketitle

\begin{abstract}
We give a detailed analysis of the stability scattering diagram for $\BP^2$ 
introduced by Bousseau in \cite{Bousseau-scatteringP2-22}. This scattering
diagram lives in a subset of $\BR^2$, and we
decompose this subset into three regions, $R_{\Delta},R_{\Diamond}$ and 
$R_{\mathrm{unbdd}}$. The region $R_{\Delta}$ has a chamber structure
whose chambers are in one-to-one correspondence with strong exceptional triples.
No ray of the stability scattering diagram enters the interior of such a
triangle, replicating a result of Prince \cite{Prince-20} and generalizing
a result of Bousseau.
The region $R_{\Diamond}$ is decomposed into \emph{diamonds}, which
are in one-to-one correspondence with exceptional bundles. Each
diamond has a vertical diagonal corresponding to a rank zero object and is traversed by a dense set of rays. Crucially, however, 
there are no collisions of rays inside diamonds, making it still 
possible to control the scattering diagram in $R_{\Diamond}$. Finally, 
the behaviour of $R_{\mathrm{unbdd}}$
is chaotic, in that every rational point inside it is a collision
of an infinite number of rays. We show that
the \emph{bounded region} $R_{\mathrm{bdd}}=
R_{\Delta}\cup R_{\Diamond}$ has as upper boundary the Le Potier curve,
thus showing that this curve arises naturally through the algorithmic
scattering process. We give an application of these results
by describing the first wall-crossing for the moduli space of
one-dimensional rank zero objects on $\BP^2$. In the sequel
\cite{Gross-Rezaee2}, we apply these results to describe the full Bridgeland 
wall-crossing for $\Hilb^n(\mathbb {P}^2)$ for any $n$.

\end{abstract}

{\hypersetup{linkcolor=black}
\tableofcontents}

\section{Introduction} \label{Section: intro}

Moduli spaces of sheaves on projective varieties $X$ have been studied for decades and play a crucial role in algebraic and enumerative geometry. 
From a modern viewpoint, it has proved to be very fruitful to study these
moduli spaces by examining how they change under variation of stability
conditions. More specifically, fix a non-singular projective
variety $X$. If $\gamma$ is a fixed
Chern character, and $\sigma$ is a Bridgeland stability condition
\cite{Bridgeland-stability} on
$D^b(X)$, denote by
$\CM^{\sigma}_{\gamma}$ the moduli space of $S$-equivalence classes
of $\sigma$-semistable objects with Chern character $\gamma$ on $X$.
Then as $\sigma$ varies, this moduli space may undergo changes as
$\sigma$ passes through a ``wall'' in the space of Bridgeland stability 
conditions. These changes typically preserve the birational equivalence
class of the moduli space, but perform some interesting non-trivial
birational operation, such as a flop.  Wall-crossing of this nature for three-folds has been studied in  \cite{Sch, Sch14, Xia, GHS, R2, R1}, where the authors use complicated methods to even list the walls. An initial motivation for the current article was to introduce a correspondence between the space of stability conditions on $\BP^3$ and a ``stability scattering diagram for $\BP^3$'' , to handle the traditional wall-crossing more systematically. This motivation led us to study the stability scattering diagram for $\BP^2$ more closely and discover several interesting facts which should eventually be applicable to the higher dimensional case as well.

There is an even longer history of study of moduli spaces of sheaves 
specifically on $\BP^2$. The
pioneering work of Drezet and Le Potier \cite{Drezet-LePotier} initiated
a systematic study of these moduli spaces. They explored
the structure of exceptional bundles, whose moduli spaces consist of
reduced points, and determined when a moduli space of
semi-stable sheaves of fixed Chern character is non-empty.
More recently, the wall-crossing point of view has been explored extensively for $\BP^2$ in many works, including \cite{ABCH,LZ,LZ2,BMW14,MacA,Toda14,Toda13,CH14b,BC13,BHLRSWZ16}. In particular, the wall-crossing point of view 
underpins the program of study of 
the birational geometry of Hilbert schemes of points in $\BP^2$ initiated
in \cite{ABCH}.

The moduli space of Bridgeland stability conditions is a complex manifold
of dimension $3$, and thus it is rather difficult to visualize walls
and think about them. However, Bridgeland \cite{Bridgeland-stability}
constructed an action of
the universal cover of $\mathrm{GL}_2(\BR)^+$ on the moduli space of
stability conditions; this action preserves moduli spaces, and in particular
all walls are pulled back from the quotient, which is now of real dimension
$2$.

Work of Bousseau \cite{Bousseau-scatteringP2-22}
then was able to give a much cleaner picture of the structure of 
wall-crossing on $\BP^2$. He chose a particular two-dimensional
slice of the above-mentioned group action, and made a clever choice
of coordinates on this two-dimensional slice.
He showed that in this two-dimensional slice, one could build an
object called a \emph{scattering diagram} which encodes the topology
of moduli spaces of semi-stable objects. Scattering diagrams
are a notion arising in the study of mirror symmetry, initially
introduced by Kontsevich and Soibelman in \cite{KS-06} and then given a 
vast generalization 
by Gross and Siebert in \cite{GS11}. Since then scattering diagram
technology has found widespread applications. In particular,
Bousseau showed \cite{Bousseau-TakahashiConjecture-23}
that the scattering diagram constructed from moduli spaces of semi-stable
objects on $\BP^2$ was directly related to the
scattering diagram which arises in constructing the mirror to $\BP^2$.
This latter scattering diagram, in the case of the mirror to $\BP^2$, was first
studied by Carl, Pumperla and Siebert in \cite{CPS}, and further analyzed by
Prince in \cite{Prince-20}.
Further applications of the $\BP^2$ scattering diagram for enumerative
purposes were explored by Gr\"afnitz in \cite{GraefnitzTropical22},  
\cite{GraefnitzTheta22} and by Gr\"afnitz, Ruddat and Zaslow in
\cite{Graefnitz-Ruddat-Zaslow-22}.

We now turn to the specifics of this paper.
Fix a lattice $M=\BZ^2$ and $M_{\BR}=M\otimes_{\BZ}\BR$, and consider
the open subset 
\[
U:=\{(x,y)\,|\,y>-x^2/2\}
\]
where $x,y$ are coordinates on
$M_{\BR}$. For the purpose of this introduction,
one can view a scattering diagram
$\foD$ as a collection of \emph{rays} $(\fod,H_{\fod})$, where
$\fod\subseteq \BR^2$ is a ray of the form $(a_0,b_0)-\BR_{\ge 0}(a,b)
\subseteq \overline{U}$
with $a,b\in \BZ$ relatively prime,
and $H_{\fod}$ is a formal power series in the monomial $z^{(a,b)}\in\BQ[M]$ 
with no constant term.
Associated to crossing a ray from one side to the other is an automorphism
of a certain algebra. A consistent scattering diagram is one for which
a composition of automorphisms associated to wall-crossings around a closed
loop in $U$ is always the identity. \cite{KS-06}
showed that one could construct a consistent scattering
diagram given an initial set of rays via an algorithmic procedure. Even
though this procedure is quite simple (to the extent that it can easily be put
on a computer), the resulting scattering diagram
can be exceedingly complicated and it can be very difficult to control.

Bousseau's main result \cite{Bousseau-scatteringP2-22} says that there is
a scattering diagram $\foD^{\mathrm{stab}}$, constructed via the
Kontsevich-Soibelman algorithm from an initial choice of rays,  such that 
the functions attached to rays are generating functions for Euler 
characteristics of moduli spaces of coherent sheaves on $\BP^2$ with
respect to various Bridgeland stability conditions. For a review of the
precise statements, see \S\ref{subsec:Bousseau scat} and in
particular Theorem \ref{thm:bousseau main}. However, roughly, if
$\sigma\in U$ is contained in some ray $(\fod,H_{\fod})$, then
the function $H_{\fod}$ will encode the Euler characteristic of
the moduli space of Bridgeland stable objects $\CM^{\sigma}_{\gamma}$
for those Chern characters $\gamma$ such that $Z^{\sigma}(\gamma)\in
i\BR_{>0}$, where $Z^{\sigma}$ is the central charge associated to the stability
condition $\sigma$.

The fact that $\foD^{\mathrm{stab}}$ essentially
coincides with the scattering diagram arising from mirror symmetry
as studied in \cite{CPS} and \cite{Prince-20} is a surprisng observation.
See \cite[\S7]{Bousseau-TakahashiConjecture-23} for
a heuristic explanation for why these two scattering diagrams should
agree. This proved to be a powerful
tool in \cite{Bousseau-TakahashiConjecture-23}, where Bousseau used sheaf
theory to solve a long-standing problem about enumeration of certain kinds
of curves on $\BP^2$. 
Going backwards, one may use results obtained about
the mirror symmetry scattering diagram to understand $\foD^{\mathrm{stab}}$.
For example, Prince in \cite{Prince-20} proved
some structure results for the scattering diagram arising from mirror
symmetry. Let us still write $\foD^{\mathrm{stab}}$ for the diagram that
Prince studied, even though it arose in a completely different way.
Prince showed that there was a region of $\BR^2$
split into a countable number of triangles, such that the interior of
each triangle is disjoint from any ray in $\foD^{\mathrm{stab}}$. 

From our point of view, $\foD^{\mathrm{stab}}$ gives us a powerful way of
visualizing wall-crossing for moduli spaces of sheaves. The basic idea
is as follows. Fix a possible Chern character $\gamma$ for a sheaf of
interest (excluding the Chern character of skyscraper sheaves).  Define
\begin{align*}L_{\gamma}:=\{\sigma\in M_{\BR}\,|\, Z^{\sigma}(\gamma)\in i\BR\};
\end{align*}
this is in fact a line in $M_{\BR}$.
Then $L_{\gamma}\cap U$ has two connected components provided 
$\CM^{\sigma}_{\gamma}$ is non-empty for some $\sigma$ (this follows
from the Bogomolov-Gieseker inequality).
If $\sigma\in L_{\gamma}$
and the moduli space $\CM^{\sigma}_{\gamma}$ is non-empty (and with
non-zero Euler characteristic), then there
will be a ray $(\fod,H_{\fod})\in \foD^{\mathrm{stab}}$ 
with $\sigma \in\fod$ and $\fod\subseteq L_{\gamma}$. Typically, for 
$\sigma$ near the endpoint of one of the two connected components of
$L_{\gamma}\cap U$, the moduli space
$\CM^{\sigma}_{\gamma}$ is empty, but for $\sigma$ very far from the
endpoint, the moduli space stabilizes and remains unchanged as $\sigma$ goes
off to infinity in the plane. As $\sigma$ moves from one extreme to the
other, the moduli space $\CM^{\sigma}_{\gamma}$ is expected to undergo
a number of birational changes. This may or may not result in a change
of Euler characteristic, so may or may not be reflected in a change of
the functions $H_{\fod}$. Regardless, this gives us a strong hint as to
where to look for wall-crossing. 

With this motivation in mind, in this paper we prove some finer points
of the structure of $\foD^{\mathrm{stab}}$. To state the results, we
introduce some notation. We define
\[
R:=\big\{(x,y)\in \BR^2\,|\, \hbox{$y\ge nx -3n^2/2$ for some $n\in \BZ$}\big\}.
\]
This set has the feature that all rays of $\foD^{\mathrm{stab}}$
are contained in $R$, and the initial rays needed to construct
$\foD^{\mathrm{stab}}$ inductively are contained in the lines
$y=nx-3n^2/2$ for various $n$, so that the boundary of $R$ is
contained in the union of these rays.
We then have a decomposition of $R$ into three regions,
\[
R=R_{\Delta}\cup R_{\Diamond}\cup R_{\mathrm{unbdd}}.
\]

Here, $R_{\Delta}$ decomposes as a countable union of triangles, 
indexed by strong exceptional triples containing at most two line bundles, with the
property that the interior of each triangle is disjoint from any ray
of $\foD^{\mathrm{stab}}$. This is Prince's chamber structure. 

The region
$R_{\Diamond}$, on the other hand, consists of a countable union of 
diamond shapes. There is a one-to-one correspondence between these diamonds
and exceptional bundles. Further, each diamond $\Diamond_v$ has a vertex
$v$ we call the \emph{base} of the diamond, and this vertex determines
the diamond $\Diamond_v$, hence the notation. Most crucially, the only rays
of $\foD^{\mathrm{stab}}$ which meet the interior of $\Diamond_v$ have
end point at $v$, and if $\fod:=v-\BR_{\ge 0}(a,b)$ with $a,b\in\BZ$
intersects the interior of $\Diamond_v$, then there is a ray of
$\foD^{\mathrm{stab}}$ with support $\fod$. Hence, $\foD^{\mathrm{stab}}$
is in fact very complicated inside $\Diamond_v$ in the sense that it is
dense, but nevertheless it is quite controlled.

Finally, $R_{\mathrm{unbdd}}$ is truly chaotic. We show that \emph{every}
rational point of $R_{\mathrm{unbdd}}$ is contained in an infinite
number of distinct rays of $\foD^{\mathrm{stab}}$, and hence we expect
that it will be impossible to make any general statements in 
$R_{\mathrm{unbdd}}$. Nevertheless, as we shall see in the sequel paper
\cite{Gross-Rezaee2}, where we apply the results in this paper to study wall-crossing
for $\mathrm{Hilb}^{n}(\BP^2)$,
it is still possible to get useful information in $R_{\mathrm{unbdd}}$.
Here,  in \S\ref{section:rankZero}, we apply our results
to describe the first walls for moduli spaces of sheaves of rank $0$ but
non-zero first Chern class. To the best of our knowledge, this is the first description of such moduli spaces, thanks to the powerful machinery of the scattering diagram.

Another lovely aspect of the scattering diagram technology is that many 
classical aspects of the theory of coherent sheaves on $\BP^2$ appear
algorithmically. As already mentioned, the triangles of $R_{\Delta}$
are indexed by strong exceptional triples. The rays which make up the boundaries
of the triangles are indexed by exceptional bundles. The \emph{Le Potier
curve} $C_{LP}$ \cite{Drezet-LePotier}, \cite[Def.~1.6]{LZ2}
is visible as the union over all diamonds of the upper two
edges of the diamond (what we call the \emph{roof} of the diamond). We
explain this last point in \S\ref{Section: LePotier}.

We now turn to more precise statements of our main results.

\subsection{Main results}

As mentioned above, the scattering diagram $\foD^{\mathrm{stab}}$ is
generated by a set $\foD^{\mathrm{stab}}_{\mathrm{in}}$ 
of initial rays, this set being
in one-to-one correspondence with $\{\CO(d),\CO(-d)[1]\,|\,d\in\BZ\}$.
We construct an additional set of rays, which we call the \emph{discrete
rays}, $\foD_{\mathrm{discrete}}$, which are in one-to-one correspondence
with the set
\[
\{\CE, \CE[1]\,|\, \hbox{$\CE$ is an exceptional vector bundle of rank $>1$}\}.
\]
We also define a set $\foD_{\mathrm{dense}}$ of rays which are the rays
with endpoints being bases of diamonds. Then one helpful way of describing our
analysis is:

    \begin{thm}[See Corollaries \ref{Cor:equivofScatPartial}
and \ref{Thm:Equivalence}] The two scattering diagrams $\foD^{\mathrm{stab}}$ and
$\foD^{\mathrm{stab}}_{\mathrm{in}}\cup\foD_{\mathrm{discrete}} \cup \foD_{\mathrm{dense}}$
are equivalent (see Definition~\ref{def:equivalent}) in $R_{\mathrm{bdd}}:=R_{\Delta}\cup R_{\Diamond}$.
    \end{thm}

   \begin{thm}[See Corollary \ref{Cor:unboundedCharacterization}]
   Every point in $R_{\mathrm{unbdd}}$ with rational coordinates 
is contained in the
interior of a countably infinite number of non-trivial rays of 
$\foD^{\mathrm{stab}}$.
    \end{thm}

These statements are really statements at the numerical level, in the sense
that the scattering diagram only records the Euler characteristic of
moduli spaces and may not see non-empty moduli spaces of Euler characteristic
zero. For applications to wall-crossing, we need stronger results, and show
the following:

\begin{thm}[See Theorems \ref{thm:bousseau generalization} and \ref{Thm:nonEmptiness}]\label{Thm: MainScattering}
We have the following additional information for moduli spaces for 
$\sigma\in R_{\mathrm{bdd}}$:
\begin{itemize}
\item[(i)] For $\sigma$ in the interior of a triangle chamber of $R_{\Delta}$
and $\gamma$ a Chern character with $Z^{\sigma}(\gamma)\in i\BR_{>0}$, the
moduli space $\CM_\gamma^\sigma$ is empty.
\item[(ii)] If $\sigma$ is in the interior of a diamond chamber $\Diamond_v$ of
$R_{\Diamond}$ and $\gamma$ a Chern character with $Z^{\sigma}(\gamma)\in
i\BR_{>0}$, then $\CM_{\gamma}^{\sigma}$ is non-empty only if $L_{\gamma}$
passes through $v$. 
\end{itemize}
\end{thm}

In particular, the first item yields Prince's chamber structure as 
a consequence of the properties of exceptional triples. The second item
is more involved and requires a careful combinatorial analysis. 
In general, the non-empty moduli spaces for
$\sigma$ in the interior of a diamond are generalizations of Steiner bundles.

Here, we state several new interesting 
consequences for wall-crossing for rank 
zero sheaves. We note that the line $L_{\gamma}$ is vertical if and only
if $\gamma$ is the Chern character of a rank zero sheaves with
non-zero first Chern class. 
Furthermore,
we show in Theorem \ref{thm: verticalDiagonal} that the line segments connecting
the base of a diamond to its apex are also vertical. We call this line segment
the \textit{vertical diagonal} of a diamond.

\begin{thm}[See Theorem \ref{Thm: FirstWallRank0}]
Writing a Chern character $\gamma$ as the triple $(\ch_0,\ch_1,\ch_2)$,
let $\gamma=(0,u,v)$ or $(0,-u,-v)$, for $u\in \BZ^{> 0}, v\in\frac{\BZ}{2}$. 
Suppose $\gamma$ is a primitive Chern character and $L_{\gamma}$ 
contains a vertical diagonal of a diamond. Then the value of
$\sigma\in L_{\gamma}$ with the smallest $y$-coordinate such that 
$\CM_{n\gamma}^{\sigma}\not=\emptyset$ for some $n>0$ is given by the 
the base of the diamond, with value
$\sigma=(\frac{v}{u},\frac{5}{8}-\frac{1+v^2}{2u^2})$.
\end{thm}

\begin{thm}[See Theorem \ref{SufficientRank0Dene}] Let $\CE$ be
an exceptional bundle, and consider the diamond $\Diamond_v$ whose base
$v$ corresponds to the exceptional bundle $\CE$. We call the upper right
edge of $\Diamond_v$ the \emph{right roof} of $\Diamond_v$. Then the
image of the projection of the right roof onto the $x$-axis is the closed
interval
\[
\left[
\frac{d}{r}-\frac{3}{2},
\frac{d}{r}
-\frac{\sqrt{9r^2-4}}{2r}
\right].
\]

If $\gamma=(0,u,v)$ is a Chern character of a rank zero sheaf such that
$v/u$ lies in the interior of this interval, then the value of
$\sigma\in L_{\gamma}$ with the smallest $y$-coordinate such that
$\CM^{\sigma}_{n\gamma}\not=\emptyset$ for some $n>0$ is given by
the intersection of $L_{\gamma}$ with the right roof of $\Diamond_v$.
\end{thm}

See \S\ref{section:rankZero} for more results, which give a complete 
description of first walls for (multiples of) Chern characters of
rank zero sheaves. The above result is also used in \S\ref{Section: LePotier}
to show how the Le Potier curve arise as the union of roofs of all
diamonds.
In particular, the following theorem
follows immediately from \cite{Drezet-LePotier}
and our own results, but is a good summation of these results:

\begin{thm}[See Corollary \ref{cor:vertex location}]
Let $(x_0,y_0)\in U$ be a point with rational coordinates. Then 
$(x_0,y_0)$ is the intersection of two rays of 
$\foD^{\mathrm{stab}}$ with distinct tangent lines if and only if
$(x_0,y_0)=\left(\frac{\ch_1(\CE)}{\ch_0(\CE)}-\frac{3}{2}, \frac{3d-\chi(\CE)}
{\ch_0(\CE)}\right)$ for 
some slope-stable coherent sheaf of positive rank $\CE$.
\end{thm}

\begin{remark}
Since our main motivation for this work is to understand wall-crossing for
moduli spaces, we point out that \cite{CoskunHuizengaWoolf} described the
first wall, and \cite[Thm.~0.1]{LZ2} gives a different description of this
first wall. Meanwhile, \cite[Thm.~0.4]{LZ2} yields an algorithm for 
determining all other walls. For certain moduli spaces, such as
the Hilbert scheme of points in $\BP^2$, we will be able to read off 
similar information, see the forthcoming \cite{Gross-Rezaee2}. However,
the slice of the stability moduli space used by Bousseau, and hence in this 
paper, is quite different than
that of \cite{LZ2}, and as a result the statements take a quite different
form. 

The stability scattering diagram has also been studied further from different
points of view in \cite{Bousseau-Descombes-LeFloch-Pioline-22}.
\end{remark}

\subsection{Outline}
\S\ref{Section: background} recalls basics about coherent sheaves on 
$\BP^2$, Bridgeland stability conditions, and finally recalls the main results
of \cite{Bousseau-scatteringP2-22}. In \S\ref{sec:discrete}, we recall basic
facts about exceptional collections of bundles on $\BP^2$, and then use
this to provide an analysis of the region $R_{\Delta}$ and 
$\foD_{\mathrm{discrete}}$. In particular, these results
reproduce Prince's results of \cite{Prince-20} about the chamber
structure on $R_{\Delta}$, but we prove a stronger result of moduli
emptiness of the interior of each triangle, modelled after a result and
proof of
Bousseau in \cite{Bousseau-scatteringP2-22}. Finally we make use of results
of \cite{Graefnitz-Luo2023} to analyze the structure of $\foD^{\mathrm{stab}}$
at each vertex of a triangle and introduce $\foD_{\mathrm{dense}}$.

In \S\ref{Section: diamonds}, we introduce diamonds. On the other hand,
\S\ref{Section: verticalDiagonal} shows that we have now accounted for all
rays intersecting each diamond. We call this ``vertex freeness'', as 
this means that in the interior of a diamond, there are no rays meeting
transversally. This is what allows us to have good control inside the diamonds.

In \S\ref{sec:unbounded}, we move on to prove our results about $R_{\mathrm{unbdd}}$. In \S\ref{section:rankZero}, we give an application to studying wall-crossing for
moduli of sheaves of rank $0$ and non-zero first Chern class. Finally, in
\S\ref{Section: LePotier}, we show how classical aspects of the theory of sheaves
on $\BP^2$, such as the Le Potier curve, show up naturally in the structure
of $\foD^{\mathrm{stab}}$.

\subsection*{Conventions} If $r=\ch_0=\ch_0(\CE)$, $d=\ch_1=\ch_1(\CE)$, $e=\ch_2=\ch_2(\CE)$ and $\chi=\ch_2(\CE)+3/2\ch_1(\CE)+\ch_0(\CE)$, for some $\CE\in \mathrm{D}^b(\BP^2)$, then we denote $L_{\ch(\CE)}$ by $L_{\CE}$. 
Also, we alternatively use $(\ch_0,\ch_1,\ch_2)$, $(r,d,e)$ or $(r,d,\chi)$, for the corresponding information of an object in the derived category. Sometimes, we use $(0,u,v)$ for the Chern character of one-dimensional rank-0 objects. We call an object with primitive Chern character a \emph{primitive object}. We call a rational number $s/t$ so that $s,t\in\BZ$ are coprime a \emph{reduced rational number}. 

\subsection*{Acknowledgments} Both authors were supported by the ERC Advanced Grant MSAG. FR was also supported by UKRI grant No. EP/X032779/1. In addition, MG was supported by the Research Institute for Mathematical Sciences,
an International Joint Usage/Research Center located in Kyoto  University, and FR was hosted by ETH Z\"urich, when some of this research was carried out.
The authors thank Arend Bayer, Pierrick Bousseau and Tim Gr\"afnitz for helpful correspondence and discussions.

\subsection*{Notation}
\begin{center}
   \begin{tabular}{ r l }

   $\foD^{\mathrm{stab}}$: &Stability scattering diagram (Definition \ref{Def: stabilityScattering}).\\

   $\foD^{\mathrm{stab}}_{\mathrm{in}}$: &Initial rays in  $\foD^{\mathrm{stab}}$ (Definition \ref{def:path ordered}).\\

   $\foD_{\mathrm{discrete}}$: &Discrete rays in  $\foD^{\mathrm{stab}}$ (Section \ref{susection: construction}).\\

      $\foD_{\mathrm{dense}}$: &Dense rays in  $\foD^{\mathrm{stab}}$ (Definition \ref{DDense}).\\
   
    $(\fod,H_{\fod})$: &Stability ray (Definition \ref{Def: stabilityRay}).\\

    $L_{\gamma}$ or $L_E$:& The line associated to an object $E$ with $\ch(E)=\gamma$ (see \eqref{eqn:LGamma}).\\

    $\CT$:& The infinite binary tree (Construction \ref{susection: construction}).\\

$\Delta_w$:& The associated triangle to a vertex $w$ of $\CT$ (Construction \ref{susection: construction}).\\
$\CT_M$:& The Markov tree (Construction \ref{susection: construction}).\\
     
     $\Diamond_v$ or $\Diamond$ : & Dense diamond  based at the vertex $v$. For simplicity, we omit $v$, when there\\&  is no confusion (Definition \ref{def: diamond})\\

   \end{tabular}
   \end{center}
     \begin{center}
   \begin{tabular}{ r l }

   $\Diamond^{\mathrm{out}}_v$ or $\Diamond^{\mathrm{out}}$ : & Outer diamond  based at the vertex $v$. For simplicity, we omit $v$, when there\\& is no confusion (Definition \ref{def: diamond}).\\

      $\Diamond^{\mathrm{init}}_v$ or  $\Diamond^{\mathrm{init}}$ : & Initial diamond generated by initial objects  based at the vertex $v$. \\&For simplicity, we omit $v$, when there is no confusion (Definition \ref{Def:OuterInitial}).\\
    
     $\Diamond^{\mathrm{sup}}_v$ or $\Diamond^{\mathrm{sup}}$ : & Supper-diamond  based at the vertex $v$. For simplicity, we omit $v$, when\\& there is no confusion. Note that a super-diamond contains infinitely\\& many diamonds, and it is not a diamond itself (Definition \ref{Def: superDiamond}).\\
     $R_{\Delta}$: &Triangle structure of the scattering diagram (Corollary \ref{Cor: RDelta}).\\
      $R_{\Diamond}$: &Diamond structure of the scattering diagram (Definition \ref{Def: RDiamond}).\\

            $R_{\mathrm{bdd}}$:& The bounded region which is define to be $R_{\Delta}\cup R_{\Diamond}$ (Definition \ref{Def: RDiamond}).\\
            $R_{\mathrm{unbdd}}$:& The  unbounded region, which is the complement of $R_{\mathrm{bdd}}$ in $\foD^{\mathrm{stab}}$\\& (Definition \ref{Def: RDiamond}).\\
        
  $C_{\mathrm{roofs}}$ : & The fractal upper boundary of the bounded region $R_{\mathrm{bdd}}$ in the scattering\\ 
     & diagram (Corollary \ref{Cor: boundedness}).\\

   \end{tabular}
\end{center}

\section{Background: Stability conditions and scattering diagram for $\mathbb{P}^2$} \label{Section: background}

\subsection{Preliminaries and stability conditions on $\mathbb{P}^2$}
We set 
\begin{equation}
\label{eq:Gamma def}
\Gamma:=K_0(D^b(\BP^2)).
\end{equation}
We identify the latter with a rank
3 lattice in two different ways. First, the map
\[
\Gamma \rightarrow \BZ^3, \quad \CE\mapsto (r(\CE), d(\CE), \chi(\CE))
\]
is an isomorphism, where $r,d,\chi$ denote rank, degree and Euler
characteristic respectively. Second, we use
\[
\ch:\Gamma \rightarrow \BZ^2\oplus (1/2)\BZ, \quad
\CE\mapsto (\ch_0(\CE),\ch_1(\CE),\ch_2(\CE)),
\]
where $\ch_i$ denotes the $i$th Chern character. Note that
$\ch_0(\CE)=r(\CE)$ and $\ch_1(\CE)=d(\CE)$, but by Riemann-Roch, 
\begin{equation}
\label{eq:RR1}
\chi(\CE)=\ch_2(\CE)+{3\over 2}\ch_1(\CE)+\ch_0(\CE)=\ch_2(\CE)+{3\over 2}d(\CE)
+r(\CE).
\end{equation}
In particular, $\ch$ induces an isomorphism with its image
\begin{equation}
\label{eq:chern image}
\{(r,d,e)\in \BZ^{\oplus 2}\oplus (1/2)\BZ)\,|\, d/2+e\in \BZ\}.
\end{equation}
We often write $(r,d,e)$ instead of 
$(\ch_0,\ch_1,\ch_2)$.

The lattice $\Gamma$ has a pairing $(\cdot,\cdot)$ given by
\begin{equation}\label{eqn: PairingGamma}
   \left((r,d,\chi),(r',d',\chi')\right) = -3dr'-rr'-dd'+r\chi'+\chi r'. 
\end{equation}
This arises from
 Riemann-Roch (see \cite[Lemma 2.2]{Bousseau-scatteringP2-22}):
if $\CE,\CE'$ are sheaves on $\BP^2$, then 
\begin{equation}
\label{eq:RR}
\chi(\CE,\CE')=\sum_{i=0}^2 (-1)^i\dim \Ext^i(\CE,\CE') =
\left((r(\CE),d(\CE),\chi(\CE)),(r(\CE'),d(\CE'),\chi(\CE'))\right).
\end{equation}

\medskip

Denote by $\mathrm{Stab}(\mathbb{P}^{2})$ the space of all Bridgeland
stability conditions on $D^b(\BP^2)$. We recall from 
\cite{BM-LocalProjPlane} 
 a two-dimensional
subfamily of $\mathrm{Stab}(\BP^2)$ parameterized by two real parameters
$\alpha\in \mathbb{R}^+, \beta \in \mathbb{R}$. 

First define the twisted Chern characters 
\[
\mathrm{ch}^{\beta}(\CE)=(\mathrm{ch}_0^{\beta}(\CE),\mathrm{ch}_1^{\beta}(\CE),\mathrm{ch}_2^{\beta}(\CE))= e^{-\beta H}.\mathrm{ch}(\CE),
\]
where $\mathrm{H}$ denotes the hyperplane class. 
Define a torsion pair of subcategories of $\Coh(\BP^2)$ by
\begin{equation*}
\begin{aligned}\label{(2.1)}
\mathcal{T}_{\beta}:= \text{{\{$\CE \in \mathrm{Coh}(\mathbb{P}^{2})\colon \mu_{\beta}(\CG)>0$  for all $\CE  \twoheadrightarrow \CG $}\}},   \\
\mathcal{F}_{\beta}:= \text{{\{$\CE \in \mathrm{Coh}(\mathbb{P}^{2})\colon \mu_{\beta}(\CF) \leq 0$ for all $\CF  \hookrightarrow \CE$}\}},
\end{aligned}
\end{equation*}
where $\mu_{\beta}(\CE):=  \frac{\ch^{\beta}_{1}(\CE)}{\mathrm{ch}_{0}(\CE)}$  if $\mathrm{ch}_{0}(\CE) \neq 0$, and $\mu_{\beta}(\CE)=+ \infty$ otherwise. Note
that $\mu_{\beta}(\CE)=\mu(\CE)-\beta$, where $\mu(\CE)=\ch_1(\CE)/\ch_0(\CE)$ is
the standard slope.
These produce via tilting a
\textit{heart of a bounded t-structure} on $D^b(\BP^2)$  given as
\[
\mathrm{Coh}^{\beta}(\mathbb{P}^{2}):= \langle \mathcal{F}_{\beta}[1], \mathcal{T}_{\beta} \rangle.
\]
We recall this is the subcategory of $D^b(\BP^2)$ consisting of objects
$\CE$ with $\CH^i(\CE)=0$ for $i\not=-1,0$, $\CH^{-1}(\CE)\in \mathcal{F}_{\beta}$
and $\CH^0(\CE)\in \mathcal{T}_{\beta}$.

The central charge of the stability condition parameterized by
$\alpha,\beta$ is defined to be
\begin{equation}\label{eqn: Zab}
Z_{\alpha, \beta}:=-(\mathrm{ch}^{\beta}_{2})+(\alpha^{2}/2)\mathrm{ch}^{\beta}_{0}+i(\mathrm{ch}^{\beta}_{1}){\alpha}. 
\end{equation}
The relevant \textit{slope function} is defined as usual as
\begin{equation*}
\nu_{\alpha, \beta}:=-\frac{\mathrm{Re}(Z_{\alpha, \beta})}{\mathrm{Im}(Z_{\alpha, \beta})}=\frac{\mathrm{ch}^{\beta}_{2}-(\alpha^{2}/2)\mathrm{ch}^{\beta}_{0}}{\mathrm{ch}^{\beta}_{1}{\alpha}}
\end{equation*}
if $\mathrm{ch}^{\beta}_{1}(E)\neq0$, and $\nu_{\alpha, \beta}(E)=+\infty$ if $\mathrm{ch}^{\beta}_{1}(E)=0$.

Then the pair $(\mathrm{Coh}^{\beta}(\BP^2),Z_{\alpha,\beta})$ form a 
Bridgeland stability condition on $D^b(\BP^2)$, see \cite{BM-LocalProjPlane}.

\begin{theorem}[Bogomolov-Gieseker inequality {\cite[Corollary 7.3.2]{BMT14:stability_threefolds}}] \label{BG}
 Any $\nu_{\alpha, \beta}$-semistable object $\CE \in \mathrm{Coh}^{\beta}(\mathbb{P}^{2})$ satisfies
 $$2\mathrm{ch}_{0}(E)\mathrm{ch}_{2}(E) \leq \mathrm{ch}^{2}_{1}(E).$$
\end{theorem}

Following Bousseau \cite{Bousseau-scatteringP2-22}, we make a change of coordinates from $\alpha,\beta$ to $x,y$ with
\[
\alpha = \sqrt{x^2+2y}, \quad \beta = x.
\]
Now $(x,y)$ takes values in
\begin{align}\label{eqn: U}
 U:=\left\{(x,y)\,\Big|\, y>-{x^2\over 2}\right\}.   
\end{align}
For the remainder of the paper, we only work with stability conditions parameterized
by $\sigma=(x,y)\in U$.

From this, the central charge \eqref{eqn: Zab} becomes
\begin{align}
\label{eq:central charge}
\begin{split}
    Z^{\sigma}=Z^{(x,y)}:=& -\ch_2+\beta \ch_1-\frac{\beta^2}{2}\ch_0+\frac{\alpha^2}{2}\ch_0+i(\ch_1-\beta \ch_0)\alpha\\ 
=&ry+dx-e+i(d-rx){\sqrt{x^2+2y}}.
\end{split}
\end{align}
We then obtain a stability condition 
\[
(\CA^{\sigma},Z^{\sigma}) := (\mathrm{Coh}^{x}(\BP^2),Z^{\sigma}).
\]

\subsection{Bousseau's scattering diagram for stability conditions in
$\BP^2$}
\label{subsec:Bousseau scat}

We fix $M=\BZ^2$ and write $M_{\BR}=M\otimes_{\BZ} \BR$, with basis
$e_1,e_2$, and set $N=\Hom_{\BZ}(M,\BZ)$ with dual basis $e_1^*,e_2^*$.
We define a skew-symmetric bilinear form on $M$ given by
\[
\langle e_1,e_2\rangle = -3.
\]

For discussions of the stability scattering diagram, we first fix
a graded Lie algebra
\[
\mathfrak{g}= \bigoplus_{m\in M}\mathfrak{g}_m,
\]
so that $[\mathfrak{g}_m,\mathfrak{g}_{m'}]\subseteq \mathfrak{g}_{m+m'}$.
For $m\in M\setminus\{0\}$ primitive, $k$ a positive integer, define
\[
\fog_{m,k}:=\bigoplus_{j=1}^k \fog_{jm}
\]
with obvious surjections $\fog_{m,k+1}\rightarrow \fog_{m,k}$, and
set
\[
\hat \fog_m=\lim_{\stackrel{\longleftarrow}{k}} \fog_{m,k}.
\]
Write
\[
\pi_{m,k}:\hat\fog_m\rightarrow \fog_{m,k}
\]
for the projection.

\begin{definition}\label{Def: stabilityRay}
A \emph{ray or line} is a pair $(\fod,H_{\fod})$ where
\begin{enumerate}
\item There exists an $m_0\in M_{\BR}$ and $m_{\fod}\in M$ primitive
such that $\fod=m_0-\BR_{\ge 0}m_{\fod}$ in the case of a ray,
and $\fod=m_0-\BR m_{\fod}$ in the case of a line. We view $m_{\fod}$
as part of the data of a ray or line although it is not made explicit.
We refer to $-m_{\fod}$ as the \emph{direction vector} of the ray and
$m_{\fod}$ as the \emph{opposite vector} of the ray. 
\item 
$H_{\fod}\in \hat\fog_{m_{\fod}}$.
\end{enumerate}
The \emph{support} of a ray or line $(\fod,H_{\fod})$ is the set $\fod$.
A ray or line $(\fod,H_{\fod})$ is \emph{non-trivial} if $H_{\fod}\not=0$.
\end{definition}

Let $\BQ[t^{\BR_{\ge 0}}]$ denote a polynomial ring with coefficients in
$\BQ$ and arbitrary non-negative real exponents. 
For any positive integer $k$, let $\fom_k$ be the ideal generated by
$t^r$ for $r>k$, so that $\fom=\fom_0$ is a maximal ideal. Set
\[
R_k:=\BQ[t^{\BR_{\ge 0}}]/\fom_k,
\]
and
\[
\widehat R:=\lim_{\stackrel{\longleftarrow}{k}} R_k
\]

We write the completed tensor
products
\[
\mathfrak{g} \widehat\otimes \widehat R:=
\lim_{\stackrel{\longleftarrow}{k}} \mathfrak{g}\otimes R_k, \quad
\mathfrak{g} \widehat\otimes \mathfrak{m}:=
\lim_{\stackrel{\longleftarrow}{k}} 
\mathfrak{g}\otimes \mathfrak{m}/\mathfrak{m}^{k+1}.
\]

\begin{definition}
Given $H\in \mathfrak{g}\widehat \otimes\mathfrak{m}$, we obtain an automorphism
\[
\theta_H:=\exp([H,\cdot]):
\mathfrak{g}\widehat \otimes\widehat R \rightarrow 
\mathfrak{g}\widehat \otimes\widehat R
\]
given by
\[
\theta_H(f) := \sum_{k=0}^{\infty} 
{[H,[H,\ldots [H,f]\cdots]]\over k!},
\]
where the $k^{th}$ term involves $k$ Lie brackets. This infinite sum makes sense
in $\mathfrak{g}\widehat\otimes\widehat R$
as $H\in\fog\widehat\otimes\fom$.
\end{definition}

For the most part, we will work with $\mathfrak{g}=\BQ[M]$, with the obvious
$M$-grading and Lie bracket defined as follows.

\begin{definition}
The skew-form $\langle \cdot,\cdot\rangle$ induces a Lie
bracket
on $\BQ[M]$, with
\[
[z^{m_1},z^{m_2}]=\langle m_1,m_2\rangle z^{m_1+m_2}.
\]
\end{definition}

\begin{remark}
\label{rem:H auto}
We can unwind the construction of $\theta_H$ for $H=H(z^{m},t)\in 
\BQ[z^m]\widehat\otimes\fom$ as follows, assuming that $H\equiv 0 \mod z^m$.

First, we introduce some additional notation.
For $m\in M$, $\langle m,\cdot
\rangle\in N$. Explicitly, if $m=ae_1+be_2$, then $\langle m,\cdot\rangle
= 3b e_1^*-3a e_2^*$. We write this as $\check m\in N$.
Thus
\[
[z^m,z^{m'}] = \check m(m')z^{m+m'}.
\]

Writing $H=\sum_{k=1}^{\infty} h_k(t) z^{km}$, we then see that
\[
[H,z^{m'}] = \sum_{k=1}^{\infty}  h_k(t) k\check m(m')z^{km+m'}
=\check m(m')z^{m'}\left(\sum_{k=1}^{\infty} kh_k(t)z^{km}\right).
\]
Repeating, and noting that $[H,z^m]=0$, we obtain
\[
\theta_H(z^{m'}) = z^{m'} \exp\left(\check m(m') \sum_{k=1}^{\infty}
k h_k(t)z^{km}\right).
\]
If we set 
\begin{equation}
\label{eq:H to f}
f:= \exp\left(3\sum_{k=1}^{\infty} kh_k(t)z^{km}\right),
\end{equation}
then we can write 
\[
\theta_H=\bar\theta_f,
\]
where
\[
\bar\theta_f(z^{m'}) = z^{m'}f^{\check m(m')/3}.
\]
The reason for the factor of $3$ in the definition of $f$ and
dividing by $3$ above is that if $m$ is primitive, so is
$\check m/3$.
\end{remark}

\begin{definition}
\label{def:varphi sigma}
For $\sigma=(x,y)\in M_{\BR}$, define $\varphi_{\sigma}:M\rightarrow\BR$
by
\[
\varphi_{\sigma}(a,b)=2(-ax-b).
\]
\end{definition}

\begin{definition}\label{Def: stabilityScattering}
A \emph{(stability) scattering diagram} $\foD$ for a Lie algebra
$\mathfrak{g}$ is a set of rays or lines whose support is contained in
the closure $\overline{U}$ of $U$ (where $U$ is defined in \eqref{eqn: U}),
satisfying the following conditions:
\begin{enumerate}
\item For $(\fod,H_{\fod})\in \foD$ and $\sigma\in\fod\cap U$,
one has $\varphi_{\sigma}(m_{\fod})>0$.
\item
 For $(\fod,H_{\fod})\in
\foD$, let $\deg H_{\fod}$ be the smallest $k$ for which 
$\pi_{m_{\fod},k}(H_{\fod})\not=0$.
Then for every compact set $K$
in $U$ and every $r\in \BR_{\ge 0}$, the set of rays
$(\fod,H_{\fod})\in \foD$ such that there exists $\sigma\in\fod\cap K$
with $(\deg H_{\fod})\varphi_{\sigma}(m_{\fod})\le k$ is finite.
\end{enumerate}
\end{definition}

\begin{definition}
We define the \emph{support} of $\foD$ to be
\[
\mathrm{Supp}(\foD):=\bigcup_{\fod\in\foD} \fod
\] 
and define the \emph{singular locus} of $\foD$ to be 
\[
\mathrm{Sing}(\foD):=\bigcup_{\fod\in\foD}\partial \fod
\cup \bigcup_{\fod_1,\fod_2\in\foD \atop \dim \fod_1\cap\fod_2=0}
\fod_1\cap\fod_2,
\]
the union of endpoints of rays in $\foD$ and intersections of
non-collinear rays in $\foD$.
\end{definition}

\begin{definition}
\label{def:total sum}
For $x\in \mathrm{Supp}(\foD)\setminus\mathrm{Sing}(\foD)$, we define
\[
H_{\foD,x}:=\sum_{(\fod,H_{\fod})\in\foD, x\in\fod} H_{\fod}.
\]
We define $\mathrm{Supp}'(\foD)$ to be the closure of the set 
\[
\{x\in \mathrm{Supp}(\foD)\setminus\mathrm{Sing}(\foD)\,|\, H_{\foD,x}\not=0\}.
\]
\end{definition}

\begin{definition}
\label{def:equivalent}
We say two scattering diagrams $\foD_1$ and $\foD_2$ are 
\emph{equivalent} if 
$\mathrm{Supp}'(\foD_1)=\mathrm{Supp}'(\foD_2)$ and for any 
$x\in \mathrm{Supp}'(\foD_i)\setminus (\mathrm{Sing}(\foD_1)\cup
\mathrm{Sing}(\foD_2))$ we have
\[
H_{\foD_1,x}=H_{\foD_2,x}.
\]
\end{definition}

\begin{definition}
\label{def:local scatter}
Let $\foD$ be a stability scattering diagram,
and $\sigma\in\mathrm{Supp}(\foD)$.
For $(\fod,H_{\fod})\in\foD$ with $\sigma\in \fod$, write
$H_{\fod}$ as a formal sum $\sum_{k\ge 1} H_{\fod,k}$ where $H_{\fod,k}
\in \fog_{km_{\fod}}$. Then we write
\[
H_{\fod}^{\sigma} :=\sum_{k\ge 1} H_{\fod,k}t^{\varphi_{\sigma}(km_{\fod})}
\in \fog\widehat\otimes\fom.
\]
Here we use the condition $\varphi_{\sigma}(m_{\fod})>0$ of Definition~\ref{Def: stabilityScattering},(1).

We define the 
\emph{local scattering diagram at $\sigma$} to be
\[
\foD_{\sigma}:=
\{(T_{\sigma}(\fod),H^{\sigma}_{\fod})\,|\,\hbox{$(\fod,H_{\fod})\in \foD$ such that $\sigma\in\fod$}\}.
\]
Here, $T_{\sigma}(\fod)$ denotes the tangent cone to $\fod$ at $\sigma$.
This coincides with the tangent line to $\fod$ if $\sigma\in\Int(\fod)$,
and coincides with $\fod-\sigma$ if $\sigma$ is the endpoint of $\fod$.
\end{definition}

\begin{definition}
\label{def:path ordered}
Let $\foD$ be a stability scattering diagram and let $\sigma\in\mathrm{Sing}(\foD)\cap U$.
For a given positive integer $k$, we write
\[
\foD_{\sigma}^k:=\{(\fod,H_{\fod}^{\sigma})\,|\, (\fod,H_{\fod}^{\sigma})
\in\foD_{\sigma},
H_{\fod}^{\sigma}\not\equiv 0 \mod \fom_k\}.
\]
By the finiteness condition of Definition~\ref{Def: stabilityScattering}, 
$\foD^k_{\sigma}$ is a finite set.
\begin{enumerate}
\item
Let $\gamma:[0,1]\rightarrow M_{\BR}$ be a small parameterized circle 
centred at $\sigma$ 
with $\gamma(0)=\gamma(1)\in M_{\BR}
\setminus\mathrm{Supp}(\foD)$.
We define
the \emph{path-ordered product} $\theta_{\gamma,\foD^k_{\sigma}}:\fog\otimes R_k\rightarrow
\fog\otimes R_k$ as follows. Let $(\fod_1,H_1),
\ldots,(\fod_n,H_n)$ be the elements of $\foD^k_{\sigma}$ traversed by $\gamma$,
with $\gamma(t_i)\in\fod_i$,
in the order traversed. We define 
\[
\theta_{\gamma,\foD_{\sigma}^k}:=\theta_{H_n}^{\pm 1}\circ\cdots\theta_{H_1}^{\pm 1},
\]
where the sign $\pm 1$ is $+1$ provided 
$\check m(\gamma'(t_i))<0$.\footnote{This is the opposite sign convention
from \cite[Def.~1.6]{Bousseau-scatteringP2-22}, correcting an oversight
in that paper.}
\item
We define
$\theta_{\gamma,\foD_{\sigma}}:\mathfrak{g}\widehat\otimes \widehat R\rightarrow 
\mathfrak{g}
\widehat \otimes \widehat R$ as the limit of the 
$\theta_{\gamma,\foD_{\sigma}^k}$.
\item
We say $\foD$ is \emph{consistent at $\sigma$} if 
$\theta_{\gamma,\foD_{\sigma}}=\mathrm{id}$.
\item
For an open subset $V\subseteq M_{\BR}$, we say $\foD$ is \emph{consistent
inside $V$} if it is consistent at each $\sigma\in V\cap \mathrm{Sing}(\foD)$.
\end{enumerate}
\end{definition}

We can now define the stability scattering diagram for $\BP^2$ with respect
to $\fog=\BQ[M]$.
 This will be the scattering diagram
written as ${\mathfrak D}^{\BP^2}_{\mathrm{cl}^+}$ in 
\cite{Bousseau-scatteringP2-22}.

Identifying $M_{\BR}$ with the $(x,y)$-plane, we recall the open subset 
$U\subseteq M_{\BR}$ given in \eqref{eqn: U}.
For $n\in\BZ$, define two rays,
\begin{align*}
\fod^+_n := {} & \left((n,-n^2/2)-\BR_{\ge 0}(-1,n), 
-\li_2(-z^{(-1,n)})\right)
\\
\fod^-_n := {} & \left((n,-n^2/2)-\BR_{\ge 0}(1,-n),
-\li_2(-z^{(1,-n)})\right),
\end{align*}
where the dilogarithm is defined by
\[
\li_2(z)=\sum_{k=1}^{\infty} {z^k\over k^2}.
\]
Set 
\[
\foD^{\mathrm{stab}}_{\mathrm{in}}:=\{\fod_n^+,\fod_n^-\,|\, n\in \BZ\}.
\]
\cite[Prop.~1.9]{Bousseau-scatteringP2-22}, following
\cite[Thm.~6]{KS-06}, proves the following statement:

\begin{theorem} 
There exists a stability scattering diagram
$\foD^{\mathrm{stab}}$ consistent in $U$ and equivalent to 
$\foD^{\mathrm{stab}}_{\mathrm{in}}$ in small neighbourhoods of the endpoints
of the rays $\fod^{\pm}_n$. Further, this diagram is unique up to equivalence.
\end{theorem}

Note that the endpoints $(n,-n^2/2)$ of $\fod_n^{\pm}$ 
lie in the boundary of $\overline{U}$,
and thus there is no consistency for loops around these points, as they are
not contained in $U$.

One of the main theorems of \cite{Bousseau-scatteringP2-22} 
relates the functions attached
to rays of $\foD^{\mathrm{stab}}$ to Euler characteristics of 
moduli spaces of stable
objects in $D^b(\BP^2)$. Explicitly, let $\sigma\in U$, and let
$\gamma\in\Gamma$ as defined in \eqref{eq:Gamma def}. We may write
$\gamma$ either as $(r,d,\chi)$ or $(\ch_0,\ch_1,\ch_2)$.

\begin{definition}
\label{def:exp line}
For $\gamma=(r,d,\chi)\in\Gamma$ or $\CE$ an object of the derived category
$D^b(\BP^2)$ with invariants $\gamma$, we 
define:
\begin{enumerate}
\item the \emph{opposite vector} of $\gamma$ or $\CE$ to be
\begin{equation}
\label{eq:mgamma def}
m_{\gamma}=m_{\CE}:=(r,-d)\in M.
\end{equation}
\item 
We define the line in $M_{\BR}$
\begin{align}\label{eqn:LGamma}
\begin{split}
L_{\gamma} = L_{\CE} :=& \{(x,y)\in U\,|\,ry+dx+r+{3\over 2}d - \chi = 0\}\\
=&\{(x,y)\in U\,|\,\ch_0y+\ch_1x=\ch_2\},
\end{split}
\end{align}
using \eqref{eq:RR1}. Note that $m_{\gamma}$ is tangent to $L_{\gamma}$.
The line $L_{\gamma}$ can be viewed as
the set of stability conditions $\sigma\in U$ for which
the central charge $Z^{\sigma}_{\gamma}$ of $\gamma$ is purely imaginary.

An object $\CE$ with Chern character $\gamma=(\ch_0,\ch_1,\ch_2)$ is called an \textit{object corresponding} to the line $L_{\gamma}=L_{\CE}$.
\item
We denote by $\CM^{\sigma}_{\gamma}$ the moduli space of $S$-equivalence
classes of $\sigma$-semistable objects of class $\gamma$ in 
$\BP^2$.
\end{enumerate}
\end{definition}
 
The following is \cite[Thm.~5.10]{Bousseau-scatteringP2-22}:

\begin{theorem}
\label{thm:bousseau main}
For a point $\sigma\in \mathrm{Supp}(\foD^{\mathrm{stab}})
\setminus\mathrm{Sing}(\foD^{\mathrm{stab}})$,
\[
H_{\foD^{\mathrm{stab}},\sigma} = \sum_{\gamma: \sigma\in L_{\gamma}} (-1)^{(\gamma,\gamma)-1}
\left(\sum_{\gamma'\in\Gamma_{\gamma}\atop \gamma=\ell\gamma'}
{1\over \ell^2} Ie^-_{\gamma',\sigma}\right) z^{m_{\gamma}}.
\]
Here
\[
\Gamma_{\gamma} = \{\gamma'\in\Gamma\,|\,\hbox{$\gamma=\ell\gamma'$
for some $\ell\in\BZ_{\ge 1}$}\}
\]
and $Ie^-_{\gamma',\sigma}$ is a signed Euler characteristic of the
intersection cohomology of $\CM^{\sigma}_{\gamma'}$, see 
\cite[\S2.4]{Bousseau-scatteringP2-22}.
\end{theorem}

For us, the precise definition of $Ie^-_{\gamma',\sigma}$
will not be important, other than the observation that if this number is
non-zero, then $\CM^{\sigma}_{\gamma'}$ is non-empty.

\begin{lemma}\label{Lem: invarianceDStab}
$\foD^{\mathrm{stab}}$ is invariant under the affine linear transformation
\[
T:M_{\BR}\rightarrow M_{\BR},\quad 
T
=\begin{pmatrix} 1 & 0\\-1&1\end{pmatrix}
\begin{pmatrix}x\\ y\end{pmatrix}
+\begin{pmatrix} 1 \\ -1/2\end{pmatrix}.
\]
We define the action of $T$ on a ray via $T(\fod,H(z^m))=(T(\fod),H(z^{T(m)}))$.
\end{lemma}

\begin{proof}
To see this invariance, it is sufficient to note that $T(\fod^\pm_n)=
\fod^\pm_{n+1}$ and then apply uniqueness of $\foD^{\mathrm{stab}}$ 
as obtained from $\foD^{\mathrm{stab}}_{\mathrm{in}}$.
\end{proof}

\begin{remark}
This affine linear transformation is induced by the action of the
autoequivalence $T:D^b(\BP^2)\rightarrow D^b(\BP^2)$ given via
$\otimes\CO_{\BP}(1)$. According to \cite[Lem.~2.24]{Bousseau-scatteringP2-22},
this autoequivalence acts on $\mathrm{Stab}(\BP^2)$ via 
$(Z,\CA)\mapsto (Z\circ T^{-1},T(\CA))$, and this action preserves
the subset of $\mathrm{Stab}(\BP^2)$ parameterized by $U\subseteq M_{\BR}$.
This action takes $\sigma\in U$ to $T(\sigma)$ with $T$ as in 
Lemma~\ref{Lem: invarianceDStab}.
\end{remark}

\begin{remark}\label{Rem: invarianceDStab}
For future reference, we note that
\[
T^d\begin{pmatrix}x\\ y\end{pmatrix}
=\begin{pmatrix} 1 & 0\\-d&1\end{pmatrix}
\begin{pmatrix}x\\ y\end{pmatrix}
+\begin{pmatrix} d \\ -d^2/2\end{pmatrix}.
\]
\end{remark}

\begin{remark}
We note that $\foD^{\mathrm{stab}}$ is also obviously invariant under
the transformation $(x,y)\mapsto (-x,y)$.
\end{remark}

\begin{remark}
\label{rem:Duv}
In fact, in \cite{Bousseau-scatteringP2-22}, Bousseau constructs a
scattering diagram which he writes as $\foD_{u,v}^{\BP^2}$. We will
not give the definition here, as it is somewhat more complicated than
$\foD^{\mathrm{stab}}$. Crucially, however, rays in this scattering
diagram encode more information about each moduli space $\CM^\sigma_{\gamma}$,
namely the Hodge polynomial associated to the intersection cohomology
of $\CM^{\sigma}_{\gamma}$. Since this Hodge polynomial is non-zero if 
$\CM^\sigma_{\gamma}$ is non-empty,

the scattering diagram sees all non-empty moduli spaces. It
is possible that in passing to $\foD^{\mathrm{stab}}$ we lose
some information as some moduli spaces might have Euler characteristic
zero. Bousseau's main result is in fact about $\foD_{u,v}^{\BP^2}$, and
the result for $\foD^{\mathrm{stab}}$ is then obtained after taking
Euler characteristics. In particular, $\foD_{u,v}^{\BP^2}$ is obtained
from an intial scattering diagram $\foD_{u,v}^{\mathrm{in}}$ in the
same way that $\foD^{\mathrm{stab}}$ is. 

Most of the results
of this paper stated for $\foD^{\mathrm{stab}}$ also hold for
$\foD_{u,v}^{\BP^2}$. Indeed, results about the non-existence of certain rays
are proved by showing the corresponding moduli spaces are empty, and hence these
rays don't exist in either $\foD^{\BP^2}_{u,v}$ or $\foD^{\mathrm{stab}}$. Results
about the existence of certain rays, however, are proven in 
$\foD^{\mathrm{stab}}$ using results about scattering, and this implies
existence of rays in $\foD^{\BP^2}_{u,v}$. 

On the other hand, $\foD^{\mathrm{stab}}$ is easier to work with
as the type of scattering which occurs there has been much studied (and
in particular the results of \cite{Graefnitz-Luo2023} can be easily
applied) and is algorithmically simple to implement.
\end{remark}

\subsection{Wall-crossing and destablizing subobjects and quotient
objects}
\label{subsec:wall crossing}

The motivating question of this paper is 
to understand the behaviour of the moduli spaces $\CM^{\sigma}_{\gamma}$
as $\sigma$ varies. 
Here we explain how Bousseau's work gives us insight into this question.
Consider a Chern character $\gamma\in\Gamma$. 
Suppose that this moduli space is non-empty for
some $\sigma=(x_0,y_0)\in U$. Then it follows from the Bogomolov-Gieseker
inequality (Theorem~\ref{BG}) that the line $L_{\gamma}$ intersects
the boundary of $\overline{U}$, namely the parabola $y=-x^2/2$. 
Indeed, replacing $y$ with $-x^2/2$ in the equation for $L_{\gamma}$,
we obtain $-\ch_0 x^2/2+\ch_1 x-\ch_2=0$, and this quadratic equation
has discriminant $\ch_1^2-2\ch_0\ch_2$, which is non-negative by the
Bogomolov-Gieseker inequality. In particular, we have solutions
\[
x= \frac{\ch_1}{\ch_0} \pm \sqrt{\left(\frac{\ch_1}{\ch_0}\right)^2-2
\frac{\ch_2}{\ch_0}}.
\]

As a consequence, we can write $L_{\gamma}\cap \overline U=L^+_{\gamma}
\cup L^-_{\gamma}$, where $L^{\pm}_{\gamma}$ is a ray with endpoint
on $\partial\overline U$. From the above formula for the $x$-coordinate at
the points of $L_{\gamma}\cap\partial\overline{U}$
and the formula \eqref{eq:central charge} for $Z^{\sigma}$, we can
choose $L^+_{\gamma}$ and $L^-_{\gamma}$ so that
for $\sigma\in \Int(L^+_{\gamma})$, $Z^{\sigma}(\gamma)
\in i\BR_{>0}$ and for $\sigma\in\Int(L^-_{\gamma})$, $Z^{\sigma}(\gamma)
\in i\BR_{<0}$. Note that replacing $\gamma$ by $-\gamma$, which
corresponds to a shift $[1]$ in $D^b(\BP^2)$, interchanges
these two rays. 

Assuming now that $\gamma$ satisfies the Bogomolov-Gieseker inequality,
we can write $L^+_{\gamma}=p+\BR_{\ge 0} m$
for some $m\in M$. For $\sigma=p+tm$ for $t$ small, the moduli space
$\CM^{\sigma}_{\gamma}$ will be empty (assuming that $L^+_{\gamma}$
does not coincide with an initial ray), as follows from 
\cite[Lemmas~4.8,4.10]{Bousseau-scatteringP2-22}.
As $t$ increases, the moduli space
$\CM^{\sigma}_{\gamma}$ may become non-empty. We call
the point on $L^+_{\gamma}$ where this first occurs \emph{the first actual wall}
for $\gamma$. \emph{Further (actual) walls} are points of $L^+_{\gamma}$
at which $\CM^{\sigma}_{\gamma}$
changes again. This leads to the following definition.

\begin{definition}\label{Def: FirstGeneratingPoint}
Let $\gamma$ be a Chern character, and let $\sigma_0=L^+_{\gamma}\cap
\partial\overline{U}$.
We define the \emph{first
generating point of $\gamma$} to be $\sigma_0-t_0 m_{\gamma}$, where
\[
t_0=\inf\{t \,|\, \hbox{$\sigma_0-t m_{\gamma}\in \fod$ for
some $(\fod,H_{\fod})\in \foD^{\mathrm{stab}}$ with $\fod\subseteq 
L^+_{\gamma}$}\}.
\]
Note this only depends on $\gamma$ up to positive multiples.
\end{definition}

Note that this is not the usual picture for walls in stability moduli space.
Rather, typically walls are viewed as codimension one sets in the space of
stability conditions. \emph{Potential walls} arise via a choice of
(non-proportional) Chern characters $\gamma_1,\gamma_2$, with the
potential wall being
\[
W_{\gamma_1,\gamma_2}:=
\{\sigma \in U\,|\,\Arg(Z^{\sigma}(\gamma_1))=\Arg(Z^{\sigma}(\gamma_2))\}.
\]
This is a conic section, see e.g., the proof of 
\cite[Lem.~3.3]{Bousseau-scatteringP2-22}.
Let $\sigma_0$ be the point of intersection between
$L_{\gamma_1}$ and $L_{\gamma_2}$. Then clearly $\sigma_0\in
 W_{\gamma_1,\gamma_2}$. \cite[Lem.~3.3]{Bousseau-scatteringP2-22}
shows that the tangent line $\ell_{\gamma_1,\gamma_2}$ 
to $W_{\gamma_1,\gamma_2}$ at $\sigma_0$ is given by
the equation $\varphi_{\sigma_0}=0$ (see Definition~\ref{def:varphi sigma}), 
after identifying the tangent
space to $U$ at $\sigma_0$ with $M_{\BR}$. See Figure 
\ref{Fig:Lgamma}.

\begin{figure}[h]
    \centering
   \includegraphics[width=11.4cm]{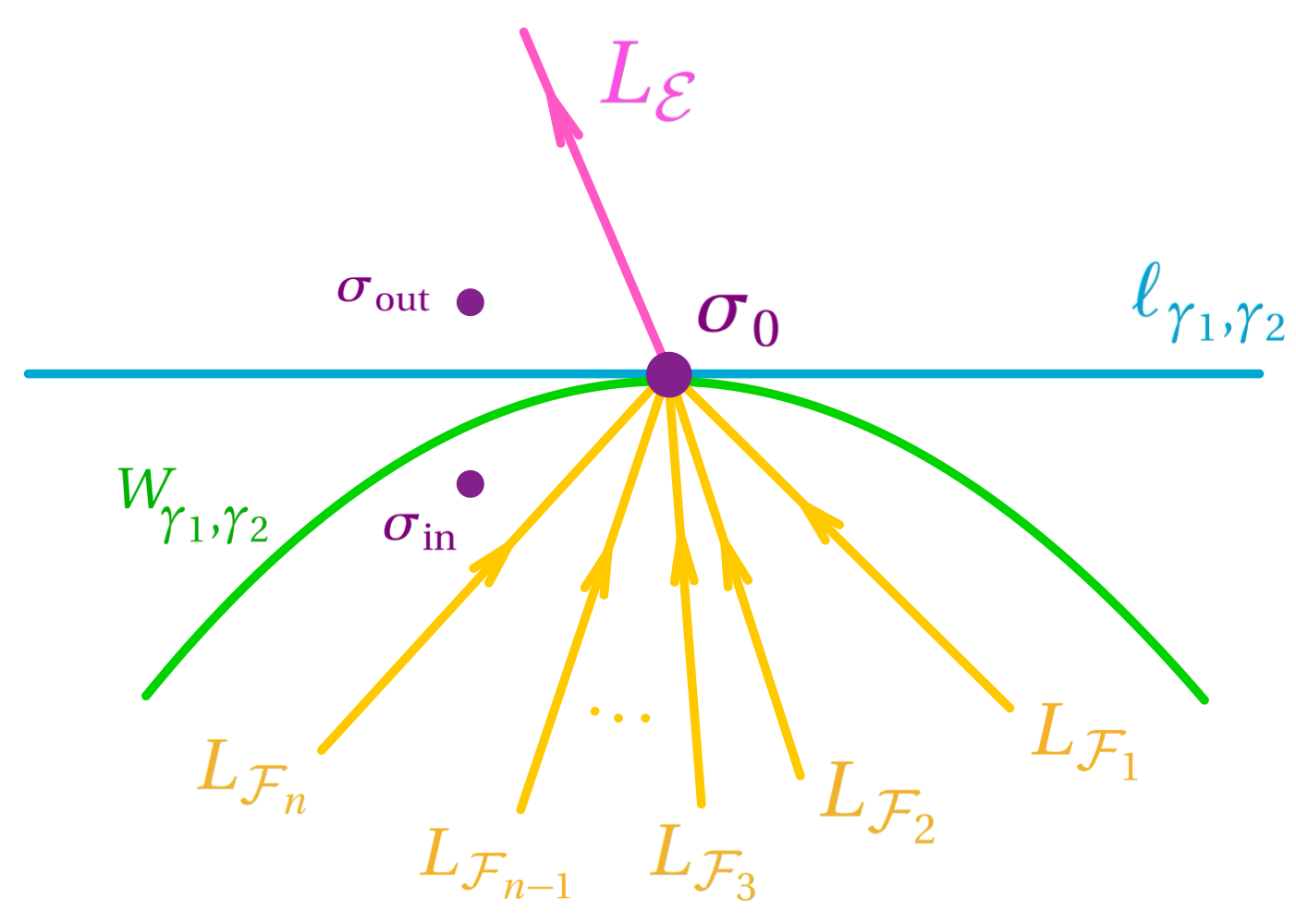}
   \caption{}
    \label{Fig:Lgamma}
\end{figure}

As in \cite{Bousseau-scatteringP2-22}, 
we choose general stability
conditions $\sigma_{\mathrm{in}}$ and $\sigma_{\mathrm{out}}$ sufficiently
close to $\sigma_0$ and below
and above $\ell_{\gamma_1,\gamma_2}$ respectively. This makes sense since 
$\ell_{\gamma_1,\gamma_2}$ is never vertical. Further,
we may choose these two stability conditions to have the same $x$-coordinate,
so that the hearts $\CA^{\sigma_{\mathrm{in}}}$ and 
$\CA^{\sigma_{\mathrm{out}}}$ of these stability 
conditions coincide, as the hearts only depend on the $x$-coordinate.
Further, as explained in \cite[\S3.2]{Bousseau-scatteringP2-22}, for
any ray $\fod$ coming into $\sigma_0$, we have 
\begin{equation}
\label{eq:varphi increasing}
\varphi_{\sigma_0}(m_{\fod})>0.
\end{equation}
(Note this is required by Definition~\ref{Def: stabilityScattering},(1).)

Now fix $\gamma\in\Gamma$ satisfying the Bogomolov-Gieseker inequality,
and let $\sigma_0$ be the intersection of $L^+_{\gamma}$
with a potential wall $W_{\gamma_1,\gamma_2}$ as above.
The key argument in \cite[\S3.5]{Bousseau-scatteringP2-22} makes use of the
following. Let $\CE$ be an object in the moduli space
$\CM_{\gamma}^{\sigma_{\mathrm{out}}}$. If $\CE$ is not stable with
respect to $\sigma_{\mathrm{in}}$, then $\sigma_0$ is an actual wall
in our language. In this case, 
$\CE$ has a Harder-Narasimhan filtration with respect to the stability
condition $\sigma_{\mathrm{in}}$, with graded pieces of the
filtration being $\CF_1,\ldots,\CF_n$, and with $\Arg(Z^{\sigma_{\mathrm{in}}}
(\CF_i))$ decreasing as $i$ increases. 
In particular, $\CF_1$ is a subobject
of $\CE$ and $\CF_n$ is a quotient object. 
The lines $L_{\CF_i}$ all pass through $\sigma_0$, as the objects
$\CF_i$ will all have the same slope as each other and as $\CE$ with
respect to the stability condition $\sigma_0$. Further, as the argument
of the central charge is decreasing,
$L_{\CF_1}$ lies to the
left of $L_{\CE}$ (relative to the direction vector $-m_{\CE}$) after 
passing $\sigma_0$
this intersection point, and $L_{\CF_n}$ lies to the right, see Figure \ref{Fig:Lgamma}. 
Thus, we will refer to the subobject $\CF_1$ as the 
\textit{left generator} and the quotient object $\CF_n$ as the 
\textit{right generator}.
We also refer to these as \textit{generators}. 
We call the point $\sigma_0=L_{\CF_1}\cap L_{\CF_n}$, the \textit{first generating point}.

We note that if $\gamma_i$ is the Chern character of $\CF_i$, then
$\gamma=\sum_i\gamma_i$ and $m_{\gamma}=\sum_i m_{\gamma_i}$. Further,
$m_{\gamma}$ is a positive rational linear combination of $m_{\gamma_1}$
and $m_{\gamma_n}$. 

This gives us the following procedure to visualize how various sheaves
are generated by the scattering process.
For a semistable object $\CF\in\CA^{\sigma}$ with $\sigma\in L_{\CF}$, 
define the line segment $E(\CF,\sigma)\subseteq L_{\CF}$ to be
the line segment with endpoints $\sigma$ and $\sigma+\alpha m_{\CF}$,
where 
\[
\alpha:=\sup \{\alpha'\,|\,\hbox{$\CF$ is $\sigma+\beta
m_{\CF}$-semistable for all $0\le \beta\le \alpha'$}\}.
\]

Now given $\CE\in \CA^{\sigma_0}$ a $\sigma_0$-semistable object,
let $E_0=E(\CE,\sigma_0)$,
and let $\sigma_1$ be the other endpoint of $E_0$. 
Then $\CE$ is unstable for a stability condition $\sigma_1+\epsilon m_{\CE}$
for small $\epsilon>0$,
and let $\CQ_1$ be the right generator in the sense of 
\S\ref{subsec:wall crossing}. We proceed similarly, taking $E_1=E(\CQ_1,
\sigma_1)$. Let $\sigma_2$ be the other endpoint of $E_1$. We
continue in this fashion, always using the right generator.
Because each of these line segments are contained in rays of $\foD^{\BP^2}_{u,v}$
as discussed in Remark \ref{rem:Duv}, as the moduli spaces of
$\CQ_i$ are non-empty for stability conditions in the interior of $E_i$,
we see that this process must terminate
when we reach some $\sigma_{n}= (d,-d^2/2)$ for some $d$. 
In particular, this gives a sequence of successive quotient objects.

Similarly, we can use the left generators to get a different sequence
of edges and this time a sequence of subobjects $\CS_i$. 
See Figure \ref{Fig: Convexity}.

\begin{figure}[h]
 \subcaptionbox*{}[0.59\linewidth]{%
    \includegraphics[width=\linewidth]{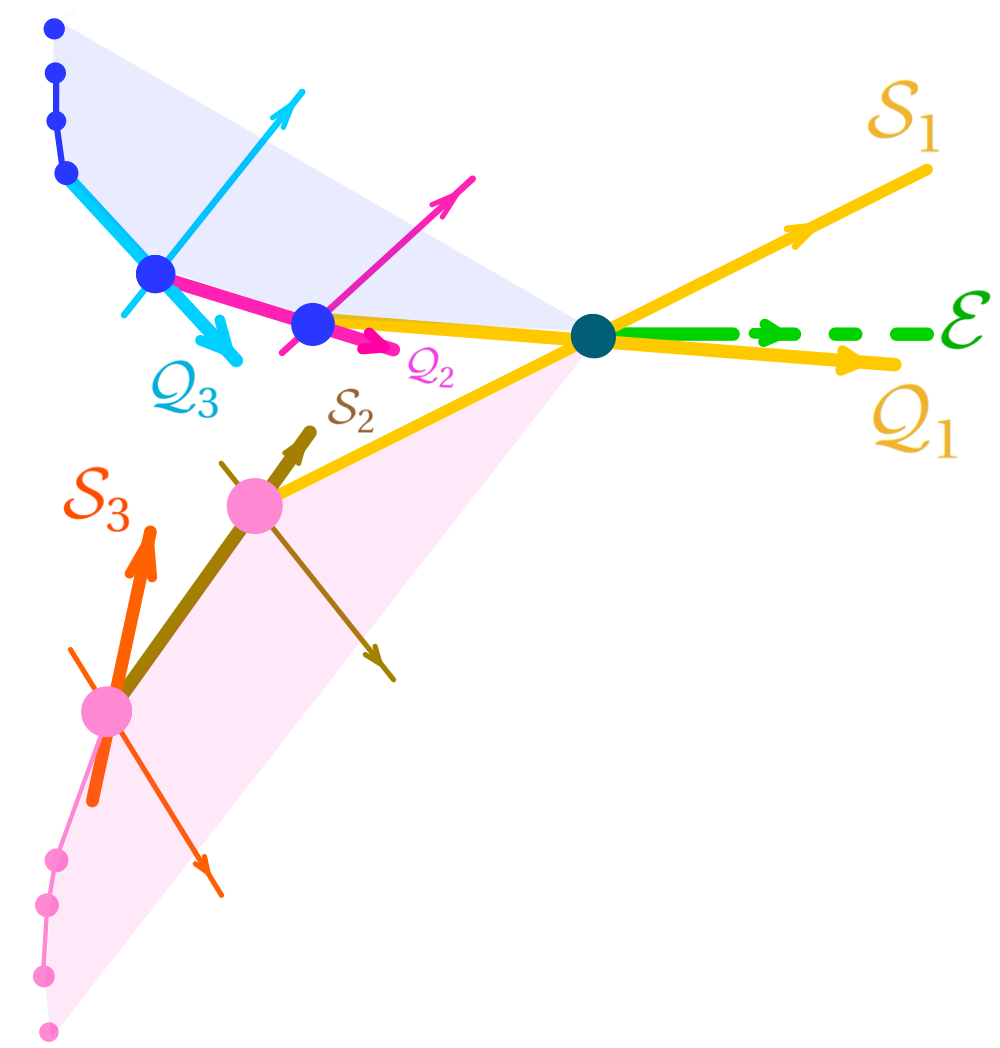}%
  }%

  \caption{The repeated subobjects and the repeated quotients. 
}
  \label{Fig: Convexity} 
\end{figure}

\section{Structure of the scattering diagram: the discrete region} 
\label{sec:discrete}
In this section we discuss the discrete region of $\foD^{\mathrm{stab}}$,
defining the set of discrete rays $\foD_{\mathrm{discrete}}\subseteq
\foD^{\mathrm{stab}}$ and showing that this gives a complete 
description of $\foD^{\mathrm{stab}}$ in a region $R_{\Delta}$ which
decomposes into an (infinite) union of triangular chambers of
$\foD^{\mathrm{stab}}$. The results in this section partially overlap with
results of \cite{Prince-20}, but we need
the refined description here, and our proofs are completely different.
We give dual descriptions of the triangles, first by describing the rays
which bound them via an inductive process, and second by describing their
vertices in terms of exceptional bundles.

\subsection{The discrete rays of $\text{\:}\foD^{\mathrm{stab}}$ and the triangles}

\medskip

\begin{construction}
\label{susection: construction}
Let $\mathcal{T}$ be an infinite binary tree with one root $w_0$ and 
every vertex having two children. Label each edge of the tree as being
either \emph{hybrid} or \emph{outgoing}. This is done so that each edge
connected to $w_0$ is hybrid, but the two edges attaching any other vertex $w$
to its children consist of one hybrid and one outgoing edge. See Figure \ref{Fig:Tree}.

\begin{figure}[h]
    \centering
   \includegraphics[width=11.9cm]{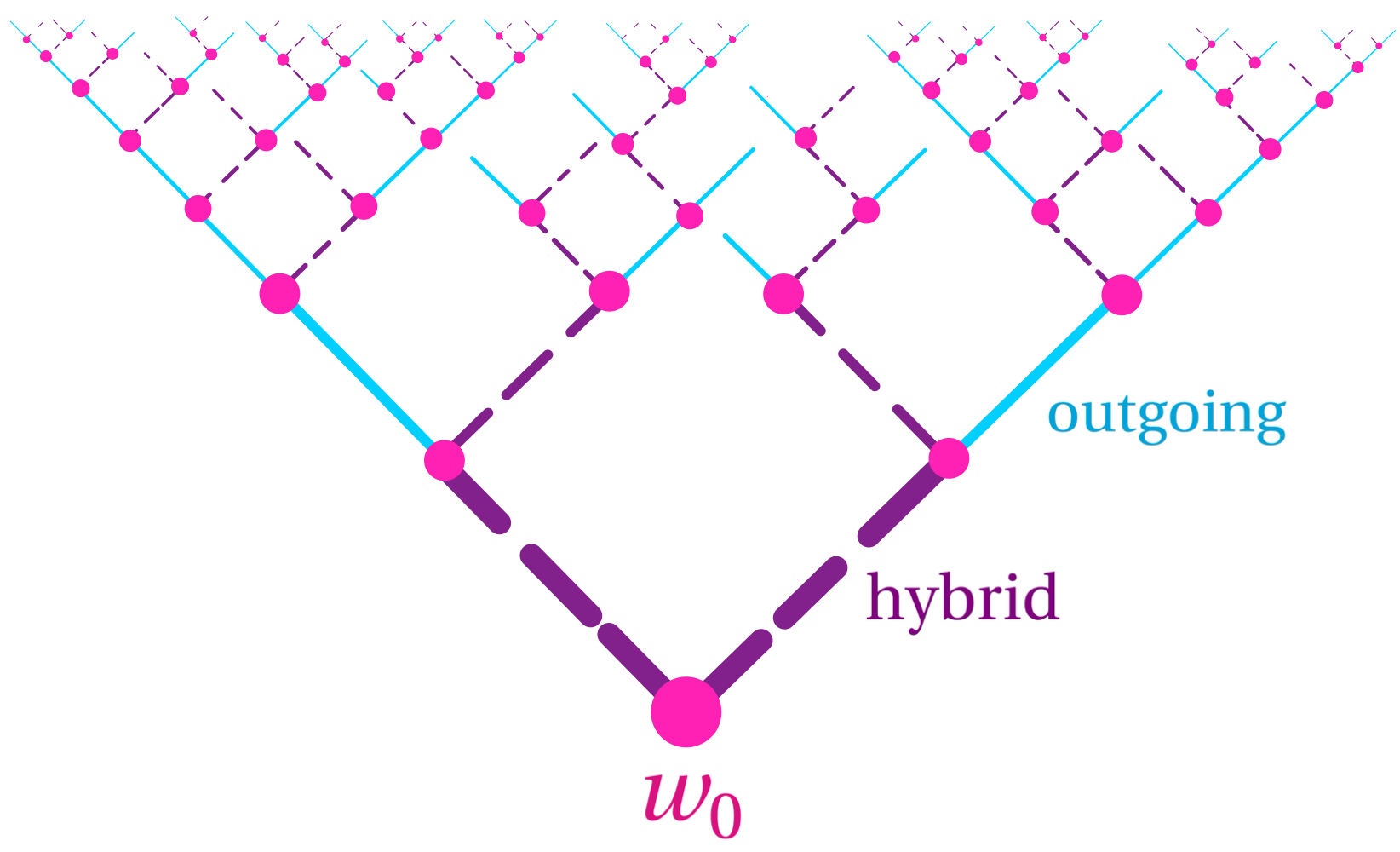}
   \caption{The infinite binary tree $\mathcal{T}$. The hybrid edges are depicted by dashed lines.}
    \label{Fig:Tree}
\end{figure}

We will associate a triangle $\Delta_w\subseteq M_{\BR}$ 
for each vertex $w$ of $\mathcal{T}$. These triangles 
will be constructed inductively, starting with $\Delta_{w_0}$. The tree
$\mathcal{T}$ can then be thought of as being dual to a two-dimensional
simplicial complex composed
of these triangles. The vertices of $\mathcal{T}$ correspond to
triangles, while edges of $\mathcal{T}$ with endpoints $w_1,w_2$ correspond
to an edge $\Delta_{w_1}\cap \Delta_{w_2}$ of both triangles.
In addition, for each vertex $w$ which is not the root, we will associate
a ray $\fod_w$. Once we do this, we may define
\begin{align}
\label{eq:diagram discrete}
\begin{split}
\foD^0_{\mathrm{discrete}}:= {} &
\{\fod_w\,|\, \hbox{$w\not=w_0$ a vertex of $\mathcal{T}$}\},\\
\foD_{\mathrm{discrete}}:= {} & \bigcup_{k\in\BZ} 
T^k(\foD^0_{\mathrm{discrete}}),
\end{split}
\end{align}
where $T$ is as defined in Lemma~\ref{Lem: invarianceDStab}.
The boundary of $\Delta_w$ consists of three intervals each contained in a 
ray of 
$\{\fod_{-1}^+,\fod_1^-\}\cup \foD^0_{\mathrm{discrete}}$, except for $\Delta_{w_0}$, 
which is slightly special. 
Explicitly, $\Delta_{w_0}$ is the triangle 
\begin{equation}
\label{eq:Tw0}
\Delta_{w_0}:= \mathrm{Conv}((-1/2,0), (1/2,0), (0,1/2)).
\end{equation}

   \begin{figure}[h]
 \subcaptionbox*{(1)}[.48\linewidth]{%
    \includegraphics[width=\linewidth]{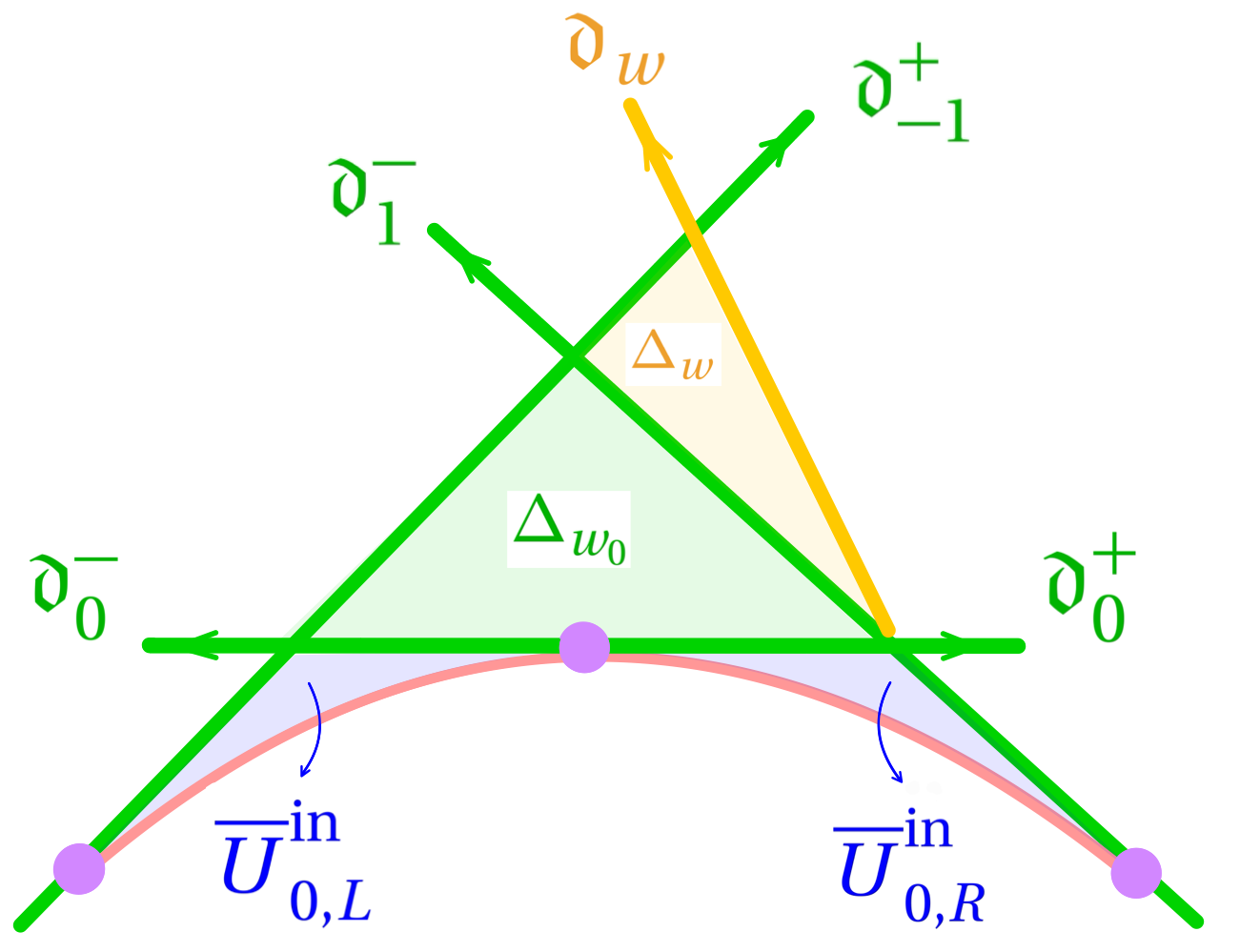}%
  }%
  \hskip5.0ex
  \subcaptionbox*{(2)}[.44\linewidth]{%
    \includegraphics[width=\linewidth]{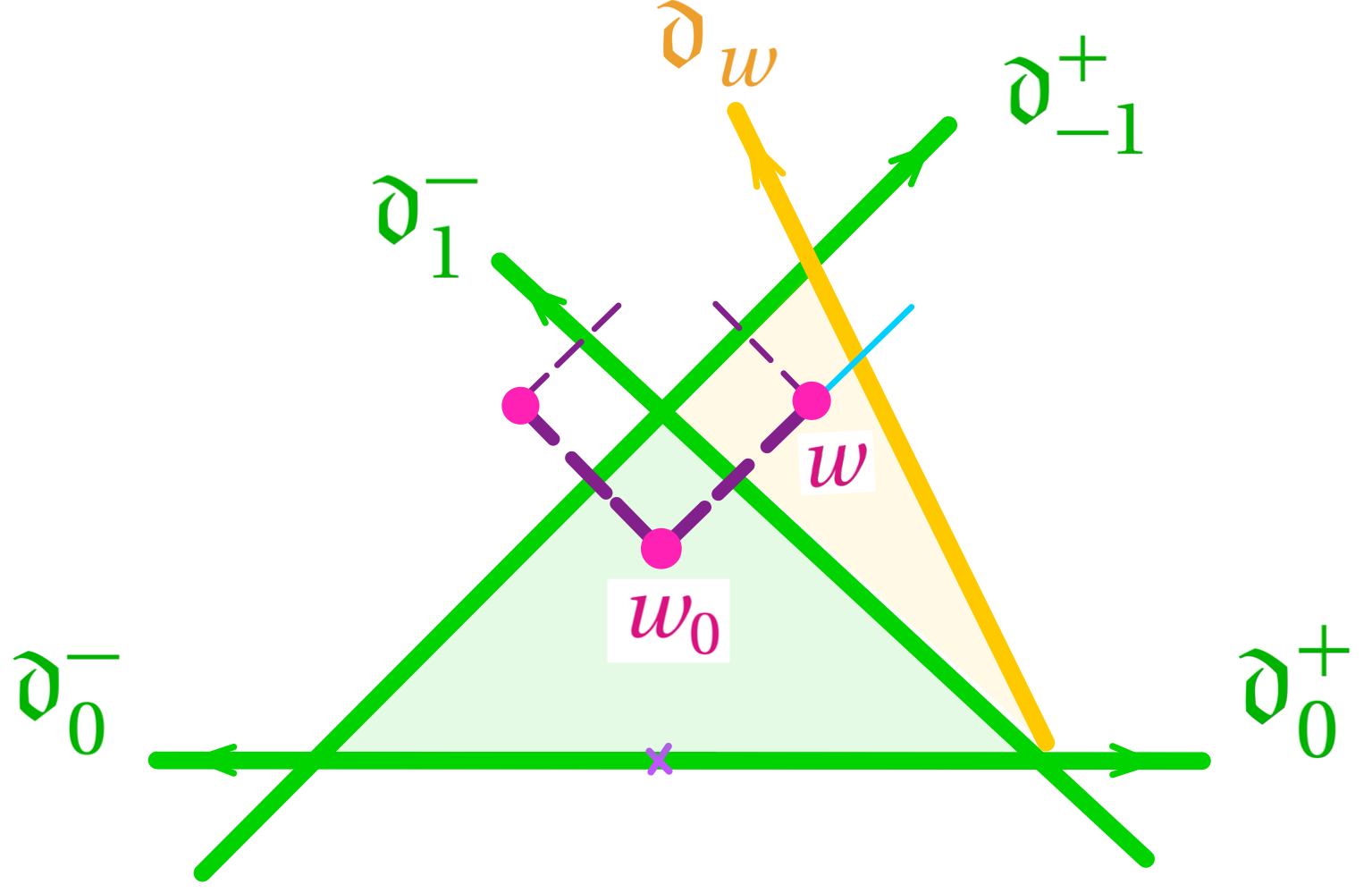}%
  }
  \caption{
}
  \label{Fig: InitialTriangle}
\end{figure}

Note that the two edges of this triangle with endpoint $(0,1/2)$ are
contained in $\fod_{-1}^+$ and $\fod_{1}^-$, while the lower
edge with endpoints $(\pm 1/2,0)$ is contained in $\fod_0^+\cup\fod_0^-$. See Figure \ref{Fig: InitialTriangle}.

We orient each ray $\fod$ in the direction
pointing out from the endpoint of the ray, i.e., in the direction of the
direction vector $-m_{\fod}$.
This induces orientations on edges of triangles contained in these rays.
We classify vertices of a given triangle based on the orientations of the
edge of the triangle. We say a vertex is \emph{outgoing} if its adjacent
edges are oriented away from the vertex. We say a vertex is \emph{incoming}
if the adjacent edges are oriented towards the vertex. Otherwise, the
vertex is said to be \emph{hybrid}. 

We note that $\Delta_{w_0}$ has two hybrid vertices and one incoming vertex,
while it will be clear by the inductive construction that all other triangles 
will have one vertex of each type.

For a vertex $v$ of $\Delta_w$, let $m_1,m_2$ 
be primitive tangent vectors to the edges adjacent to $v$. 
We define the 
\emph{determinant} of the vertex $v$ to be 
\begin{equation}
\label{eq:det def}
D_v:=
|m_1\wedge m_2|,
\end{equation}
i.e., the absolute value of the determinant of the $2\times 2$ matrix with
rows given by $m_1,m_2$. A priori, $D_v$ depends not just on $v$ but the
triangle $\Delta_w$ it is a vertex of. In Proposition \ref{Lem: degree=rank},
we will show that in fact $D_v$ only depends on the $x$-coordinate
of $v\in M_{\BR}$. If $v$ is clear, for simplicity, we denote the determinant by $D$.

Given a triangle $\Delta_w$, let $v$ be the unique incoming vertex, and let
$E_1,E_2$ be the adjacent edges. We call $E_i$ \emph{hybrid} if its other
vertex is hybrid, and \emph{outgoing} if its other vertex is outgoing.
We may then equate $E_1,E_2$ with the edges of
$\mathcal{T}$ attaching $w$ to its two children, $w_1,w_2$, in a way
which respects the labeling of hybrid and outgoing.
 
Denote by $v_i$ the other vertex of $E_i$, and by $\fod_i$ the ray
containing $E_i$. Let $m_i$ be the primitive tangent vector to
$E_i$, oriented in the opposite direction of $E_i$ (hence $m_i$
agrees with the opposite vector $m_{\fod_i}$).
Let $E$ be the remaining edge of $\Delta_w$.
To obtain the triangle $\Delta_{w_j}$, let $j'$ be such that $\{j,j'\}=\{1,2\}$.
Let $m$ be a primitive tangent vector to $E$ pointing towards $v_j$.
Define
\begin{equation}
\label{eq:dwj def}
\fod_{w_j}:=\left((v_j-\BR_{\ge 0} (3D_{v_j} m_j - m), 
-\li_2(-z^{ 3D_{v_j}m_j-m})\right).
\end{equation}

Let $v'_j$ be the intersection of the ray $\fod_{j'}$ and the
ray $\fod_{w_j}$. We then define
\[
\Delta_{w_j}:=\mathrm{Conv}(E_j\cup \{v'_j\}),
\]
where $\mathrm{Conv}$ denotes convex hull.
See Figure \ref{Fig: Tw}.

  \begin{figure}[h]
    \centering
    \includegraphics[width=9.7cm]{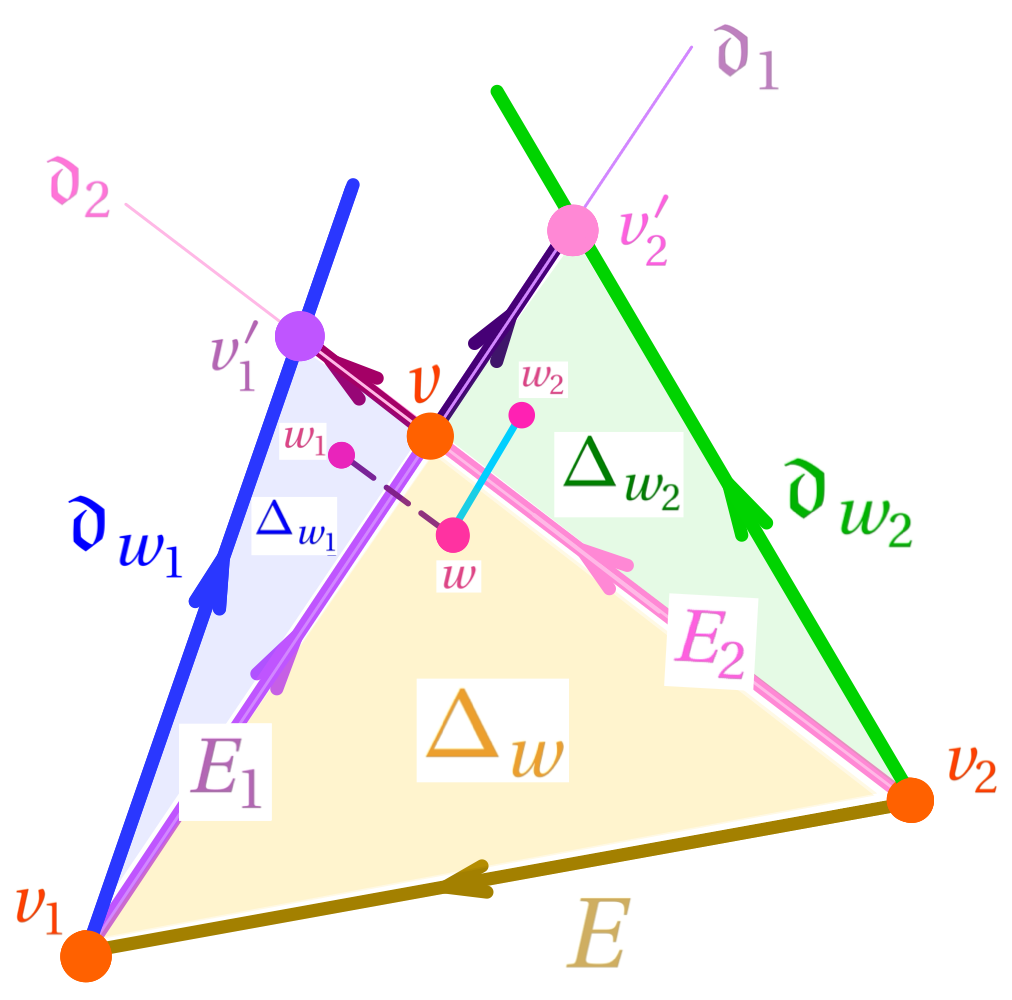}
    \caption{$\Delta_w$ along with $\Delta_{w_1}$ and $\Delta_{w_2}$. Note that $E_1$ is hybrid and $E_2$ is outgoing. Part of the tree corresponding to this diagram is also pictured.}
    \label{Fig: Tw}
\end{figure}

Note that the vertices of $\Delta_{w_j}$ are $v, v_j$ and $v'_j$, and the
edge joining $v$ and $v'_j$ is contained in the ray $\fod_{j'}$.
Further, $v'_j$ is now (in $\Delta_{w_j}$) an incoming vertex and $v$ is a hybrid vertex.
\end{construction}

\begin{definition}[Mutation]\label{Def: mutation}
We call the construction of $\Delta_{w_j}$ from $\Delta_w$ the \emph{mutation of
$\Delta_w$ at the vertex $v_{j'}$} or the \emph{mutation of $\Delta_w$ along the ray $\fod_j$}, so that $\fod_j$ remains an edge during the mutation process. Also, sometimes we call $\fod_{w_j}$ the \emph{mutation of $\text{\:}\fod_j$ at $v_j$ or along $\text{\:}\fod_{j'}$}. 
\end{definition}

\bigskip

\subsection{Moduli emptiness of the triangles}
\label{subsec:moduli emptiness}
In this subsection, we will connect the triangles constructed in the
previous subsection to strong exceptional triples. This will allow us, in particular,
to give a proof of Theorem \ref{Thm: MainScattering}(i), as well as give a 
powerful additional tool for understanding $\foD^{\mathrm{stab}}$.

We begin by recalling basics about exceptional bundles on $\BP^2$,
drawing from \cite{Drezet-LePotier}, \cite{Goro-Rud}, 
\cite{Rudakov89}.

\begin{definition}
\begin{enumerate}
\item An \emph{exceptional bundle} on $\BP^2$ is a bundle $\CE$ such that
$\Hom(\CE,\CE)=\BC$ and $\Ext^1(\CE,\CE)=\Ext^2(\CE,\CE)=0$.
\item An ordered tuple of exceptional bundles
$(\CE_0,\ldots,\CE_n)$ is a \emph{exceptional collection} 
if we have $\Ext^{\bullet}(\CE_i,\CE_j)=0$ for $i>j$.
\item An exceptional collection is a \emph{strong exceptional collection}
if $\Ext^k(\CE_i,\CE_j)=0$ for $k>0$.
\end{enumerate}
\end{definition}

Given a strong exceptional pair $(\CE,\CF)$, 
one defines the left and right mutations via the exact sequences
\begin{align*}
0\rightarrow & \CL_{\CE} \CF \rightarrow \Hom(\CE,\CF)\otimes\CE \rightarrow \CF
\rightarrow 0,\\
0\rightarrow & \CE\rightarrow \Hom(\CE,\CF)^*\otimes \CF\rightarrow \CR_{\CF}
\CE\rightarrow 0,
\end{align*}
with the second and first non-zero maps in the two rows respectively
given canonically by evaluation.

It is well-known (see \cite{Goro-Rud},\cite{Rudakov89})
that all strong exceptional collections on $\BP^2$ can
be produced by a process of mutation, starting an \emph{initial
exceptional collection} with
$(\CE_0,\CE_1,\CE_2)=(\CO(m),\CO(m+1),\CO(m+2))$, replacing at each step
a strong exceptional collection $\CC=(\CE_0,\CE_1,\CE_2)$ with
either 
\[
\hbox{$\CL_{\CC}:=(\CE_0, \CL_{\CE_1}\CE_2, \CE_1)$ or $\CR_{\CC}:=
(\CE_1,\CR_{\CE_1}\CE_0, \CE_2)$.}
\]

We recall a few facts. For a sheaf $\CE$, as usual we write the slope 
$\mu(\CE)=d(\CE)/r(\CE)$. Note that the
slope triple for the initial exceptional collection is $(\mu_0,\mu_1,\mu_2)
= (m,m+1,m+2)$. 

\begin{lemma}[Properties of strong exceptional triples]
\label{lem:slope inequalities}
Let $\CC=(\CE_0,\CE_1,\CE_2)$ be a strong exceptional triple with slopes
$\mu_0,\mu_1,\mu_2$, ranks $r_0,r_1,r_2$, 
and first Chern classes $d_0,d_1,d_2$. 
Then
\begin{enumerate}
\item The rank triple of $\CL_{\CC}$ is $(r_0, 3r_0r_1-r_2, r_1)$
and the rank triple of $\CR_{\CC}$ is $(r_1,3r_1r_2-r_0,r_2)$.
\item $\mu_0<\mu_1<\mu_2$.
\item Provided that $\CC$ is not an initial exceptional collection,
then
\[
\mu_1-\mu_0 \le 1/2,\quad \mu_2-\mu_1\le 1/2.
\]
\item
There are exact sequences
\begin{align*}
0\rightarrow & \CL_{\CE_1} \CE_2 \rightarrow \Hom(\CE_1,\CE_2)\otimes\CE_1 
\rightarrow \CE_2
\rightarrow 0,\\
0\rightarrow & \CE_0\rightarrow \Hom(\CE_0,\CE_1)^*\otimes \CE_1\rightarrow 
\CR_{\CE_1}
\CE_0\rightarrow 0,\\
0\rightarrow & \CE_2(-3)\rightarrow \Hom(\CE_2(-3),\CE_0)^*\otimes \CE_0
\rightarrow \CL_{\CE_1}
\CE_2\rightarrow 0,\\
0\rightarrow & \CR_{\CE_1} \CE_0 \rightarrow \Hom(\CE_2,\CE_0(3))\otimes\CE_2 
\rightarrow \CE_0(3)
\rightarrow 0.
\end{align*}
\item $\dim \Hom(\CE_0,\CE_1)=3(d_1r_0-d_0r_1)$ and $\dim\Hom(\CE_1,\CE_2)
= 3(d_2r_1-d_1r_2)$. Also, 
\[
\dim \Hom(\CE_2(-3),\CE_0)=
\dim\Hom(\CE_2,\CE_0(3)) = 3(3r_0r_2+d_0r_2-d_2r_0).
\]
\item For any exceptional bundle with Chern character $(r,d,e)$, we have
$r^2-d^2+2re=1$. Further, $\gcd(r,d)=1$.
\item There is at most one exceptional bundle of a given slope.
\item For $\CE$ an exceptional bundle, there is at most one strong
exceptional triple $(\CF_0,\CE,\CF_2)$ containing $\CE$ as its middle
member.
\end{enumerate}
\end{lemma}

\begin{proof}
(1) is standard, see e.g., \cite[Thm.~4.2,3)]{Rudakov89}.
For (2), the statement is certainly true for an initial triple.
The proof of \cite[Thm.~4.3(3)]{Rudakov89} shows in any event that
the slopes satisfy, for some odd integer $a$ and non-negative integer $b$,
the equalities
\begin{equation}
\label{eq:mu epsilon}
\mu_0=\varepsilon\left({a-1\over 2^b}\right),\quad
\mu_1=\varepsilon\left({a\over 2^b}\right),\quad
\mu_2=\varepsilon\left({a+1\over 2^b}\right),
\end{equation}
where $\varepsilon$ is defined in \cite[\S5.1]{Drezet-LePotier}.
Note $\mu_0<\mu_2$ follows inductively from the construction of
a mutation of the triple, and by definition of $\varepsilon$,
$\mu_1=\mu_0.\mu_2$, where this notation
is as defined in \cite[\S5.1]{Drezet-LePotier}. 
By \cite[Prop.~(5.1)]{Drezet-LePotier}, $\mu_0<\mu_0.\mu_2<\mu_2$, hence
the claim in (2).

(3) now follows from \cite[Lem.~(5.4)]{Drezet-LePotier}.

The first two exact sequences of (4) are the definitions of the mutation,
and the last two are \cite[Thm.~4.2,2)]{Rudakov89}, correcting the obvious
typographical errors.

(5) is \cite[Lem.~4.1,2),5)]{Rudakov89}, bearing in mind 
the vanishing
of $\Ext$'s required for a strong exceptional collection.

For (6), we have for an exceptional bundle $1=\chi(\CE,\CE)$, and the
claimed equality then follows from Riemann-Roch \eqref{eq:RR}. Since
$e\in \BZ/2$, the equality immediately implies $\gcd(r,d)=1$.

(7) is \cite[Lemme~(4.3)]{Drezet-LePotier}.

For (8), note that \eqref{eq:mu epsilon} implies that the slopes of
$\CF_0$ and $\CF_2$ are determined by the slope of $\CE$. The result then
follows from (7).
\end{proof}

\begin{definition}[Associated vertex to an exceptional object]\label{Def: V_E}
For $\CE$ an object of rank $r\not=0$, first Chern class $d$, and
Euler characteristic $\chi$, define
\begin{align}\label{eq: VertexPE}
    V_{\CE}:=\left( {d\over r}-{3\over 2}, {3d-\chi\over r}\right)\in M_{\BR}.
\end{align}

Given a strong exceptional triple $\CC=(\CE_0,\CE_1,\CE_2)$, we define
$\Delta_{\CC}$ to be the triangle in $M_{\BR}$ with vertices $V_{\CE_0},
V_{\CE_1},V_{\CE_2}$.
\end{definition}

\begin{example}
\label{ex:initial triangle}
Starting with $\CC=(\CO,\CO(1),\CO(2))$, we have
$\CR_{\CC}=(\CO(1),\CT_{\BP^2},\CO(2))$. The vertices of $\Delta_{\CR_{\CC}}$
are $(-1/2,0),(0,1/2)$ and $(1/2,0)$. We call this triangle an \emph{initial
triangle}.
\end{example}

The following theorem generalizes \cite[Lem.~4.8]{Bousseau-scatteringP2-22},
and the proof is very similar.

\begin{theorem}[Moduli emptiness]
\label{thm:bousseau generalization}
Let $\CC$ be a strong exceptional triple not equal to 
$(\CO(m),\CO(m+1),\CO(m+2))$ for any $m$, i.e., not an initial exceptional
collection. Then for $\sigma$ in the interior of an edge of
$\Delta_{\CC}$, we have
\[
H_{\foD^{\mathrm{stab}},\sigma}=-\li_2(-z^{m_{\CF}})
\]
for some object
$\CF=\CE$ or $\CE[1]$, for $\CE$ an exceptional bundle.
Further, for any $\sigma\in\Int(\Delta_{\CC})$, there are no stable
objects $\CE$ in $\CA^{\sigma}$ with $\Re Z^{\sigma}(\gamma(\CE))=0$. 
\end{theorem}

\begin{proof}
{\bf Step I.} \emph{Identifying $D^b(\BP^2)$ with 
a quiver representation category.}
Since $\CC=(\CE_0,\CE_1,\CE_2)$ is not an initial exceptional collection, 
we can write it as the mutation of another strong exceptional collection
$\CC'=(\CE_0',\CE_1',\CE_2')$. 

Let $\mathbf{T}=\CE_0\oplus\CE_1\oplus\CE_2$, and let 
\[
A_0:=\Hom(\mathbf{T},\mathbf{T})^{\mathrm{op}}.
\]
Let $\CA_0$ be the abelian category of finitely generated left $A_0$-modules.
By a result due to Bondal, see \cite[Thm.~3.11]{Macri-Curves-07},
there is an equivalence of categories
\[
D^b(\BP^2)\rightarrow D^b(\CA_0), \quad E\mapsto {\bf R}\Hom(\mathbf{T},E).
\]
This allows us to view $\CA_0$ as a subcategory of $D^b(\BP^2)$.

We note that $A_0$ is the quotient of the path algebra associated
to a quiver with vertices $v_0,v_1,v_2$ and with
$\dim \Hom(\CE_i,\CE_{i+1})$ arrows from $v_{i+1}$ to $v_i$.
We may label these arrows with a choice of basis $\{e_{i,j}\}$ for
$\Hom(\CE_i,\CE_{i+1})$.
 We divide out
by an ideal of relations coming from the kernel of the composition
map
\[
\Hom(\CE_1,\CE_2)\otimes \Hom(\CE_0,\CE_1)\rightarrow\Hom(\CE_0,\CE_2).
\]
In other words, if $\sum_{j_1,j_0} c_{j_1,j_0} e_{1,j_1}\otimes e_{0,j_0}$
lies in the kernel, then we include the corresponding linear combination
of paths $\sum_{j_1,j_0} c_{j_1,j_0}e_{1,j_1}e_{0,j_0}$ in the ideal.
Thus a finite-dimensional left $A_0$-module consists of the data
of finite-dimensional $\BC$-vector spaces $V_0,V_1,V_2$, and for
every arrow between $v_i$ and $v_j$ a linear transformation $V_i\rightarrow
V_j$. These must satisfy the given relations. 

Note that $\CA_0$ then has simple objects $S_0,S_1,S_2$ with dimension vectors
$(1,0,0)$, $(0,1,0)$ and $(0,0,1)$ respectively, with $S_j$ given by
$V_j=\BC$ and $V_k=0$, $j\not=k$.

\medskip

{\bf Step II.} \emph{Identifying simple objects.}
We analyze right and left mutations separately.
\medskip

{\bf Right mutations.} Suppose $\CC=\CR_{\CC'}$. 
In this case 
\[
(\CE_0,\CE_1,\CE_2)=(\CE_1',\CR_{\CE_1'}\CE_0',\CE_2').
\] 
The following triple
\[
(\CS_0,\CS_1,\CS_2)=(\CE_1', 
\CE_0'[1],
\CE_2'(-3)[2]) 
\]
corresponds to the three simple objects. To see this, we just need to show
that 
$\mathbf{R}\Hom(\mathbf{T},\CS_i)$ defines the $A_0$-module $S_i$
for $i=1,2,3$. In other words, we need to check that
\[
\Ext^k(\CE_i,\CS_j)=\begin{cases} 0 & \hbox{if $k\not=0$ or $i\not=j$}\\
\BC&k=0, i=j
\end{cases}
\]
These are all mostly straightforward using the fact that $\CC$ and
$\CC'$ are strong exceptional collections. For example,
$\Ext^*(\CE_0,\CS_0)=\Ext^*(\CE_1',\CE_1')$ which is $\BC$ in degree $0$
and $0$ otherwise. On the other hand, $\Ext^*(\CE_0,\CS_2)=
\Ext^*(\CE_1',\CE_2'(-3)[2])=\Ext^{*+2}(\CE_1',\CE_2'(-3))
=\Ext^{-*}(\CE_2',\CE_1')^{\vee} = 0$ by Serre duality and again the
fact that $\CC'$ is an strong exceptional collection.

The only slightly subtle calculation is that of
$\Ext^*(\CE_1,\CS_1) = \Ext^{*+1}(\CR_{\CE_1'}\CE_0',\CE_0')$. Applying
$\Ext^*(\cdot,\CE_0')$ to the defining exact sequence for
$\CR_{\CE_1'}\CE_0'$, we have a long exact sequence
\[
\rightarrow \Hom(\CE_0',\CE_1')^*\otimes \Ext^i(\CE_1',\CE_0')
\rightarrow \Ext^i(\CE_0',\CE_0')\rightarrow \Ext^{i+1}(\CR_{\CE_1'}\CE_0',
\CE_0')\rightarrow
\Hom(\CE_0',\CE_1')^*\otimes \Ext^{i+1}(\CE_1',\CE_0')\rightarrow,
\]
using
$ \Ext^i(\Hom(\CE_0',\CE_1')^*\otimes \CE_1',\CE_0')=\Hom(\CE_0',\CE_1')^*
\otimes \Ext^i(\CE_1',\CE_0')$.
Since $\Ext^*(\CE_1',\CE_0')=0$ and $\CE_0'$ is exceptional, we obtain
\[
\Ext^i(\CE_1,\CS_1)=\begin{cases} \BC & i=0\\ 0 & i\not=0.\end{cases}
\]

    \begin{figure}[h]
 \centering    \includegraphics[width=9.6cm]{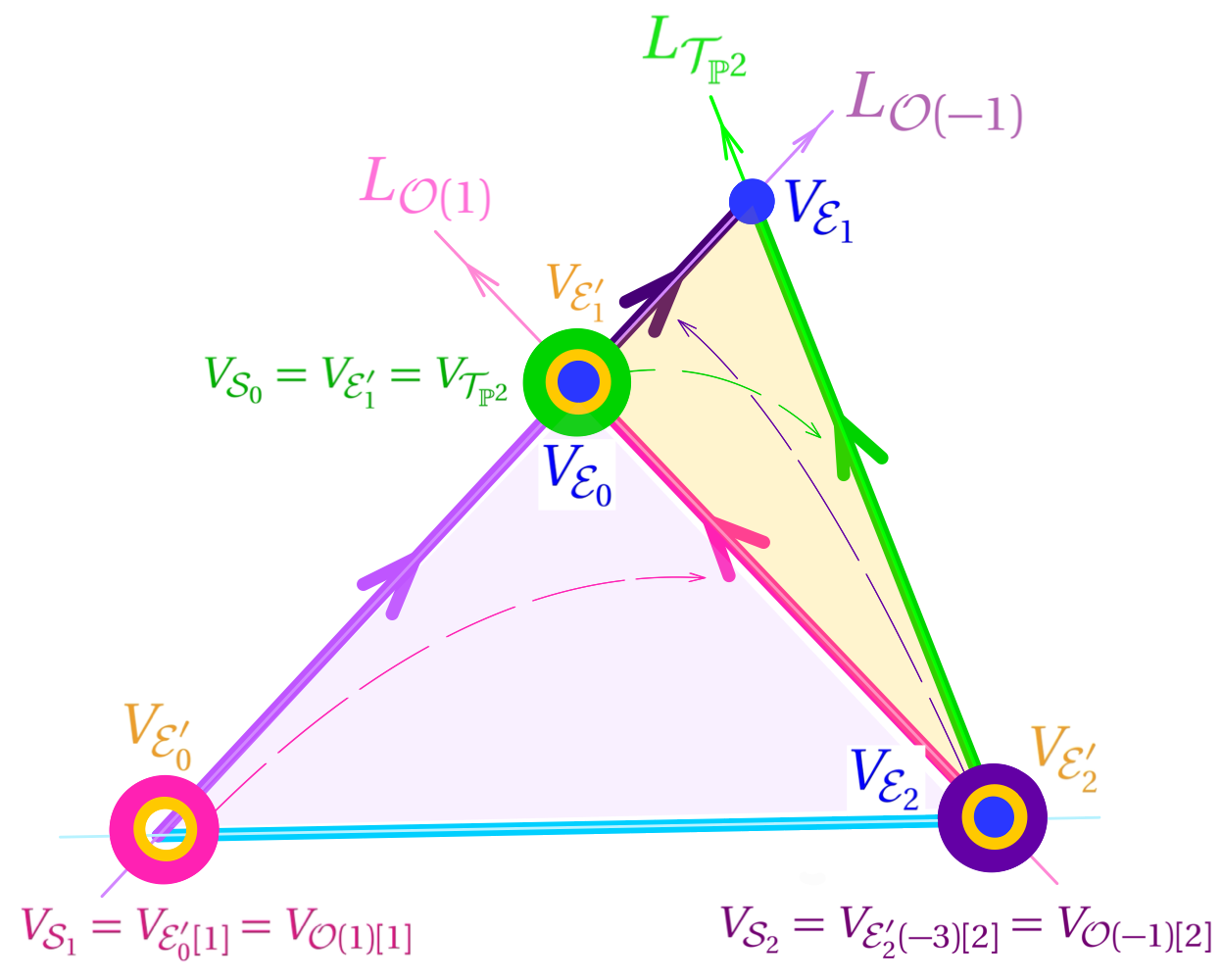}
  \caption{Right mutation of the central initial triangle. Dashed arrows show the corresponding edges to the simple vertices $V_{\CS_i}$. Here, $\CE_0=\CT_{\BP^2}$,  $\CE_2=\CO(2)$, and $\CE_1$ is an object of Chern character $(5,8,4)$.}
  \label{fig: RightMutation}
\end{figure}

{\bf Left mutations}. Suppose
$\CC=\CL_{\CC'}$. In this case 
\[
(\CE_0,\CE_1,\CE_2)=(\CE_0', \CL_{\CE_1'}\CE_2',\CE'_1).
\]
Now the simples are
\[
(\CS_0,\CS_1,\CS_2)=(
\CE_0',
\CE_2'(-3)[1], 
\CE_1'(-3)[2]). 
\]
This is checked similarly.

\medskip

{\bf Step III.} \emph{The central charges of the simple objects
and the categories $\CA^{\sigma}[Z^{\sigma},\phi]$.}

Write $(r_i,d_i,\chi_i)=
(r(\CE_i),d(\CE_i),\chi(\CE_i))$ and $(r_i',d_i',\chi_i')=
(r(\CS_i),d(\CS_i),\chi(\CS_i))$.
Following \eqref{eqn:LGamma}, set
\[
L_i=L_{(r_i',d_i',\chi_i')}=L_{\CS_i}.
\]
Recall from \eqref{eq:central charge} that, with $\sigma=(x,y)$,
\begin{equation}
\label{eq:recall Z}
Z^{\sigma}(\CS_j)=
r'_jy+d'_jx+r'_j+{3\over 2}d'_j-\chi_j'
+ i(d'_j-r'_jx)\sqrt{x^2+2y}.
\end{equation}
On the other hand, by the calculations of $\Ext^i(\CE_j,\CS_k)$ of 
Step II for the first equality 
and Riemann-Roch \eqref{eq:RR} for the second,
\begin{align}\label{eq: Kronrcker}
   \delta_{ij}=\chi(\CE_i,\CS_j)=
-3d_ir_j'-r_ir_j'-d_id_j'+r_i\chi_j'+\chi_i r_j' 
\end{align}
where $\delta_{ij}$ is the Kronecker delta.
Note that
\begin{align}
\label{eq:Z chi}
\begin{split}
-r_i\Re Z^{V_{\CE_i}}(\CS_j)= {} & -r_ir_j'{3 d_i-\chi_i \over r_i}-r_i
d'_j\left({d_i\over r_i}-{3\over 2}\right)-r_ir'_j-{3\over 2}r_id'_j+r_i\chi_j'
\\
={}&\chi(\CE_i,\CS_j).
\end{split}
\end{align}
Thus, if $\{i,j,k\}=\{0,1,2\}$, we see that $V_{\CE_i}= L_j \cap L_k$  and
that 
\begin{equation}
\label{eq:real negative}
\Re Z^{V_{\CE_i}}(\CS_i) < 0.
\end{equation}
In particular, the $L_j$ are the extensions of the three edges of
$\Delta_{\CC}$ to lines.

We first show that if $\sigma\in\Int(\Delta_{\CC})$, then there
are no stable objects $\CE$ in $\CA^{\sigma}$ with $\Re Z^{\sigma}(\CE)=0$.
This was already shown for the triangles obtained when $\CC'$ is an
initial strong exceptional triple in \cite[Lemma 4.8]{Bousseau-scatteringP2-22},
and therefore we will assume that $\CC'$ is not an initial triple. This
will enable us to avoid a special case.

Recall from \cite[Section 4.3]{Bousseau-scatteringP2-22}, the construction of a new stability
condition $(Z[\phi], \CA[Z,\phi])$ from a stability condition 
$(Z,\CA)$ for $\phi\in\BR$. One 
defines $Z[\phi](\gamma):=e^{-i\pi\phi} Z(\gamma)$.
Let $\CQ_{\phi}$ be the subcategory of $\CA$ generated via extensions
of the semistable objects $\CE$ with ${1\over\pi}\Arg Z(\CE)>\phi$,
and $\CF_{\phi}$ the subcategory of $\CA$ similarly generated by the
semistable objects $\CE$ with ${1\over\pi} \Arg Z(\CE) \le \phi$. Then 
denoting $\CH^i_{\CA}$ the cohomology functors with respect to the bounded
$t$-structure of heart $\CA$, the category $\CA[Z,\phi]$ is the subcategory of
$D^b(\BP^2)$ of objects with $\CH_{\CA}^i=0$ for $i\not=-1,0$, 
$\CH_{\CA}^{-1}$ an object of $\CF_{\phi}$, and $\CH_{\CA}^0$ an object
of $\CQ_{\phi}$.

We can write $(x,y)=\sigma\in\Int(\Delta_{\CC})$ as $\sum_{i=0}^2 a_iV_{\CE_i}$ for
$a_i>0$, $a_0+a_1+a_2=1$. 

Then $\sum_{i=0}^2 a_i \Re Z^{V_{\CE_i}}(\CE)
=\Re Z^{\sigma}(E)$, so by \eqref{eq:real negative}, 
\begin{equation}
\label{eq:real negative 2}
\hbox{$\Re Z^{\sigma}(\CS_i)<0$ for $i=0,1,2$.}
\end{equation}

We also need to understand the sign of $\Im Z^{\sigma}(\CS_i)$, which
requires slightly different analyses in the left and right mutation cases.
Note that the sign of $\Im Z^{\sigma}(\CS_i)$ agrees with the sign of
$\operatorname{sign}(r_j')\left({d_j'\over r_j'}-x\right)$, 
and $d_j'/r_j'$ is the slope of $\CS_i$. (Bear in mind that the rank of 
$\CE[1]$ is the negative of the rank of $\CE$).

{\bf The right mutation case.} 
Suppose $\CC=\CR_{\CC'}$. Let $\mu_i'$ be the slope of $\CE_i'$.
Then the slopes of $\CS_i$, $i=0,1,2$, are $\mu_1', \mu_0', \mu_2'-3$
respectively. Note further that by the definition of $V_{\CE_i}$ and
Lemma~\ref{lem:slope inequalities}, (2), $\mu_0-3/2 < x < \mu_2-3/2$,
or $\mu_1'-3/2 < x < \mu_2'-3/2$. By Lemma~\ref{lem:slope inequalities},
(3), $\mu_2'-\mu_1'=\mu_2-\mu_0\le 1$. Thus we have 
\begin{equation}
\label{eq:mu ineq right}
\mu_1' > x, \quad \mu_2'-3 < x.
\end{equation}
Provided that $\CC'$ was not an initial exceptional collection,
then again $\mu_2'-\mu_0'\le 1$ and 
\begin{equation}
\label{eq:mu ineq right2}
\mu_0' > x.
\end{equation}
Thus we see that (taking into account the shift and twist on $\CS_1,\CS_2$)
\[
\Im Z^{\sigma}(\CS_0) >0,\quad \Im Z^{\sigma}(\CS_1)<0,
\quad \Im Z^{\sigma}(\CS_2) <0.
\]
Note further that the inequalities \eqref{eq:mu ineq right}, 
\eqref{eq:mu ineq right2} imply that $\CE_0',\CE_1'\in \CA^{\sigma}$
and $\CE'_2(-3)[1]\in\CA^{\sigma}$. 

Putting this all together, taking $\phi=1/2$, we see that
$\CE_1'\in \CQ_{\phi}$, $\CE_0'\in\CF_{\phi}$ (remembering that
$\CS_1=\CE_0'[1]$), $\CE_2'(-3)[1]\in \CF_{\phi}$, so we get
$\CS_i\in \CA^{\sigma}[Z^{\sigma},\phi]$. 

\medskip

{\bf The left mutation case.} Suppose $\CC=\CL_{\CC'}$ with $\CC'$
not an initial collection. 

Similarly as in the right mutation case, we obtain
\[
\Im Z^{\sigma}(\CS_0) >0,\quad \Im Z^{\sigma}(\CS_1)>0,
\quad \Im Z^{\sigma}(\CS_2) <0.
\]
In addition, $\CE_0'\in \CA^{\sigma}$
and $\CE_1'(-3)[1],\CE'_2(-3)[1]\in\CA^{\sigma}$. 

Putting this together, taking $\phi=1/2$, we see that
$\CE_0',\CE_2'(-3)[1]\in \CQ_{\phi}$, $\CE_1'(-3)[1]\in\CF_{\phi}$,
so we get $\CS_i\in \CA^{\sigma}[Z^{\sigma},\phi]$. 

\medskip

{\bf Step IV.} \emph{Completing the argument.}
We may now finish the argument as in the proof of \cite[Lemma 4.8]{Bousseau-scatteringP2-22}. In both cases, $\CA^{\sigma}[Z^{\sigma},\phi]$
contains the simple objects of the quiver representation category $\CA_0$
and ${1\over \pi} \Arg Z^{\sigma}[\phi](\CS_i)\in (0,1)$. Further,
$(\CS_0,\CS_1,\CS_2)$ form a complete $\Ext$-exceptional collection 
in the sense of \cite[Def.~3.10]{Macri-Curves-07}.
Thus by \cite[Lem.~3.16]{Macri-Curves-07},
$\CA^{\sigma}[Z^{\sigma},\phi]=
\CA_0$. We deduce that the central charge with respect to $Z^{\sigma}[\phi]$
of any stable object in 
$\CA^{\sigma}[Z^{\sigma},\phi]$, being a non-negative linear combination
of the central charges of the $\CS_i$, lie in the open upper half-plane.
Correspondingly, any stable object of $\CA^{\sigma}$ has $Z^{\sigma}$-central
charge lying in the open left half-plane, and hence has negative real part.
In particular, the real part of the central charge is non-zero. This proves
the emptiness of $\Int(\Delta_{\CC})$.

If instead $\sigma$ lies in the interior of an edge of $\Delta_{\CC}$,
say contained in the line $L_i$, containing $V_{\CE_j},V_{\CE_k}$,
$\{i,j,k\}=\{0,1,2\}$, then $\Re Z^{\sigma}(\CS_i)=0$. Further, 
the argument above in the two cases still shows that 
$\CS_i\in \CA^{\sigma}[Z^{\sigma},1/2]$; indeed, the inequalities
\eqref{eq:mu ineq right}, \eqref{eq:mu ineq right2}
remain strict inequalities, as well as the analogous inequalities in 
the left mutation case. Since $\CS_i$ is
simple,
it must be stable. Thus $\CS_i$ corresponds to a stable 
object in $\CA^{\sigma}$,
necessarily (a shift) of an exceptional bundle. So there is a one-point
moduli space and the edge is thus contained in a non-trivial ray. The
precise form of the function attached to the ray then follows from 
Theorem~\ref{thm:bousseau main}.
\end{proof}

\begin{remark}
\label{rmk:left right exceptionals}
For future reference, we observe that in Step III of the proof,
the object of $\CA^{\sigma}$
corresponding to $\CS_0$ is $\CE_0$ in both the left and right mutation
cases. The object corresponding to $\CS_2$ in $\CA^{\sigma}$ is
$\CE_2(-3)[1]$ in both cases.
\end{remark}

\begin{remark}
\label{rem:bottom exceptionals}
The object $\CS_1$ cannot be read
off from the strong exceptional collection $\CC$, as it depends on whether
$\CC$ is obtained by right or left mutation. However, 
$\CS_1=\CE_0'[1]$ or $\CE_2'(-3)[1]$ in the right or left mutation 
case, and thus in either case $r_1'<0$. From \eqref{eq:recall Z}, 
\eqref{eq: Kronrcker}
and \eqref{eq:Z chi}, we then see that $V_{\CE_1}$ always lies above
the line $L_{\CS_i}$. Thus 
$V_{\CE_1}$ always lies above the line containing $V_{\CE_0},V_{\CE_2}$.
\end{remark}

\begin{definition}[Vertex-line duality]\label{Def:vertexRayDuality}
Let $\CE_i$, $i=0,1,2$, be as in the proof of Theorem \ref{thm:bousseau generalization}. We say that the vertex $V=V_{\CE_i}$ (defined in Definition \ref{Def: V_E}) is the \emph{dual vertex to $\CE_i$}, or the \emph{dual vertex to  $L_{\CE_i}$}. On the other hand,  $\CE_i$, or $L_{\CE_i}$ is called the \emph{dual object/line to $V_{\CE_i}$}. 
\end{definition}

Recall the tree $\CT$ of Construction \ref{susection: construction}. Write
$\CT_E$ for a copy of this tree. We decorate each vertex $w\in V(\CT_E)$
with a strong exceptional triple $\CC(w)$ in the following way. First, we
assign to the root of $\CT_E$ the strong exceptional triple
$(\CO(1),\CT_{\BP^2},\CO(2))$. Inductively, if $\CC(w)$ has been defined,
we assign the triples $\CL_{\CC(w)}$ and $\CR_{\CC(w)}$ to the two children
of $w$.

We now identify the triangles constructed algorithmically in
Construction~\ref{susection: construction} with the triangles associated
to strong exceptional triples.

\begin{theorem}
\label{thm:mu build}
\begin{enumerate}
\item
There exists an isomorphism 
$\eta_E:\CT\rightarrow\CT_E$ so that
for $w$ a vertex of $\CT$, $\Delta_w=\Delta_{\CC(\eta_E(w))}$.
\item
This isomorphism has the property that if $\CC(\eta_E(w))=(\CE_0,\CE_1,\CE_2)$,
then $V_{\CE_1}$ is the incoming vertex of $\Delta_w$.
\item
Except for the horizontal edge of $\Delta_{w_0}$, which is contained in
$\fod_0^-\cup\fod_0^+$, every edge $E$ of every triangle $\Delta_w$ is 
contained in precisely one ray $(\fod,H_{\fod})$ of
$\{\fod_{-1}^+,\fod_1^-\}
\cup \foD^0_{\mathrm{discrete}}$. Furthermore, if 
$\sigma\in U\subseteq M_{\BR}$ lies in the interior of $E$, then
$H_{\fod}=H_{\foD^{\mathrm{stab}},\sigma}$, in the sense
of Definition \ref{def:total sum}.
\end{enumerate}
\end{theorem}

\begin{proof}
We construct $\eta_E$ inductively, and prove statement (2)
at the same time. We check most of (3) after this is done.

\medskip

{\bf Step I.} \emph{The base case}.
The base case is the vertex $w_0$. From
Example~\ref{ex:initial triangle}, we have that 
$\Delta_{w_0}=\Delta_{\CC}$
with $\CC=(\CO(1),\CT_{\BP^2},\CO(2))$. Note that $V_{\CT_{\BP^2}}$
is the incoming vertex of $\Delta_{w_0}$ and $\Delta_{w_0}$
has a horizontal edge $E$ contained in a union of the two rays
$\fod_0^-$ and $\fod_0^+$. \cite[Lem.~4.11]{Bousseau-scatteringP2-22}
shows that $H_{\foD^{\mathrm{stab}},\sigma}$ agrees with the 
attached wall functions to $\fod_0^{\pm}$ for $\sigma$ lying in
$\Int(E)\cap\Int(\fod^{\pm}_0)$. Except for the uniqueness
statements, this shows (3) for $\sigma\in E\cap U$.

\medskip
{\bf Step II.} \emph{The induction step}.
Now assume that for a vertex $w\in V(\CT)$, we have an assignment
$\eta_E(w)\in V(\CT_M)$ so that $\Delta_w = \Delta_{\CC(w)}$, with
$\CC(w)=(\CE_0,\CE_1,\CE_2)$ and $V_{\CE_1}$ the incoming vertex of
$\Delta_w$. Let $\CS_0,\CS_1,\CS_2$ be the
corresponding simple objects as described in Step II of the proof
of Theorem~\ref{thm:bousseau generalization}, so that $V_{\CE_i}
\in L_{\CS_j}$ for $i\not=j$. 

From the precise description of
the $\CS_i$ in the right and left mutation cases, we see that
the edges  of $\Delta_{\CC(w)}$ joining $V_{\CE_1}$ to $V_{\CE_0}$ 
and $V_{\CE_2}$ are contained in the lines
$L_{\CE_2(-3)[1]}$ and $L_{\CE_0}$ respectively, see Remark~\ref{rmk:left right exceptionals}.
Again, from the explicit
description of the simples for $\CR_{\CC(w)}$ and $\CL_{\CC(w)}$, one sees
that $V_{\CR_{\CE_1}\CE_0}$ is also contained in the line 
$L_{\CE_2(-3)[1]}$ and the vertex $V_{\CL_{\CE_1}\CE_2}$ is also contained
in the line $L_{\CE_0}$. Further, the edge joining $V_{\CE_0}$
and $V_{\CL_{\CE_1}\CE_2}$ is contained in $L_{\CE_1(-3)[1]}$ and
the edge joining $V_{\CE_2}$ and $V_{\CR_{\CE_1}\CE_0}$ is contained
in $\CL_{\CE_1}$. See Figure \ref{Fig: RL}.

\begin{figure}[h]
    \centering
    \includegraphics[width=8.9cm]{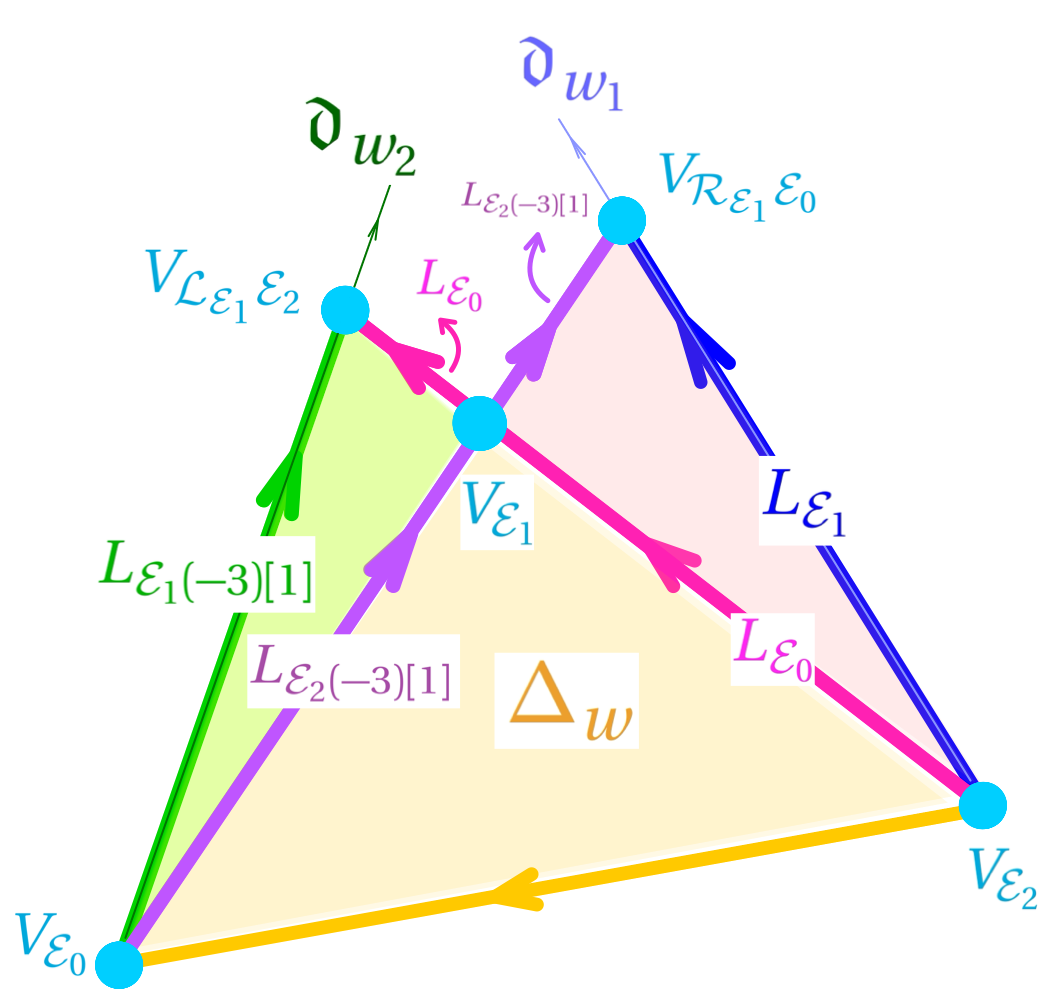}
    \caption{
}
    \label{Fig: RL}
\end{figure}

Since we have assumed that $V_{\CE_1}$ is the incoming vertex,
in the language of Construction~\ref{susection: construction}, 
we have $v=V_{\CE_1}$, and can take
$v_1=V_{\CE_0}$, $v_2=V_{\CE_2}$. 
Again in the notation of 
Construction~\ref{susection: construction}, $m_i$ is a primitive
tangent vector to the
edge joining $v$ to $v_i$, oriented from $v$ to $v_i$. Then we see that
\[
m_1=m_{\CE_2(-3)[1]}=\big(-r(\CE_2(-3)),d(\CE_2(-3))\big),\quad 
m_2=m_{\CE_0}=\big(r(\CE_0),
-d(\CE_0)\big),
\] 
recalling the notation of \eqref{eq:mgamma def}. To see these equalities,
first note that the opposite vectors $m_{\CE_2(-3)[1]}$ and $m_{\CE_0}$
are tangent to $L_{\CE_2(-3)[1]}$ and $L_{\CE_0}$ 
respectively. They are have the same orientation as $m_1$ and $m_2$
by the fact that the $x$-coordinates of
$v_1,v,v_2$ are increasing by Lemma~\ref{lem:slope inequalities}(2).
Finally, they are primitive by Lemma~\ref{lem:slope inequalities},(6).

We may order the children $w_1,w_2$  of $w$ so that
$\Delta_{w_j}$ is the mutation of $\Delta_w$ at $v_j$. We will
define $\eta_E(w_1):=\CR_{\CC(w)}$ and $\eta_E(w_2)=\CL_{\CC(w)}$.
Thus we need to check that $\Delta_{w_i}=\Delta_{\CC(\eta_E(w_i))}$.

Further analysis requires two cases, depending on whether
$(\CE_0,\CE_1,\CE_2)$ is obtained as a right or left mutation from
$(\CE_0',\CE_1',\CE_2')$. 

{\bf Right mutation.}
Suppose $(\CE_0,\CE_1,\CE_2)=(\CE_1', \CR_{\CE_1'}\CE_0',
\CE_2')$. Then the edge of $\Delta_{w}$ joining $V_{\CE_0}$ and
$V_{\CE_2}$ is $L_{\CE'_0[1]}$, see Remark~\ref{rem:bottom exceptionals}.
A primitive tangent vector to this
edge, pointing from $v_1$ to $v_2$, is 
\[
m=m_{\CE_0'}=\big(r(\CE_0'),-d(\CE_0')\big).
\] 
If we mutate $\Delta_w$ at the vertex $v_1$, we obtain the triangle 
$\Delta_{w_1}$ with
vertices $V_{\CE_1}, V_{\CE_2}$, and the third vertex at the intersection
of $L_{\CE_2(-3)[1]}$ and $\fod_{w_1}=V_{\CE_2}-\BR_{\ge 0} (3D_{v_2}m_2-m)$.
If we show that this intersection is $V_{\CR_{\CE_1}\CE_0}$, this
will show that $\Delta_{w_1}=\Delta_{\CR_{\CC(w)}}$ as needed.
Thus
it is sufficient to show that $\fod_{w_1}$ 
is contained in the line
$L_{\CE_1}$, i.e.,
that the tangent vector to $L_{\CE_1}$ oriented from
$V_{\CR_{\CE_1}\CE_0}$ to $V_{\CE_2}=v_2$ is $3D_{v_2} m_2 - m$. Put another
way, we require
\begin{equation}
\label{eq:ddr1}
m_{\CE_1}=3D_{v_2}m_{\CE_0}-m_{\CE_0'}
\end{equation}
However, using $\CE_1'=\CE_0$, we have the second exact sequence
of Lemma~\ref{lem:slope inequalities}(4) is
\[
0\rightarrow  \CE'_0\rightarrow \Hom(\CE'_0,\CE_0)^*\otimes \CE_0\rightarrow 
\CE_1
\rightarrow 0.
\]
By Lemma~\ref{lem:slope inequalities},(5),
$\dim\Hom(\CE_0',\CE_0) = 3(d(\CE_0)r(\CE_0')-d(\CE_0')r(\CE_0))
= 3|m\wedge m_2| = 3D_{v_2}$. Then \eqref{eq:ddr1} follows
from the exact sequence.

Similarly, if 
instead we mutate $\Delta_w$ at the vertex $v_2$ to get 
the triangle $\Delta_{w_2}$, we wish to show $\Delta_{w_2}=\Delta_{\CL_{\CC(w)}}$. For this analysis, we need to change
the orientation of $m$, replacing it with 
\[
m=m_{\CE_0'[1]}=\big(-r(\CE_0'),d(\CE_0')\big).
\]
Noting that $\fod_{w_2}= V_{\CE_0}-\BR_{\ge 0}(3D_{v_1}m_1-m)$, 
we need to show that $\fod_{w_2}\subseteq L_{\CE_1(-3)[1]}$, i.e., that
\begin{equation}
\label{eq:ddr2}
m_{\CE_1(-3)[1]} = 3D_{v_1}m_{\CE_2(-3)[1]} - m_{\CE_0'[1]}.
\end{equation}
By twisting the fourth exact sequence of 
Lemma~\ref{lem:slope inequalities}(4) by $-3$, we get the exact sequence
\[
0\rightarrow \CE_1(-3)\rightarrow \Hom(\CE_2,\CE_0'(3))\otimes\CE_2(-3)
\rightarrow\CE_0'\rightarrow 0.
\]
By Lemma \ref{lem:slope inequalities},(5),
\begin{align}
\label{eq:HomE2E0(3)}
\begin{split}
\dim\Hom(\CE_2,\CE_0'(3))= {} & 3\left[3r(\CE_2)r(\CE_0')+d(\CE_0')r(\CE_2)
-d(\CE_2)r(\CE_0')\right]\\
= {}&
 3\left[d(\CE_0')r(\CE_2)-r(\CE_0')\big(d(\CE_2)-3r(\CE_2)\big)\right]\\
= {} & 3\left[d(\CE_0')r(\CE_2(-3))-r(\CE_0')d(\CE_2(-3))\right]\\
= {} & 3|m_1\wedge m|\\
= {} & 3D_{v_1}.
\end{split}
\end{align}
Thus the exact sequence implies the desired equality on tangent vectors
again. 

{\bf Left mutation.}
Suppose instead that $(\CE_0,\CE_1,\CE_2)=(\CE_0', \CL_{\CE_1'}\CE_2',
\CE_1')$. Then the edge of $\Delta_{w}$ joining $V_{\CE_0}$ and
$V_{\CE_2}$ is $L_{\CE'_2(-3)[1]}$, again by 
Remark~\ref{rem:bottom exceptionals}, and a primitive tangent vector to this
edge, pointing from $v_1$ to $v_2$, is 
\[
m=m_{\CE_2'(-3)}=\big(r(\CE_2'(-3)),-d(\CE_2'(-3))\big).
\] 
If we mutate $\Delta_w$ at the vertex $v_1$ to get the triangle
$\Delta_{w_1}$, we have as previously $\fod_{w_1}=V_{\CE_2}-
\BR_{\ge 0}(3D_{v_2}m_2-m)$ and we need to show that
$\fod_{w_1}\subseteq L_{\CE_1}$, i.e., that
\begin{equation}
\label{eq:ddr3}
m_{\CE_1}= 3D_{v_2} m_{\CE_0}-m_{\CE_2'(-3)}.
\end{equation}
As in the right mutation case, this follows from
the third exact sequence of Lemma~\ref{lem:slope inequalities}(4).

If instead we mutate $\Delta_w$ at the vertex $v_2$ to get the triangle
$\Delta_{w_2}$, we need to change
the orientation of $m$, replacing it with 
\[
m=m_{\CE_2'(-3)[1]}=\big(-r(\CE_2'(-3)), d(\CE_2'(-3))\big).
\]
We have as previously $\fod_{w_2}=V_{\CE_0}-\BR_{\ge 0}(3D_{v_1}m_1-m)$
and need to show that $\fod_{w_2}\subseteq L_{\CE_1(-3)[1]}$, i.e., that
\begin{equation}
\label{eq:ddr4}
m_{\CE_1(-3)[1]}=3D_{v_1} m_{\CE_2(-3)[1]}-m_{\CE_2'(-3)[1]}.
\end{equation}
This time this follows by twisting the first exact sequence of 
Lemma~\ref{lem:slope inequalities}(4) by $-3$.

{\bf Step III}. \emph{Proof of (3).}
Let $\Delta_w$ be a triangle as above, with $(\CE_0,\CE_1,\CE_2)$
the corresponding strong exceptional collection. By the inductive construction
of the triangles, Construction \ref{susection: construction}, 
the edge $E_i$ joining $v$ to $v_i$, $i=1,2$,
is contained in a ray $(\fod_i,H_{\fod_i})$ of $\{\fod_{-1}^+,\fod_{1}^-\}\cup
\foD_{\mathrm{discrete}}^0$, with $H_{\fod_i}= -\li_2(-z^{m_i})$.
But by Remark~\ref{rmk:left right exceptionals} and 
Theorem~\ref{thm:bousseau generalization}, this agrees with
$H_{\foD^{\mathrm{stab}},\sigma}$ for $\sigma\in\Int(E_i)$.

It remains to show uniqueness of this ray. We first observe that if an
edge of a triangle is contained in some $L_{\CE}$ for $\CE$ an exceptional
bundle, in fact $\CE$ is uniquely determined. Indeed,
the slope of
$\CE$ is determined by the slope of the line $L_{\CE}$, and $\CE$ is
uniquely determined by its slope by Lemma~\ref{lem:slope inequalities},(7).

We have seen in Step II that every ray
of $\foD_{\mathrm{discrete}}^0$ is contained in a line
of the form $L_{\CE}$ or $L_{\CE(-3)[1]}$ for $V_{\CE}$ an incoming
vertex of a triangle $\Delta_w$, or equivalently, the middle member
of a strong exceptional triple. The middle member of the triple uniquely
determines the strong exceptional triple by Lemma~\ref{lem:slope inequalities},(8)
hence the triangle. Note that the ray is contained in a line of the
form $L_{\CE}$ if its direction vector has negative $x$-coordinate and
is contained in a line of the from $L_{\CE(-3)[1]}$ if its direction
vector has positive $x$-coordinate.
Thus the map which assigns to 
$(\fod,H_{\fod})\in\foD_{\mathrm{discrete}}^0$ the unique 
exceptional bundle $\CE$ or shifted exceptional bundle $\CE[1]$
with $\fod\subseteq L_{\CE}$ is an injective map
into the set of exceptional bundles and their shifts.
In particular, an edge of
a triangle $\Delta_w$ cannot be contained in two different such rays.
Note also that none of the exceptional bundles which appear in the image
of this map are line bundles. Since the exceptional bundles whose corresponding
lines contain $\fod_{-1}^+$ and $\fod_1^-$ are the line bundles $\CO(-1)$
and $\CO(1)$ respectively, we see that
no edge of any triangle $\Delta_w$
can be contained in more than one ray of $\{\fod_{-1}^+,
\fod_1^-\}\cup \foD_{\mathrm{discrete}}^0$. This complete the proof of (3).
\end{proof}

\begin{corollary}[Collection of triangles] \label{Cor: RDelta}
\begin{enumerate}
\item The collection of triangles $\{\Delta_w\,|\,w\in V(\CT)\}$
forms a triangulation of the region
\[
R^0_{\Delta}:=\bigcup_{w\in V(\CT)} \Delta_w.
\]
\item The collection of triangles $\{T^k(\Delta_w)\,|\, w\in V(\CT), k\in\BZ\}$
forms a triangulation of the region
\[
R_{\Delta}:=\bigcup_{k\in\BZ} T^k(R^0_{\Delta}).
\]
\item
$T^k(R^0_{\Delta}) = R_{\Delta}\cap \{(x,y)\in M_\BR\,|\,(k-1)/2\le x\le (k+1)/2
\}$.
\end{enumerate}
\end{corollary}

\begin{proof}
For (1), by Theorem~\ref{thm:mu build}, it is enough to observe that
$\bigcup_{w\in V(\CT_M)} \Delta_{\CC(w)}$ is triangulated by
$S:=\{\Delta_{\CC(w)}\,|\,w\in V(\CT(M))\}$. By construction,
each triangle $\Delta_{\CC(w)}$ meets three other triangles in $S$
along the three edges of $\Delta_{\CC(w)}$ determined
by the parent and the two children of $w$. There is of course one
exception, namely when $w=w_0$ is the root, in which case there are
only two triangles. To show $S$ triangulates the given region, we just
need to check that for any other triangle $\Delta$ of $S$, $\Delta\cap
\Delta_{\CC(w)}$ is a face of both if non-empty. However, if this intersection
is not a face, then one of the edges of $\Delta$ must meet the interior
of either $\Delta_{\CC(w)}$ or one of the three adjacent triangles. However,
since the edges of $\Delta$ are contained in non-trivial rays 
of $\foD^{\mathrm{stab}}$ and the interior
of the triangles $\Delta_{\CC(w)}$ contain no non-trivial rays of
$\foD^{\mathrm{stab}}$, this is 
not possible.
Again, there is a special case if $w=w_0$, in which case the
edge with vertices $(-1/2,0), (1/2,0)$ does not have an adjacent
triangle. But the region below this edge also does not contain any
rays, by \cite[Lem~4.10]{Bousseau-scatteringP2-22} (where this region
is split into two regions $\overline{U}_{0,L}^{\mathrm{in}}$ and
$\overline{U}_{0,R}^{\mathrm{in}}$, see Figure \ref{Fig: InitialTriangle}(1)).
This proves (1).

For (2), every non-initial strong exceptional
collection is a twist of a non-initial strong exceptional collection obtained
from mutating $(\CO(1),\CT_{\BP^2},\CO(2))$. Thus
the collection of triangles $\{T^k(\Delta_w)\,|\, w\in V(\CT)\}$
coincides with the collection of triangles 
\[
\{\Delta_{\CC}\,|\,\hbox{$\CC$
is a non-initial strong exceptional collection}\}.
\]
Given the previous paragraph,
this clearly triangulates $R_{\Delta}$.

For (3), it is sufficient to note that the initial
triangle $\Delta_{w_0}$ has $x$-coordinate lying in $[-1/2,1/2]$,
and by Lemma~\ref{lem:slope inequalities},(2), Definition~\ref{Def: V_E},
and Theorem~\ref{thm:mu build},
all vertices of triangles obtained via mutation from $\Delta_{w_0}$ also
have $x$-coordinate lying in $[-1/2,1/2]$.
\end{proof}

The following now follows immediately from Theorem \ref{thm:mu build}(3):

\begin{corollary}\label{Cor:equivofScatPartial}
The two scattering diagrams $\foD^{\mathrm{stab}}$ and
$\foD^{\mathrm{stab}}_{\mathrm{in}}\cup\foD_{\mathrm{discrete}}$
are equivalent in $R_{\Delta}$.
\end{corollary}

\subsection{Determinants of vertices and Markov numbers.}

We recall a Markov triple $(x,y,z)$ is a triple of positive integers satisfying
the equation
\[
x^2+y^2+z^2=3xyz.
\]
Given a Markov triple $(x,y,z)$, any permutation is also a Markov triple.
Also, if $(x,y,z)$ is a Markov triple, then so is $(x,y,3xy-z)$; 
this is \emph{the mutation of $(x,y,z)$ at $z$}. The mutations at $x$ and $y$
are defined similarly. The
Markov tree is obtained by taking as vertices Markov triples up to permutation,
and connecting two Markov triples by an edge
if they are related by permutation and one mutation.
The Markov tree has one univalent vertex, corresponding to $(1,1,1)$,
an adjacent bivalent vertex, corresponding to $(1,1,2)$, also adjacent
to $(1,2,5)$. All remaining vertices, including $(1,2,5)$, are trivalent.

\begin{definition}
\label{def:modified Markov tree}
We define a modified Markov tree $\CT_M$ as follows.
Let $\mathcal{T}'$ be the subtree obtained
by deleting the two vertices $(1,1,1)$ and $(1,1,2)$. Take two copies
of $\mathcal{T}'$ and attach them to a vertex labeled with $(1,1,2)$,
to obtain an infinite binary tree with root given by $(1,1,2)$. We call
this tree $\mathcal{T}_M$. 
\end{definition}

\begin{lemma} \label{Lem: TE to TM}
There is an isomorphism  $\eta_M:\mathcal{T}_E\rightarrow\mathcal{T}_M$ with
the property that for $w\in V(\CT_E)$ with $\CC(w)=(\CE_0,\CE_1,\CE_2)$,
the Markov triple $(x,y,z)$ associated to the vertex $\eta_M(w)$
is a permutation
of $(r(\CE_0),r(\CE_1),r(\CE_2))$. 
\end{lemma}

\begin{proof}
This is well-known, and follows immediately inductively. The base case
is just fact that the 
rank triple of $\CC(w_0)=(\CO(1),\CT_{\CP^2},\CO(2))$ is $(1,2,1)$ and
the induction step is Lemma~\ref{lem:slope inequalities},(1).
\end{proof}

\begin{lemma} 
\label{Lem: degree=rank}
For $w\in V(\CT)$, let $v$ be a vertex of $\Delta_w$, and let $D_v$
(as defined in \eqref{eq:det def})
be calculated using the triangle $\Delta_w$. Then there is an exceptional
bundle $\CE$ such that $V_{\CE}=v$ and $D_v=r(\CE)$. Further, if $v=(x,y)$,
then $D_v$ is the denominator of the rational number $x+3/2$ written
in reduced form. In particular, $D_v$ only depends on $v$.
\end{lemma}

\begin{proof}
We proceed by induction. 
The root $w=w_0$ of $\mathcal{T}$ corresponds to the triangle
of \eqref{eq:Tw0}, and also has vertices $V_{\CO(1)},V_{\CT_{\BP^2}},V_{\CO(2)}$
by Example~\ref{ex:initial triangle}.
The primitive tangent vectors at the vertex $v=(0,1/2)=V_{\CT_{\BP^2}}$ 
are $(1,-1)$ and $(-1,-1)$, $D_v=|(1,-1)\wedge (-1,-1)|=2=
\rank(\CT_{\BP^2})$.
The determinants of the vertices 
$(\pm 1/2,0)$ are similarly calculated to be $1$,
the ranks of $\CO(1)$ and $\CO(2)$. 

For the induction step, suppose that we have shown the statement
for a given triangle $\Delta_w$, $w\in V(\CT)$. Let $w'$ be a child of
$w$, and assume first that $\CC(\eta_E(w'))$ obtained via left mutation.
Thus by Theorem~\ref{thm:mu build}, if $\CC(\eta_E(w))=(\CE_0,\CE_1,\CE_2)$,
then $\Delta_w$ has vertices $V_{\CE_0},V_{\CE_1},V_{\CE_2}$ and
$\Delta_{w'}$ as vertices $V_{\CE_0},V_{\CL_{\CE_1}\CE_2},V_{\CE_1}$.
Consulting Figure \ref{Fig: RL}, we see that the number $D_{V_{\CE_1}}$
is the same when calculated in either $\Delta_w$ or
$\Delta_{w'}$ as it involves the wedge product of the same primitive tangent
vectors. 

On the other hand, we note that in the notation of
the proof of Theorem~\ref{thm:mu build}, by \eqref{eq:ddr2} and \eqref{eq:ddr4}
we have that $3D_{V_{\CE_0}}m_1-m = m_{\CE_1(-3)[1]}$. Thus by
Lemma~\ref{lem:slope inequalities},(6), $3D_{V_{\CE_0}}m_1-m$ is primitive. So 
we may compute $D_{V_{\CE_0}}$ in the triangle $\Delta_{w'}$ as 
\[
|m_1 \wedge (3D_{V_{\CE_0}} m_1-m)| = |m_1\wedge m| = D_{V_{\CE_0}},
\]
the latter computed in $\Delta_w$. Thus $D_{V_{\CE_0}},D_{V_{\CE_1}}$
are unchanged between the two triangles, and it remains
to compute $D_{V_{\CL_{\CE_1}\CE_2}}$. This is then
$| m_2\wedge (3D_{V_{\CE_0}}m_1-m)|$. We note that $m_2\wedge m_1$ and
$m_2\wedge m$ have the same sign with our orientation 
conventions. See Figure \ref{Fig: TreeMarkov} (2).
Thus we obtain $D_{V_{\CL_{\CE_1}\CE_2}}=3D_{V_{\CE_0}}D_{V_{\CE_1}} - 
D_{V_{\CE_2}}=\rank(\CL_{\CE_1}\CE_2)$ by Lemma~\ref{lem:slope inequalities},(1)
and the induction hypothesis. This is sufficient to prove the induction step
for right mutations.

The argument assuming $\CC(\eta_E(w'))$ is obtained via right mutation
of $\CC(\eta(w))$ is the same, this time using \eqref{eq:ddr1} and
\eqref{eq:ddr3}.
\end{proof}

\begin{remark}
In particular, this shows that for a triangle $\Delta_w$ with vertices
$v_0,v_1,v_2$, the triple $(D_{v_0},D_{v_1},D_{v_2})$ forms a Markov triple. See Figure \ref{Fig: TreeMarkov}.
\end{remark}

   \begin{figure}[h]
 \subcaptionbox*{(1) $\mathcal{T}_{M}$}[.59\linewidth]{%
    \includegraphics[width=\linewidth]{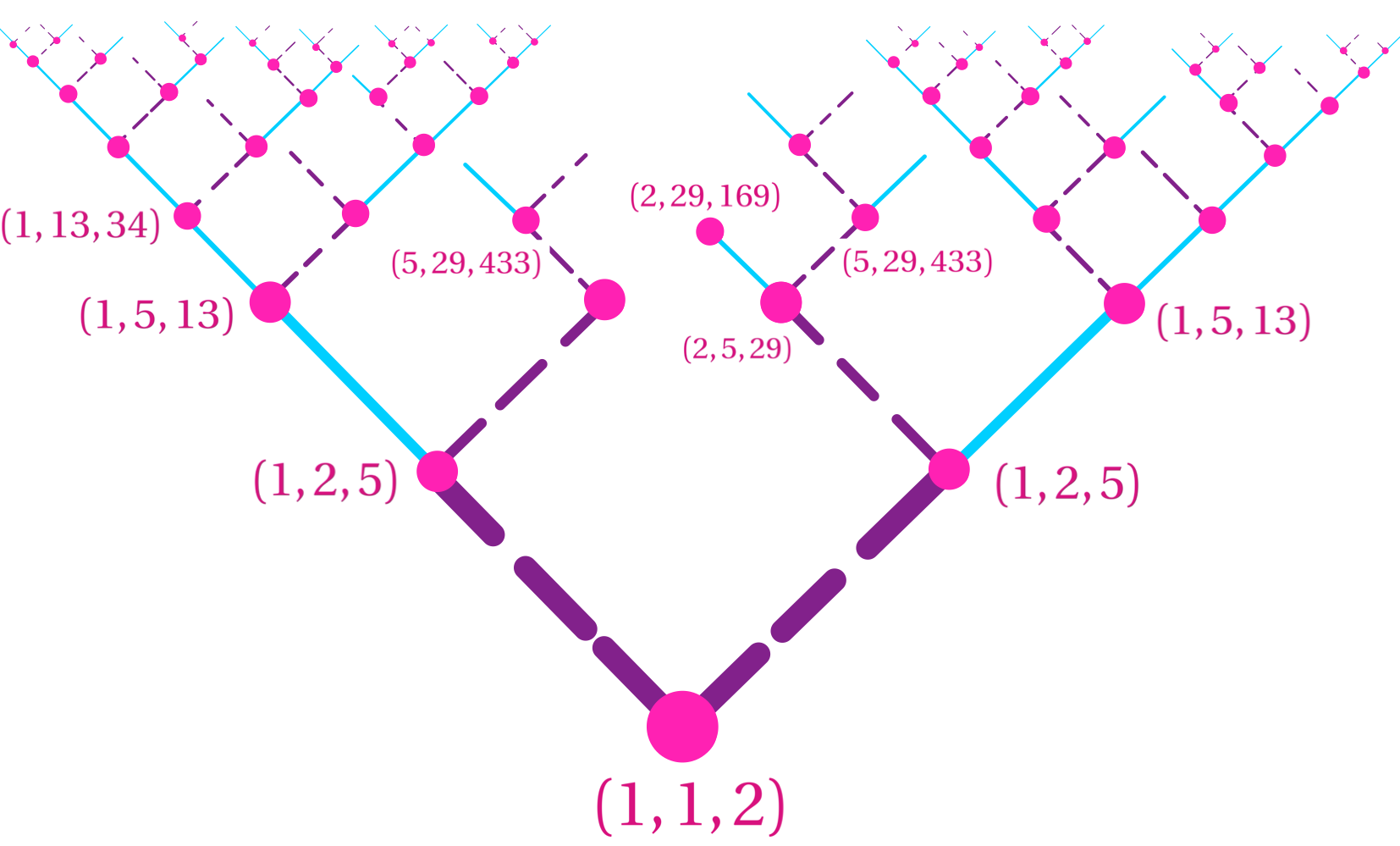}%
  }%
   \hspace{-0.75cm}
  \hskip0.59ex
  \subcaptionbox*{(2)}[.44\linewidth]{%
    \includegraphics[width=\linewidth]{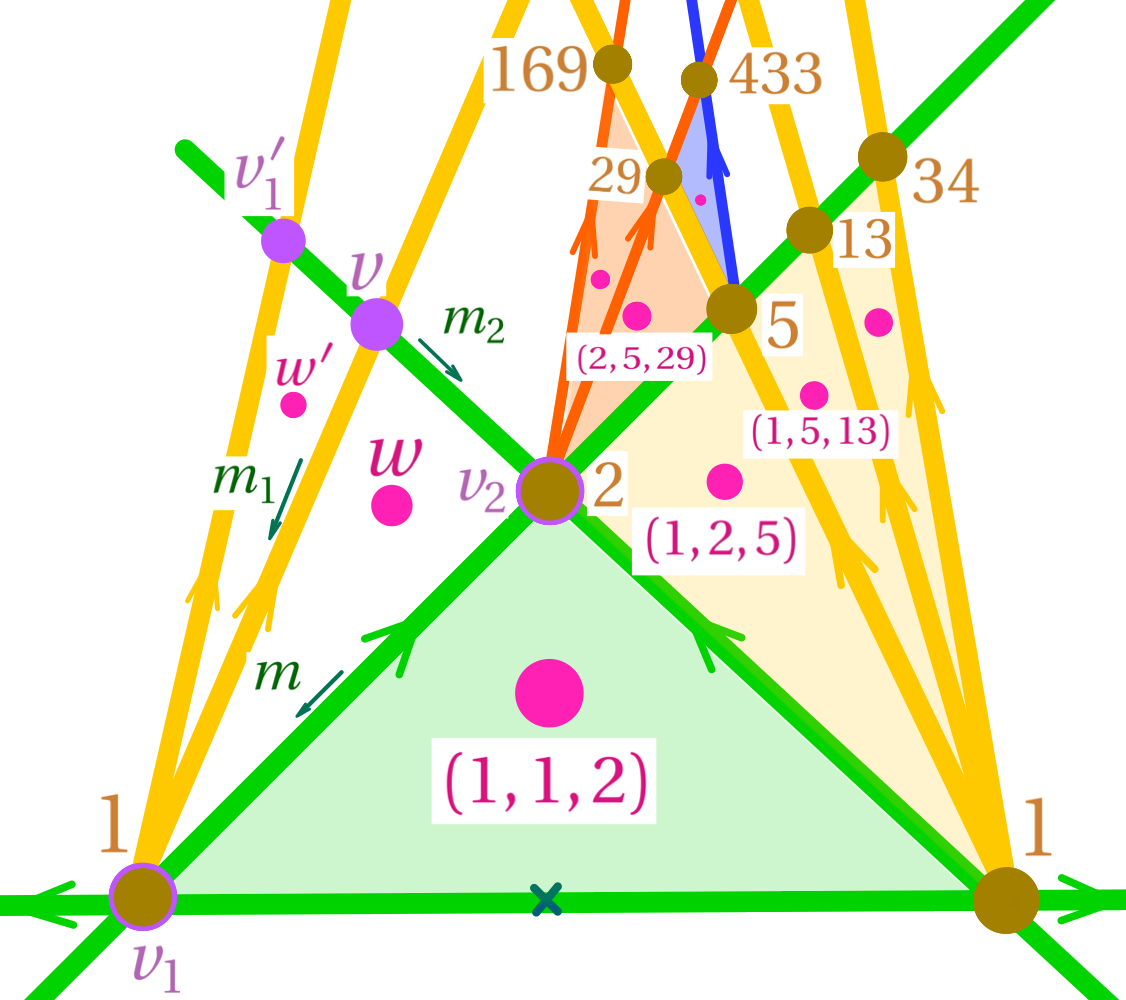}%
  }
  \caption{(1) The corresponding Markov tree $\mathcal{T}_{M}$ to $\CT$ in Figure \ref{Fig:Tree}.
  (2) The corresponding triangles to $\CT_M$: The numbers at the vertices are determinants at each vertex, and each triangle corresponds to a Markov triple given by the determinants at each of its vertices.
}
  \label{Fig: TreeMarkov}
\end{figure}

\begin{prop}
\label{prop:det order}
Let $w$ be a vertex of $\CT$ with $w\not=w_0$. Let $D,D''$ and $D'$ be
the determinants of the incoming, hybrid and outgoing vertices of 
$\Delta_w$ respectively. Then $D>D''>D'$.
\end{prop}

\begin{proof}
From Lemma~\ref{lem:slope inequalities},(1), one sees inductively
that for a strong exceptional triple $\CC=(\CE_0,\CE_1,\CE_2)$,
we have $r(\CE_1)\ge r(\CE_0),r(\CE_2)$, with equality 
if and only if $\CC$ is an initial strong
exceptional triple. Since $V_{\CE_1}$ is always the incoming vertex
of the corresponding triangle $\Delta_{\CC}$, we see that $D>D',D''$.
The ordering of $D'$ and $D''$ is proved inductively by observing,
e.g., from Figure~\ref{Fig: RL}, that if $w_1,w_2$ are the children
of $w$ in $\CT$, then the incoming vertex of $\Delta_w$ becomes
a hybrid vertex of $\Delta_{w_i}$, while the outgoing 
vertex of $\Delta_w$ becomes the outgoing vertex of one of $\Delta_{w_1}$,
$\Delta_{w_2}$, and the hybrid vertex of $\Delta_w$ becomes the outgoing
vertex of the other one of $\Delta_{w_1},\Delta_{w_2}$.
\end{proof}

\subsection{The structure of scattering at discrete points of 
$\text{\:}\foD^{\mathrm{stab}}$}
\label{subsec:discrete scat structure}
The discrete region of $\foD^{\mathrm{stab}}$ consists of triangles, and
each vertex of each triangle has two rays coming into the vertex. This
will always produce scattering of a standard, now much-studied form. We
review this here.

We first recall a slightly different notion of scattering diagram in 
$M_{\BR}$ than that of \S\ref{subsec:Bousseau scat}. 
A ray or line is of the form $(\fod,f_{\fod})$ with
$\fod$ as in Definition~\ref{Def: stabilityRay}, but $f_{\fod}(z^{m_{\fod}},
t_1,t_2)
\in \BQ[M][[t_1,t_2]]$ is a formal power series which satisfies 
$f_{\fod}\equiv 1 \bmod (t_1,t_2)$ and
$f_{\fod}\equiv 1\mod z^{m_{\fod}}$. 
Given a scattering diagram $\foD$ consisting
of such rays and a path $\gamma$ as in 
Definition~\ref{def:path ordered}, we may define the path-ordered
product $\theta_{\gamma,\foD}:\BQ[M][[t_1,t_2]]\rightarrow \BQ[M][[t_1,t_2]]$
by taking a composition of automorphisms associated to the rays or lines
crossed by $\gamma$. If at time $t_0$ the path $\gamma$ crosses 
$(\fod,f_{\fod})$, and $n\in N$ is the unique primitive element 
annihilating
$m_{\fod}$
and with $\langle n,\gamma'(t_0)\rangle <0$, then the associated
automorphism is given by $z^{m'}\mapsto z^{m'}f_{\fod}^{\langle n,m'\rangle}$.

Given a scattering diagram $\foD$, it is \emph{consistent} if $\theta_{\gamma,
\foD}$ is the identity for any loop for which $\theta_{\gamma,\foD}$ is
defined. If $\foD$ is a scattering diagram, there is a consistent scattering
diagram $\mathsf{S}(\foD)$ containing $\foD$ and such that 
$\mathsf{S}(\foD)\setminus\foD$ consists only of rays. Again, this is
\cite[Thm.~6]{KS-06} or \cite[Thm.~1.4]{GPS-10}.

\begin{definition}\label{Def: Rj}
Let $D$ be a positive integer. Define the sequence of integers $R_i(D)$
for $i\ge 0$ recursively by 
\[
\hbox{$R_0(D)=0, \quad R_1(D) = 1, \quad R_i(D)= 3DR_{i-1}(D)-R_{i-2}(D)$
for $i\ge 2$.}
\]
Define
\[
r_i(D):= R_{i+1}(D)/R_i(D)
\]
and 
\[
r_{\infty}(D):=\lim_{i\rightarrow \infty} r_i(D),
\]
which is easily seen to be
\[
r_{\infty}(D)={3D+\sqrt{9D^2-4}\over 2}.
\]
\end{definition}

We then have the following now well-studied example.

\begin{definition}
\label{def:notation for tims diagram}
Let $D$ be a positive integer. We define
\[
\foD'_{\mathrm{in}}:=\left\{\left(\BR (1,0), (1+t_1z^{(1,0)})^{3D}\right), 
\left(\BR (0,1), (1+t_2z^{(0,1)})^{3D}\right)\right\}.
\]
Define 
\[
m'_{i,1}=(R_{i+1}(D),R_i(D)),\quad m'_{i,2} = (R_i(D), R_{i+1}(D)),
\]
for $i\ge 1$, and define the ray
\[
\fod'_{i,j} := \left(\BR_{\le 0}m'_{i,j}, (1+t_1^{R_i(D)}t_2^{R_{i+1}(D)}
z^{m'_{i,j}})^{3D}\right),
\]
for $i\ge 1$, $j=1,2$.
\end{definition}

\begin{lemma}
\label{lem:basic scattering}
With notation as in Definition \ref{def:notation for tims diagram},
we may write
\[
{\mathsf S}(\foD')= \foD'_{\mathrm{in}} \cup \{\fod'_{i,j}\,|\, i\ge 1, j=1,2\}
\cup \foD'_{\mathrm{dense}},
\]
where $\foD'_{\mathrm{dense}}$ is a scattering diagram containing one 
non-trivial ray with support $\BR_{\le 0}(a,b)$ for $a,b$ relatively prime 
precisely when $a,b>0$ and
\begin{equation}
\label{eq:dense range}
r_{\infty}(D)^{-1}< b/a < r_{\infty}(D).
\end{equation}
None of the $\fod_{i,j}$ have slope in this range. Finally,
if $a,b$ are relatively prime and satisfy the inequality 
\eqref{eq:dense range} then the function $f$
attached to the ray $\BR_{\le 0}(a,b)$ in $\foD'_{\mathrm{dense}}$ takes
the form
\[
f=\prod_{k=1}^{\infty} (1+t_1^{ka}t_2^{kb}z^{(ka,kb)})^{c_{ka,kb}}
\]
for $c_{ka,kb}$ strictly positive integers.
\end{lemma}

\begin{proof}
This follows from \cite[Thm.~7, Lemmas~5.2,
Lemma~5.3]{Gross-Pandharipande-10}  and \cite[Thm.~1.4]{Graefnitz-Luo2023}.
\end{proof}

\begin{figure}[h]
 \centering
    \includegraphics[width=9.2cm]{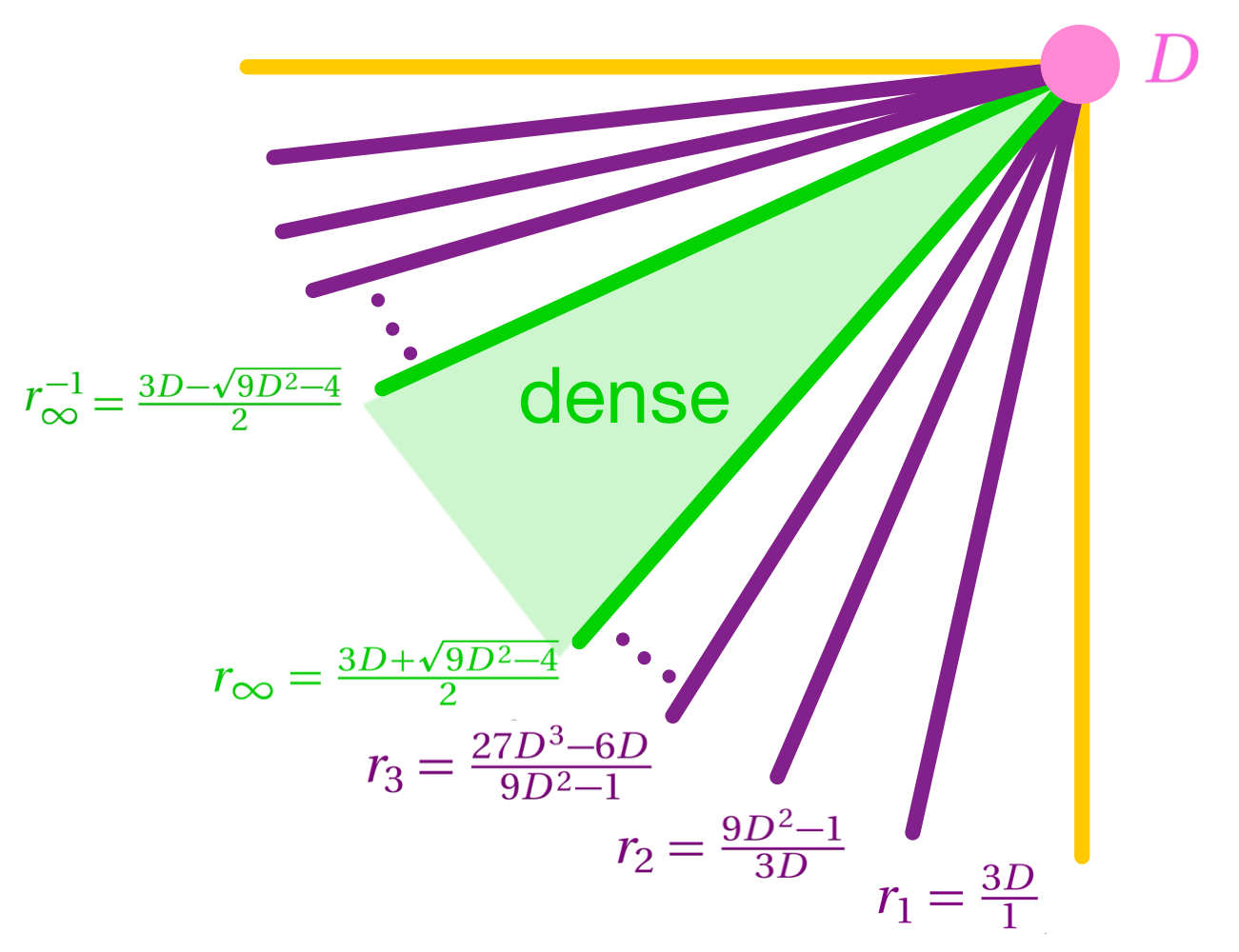}
  \caption{The scattering diagram of Lemma~\ref{lem:basic scattering}.
  }
  \label{Fig: dense region}
\end{figure}

We may then use this to analyze the scattering which will happen in
$\foD^{\mathrm{stab}}$. Assume given a stability scattering diagram $\foD$
consistent at a point
$\sigma\in M_{\BR}$, with local scattering diagram $\foD_{\sigma}$
as in Definition~\ref{def:local scatter}.
Let $m_1,m_2\in M$ be primitive and linearly independent. 
Assume that there are only two lines (as opposed to rays) in 
$\foD_{\sigma}$, with opposite vectors $m_1,m_2\in M$, primitive and
linearly independent. Further
assume these lines are of the form $(\fod_i,H^{\sigma}_i)$, 
where
\[
\fod_i=\BR m_i, \quad H_i = -\li_2(-t^{\varphi_{\sigma}(m_i)}z^{m_i}).
\]
Let $D=|m_1\wedge m_2|$, and set for $i\ge 0$,
\[
(a_{i,1},b_{i,1}) = (R_{i+1}(D),R_i(D)),\quad
(a_{i,2},b_{i,2}) = (R_i(D),R_{i+1}(D)),
\]
and 
\begin{equation}
\label{eq:mi1}
m_{i,j} = a_{i,j} m_1 + b_{i,j} m_2
\end{equation}
for $i\ge 0$, $j=1,2$.
By the definition of $R_i(D)$, it follows that
\begin{align}
\label{eq:recursive mij}
\begin{split}
m_{0,1}= m_1, \quad  m_{1,1}=3Dm_1+m_2, \quad & m_{i+1,1}
= 3D m_{i,1}-m_{i-1,1}\\
m_{0,2}= m_2, \quad  m_{1,2}=3Dm_2+m_1, \quad & m_{i+1,2}
= 3D m_{i,2}-m_{i-1,2}
\end{split}
\end{align}
Define
\[
\fod_{i,j}:=\left(-\BR_{\ge 0}m_{i,j}, 
-\li_2(-t^{\varphi_{\sigma}(m_{i,j})}z^{m_{i,j}})\right).
\]

\begin{lemma}
\label{lem:basic scattering2}
In the above situation, we may write
\begin{equation}
\label{eq:D from D'}
\foD_{\sigma}= \{(\fod_1,H^{\sigma}_1),(\fod_2,H^{\sigma}_2)\} 
\cup \{\fod_{i,j}\,|\, i\ge 1, j=1,2\}
\cup \foD_{\sigma,\mathrm{dense}},
\end{equation}
where $\foD_{\sigma,\mathrm{dense}}$ is a scattering diagram containing one
non-trivial ray with support 
$-\BR_{\ge 0} (am_1+bm_2)$ for $a,b$ relatively prime precisely when
$a,b>0$ and
\begin{equation}
\label{eq:slope range}
r_{\infty}(D)^{-1}< b/a < r_{\infty}(D).
\end{equation}
None of the $\fod_{i,j}$ have slope in this range.
\end{lemma}

\begin{proof}
This will follow from Lemma~\ref{lem:basic scattering} via
the \emph{change of lattice trick},
see \cite[Proposition C.13, Step IV]{GHKK} or 
\cite[Section 2.3]{Graefnitz-Luo2023}.

Let $\foD_{\mathrm{in}}=\{(\fod_1,H^{\sigma}_1),(\fod_2,H_2^{\sigma})\}$, 
recalling that
$H_i^{\sigma}=-\li_2(-t^{\varphi_{\sigma}(m_i)}z^{m_i})$.
Note that if we
instead define $\foD_{\mathrm{in}}'=\{(\fod_1,H_1'),
(\fod_2,H_2'))\}$ with $H_i'=-\li_2(-t_iz^{m_i})$,  working in
$\BQ[M][[t_1,t_2]]$ rather than $\BQ[M][t^{\BR_{\ge 0}}]$,
we may apply \cite[Thm.~6]{KS-06} or 
\cite[Thm.~1.4]{GPS-10} to obtain a consistent diagram $\foD'_{\sigma}$ from
$\foD_{\mathrm{in}}'$. We may then obtain
$\foD_{\sigma}$ from $\foD'_{\sigma}$ by applying the ring homomorphism
$\BQ[M][[t_1,t_2]]\rightarrow \BQ[M][t^{\BR_{\ge 0}}]$ given by
$t_i\mapsto t^{\varphi_{\sigma}(m_i)}$.

We then further replace
$H^{\sigma}_i$ with the corresponding function
$f_i$ in \eqref{eq:H to f}, 
in order to convert $\foD'_{\mathrm{in}}$ into
a scattering diagram of the sort used in Lemma~\ref{lem:basic scattering}.
Note that
$f_i= (1+t_iz^{m_i})^3$, so that
$\theta_{H_i'}=\bar\theta_{f_i}$ 
in the notation of the remark.

We now apply the change of lattice trick. We replace the
lattice $M$ with the sublattice $M'$ generated by
$m_1$ and $m_2$. The index of this sublattice in $M$
is $D$. The dual lattice $N'$ of $M'$ is
a superlattice of $N$. Note that if $n_i\in N$
is a primitive vector with $\langle n_i,m_i\rangle = 0$,
$\langle n_i,m_{i'}\rangle >0$ for $\{i,i'\}=\{1,2\}$, then one
checks easily, e.g., by writing $m_i, n_i$ in the standard basis 
for $M$ and $N$,
that $n_i = D n_i'$ for $n_i'\in N'$ primitive. 

Consider for $i=1,2$ the line
$\fod_i':=\big(\BR m_i, (1+t_iz^{m_i})^{3D}\big)$ viewed as a line in
a scattering diagram in the lattice $M'$. Because
$n_i=Dn_i'$, in fact 
the wall-crossing automorphism of $\BQ[M][[t_1,t_2]]$
associated to $\fod_i$ agrees, after restriction to $\BQ[M'][[t_1,t_2]]$,
with the wall-crossing automorphism associated to $\fod_i'$.

We may now consider the initial scattering diagram $\foD'_{\mathrm{in}}$
consisting of the
lines $\fod_1',\fod_2'$, again with respect to the lattice $M'$.
Let $\foD'=\mathsf{S}(\foD'_{\mathrm{in}})$. This is the scattering
diagram analyzed in Lemma~\ref{lem:basic scattering}. We just need to 
check that the rays described in Lemma~\ref{lem:basic scattering} for
the lattice $M'$ then correspond to the rays described in the current
lemma.

Using the notation $m_{i,j}$ of the statement, we note that if
we write an element of $M'$ as vector using the basis $m_1,m_2$,
in fact $m_{i,j}\in M$ agrees with $m'_{i,j}\in M'$ of
Lemma~\ref{lem:basic scattering} under the inclusion
$M'\subseteq M$. Thus the rays $\fod'_{i,j}$ and $\fod_{i,j}$ have
the same support in $M'_{\BR}=M_{\BR}$. 

We may then recover $\foD$ from $\foD'$ by modifying each function attached
to a ray $\fod'$ in $\foD'$ to give a ray $\fod$ so that the automorphism
associated to $\fod'$ is the restriction 
to $\BQ[M'][[t_1,t_2]]$ of the automorphism associated to $\fod$, 
see \cite[Prop.~2.16]{Graefnitz-Luo2023}. Thus it only remains to check
that $\fod'_{i,j}$ induces the same automorphism of $\BQ[M'][[t_1,t_2]]$
as $\fod_{i,j}$ does. For this, note that in $N'$, a primitive
vector annihilating $m_{i,j}$ is $n'_{i,j}=(-a_{i,j} n'_2 + b_{i,j} n'_1)/\gcd(a_{i,j},
b_{i,j})$. However, say
for $j=1$, we have
\[
\gcd(a_{i,1},b_{i,1})=\gcd(3D a_{i-1,1}-b_{i-1,1}, a_{i-1,1})=
\gcd(a_{i-1,1},b_{i-1,1})=\cdots=\gcd(a_{0,1},b_{0,1})=1.
\]
Similarly by induction, note that for each $i$, either $3D|a_{i,j}$
and $\gcd(3D,b_{i,j})=1$, or $3D|b_{i,j}$ and $\gcd(3D,a_{i,j})=1$, with
the two choices depending on the parity of $i$. Without loss of generality,
assuming the former
of the two choices, we see that $-a_{i,j} n'_2 + b_{i,j} n'_1 \equiv
b_{i,j} n'_1 \mod N$, and thus this normal vector 
represents an element of order $D$ in $N'/N$.
From this it follows that 
the primitive normal vector to $m_{i,j}$ in $N$
must be $D n'_{i,j}$. 
Hence indeed $\fod_{i,j}$ and $\fod'_{i,j}$ induce
the same automorphism on $\BC[M'][[t_1,t_2]]$. 

By following a similar process for dense rays $(\fod',f_{\fod'})\in
\foD'_{\mathrm{dense}}$ to obtain $(\fod,f_{\fod})$ inducing the
same automorphism of $\BC[M'][[t_1,t_2]]$, we obtain a scattering diagram 
$\foD_{\sigma,\mathrm{dense}}$. 
We note that the relationship between $f_{\fod'}$ and $f_{\fod}$ may
not be as simple as described above, but we do not need the precise
form. This then defines a scattering diagram $\foD$ as in \eqref{eq:D from D'}
with the property that, for a loop $\gamma$ around the origin,
$\theta_{\gamma,\foD}$ and $\theta_{\gamma,\foD'}$ agree on 
$\BC[M'][[t_1,t_2]]$, i.e., are both the identity. Since $\theta_{\gamma,\foD}$
is a ring automorphism, this implies that for $m\in M$,
$\theta_{\gamma,\foD}(z^m) = \xi z^m$ for some root of unity $\xi\in\BC$.
Since $\theta_{\gamma,\foD}$ is the identity modulo $\mathfrak{m}$,
we conclude $\xi=1$ always. Thus $\theta_{\gamma,\foD}$ is the identity.
\end{proof}

The following formula will be useful later:

\begin{prop}[Closed formula for $R_j(D)$]\label{Prop: Rj}

Let $m =\lfloor j/2\rfloor$.
Then 
    \begin{align}\label{eq:Rk}
    R_{j}(D)=\sum_{l=0}^{m}(-1)^l{j-l-1 \choose l}\big(3D\big)^{j-2l-1}.
\end{align}
Here we use the convention that ${m-1\choose m} =0$.
\end{prop}

\begin{proof}
    We prove this by induction. For $j=0$ and $1$, the formula obviously holds. Now, we assume that the formula holds for $j\leq p-1$, and we want to show that it holds for $j=p$ as well. 

Let $m':=\lfloor (p-1)/2\rfloor$ and $m''=\lfloor (p-2)/2\rfloor$. Then
$m''=m-1$ and either $m=m'$ or $m'=m''$ depending on whether $p$ is odd
or even respectively. We get
\begin{align*}
R_{p}(D)={}& (3D)R_{p-1}(D)-R_{p-2}(D)\\
={}&\sum_{l=0}^{m'}(-1)^l{p-l-2 \choose l}\big(3D\big)^{p-2l-1}-\sum_{l=0}^{m''}(-1)^l{p-l-3 \choose l}\big(3D\big)^{p-2l-3}\\
={}&\sum_{l=0}^{m'}(-1)^l{p-l-2 \choose l}\big(3D\big)^{p-2l-1}-\sum_{l=1}^{m}(-1)^{l-1}{p-l-2 \choose l-1}\big(3D\big)^{p-2l-1},
\end{align*}
where in the last summation we have shifted the index of summation by $1$.
Using that ${p-l-2 \choose l-1}={p-l-1 \choose l}-{p-l-2 \choose l}$ in the
last term, 
the above yields
\begin{align*}
 &(3D)^{p-1}+\sum_{l=1}^{m'}(-1)^l{p-l-2 \choose l}\big(3D\big)^{p-2l-1}\\
-&\sum_{l=1}^{m}(-1)^{l-1}{p-l-1 \choose l}\big(3D\big)^{p-2l-1}+\sum_{l=1}^{m}(-1)^{l-1}{p-l-2 \choose l}\big(3D\big)^{p-2l-1}.
\end{align*}
The first and third sums cancel, noting that in the case that $m'=m-1$,
the $l=m$ contribution in the third term vanishes. Thus we obtain
\[
(3D)^{p-1}+\sum_{l=1}^{m}(-1)^{l}{p-l-1 \choose l}\big(3D\big)^{p-2l-1}=\sum_{l=0}^{m}(-1)^{l}{p-l-1 \choose l}\big(3D\big)^{p-2l-1},
\]
as claimed.
\end{proof}

\begin{corollary}\label{cor: RjIncreasing} For each $j$, we have $$R_{j+1}(D)>R_{j}(D).$$  
\end{corollary}

\begin{theorem}
\label{thm:local structure}
Let $v\in M_{\BR}$ be either an incoming vertex of a triangle 
$\Delta_w$ or any vertex of $\Delta_{w_0}$. 
Then the local scattering diagram $\foD^{\mathrm{stab}}_v$ contains 
precisely two lines,
$(\fod_1,H_1)$, $(\fod_2,H_2)$,
corresponding to the two rays of $\text{\:}\foD^{\mathrm{stab}}_{\mathrm{in}}
\cup\foD_{\mathrm{discrete}}$ containing the two edges of $\Delta_w$
containing $v$. It is equivalent to the scattering diagram
$\foD_{\sigma}$ described in Lemma~\ref{lem:basic scattering2} with $D=D_v$. Further,
the rays of $\text{\:}\foD^{\mathrm{stab}}_v$ arising from elements of
$\text{\:}\foD_{\mathrm{discrete}}$ are precisely the rays $\{\fod_{i,j}\,|\,
i\ge 1, j=1,2\}$ of Lemma~\ref{lem:basic scattering2}.
\end{theorem}

\begin{proof}
We first observe that the only lines in $\foD^{\mathrm{stab}}_v$ have
support as advertised. 
Note that if $w_1,w_2$ are the two children of $w$, then
locally near $v$, $M_{\BR}\setminus (\Delta_{w_1}\cup\Delta_{w_2}\cup
\Delta_w)$ is a convex cone, see Figure \ref{Fig: Tw}. Thus any ray $\fod$
containing $v$ for which $\fod$ is not an endpoint must intersect
one of these three triangles. However, since there are no rays
of $\foD^{\mathrm{stab}}$ intersecting the interiors of these triangles
by Theorem \ref{thm:bousseau generalization}, the only choice is
that $\fod$ contains one of the two edges of $\Delta_w$ containing $v$.

If $v$ is instead one of the hybrid vertices of $\Delta_{w_0}$, then it is
the point of intersection of $\fod_{n-1}^+$ and $\fod_n^-$ for $n=0$ or $n=1$.
The argument of the
first paragraph also applies in this case, as the regions
$\overline{U}^{\mathrm{in}}_{0,L},\overline{U}^{\mathrm{in}}_{0,R}$
(see Figure~\ref{Fig: InitialTriangle})
below these
two rays is empty by \cite[Lem.~4.10]{Bousseau-scatteringP2-22}.

In either case, we know from Theorem \ref{thm:mu build} that the wall functions
attached to these two lines are as given in Lemma~\ref{lem:basic scattering2}.
Given this, by uniqueness of the construction of scattering diagrams
$\foD_v^{\mathrm{stab}}$ must be equivalent to the scattering diagram
given in Lemma~\ref{lem:basic scattering2}. 

The only thing remaining to
check is that all the rays $\{\fod_{i,j}\,|\,i\ge 1, j=1,2\}$
are accounted for in $\foD_{\mathrm{discrete}}$. This follows, however,
from the construction of $\foD_{\mathrm{discrete}}$. Assuming that
$v$ is an incoming vertex of $\Delta_w$, we obtain all the discrete
rays with endpoint $v$ as follows. First one chooses one of the children
$w_1$
of $w$ in $\CT$.
One then repeatedly follows hybrid edges in $\CT$ to obtain
an infinite chain of vertices $w_2,w_3,\cdots$. Then $\fod_{w_j}-v$
coincides with $\fod_{i,j}$ for some $i=1$ or $2$. Which we get depends
on which child of $w$ was chosen. This follows by
comparing \eqref{eq:dwj def} and \eqref{eq:recursive mij}.

If instead $v$ is one of the hybrid vertices of $\Delta_{w_0}$, the
situation is slightly more complicated. Without loss of generality, 
take $v=(-1/2,0)$. Then by first choosing the child $w_1$ of $w_0$
which corresponds to mutating at the vertex $(1/2,0)$ and then following
a path of outgoing edges, we obtain the rays $\fod_{i,j}$ for one choice
of $i$. To get the other choice of $i$, we follow the same procedure at
$(1/2,0)$ instead of $(-1/2,0)$ and then apply $T^{-1}$.

  \begin{figure}[h]
    \centering
    \includegraphics[width=7.8cm]{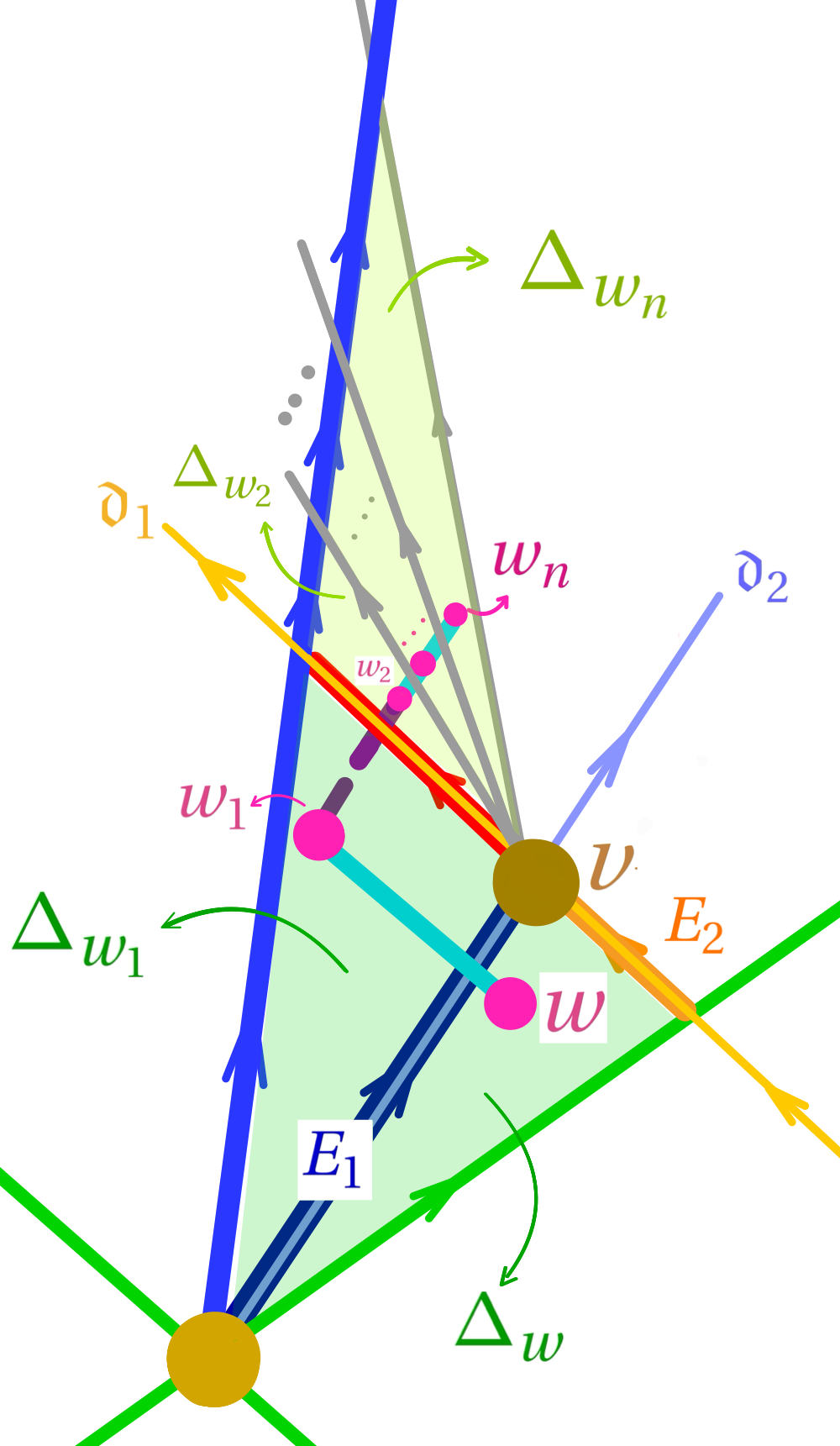}
    \caption{$\Delta_{w_1},\ldots,\Delta_{w_n}$ together with $\Delta_{w}$ with a common vertex $v$. Also, the corresponding part of $\CT$ is pictured here. 
} 
    \label{Fig: Twn}
\end{figure}

\end{proof}

\begin{definition}[ $\foD_{\mathrm{dense}}$]\label{DDense}
We define $\foD_{\mathrm{dense}}$ to be the set of rays of 
$\foD^{\mathrm{stab}}$ which are not elements of
$\foD^{\mathrm{stab}}_{\mathrm{in}}\cup \foD_{\mathrm{discrete}}$
but have endpoint a vertex of some triangle $T^k(\Delta_w)$.
\end{definition}

It follows from Theorem~\ref{thm:local structure} 
that the local scattering diagram $\foD_{\mathrm{dense},v}$ agrees with the 
diagram $\foD_{v,\mathrm{dense}}$ of Lemma~\ref{lem:basic scattering2}.

\section{Diamonds}\label{Section: diamonds}
In this brief section, we provide definitions related to the structure of diamonds.

\begin{definition}
\label{Def:dense cone}
Let $w$ be a vertex of $\mathcal{T}$ and $v$ the incoming vertex of
$\Delta_w$. Then $v$ is the intersection of two rays of 
$\{\fod_{-1}^+,\fod_1^-\}\cup\foD^0_{\mathrm{discrete}}$,
say $\fod_1,\fod_2$, with $v$ lying in the interior of each, see
Theorem~\ref{thm:local structure}. If $m_1,m_2$ are the
opposite vectors of $\fod_1,\fod_2$, define the cone
\[
C_v:=\operatorname{closure}\left(\left\{v-(am_1+bm_2)\,\Big|\, a,b\in\BR_{>0}, \quad
{3D_v-\sqrt{9D_v^2-4}\over 2}<{b\over a} <
{3D_v+\sqrt{9D_v^2-4}\over 2}\right\}\right).
\]
We can then define $C_{T^k(v)}=T^k(C_v)$ for any
$k\in\BZ$. 

If instead $v$ is the intersection of
$\fod^+_n$ and $\fod^-_{n+1}$ for some $n$,\footnote{We note that
such a vertex is never an incoming vertex of a triangle.}
then we may define
\[
C_v := \operatorname{closure}\left(\left\{a (1,-n) + b (-1, (n+1))\,\Big|\,
a,b\in \BR_{>0},
\quad
{3-\sqrt{5}\over 2} <{b\over a}<{3+\sqrt{5}\over 2}\right\}\right).
\]
We call such a vertex an \emph{initial vertex}. 

See Figure \ref{Fig: Cone}.
\end{definition}

\begin{remark}
By Theorem~\ref{thm:local structure} every ray of rational slope with endpoint
$v$ and contained in $C_v$ is in fact a non-trivial ray of
$\foD^{\mathrm{stab}}$.
\end{remark}

\begin{lemma}
\label{lem:prince}
\begin{enumerate}
\item
Let $v$ be an incoming vertex of a triangle $\Delta_w$, with $w_1,w_2$ the
children of $w$ in $\mathcal{T}$. Then
$\fod_{w_1}\cap\fod_{w_2}$ is non-empty and lies in $C_v$. See Figure \ref{Fig: Cone}.
\item
If $v$ is an initial vertex, the intersection of
$\text{\:}\fod_n^+$ and $\text{\:}\fod_{n+1}^-$,
then (1) is still true after replacing 
$\text{\:}\fod_{w_1},\fod_{w_2}$ with $\fod_{n-1}^+,\fod_{n+2}^-$ respectively.
\end{enumerate}
\end{lemma}

\begin{proof}
For (1), let
$(\CE_0,\CE_1,\CE_2)$ be the strong exceptional triple associated to $\Delta_w$.
By \eqref{eq:ddr1},\eqref{eq:ddr2},\eqref{eq:ddr3}, and \eqref{eq:ddr4},
we have (see Figure~\ref{Fig: Cone}), with suitable ordering of $w_1,w_2$,
\[
\fod_{w_1}=V_{\CE_2}-\BR_{\ge 0} m_{\CE_1},
\quad \fod_{w_2} = V_{\CE_0}-\BR_{\ge 0} m_{\CE_1(-3)[1]}.
\]
Since $m_{\CE_1(-3)[1]}=-m_{\CE_1(-3)}=\big(-r(\CE_1(-3)),d(\CE_1(-3))\big)$, 
we see that $\fod_{w_1}$ heads to the left and $\fod_{w_2}$ heads to the
right, as depicted in Figure \ref{Fig: Cone}. This makes it clear that
these two rays intersect. Suppose their intersection point does not lie
in $C_v$. In this case, they must intersect at a point contained in
the edge of some triangle $\Delta_{w'}$ with $w'$ a descendent of $w$
and with the property that
one of $\fod_{w_1}$ or $\fod_{w_2}$ 
contains this edge. 
However, this is impossible by Theorem~\ref{thm:bousseau generalization}.

We can check (2) by hand, say for $n=0$, where $C_v$
is the cone based at $v=(1/2,0)$ generated by the two vectors 
$\big((-1+\sqrt{5})/2, (3-\sqrt{5})/2 \big)$ and
$\big((-1-\sqrt{5})/2, (3+\sqrt{5})/2 \big)$. These have
slope $\mu_1=(\sqrt{5}-1)/2>0$ and $\mu_2=-(\sqrt{5}+1)/2<0$ 
respectively. The intersection of the rays $\fod_{-1}^+$ and $\fod_2^-$ is $(1/2,1)$,
which lies vertically over $v$, and hence
clearly lies in $C_v$. 

 \begin{figure}[h]
    \centering
    \includegraphics[width=9.4cm]{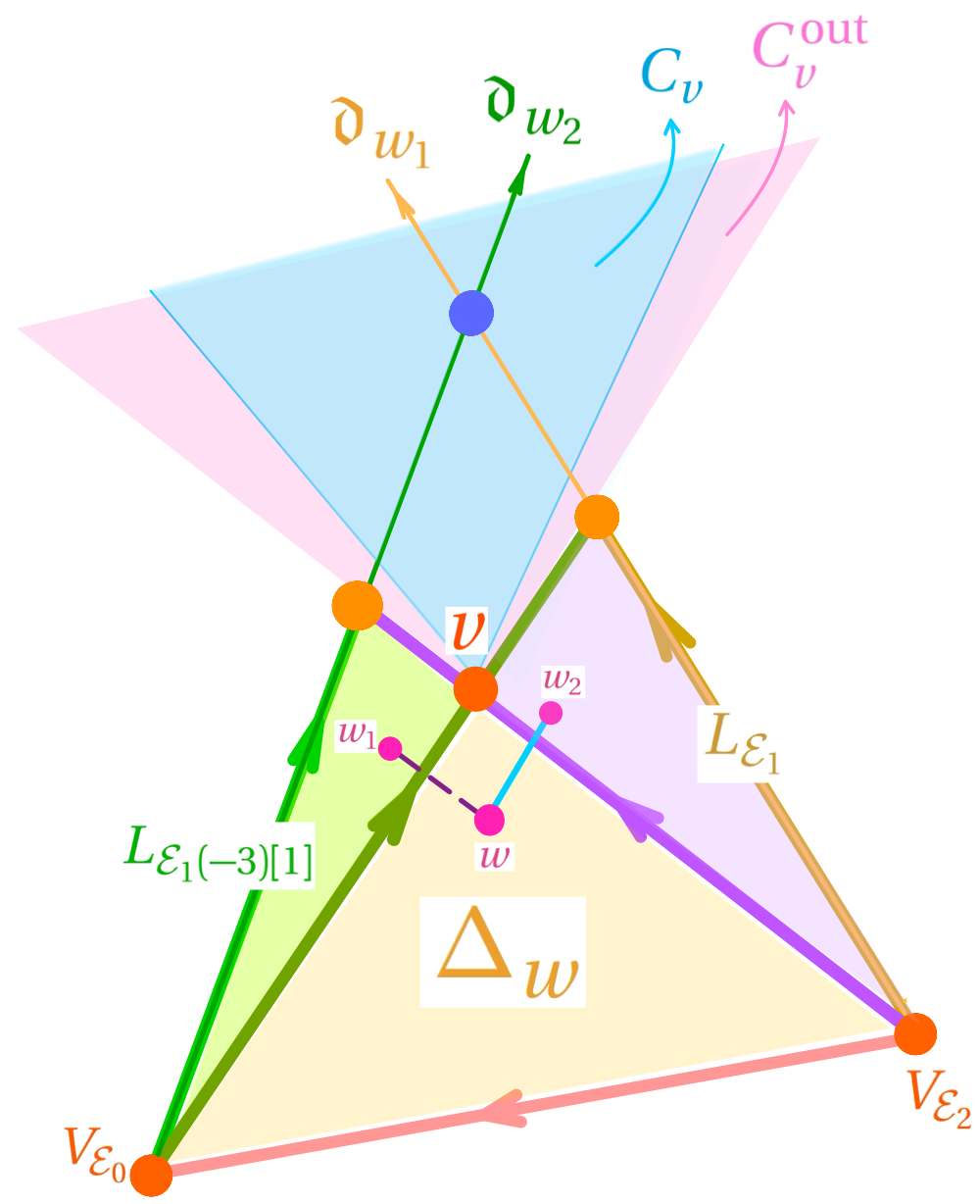}
    \caption{}
    \label{Fig: Cone}
\end{figure}
\end{proof}

\begin{definition}[Dense diamond]\label{def: diamond}
If $v$ is not an initial vertex, 
we define the \emph{(dense) diamond based at $v$ (or at $v$)}, $\Diamond_v$, to be, in the notation of Lemma~\ref{lem:prince}, the
closure of the connected component of $C_v\setminus(\fod_{w_1}\cup
\fod_{w_2})$ containing the vertex $v$. If $v$ is an initial vertex,
the intersection of $\fod_n^+$ and $\fod_{n+1}^-$, 
we instead take $\Diamond_v$ to be the closure of the
connected component of $C_v\setminus(\fod_{n-1}^+\cup \fod_{n+2}^-)$.
\end{definition}

In other words, for any vertex $v$, we can define the dense diamond at $v$ as a quadrilateral with sides
as follows. The lower sides are the two limits of the discrete rays from
the left side and right side at $v$; these have an irrational slope.
The upper right and left sides, which we call the \textit{roofs} of
the diamond, are $\fod_{w_1}$ and $\fod_{w_2}$ in the notation above
if $v$ is not an initial vertex, and are $\fod_{n-1}^+$ and $\fod_{n+2}^-$
if $v$ is an initial vertex.
We denote the diamond based at $v$ by $\Diamond_v$ or 
$\Diamond$ (when there is no ambiguity about the vertex).

\begin{definition}[Outer and initial diamond]\label{Def:OuterInitial}
    Let $v$ be an incoming vertex of a triangle $\Delta_w$. Define the
cone 
\begin{equation}
\label{eq:outer cone}
C_v^{\mathrm{out}}:=
v-(\BR_{\ge 0}m_1+\BR_{\ge 0}m_2),
\end{equation}
where $m_1,m_2$ are
tangent vectors to the edges of $\Delta_w$ containing $v$ oriented away from $w$ as usual.
We define the \emph{outer diamond
based at $v$ (or at $v$)}, $\Diamond_v^{\mathrm{out}}$, to be the closure of the
connected component of $C_v^{\mathrm{out}}\setminus (\fod_{w_1}\cup
\fod_{w_2})$ containing the vertex $v$. 

Similarly, if $v$ is an initial vertex, the intersection of 
$\fod_n^+$ and $\fod_{n+1}^-$, take $m_1,m_2$ to be the opposite vectors
of $\fod_n^+$ and $\fod_{n+1}^-$ and define
$C_v^{\mathrm{out}}$ as in \eqref{eq:outer cone}. Then the \emph{outer
diamond based at $v$} in this case is the closure of the connected component
of $C_v^{\mathrm{out}}\setminus(\fod_{n-1}^+\cup\fod_{n+2}^-)$ containing
the vertex $v$.
We call this diamond an \textit{initial diamond} and denote it by $\Diamond^{\mathrm{init}}$ or  $\Diamond_v^{\mathrm{init}}$.

Alternatively, an initial diamond $\Diamond^{\mathrm{init}}$ is a quadrilateral bounded by the four lines $L_{\CO(d)}$ for $d\in\{m-1,m,m+1,m+2\}$. This is an example of an outer diamond.
See Figure \ref{Fig: Diamonds}.

For $i=1,2$, we call the rays $\fod_{i}$ with direction vectors $m_i$, the left and right \textit{generators} (or the left and right \textit{generating rays}) of the outer diamond. The generators generate any other rays (dense or discrete) at the base vertex $v$, i.e., any such ray has direction vector
a positive linear combination of the direction vectors of
$\fod_{1}$ and $\fod_{2}$. See Figure \ref{Fig: Diamond}.
\end{definition}

An outer diamond at $v$, $\Diamond_v^{\mathrm{out}}$ or $\Diamond^{\mathrm{out}}$ can alternatively be defined as follows: it is bounded by the roofs of the dense diamond and the left and right generators (rather than the irrational rays). Note that any outer diamond contains a unique dense diamond. The complement of the dense diamond in an outer diamond consists of \emph{discrete triangles}. See Figure \ref{Fig: Diamond}.

\begin{definition}[Discrete and dense ray or object] A ray (or equivalently the object corresponding to the ray) is called a \emph{dense ray or object} if the ray lies in $\foD_{\mathrm{dense}}$. A ray (or
equivalently an object corresponding to the ray) is called a
\emph{discrete ray or object} if the ray lies in
$\foD_{\mathrm{in}}^{\mathrm{stab}}\cup \foD_{\mathrm{discrete}}$. We note that the discrete objects are exceptional bundles $\CE$ and shifts of
such $\CE[1]$.
\end{definition}

    \begin{figure}[h]
    \centering
    \includegraphics[width=14.8cm]{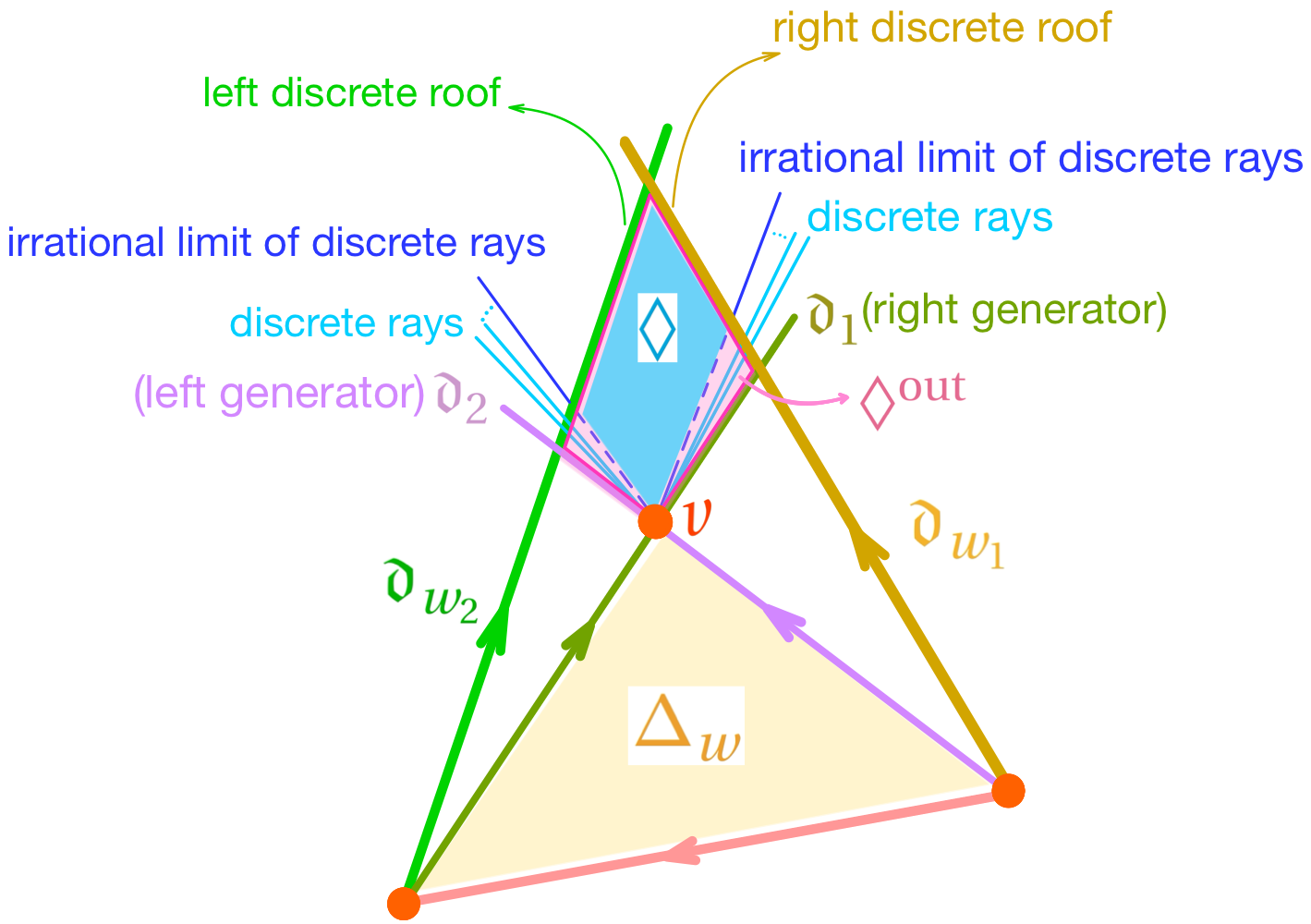}
    \caption{The dense diamond $\Diamond$ at vertex $v$. It is \textit{dense}, in the sense that it consists of all rays generated at $v$ with the slope between the two irrational limits. The outer diamond at $v$, $\Diamond^{\mathrm{out}}$, is given by the left and right generators ($\fod_2$ and $\fod_1$) together with the left and right roofs ($\fod_{w_1}$ and $\fod_{w_2}$). The triangle
$\Delta_w$ is the base triangle for both $\Diamond$ and $\Diamond^{\mathrm{out}}$.} 
    \label{Fig: Diamond}
\end{figure}

\begin{definition}[$R_{\Diamond}$,  $R_{\mathrm{bdd}}$ and $R_{\mathrm{unbdd}}$]\label{Def: RDiamond}
Let 
\[
R:=\big\{(x,y)\in \BR^2\,|\, \hbox{$y\ge nx -3n^2/2$ for some $n\in \BZ$}\big\}.
\]
We define
\[
R_{\Diamond}:= \bigcup_{k\in\BZ} T^k\left(
\bigcup_v\Diamond_v\right),
\]
where the union is over all incoming vertices $v$ of triangles $\Delta_w$
and the initial vertex which is the intersection of $\fod_0^+$ and
$\fod_1^-$. We write
\[
R_{\mathrm{bdd}}:=R_{\Delta}\cup R_{\Diamond}
=\bigcup_{k\in\BZ} T^k\left(\bigcup_v \Diamond^{\mathrm{out}}_v\right),
\]
which we call the \emph{bounded region}.

Finally, we denote by $R_{\mathrm{unbdd}}$ the complement of $R_{\mathrm{bdd}}$ in $R$, and call it the \textit{unbounded region}.
\end{definition}

\begin{remark}
    [Uniqueness of diamond at each vertex]\label{rem: uniquenessDiamond}
    By definition, each vertex $v$ uniquely determines the diamond at $v$.

\end{remark}

\begin{definition}[Initial triangle] We say that a triangle in the scattering diagram is \textit{initial} if it is of the form $T^k(\Delta_{w_0})$ for
some $k\in\BZ$.
In this case, all the sides are contained in unions of initial rays. We denote such a triangle by $\Delta^{\mathrm{init}}$. See Figure \ref{Fig: Diamonds}.
\end{definition}

\begin{definition}[Super-diamond]\label{Def: superDiamond} 
 A \textit{super-diamond}, $\Diamond^{\mathrm{sup}}$, is a quadrilateral generated by the lines corresponding to $\CO(m-1),\CO(m),\CO(m+2),\CO(m+3)$. See Figure \ref{Fig: Diamonds}.
\end{definition}

\begin{definition}[Base triangle]\label{def: baseTriangle} For non-initial (outer or dense) diamonds, the triangle $\Delta_w$ with its incoming vertex being the base of the diamond is called \emph{base triangle} of the diamond. See Figure \ref{Fig: Diamond}.

\end{definition}

  \begin{figure}[h]
 \subcaptionbox*{}[.69\linewidth]{%
    \includegraphics[width=\linewidth]{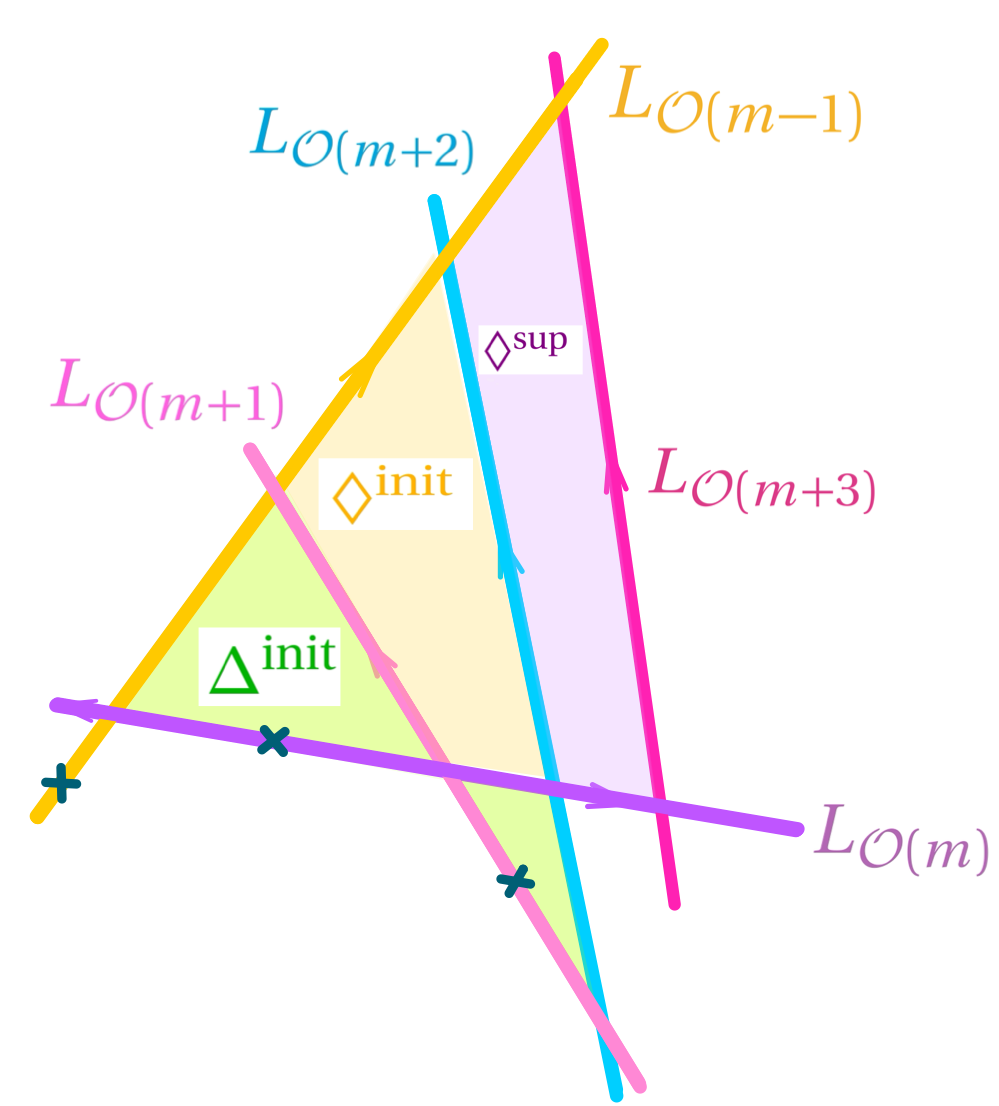} }%
 
  \caption{Super-diamond, $\Diamond^{\sup}$,   initial diamond, $\Diamond^{\mathrm{init}}$, and initial triangle, $\Delta^{\mathrm{init}}$. All the non-initial (outer or dense) diamonds are contained in the super-diamond. $\Diamond^{\mathrm{init}}$ is an example of an outer diamond.}

  \label{Fig: Diamonds}
\end{figure}

\begin{definition}[Degenerate diamond] \label{Def: positiveZeroDiamond}
 
We define a \textit{degenerate diamond} to be the intersection of
a dense ray with endpoint $v$ with the roof of $\Diamond_v$. 
We call the vertical line through a degenerate diamond the
\emph{vertical diagonal} of the degenerate diamond.

  See Figure \ref{Fig: ZeroPositiveDiamond}.

  \begin{figure}[h]
    \centering
    \includegraphics[width=9.9cm]{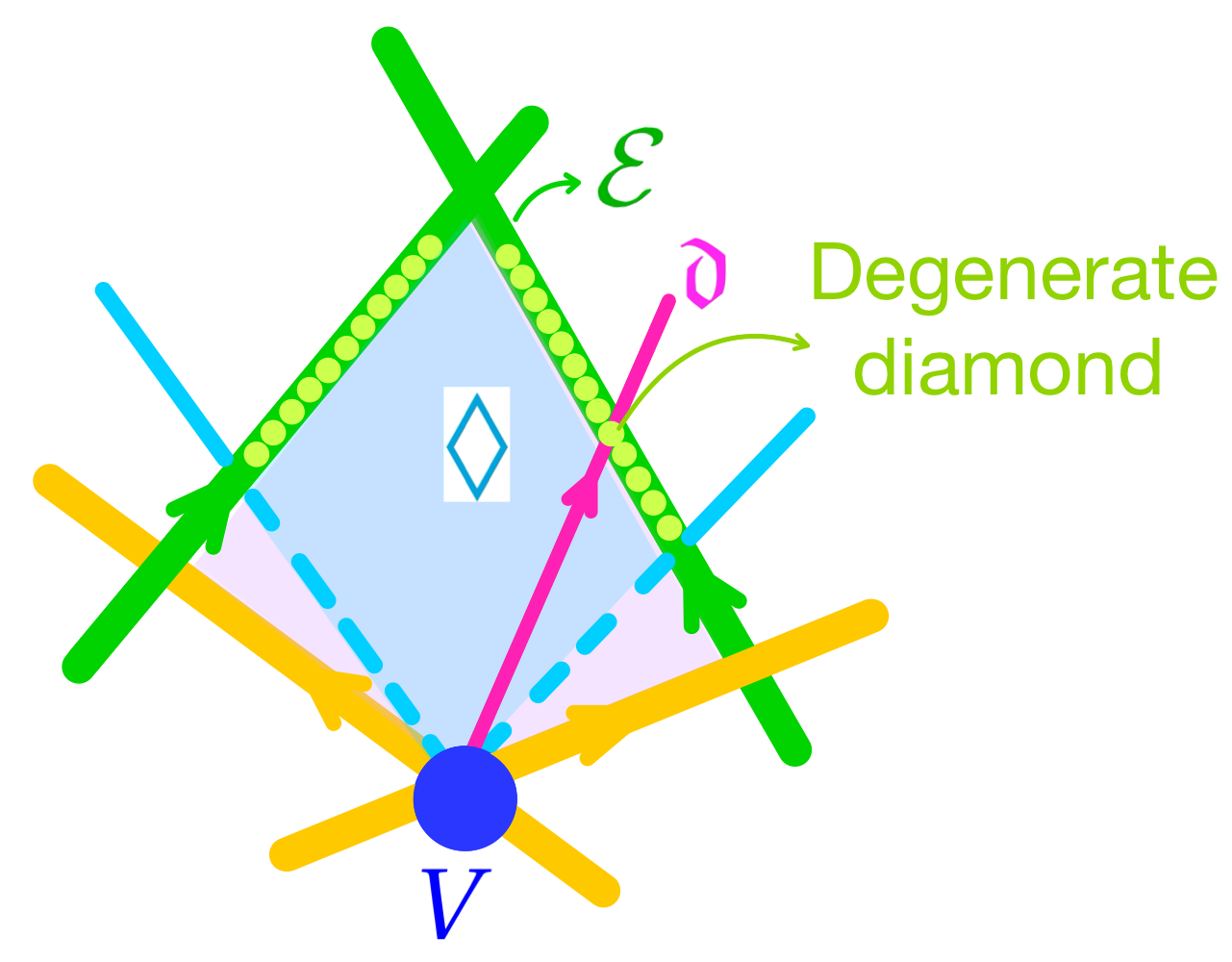}
    \caption{The intersection points of the roof $\CE$ with any rational dense ray $\fod$ is a degenerate diamond.}
    \label{Fig: ZeroPositiveDiamond}
\end{figure}

\end{definition}

\begin{remark}
    Note that passing from a triangle contained in one outer diamond
to an adjacent triangle contained in a neighbouring outer diamond
corresponds to traversing a hybrid edge of $\CT$. On the other hand,
passing from a triangle contained in an outer diamond to an adjacent
triangle contained in the same outer diamond corresponds to
traversing an outgoing edge of $\CT$.
See Figure \ref{Fig: Twn}.
\end{remark}

\section{Vertical diagonals and vertex-freeness of the outer diamonds}\label{Section: verticalDiagonal}

In this section, we prove the vertex-freeness for outer diamonds. 

This means that no two rays of $\foD^{\mathrm{stab}}$
intersect in the interior of an outer diamond.

\begin{theorem}
[Vertical diagonals]\label{thm: verticalDiagonal}
For an (outer or dense) positive diamond, the diagonal joining the base $v$ of
the diamond to the vertex of the diamond which is the intersection of the
two roofs is  vertical. 
\end{theorem}

\begin{proof}

Note that the twist operator $T$ takes vertical line segments to
vertical line segments. Thus for an initial
diamond, we only need to check one such diamond. The proof of
Lemma \ref{lem:prince} shows that the two vertices in question for
$n=0$ are $(1/2,0)$ and $(1/2,1)$, thus giving a vertical diagonal.
See Figure \ref{Fig:verticalProof}.

\begin{figure}[h]
    \centering
   \includegraphics[width=9.5cm]{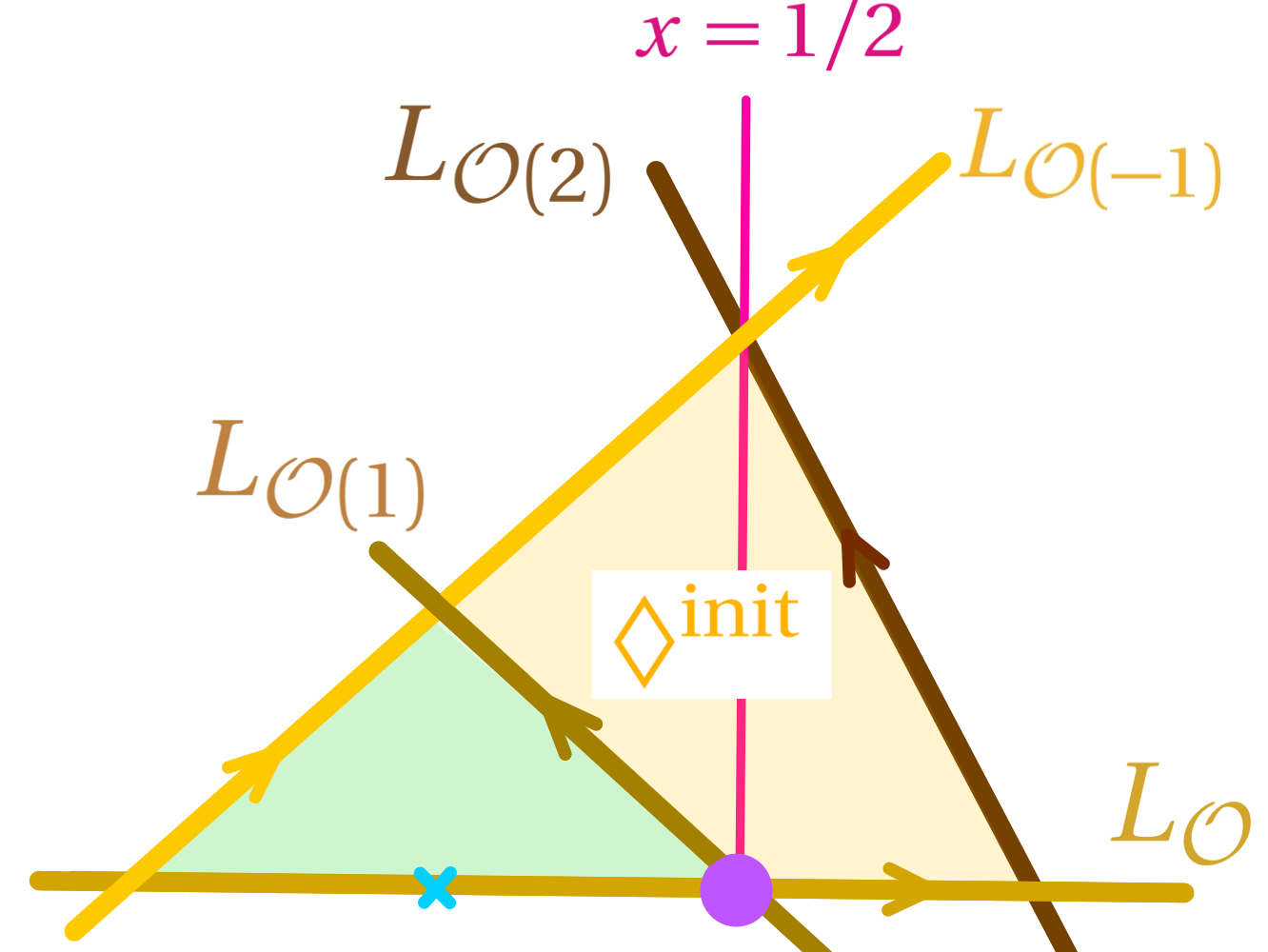}
   \caption{
   $\Diamond^{\mathrm{init}}$ has a vertical diagonal. 
   }
\label{Fig:verticalProof}
\end{figure}

Further, following the first paragraph of the
proof of Lemma~\ref{lem:prince}, a diamond based at an incoming vertex $V_{\CE_1}$
has a vertex at the intersection of the two roofs $\fod_{w_1}\cap\fod_{w_2}$
the intersection of the lines $L_{\CE_1(-3)[1]}$ and $L_{\CE_1}$. Writing
$(r_1,d_1,e_1):=(r(\CE_1),d(\CE_1),e(\CE_1))$, we have
$(r(\CE_1(-3)),d(\CE_1(-3)),e(\CE_1(-3)))=(r_1,d_1-3r_1, e_1-3d_1+9r_1/2)$.
Thus the point of intersection is the solution $(x,y)$ to the equations
\begin{align*}
r_1y+d_1x= {} & e_1,\\
r_1y + (d_1-3r_1)x = {} &  e_1-3d_1+9r_1/2,
\end{align*}
i.e.,
\begin{equation}
\label{eq:apex coordinates}
(x,y)=\left( {d_1\over r_1}-{3\over 2}, {e_1\over r_1} -{d_1^2\over r_1^2} 
+{3d_1\over 2r_1}
\right).
\end{equation}
Note this has the same $x$-coordinate as $V_{\CE_1}$, and thus this diagonal
is vertical as promised.

\end{proof}

\begin{definition}
\label{def:jth discrete ray}
Let $v$ be the incoming vertex of a triangle $\Delta_w$ and $\fod$ one
of the two initial or discrete rays containing $v$ in its interior. Let
$v'$ be the vertex of $\Delta_w$ not contained in $\fod$, and let
$\fod_1$ be the ray containing the edge of $\Delta_w$ with endpoint $v,v'$.
Let $\fod_2,\fod_3,\ldots$ be the sequence of discrete rays with endpoint $v'$
encountered going clockwise (resp.\ counterclockwise) from $\fod_1$
around $v'$ if
$v'$ is a right-most (resp.\ left-most) vertex of $\Delta_w$. Let
$(a_i,b_i)$, $(c,d)$ be the direction vectors of $\fod_i$ and $\fod$
respectively. Let $v_i$ be the
intersection point of $\fod_i$ with $\fod$.
\begin{enumerate}
\item
By the \textit{$j$-th discrete ray $\fod^j_i$ generated along $\fod$
at $v_i$}, we mean the ray with endpoint $v_i$ with direction vector
$R_{j+1}(D_i){\begin{pmatrix}
c\\
d\\
\end{pmatrix}}+R_j(D_i){\begin{pmatrix}
a_i\\
b_i\\
\end{pmatrix}}$,
where $D_i$ is the determinant of $v_i$. We write
$\fod_i^{\infty}$ for the limit of the rays $\fod_i^j$ as $j\rightarrow\infty$;
this is a ray of irrational slope. 
\item By the 
\textit{$j$-th opposite discrete ray $\fod^j_{i,\mathrm{op}}$ generated 
along $\fod$ at $v_i$}, we mean the ray with endpoint $v_i$ and 
direction vector $R_{j}(D_i){\begin{pmatrix}
c\\
d\\
\end{pmatrix}}+R_{j+1}(D_i){\begin{pmatrix}
a_i\\
b_i\\
\end{pmatrix}}$. 
\end{enumerate}
Note that $\fod^0_i$ and $\fod^0_{i,\mathrm{op}}$
are the rays with direction vectors $(c,d)$ and $(a_i,b_i)$ respectively.
\end{definition}

\begin{definition}
Let $\fod$ be a ray in $\foD^{\mathrm{stab}}$.
If its direction vector $-m_{\fod}$ has 
negative (resp.\ positive) $x$-coordinate, we call it a left-moving 
(resp.\ right-moving) ray. 
\end{definition}

\begin{corollary}[Decreasing slopes]\label{cor: decreasingSlope}

Fix an initial or discrete ray $\fod$.
\begin{itemize}
    \item[(1)] For a fixed $i$ and every $j$, the slope of $\text{\:}\fod^1_i$ is greater than the 
slope of $\text{\:}\fod^{j}_{i+1}$, in the
sense that the angle between $\fod$ and $\text{\:}\fod^1_i$ is larger than the
angle between $\fod$ and $\text{\:}\fod^j_{i+1}$.  In particular, the slope of 
$\text{\:}\fod^1_i$ is greater than the slope of $\text{\:}\fod^{\infty}_{i+1}$. 
    \item[(2)] 

For a fixed $i$ and every $j,j'\ge 0$, along $\text{\:}\fod$, the slope of $\text{\:}\fod^j_{i,\mathrm{op}}$ is greater than the slope of $\text{\:} \fod^{j'}_{i+1,\mathrm{op}}$, in the sense
that the angle between $\fod$ and $\fod^j_{i,\mathrm{op}}$ is
larger than the angle between $\fod$ and $\fod_{i+1,\mathrm{op}}^{j'}$.
In particular, the slope of $\text{\:}\fod^{\infty}_{i,\mathrm{op}}$ is greater than the slope of $\text{\:}\fod^0_{i+1,\mathrm{op}}$.
\end{itemize}

\end{corollary}

\begin{proof}

(1) Let $\fod$ be an arbitrary (discrete or initial) ray along which we want to show the statement. Without loss of generality, we show the statement if
$\fod$ is moving to the right. Using the notation of 
Definition~\ref{def:jth discrete ray}, we need thus need to show

the inequality $\frac{3D_id+b_i}{3D_ic+a_i}> \frac{R_{j+1}(D_{i+1})d+R_j(D_{i+1})b_{i+1}}{R_{j+1}(D_{i+1})c+R_j(D_{i+1})a_{i+1}}$. As in Figure \ref{Fig: slopeDecreasing}(1) and without loss of generality, we can  assume that $3D_ic+a_i>0$,  $3D_id+b_i>0$, $R_{j+1}(D_{i+1})d+R_j(D_{i+1})b_{i+1}>0$ and $R_{j+1}(D_{i+1})c+R_j(D_{i+1})a_{i+1}>0$. 
With $\fod$ right-moving, $(a_{i+1},b_{i+1}),(a_i,b_i)$ form an oriented
basis for $M_{\BR}$, and hence $b_ia_{i+1}-a_ib_{i+1}>0$. The desired 
inequality then simplifies to 
  \begin{align}\label{eq:desiredEq}
  (b_ia_{i+1}-a_ib_{i+1})R_j(D_{i+1})>
D_i(3D_{i+1}R_j(D_{i+1})-R_{j+1}(D_{i+1}))=
D_iR_{j-1}(D_{i+1}),
\end{align}
the latter equality by Definition \ref{Def: Rj}. By
the same definition, we have $R_j(D_{i+1})=3D_{i+1}R_{j-1}(D_{i+1})-R_{j-2}(D_{i+1})$. Hence, \eqref{eq:desiredEq}, becomes
  \begin{align}\label{eq:desiredEq2}
      (b_ia_{i+1}-a_ib_{i+1})\big(3D_{i+1}R_{j-1}(D_{i+1})-R_{j-2}(D_{i+1})\big)>D_iR_{j-1}(D_{i+1}).
  \end{align}
However, using $D_{i+1}>D_i$, $b_ia_{i+1}-a_ib_{i+1}\ge 1$,
and Corollary \ref{cor: RjIncreasing}, we have
\[
0<(b_ia_{i+1}-a_ib_{i+1})(D_{i+1}-D_i)R_{j-1}(D_{i+1})
+(b_ia_{i+1}-a_ib_{i+1})(2D_{i+1}R_{j-1}(D_{i+1})-R_{j-2}(D_{i+1})),
\]
yielding \eqref{eq:desiredEq2}, proving the result.

Finally, since we have the statement for every $j$, taking the limit of $\{\fod^j_{i+1}\}$, when $j\to\infty$ we get the final statement.

Part (2) is similar as in Figure \ref{Fig: slopeDecreasing}(2). 

In this case we can show an even stronger statement: We will show that, along $\fod$, $(a_{i+1},b_{i+1})$ (hence, each $\fod^{j}_{i+1,\mathrm{op}}$) has a slope smaller than $\fod^{j}_{i,\mathrm{op}}$. For this, we need to show
\begin{align}\label{eq:firsteq}
    \frac{R_j(D_i)d+R_{j+1}(D_i)b_i}{R_j(D_i)c+R_{j+1}(D_i)a_i}>
    \frac{b_{i+1}}{a_{i+1}}.
\end{align}
Note the from Figure \ref{Fig: slopeDecreasing}(2) and without loss of generality, we can assume $a_{i+1}<0$ and $R_j(D_i)c+R_{j+1}(D_i)a_i<0$; hence, \eqref{eq:firsteq} is equivalent to 
\begin{align}\label{eq:2ndeq}
    (a_{i+1}b_i-b_{i+1}a_i)R_{j+1}(D_i)>(b_{i+1}c-a_{i+1}d)R_{j}(D_i).
\end{align}
But as in  Figure \ref{Fig: slopeDecreasing}(2), we can assume $D=a_{i+1}b_i-b_{i+1}a_i>0$ and $D_{i+1}=b_{i+1}c-a_{i+1}d>0$. Hence, \eqref{eq:2ndeq} is equivalent to $DR_{j+1}(D_i)-D_{i+1}R_{j}(D_i)>0$. By definition, this is equivalent to the positivity of
\begin{align}\label{eq:ineqEquiv}
    D(3D_iR_j(D_i)-R_{j-1}(D_i))-D_{i+1}R_{j}(D_i)=(3DD_i-D_{i+1})R_j(D_i)-DR_{j-1}(D_i).
\end{align}

Note that since $(D,D_i,D_{i+1})$ and $(D,D_{i-1},D_i)$ are Markov triples
related by mutation, we have
$3DD_{i}=D_{i+1}+D_{i-1}$.
Thus, \eqref{eq:ineqEquiv} is equivalent to
$$D_{i-1}R_j(D_i)-DR_{j-1}(D_i),$$
which is positive since $D_{i-1}\geq D$ (by Proposition~\ref{prop:det order})
and $R_j(D_i)>R_{j-1}(D_i)$ (Corollary \ref{cor: RjIncreasing}).

  \begin{figure}[h]
  \subcaptionbox*{(1)}[.9\linewidth]{%
    \includegraphics[width=\linewidth]{
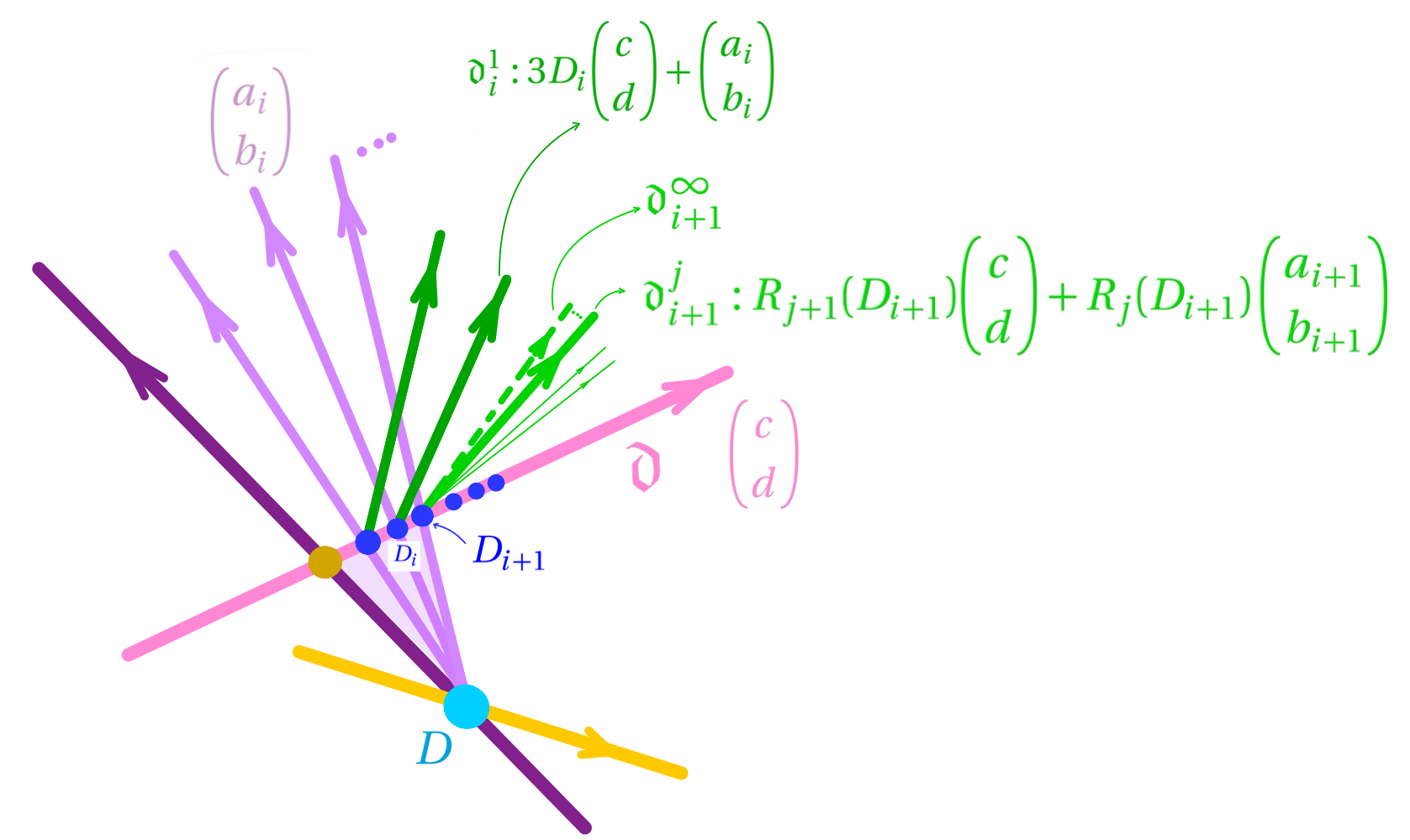}%
  }%
  \hskip5.0ex
  \subcaptionbox*{(2)}[.8\linewidth]{%
    \includegraphics[width=\linewidth]{
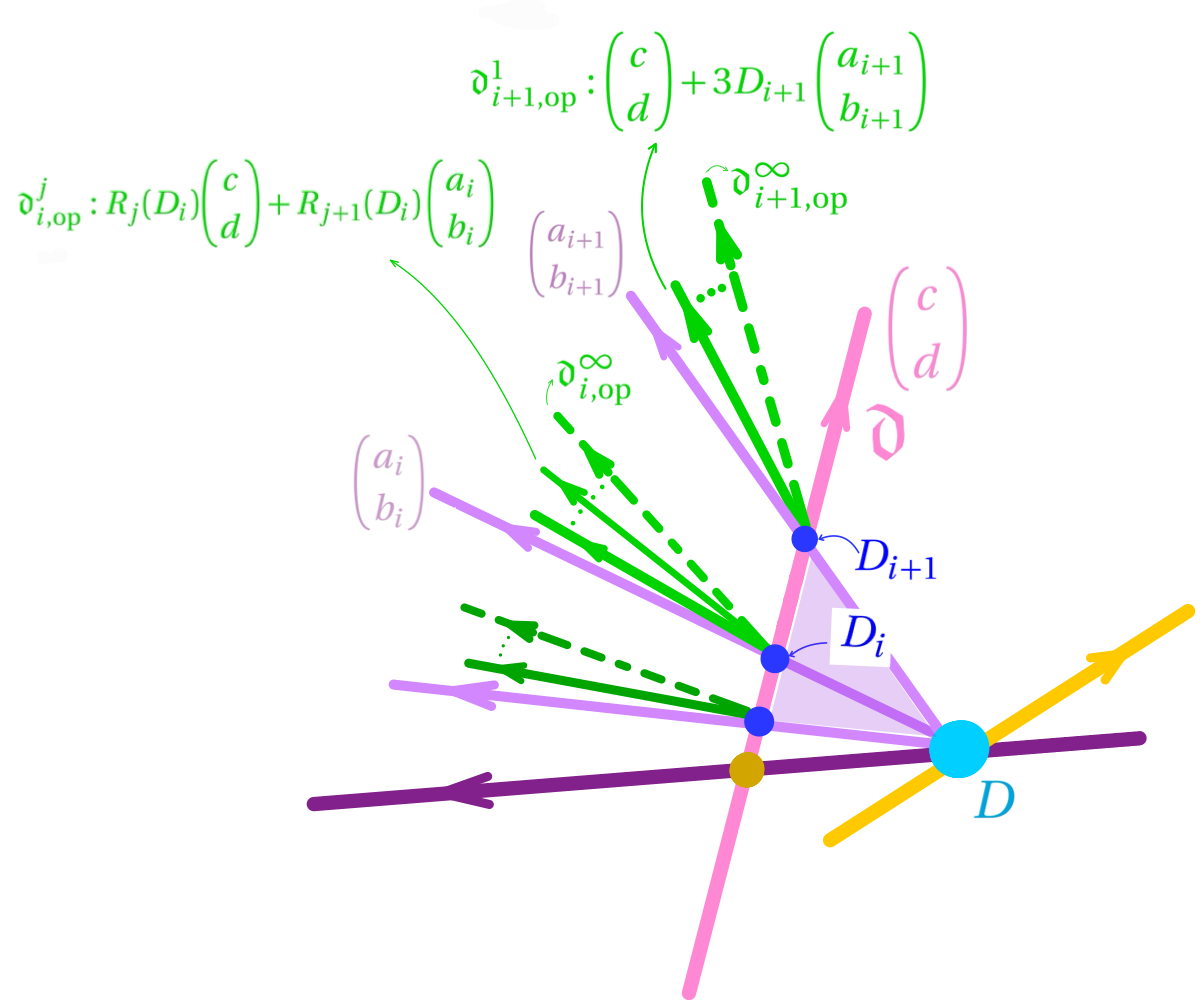}%
  }
  \caption{(1) Along $\fod$, $\fod^{\infty}_{i+1}$ (hence, each $\fod^{j}_{i+1}$) has a slope smaller than $\fod^{j}_{i}$.  (2) Along $\fod$, $(a_{i+1},b_{i+1})$ (hence, each $\fod^{j}_{i+1,\mathrm{op}}$) has a slope smaller than $\fod^{j}_{i,\mathrm{op}}$. Here $D_i=b_ic-a_id>0$ and $D=a_{i+1}b_i-a_ib_{i+1}>0$ are the corresponding determinants at the vertices.}
  \label{Fig: slopeDecreasing}
\end{figure}
\end{proof}

The following statement is analogous to \cite[Lem.~3.10]{Prince-20}.

\begin{corollary}
\label{cor:bigger slopes}
For $w$ a vertex of the Markov tree $\CT$, let $v,v',v''$ be the incoming,
left-most and right-most vertices of the triangle $\Delta_w$. Denote
by $\fod_{v}^-$ and $\fod_v^+$ the left-hand and right-hand rays
respectively generating the cone $C_v$. Let $\fod^1_{v'}$ be the 
first ray containing $v'$ counter-clockwise from the edge of $\Delta_w$
joining $v$ and $v'$, and let $\fod^1_{v''}$ be the first ray containing
$v''$ clockwise from the edge of $\Delta_w$ joining $v$ and $v''$.

Then the slope of $\text{\:}\fod_{v'}^1$ is greater than the slope
of $\text{\:}\fod_{v}^+$ and the slope of $\text{\:}\fod_{v''}^1$ is less than the
slope of $\text{\:}\fod_v^-$. See Figure~\ref{Fig:SlopeDense}.
\end{corollary}

\begin{figure}[h]
    \centering
   \includegraphics[width=9.1cm]{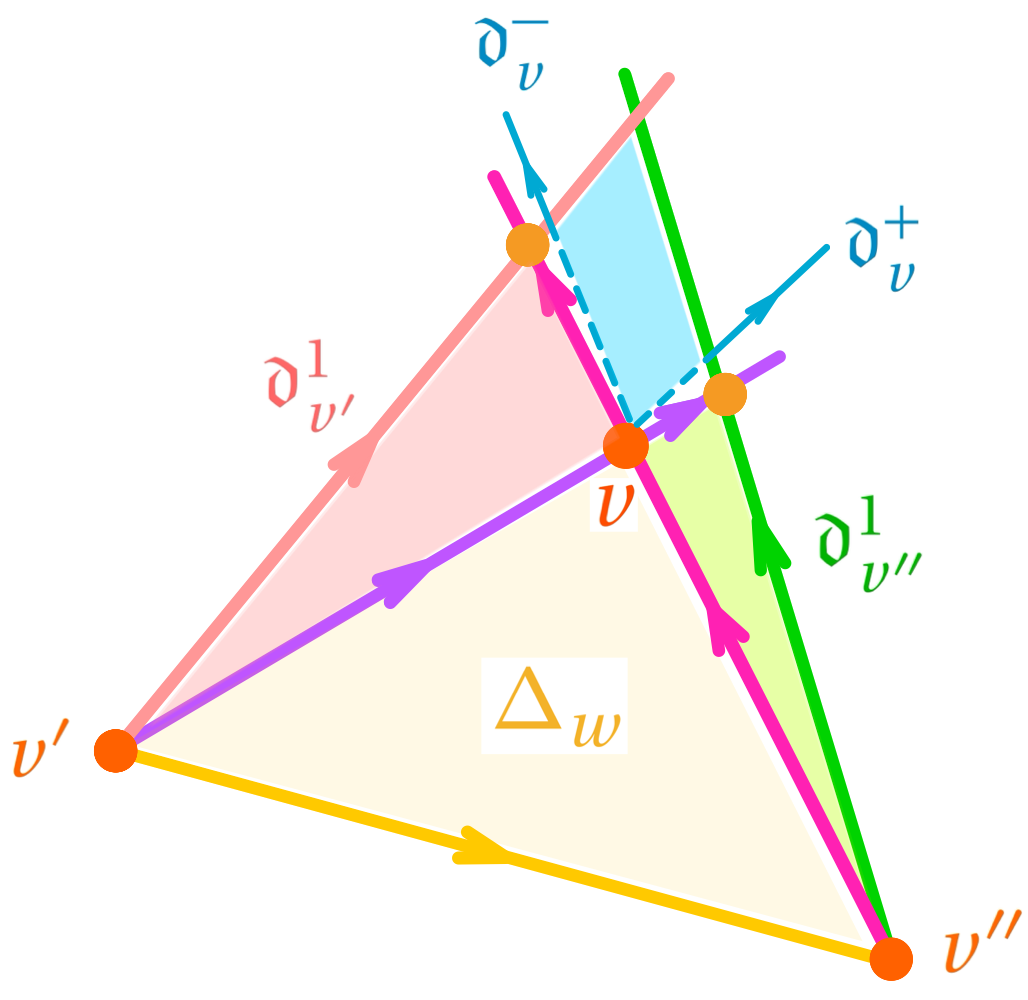}
   \caption{}
    \label{Fig:SlopeDense}
\end{figure}

\begin{proof}
We consider the case of $v'$, as the situation is symmetric. If $\Delta_w$ is
the initial triangle, then one immediately checks this by hand, the
slope of $\fod_{v'}^1$ being $3/2$
and the slope of $\fod_v^+$ being $\sqrt{2}$.

Otherwise, the argument depends on whether $v'$ is a hybrid or outgoing
vertex. If it is hybrid, then $\Delta_w$ is as in Figure
\ref{Fig: slopeDecreasing},(1), with vertices $v,v',v''$ having degrees
$D_{i+1},D_i$ and $D$ respectively. The claim then follows from
Corollary~\ref{cor: decreasingSlope},(1). On the other hand, if $v'$ is
an outgoing vertex, the picture is as in Figure \ref{Fig: slopeDecreasing},(2)
(reflected about the $y$-axis), with $\Delta_w$ corresponding to
the triangle with vertices $v,v',v''$ having degrees
$D_{i+1},D$ and $D_i$ respectively. The result in this case
follows from Corollary~\ref{cor: decreasingSlope},(ii).
\end{proof}

We will need the following:

  \begin{theorem}
      [Inclusion in super-diamonds] \label{thm: Inclusion}
\begin{enumerate}
\item 
Let $\fod\in\foD^0_{\mathrm{discrete}}$. 
If $\fod$
is right-moving the slope of $\fod$ is less than $2$, while if $\fod$
is left-moving then the slope of $\fod$ is greater than $-2$. 
\item
Every non-initial outer diamond is contained in a super-diamond.
\end{enumerate}
  \end{theorem} 
  
\begin{figure}[h]
 \subcaptionbox*{}[.99999999\linewidth]{%
 \includegraphics[width=\linewidth]{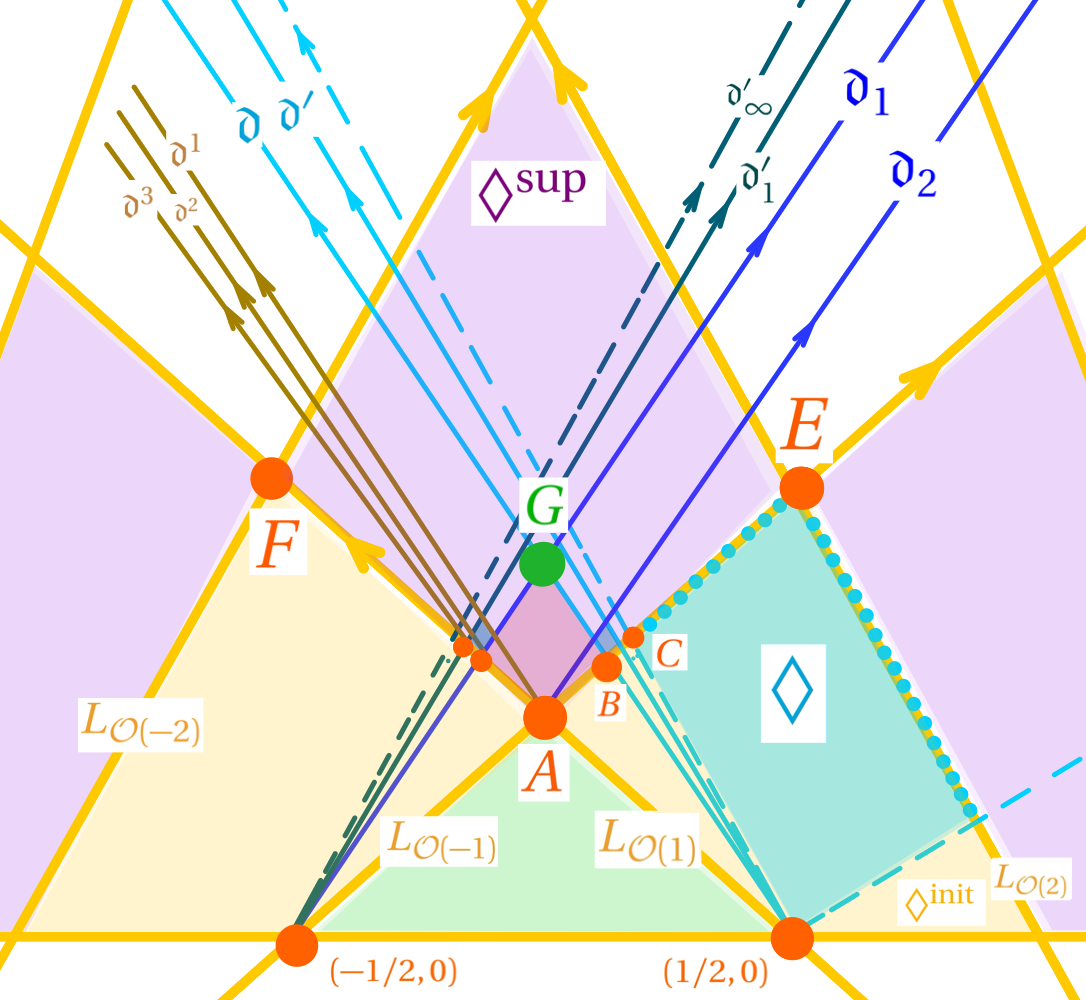} }%
  \caption{$L_{\CO(-1)}$ bounds an initial diamond, $\Diamond^{\mathrm{init}}$ (its big diamond) at $(1/2,0)$ on the one side and infinitely many on the other.}
  \label{Fig: init}
\end{figure}

  \begin{proof}
By the analysis of \S\ref{subsec:moduli emptiness}, and in 
particular from Figure~\ref{Fig: RL}, every ray of
$\foD^0_{\mathrm{discrete}}$ is contained in $L_{\CE(-3)[1]}$
for a right-moving ray or in $L_{\CE}$ for a left-moving ray, for some
exceptional bundle $\CE$ obtained
via mutation from the strong exceptional triple $\CC:=(\CO_{\BP^2}(1), \CT_{\BP^2},
\CO_{\BP^2}(2))$. Since $\mu(\CO_{\BP^2}(1))=1$, $\mu(\CO_{\BP^2}(2))=2$,
we have by Lemma~\ref{lem:slope inequalities},(2) that $\mu(\CE)\in [1,2]$
(with $\mu(\CE)=1$ or $2$ if and only if $\CE=\CO_{\BP^2}(1)$ or
$\CO_{\BP^2}(2)$ respectively).
The slope of the line $L_{\CE(-3)[1]}$ is $-\mu(\CE(-3))=
-\mu(\CE)+3 \in [1,2]$. The slope of the line $L_{\CE}$ is $-\mu(\CE)
\in [-2,-1]$. This shows (1). Note further this shows that if
$\fod,\fod'\in \{\fod_{-1}^+,\fod^-_1\}\cup
\foD^0_{\mathrm{discrete}}$ are a right-moving and left-moving
ray respectively, and the endpoints of these rays each lie below $L_{\CO(2)}$
and $L_{\CO(-2)}$, then the intersection $\fod\cap\fod'$, if non-empty, also
lies below $L_{\CO(2)}$ and $L_{\CO(-2)}$.

For (2), first
note that the projection of a super-diamond to the $x$-axis is
$[n-1/2,n+1/2]$ if the superdiamond is bounded by the lines
$L_{\CO(n-2)}, L_{\CO(n-1)}, L_{\CO(n+1)}, L_{\CO(n+2)}$, see 
Figure \ref{Fig: init}. Focusing without loss of generality on the case
$n=0$, it is thus enough to show the following.
Let $\CE$ be an exceptional bundle of rank $>1$
for which the $x$-coordinate
of $V_{\CE}$ lies in $(-1/2,1/2)$, i.e., $\mu(\CE)\in (1,2)$. Then the
outer diamond based at $V_{\CE}$ is contained in the super-diamond
$\Diamond^{\mathrm{sup}}$ corresponding to $n=0$. 

Denote by $\mathbf{Exc}$ the set of exceptional
bundles on $\BP^2$ obtained via mutation from $\CC$ (including
the three elements of $\CC$ itself). For $\CE\in \mathbf{Exc}$
not a line bundle, the vertex $V_{\CE}$ is the intersection of a left-moving
and right-moving ray in $\{\fod_{-1}^+,\fod^-_1\}\cup
\foD^0_{\mathrm{discrete}}$, and hence $V_{\CE}$
lies below $L_{\CO(-2)}$ and $L_{\CO(2)}$ respectively provided the
endpoints of these two rays do. Thus, starting from $V_{\CE}$ for
$\CE$ in the strong exceptional triple $\CC$, we see inductively that for all
$\CE\in \mathbf{Exc}$, $V_{\CE}$ lies below
$L_{\CO(-2)}$ and $L_{\CO(2)}$. Similarly, again
consulting Figure~\ref{Fig: RL}, we see inductively that the roofs
of the outer diamond based at $V_{\CE}$ lie below $L_{\CO(-2)}$ and
$L_{\CO(2)}$. Thus $\Diamond^{\mathrm{out}}_\CE$ lies below these
two lines.

On the other hand, provided that $\CE$ is rank $>1$, then $V_{\CE}$
lies on or above the lines $L_{\CO(1)}$ and $L_{\CO(-1)}$, as
clearly $V_{\CE}$ does not lie in the interior of the outer diamond based at 
$(-1/2,0)$ or $(1/2,0)$, see Figure~\ref{Fig: init}. Thus we see that
$\Diamond^{\mathrm{out}}_{\CE}$ lies in $\Diamond^{\mathrm{sup}}$.
\end{proof}

\begin{lemma}
\label{lem:no cross}
Let $\fod$ be an initial or discrete ray of $\foD^{\mathrm{stab}}$. 
If $\fod$ is right-moving (resp.\ left-moving),
then there is no $\sigma\in\Int(\fod)$ and object $\CE\in\CA^{\sigma}$
which is semistable with respect to $\sigma$ and has direction vector
$-m_{\CE}$ pointing to the right (resp.\ left) of $\fod$.
\end{lemma}

\begin{proof}
After twisting, we may assume that $\fod\in \{\fod_{-1}^+,\fod_0^{\pm},
\fod_1^-\}\cup \foD^0_{\mathrm{discrete}}$. Now suppose that there
exists
$\sigma_0\in \Int(\fod)$, $\CE\in
\CA^{\sigma}$ which violates the statement of the lemma. We will first build
a sequence of line segments or rays $E_0,\ldots,E_n$ and 
$E_0',\ldots,E_{n'}'$ from this data.

Without loss of generality, let us focus on the right-moving case.
Set $E_0=\fod$. If
$\fod$ is an initial ray, we stop.
Otherwise, we define the sequence $E_1,\ldots,E_n$ inductively.
The endpoint
of $\fod$ is an incoming vertex $v_1$ of some triangle $\Delta_1$. 
Take $E_1$ to be the left-hand edge of this triangle, with other endpoint
$v_2$. If $E_1$ contains one of the points of the form $(d,-d^2/2)$ from which initial
rays emanate (so that $E_1$ is the edge of an initial triangle), then
we replace $E_1$ with the subsegment of $E_1$ with endpoints $(d,-d^2/2)$
and $v_1$. At this
point we are done. Otherwise, we continue the process 
inductively. If we have chosen $E_i$ with vertices $v_i,v_{i+1}$ and
$E_i$ does not contain a point of the form  $(d,-d^2/2)$, then
$v_{i+1}$ is an incoming vertex of a triangle $\Delta_{i+1}$, and we
take $E_{i+1}$ to be the left-hand edge of $\Delta_{i+1}$ containing
$v_{i+1}$. If $E_{i+1}$ contains a point of the form $(d,-d^2/2)$,
then we replace $E_{i+1}$ with the line segment joining
$(d,-d^2/2)$ and $v_i$ and are done. Thus we have the data of a ray $E_0$
with endpoint $v_1$, and line segments $E_1,\ldots,E_n$ with $E_n$
having endpoints $v_i,v_{i+1}$ and $v_{n+1}=(d,-d^2/2)$ for some
integer $d$. Since $\fod$ was chosen to be right-moving and
in $\{\fod_{-1}^+,\fod_0^{\pm}, \fod_1^-\}\cup 
\foD^0_{\mathrm{discrete}}$, note that $d=-1$ or $0$, and $d=0$ only
if $\fod=\fod_0^+$.

We may similarly follow the procedure given in \S\ref{subsec:wall crossing},
starting with $\CE$ and obtaining a sequence of quotient objects
$\CQ_1,\ldots,\CQ_{n'}$
and edges $E_0'=E(\CE,\sigma_0)$, $E_1',\ldots,E'_{n'}$. 
This sequence terminates at a point $(d',-(d')^2/2)$ for $d'$ some integer.
See Figure~\ref{Fig: ConcaveOuterRay} for the first steps of this process.

\begin{figure}[h]
 \subcaptionbox*{}[0.95\linewidth]{%
   \includegraphics[width=\linewidth]{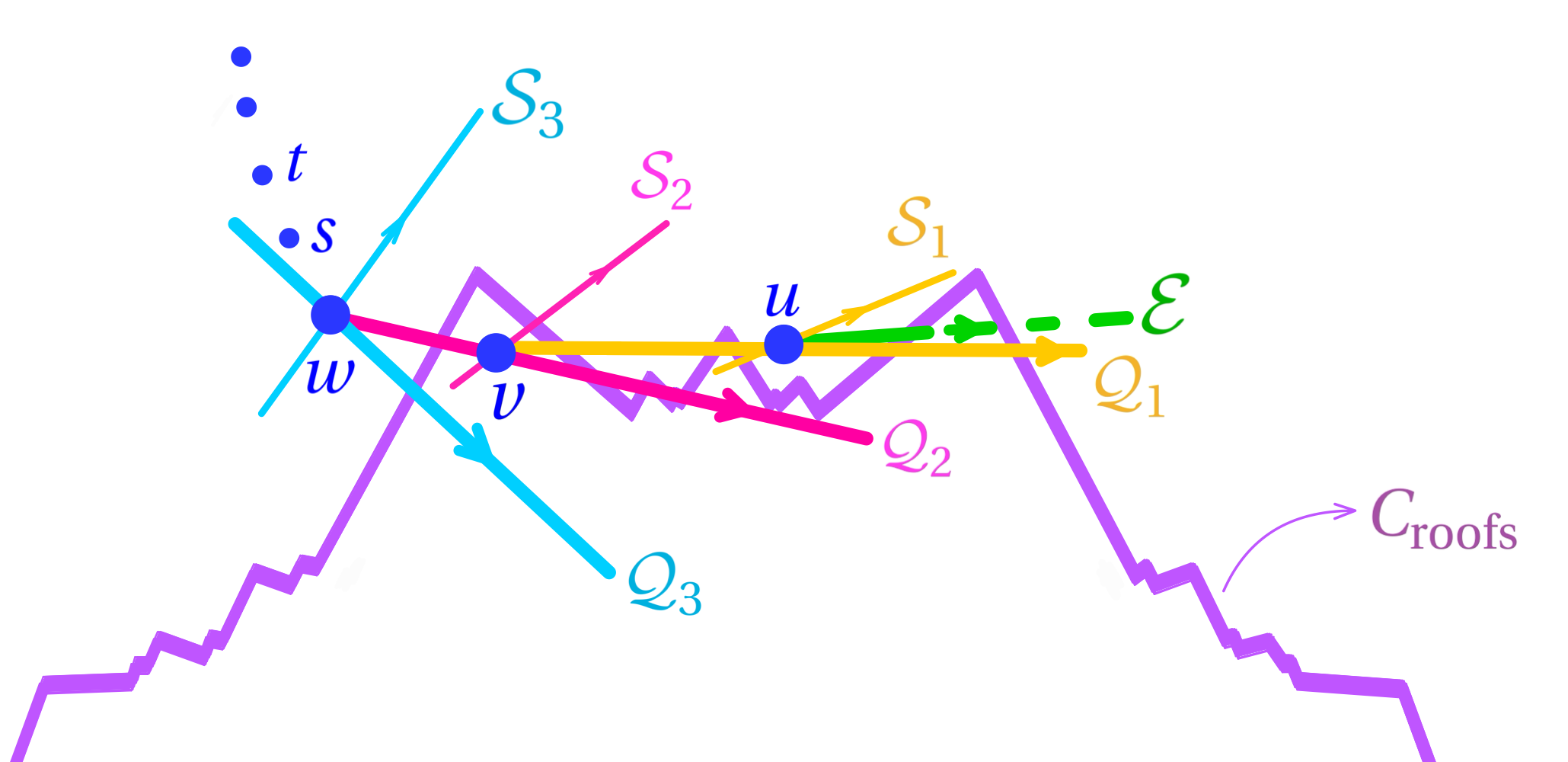}%
  }%
  \caption{}

 \label{Fig: ConcaveOuterRay} 
\end{figure}

For $E=E_i$ or $E_i'$, write $\mu(E)$ for the slope of the line segment
$E$. Note that if $E\subseteq L_{\CE}$ for some object $\CE$, then
$\mu(E)=-\mu(\CE)$. By construction
$\mu(E_i)\ge \mu(E_{i+1})$ for all $i$. We will now show similarly that
\begin{equation}
\label{eq:mu increasing}
\mu(E_i')>\mu(E_{i+1}')
\end{equation}
for all $i$. 

Indeed, this almost follows from the fact that at each step
of the construction of the $E_i'$,
$\CQ_i$ is chosen as a right generator, as in Figure \ref{Fig: Convexity}.
However, \eqref{eq:mu increasing} would fail if $m_{\CQ_i}$ points to the
left but $m_{\CQ_{i+1}}$ points to the right. To rule out this case, we
proceed as follows. Extend $E_0'$ in the direction $-m_{\CE}$,
i.e., replace $E_0'$ with the ray $\sigma_1-\BR_{\ge 0}m_{\CE}$.
Since the
sequence of edges terminates at a point on the parabola $y=-x^2/2$ and
all edges $E_i'$ (except the last one which has endpoint on this parabola)
lie about this parabola, there is no choice but for this sequence
of edges to cross itself, see Figure \ref{Fig: loop}.
Thus there exists some $i,j$ with $j>i+1\ge 1$ for which
$E_i'\cap E_j'\not=\emptyset$. After subdividing these two edges,
we may assume they intersect in a vertex. Thus $E_i',E_{i+1}',\ldots
E_j'$ form a closed loop. Since this loop is compact, there is a point
$(x_0,y_0)$ on this loop
for which the function $h(x,y)=y+x^2/2$ achieves its maximum. This means
that all $E_k'$ with $i\le k \le j$ containing $(x_0,y_0)$ must lie in the same half-plane
with boundary the line $x_0(x-x_0) + (y-y_0)=0$. If $(x_0,y_0)$ lies
in the interior of an edge $E_k'$, then $E_k'$ must lie in this
line, contradicting \eqref{eq:varphi increasing}, as then 
$\varphi_{(x_0,y_0)}(m_{\CQ'_k})=0$.
If this maximum occurs at the intersection of two adjacent edges
$E_k',E_{k+1}'$, then 
$\varphi_{(x_0,y_0)}(m_{\CQ_k})$ and
$\varphi_{(x_0,y_0)}(m_{\CQ_{k+1}})$ have the opposite sign, 
again contradicting
\eqref{eq:varphi increasing}. The same contradiction occurs 
if $(x_0,y_0)$ is the intersection point of $E_i'$ and $E_j'$. This proves
that \eqref{eq:mu increasing} holds.

\begin{figure}[h]
 \subcaptionbox*{}[.67\linewidth]{%
 \includegraphics[width=\linewidth]{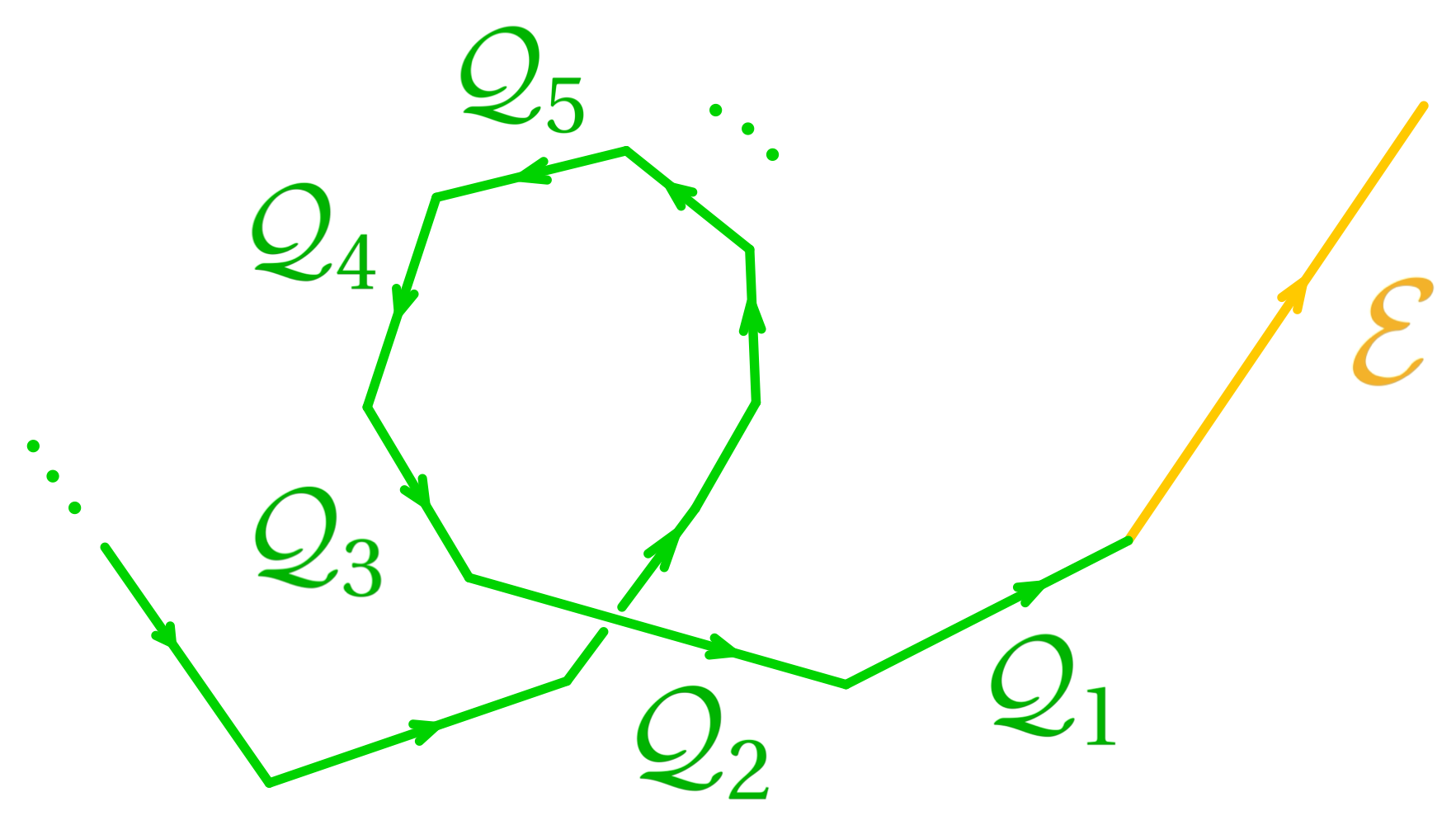} }%
  \caption{}
  \label{Fig: loop}
\end{figure}

Now note that $d'\ge -1$. Indeed, if $d'\le -2$, then $\mu(E_{n'}')\ge 2$.
Using \eqref{eq:mu increasing} repeatedly, we thus have $\mu(E_1')\ge 2$.
But by Theorem \ref{thm: Inclusion},(1), $\mu(\fod)<2$. Thus $-m_{\CE}$
cannot point to the right of $\fod$, contradicting our assumption. 

It thus follows that $\cup_i E_i$ and $\cup_i E'_i$ must have 
an intersection point other than $\sigma_0$. Thus, there is some $i$
for which there exists a $j$ with
$E_i'\cap E_j\not=\emptyset$ and $\mu(E_i')>\mu(E_j)$.
Take $i$ as small as possible, and given $i$, take $j$ as small as possible. 
If $j=0$, then by \eqref{eq:mu increasing},
$\mu(E_i')<\mu(E_1')$, but by assumption $\mu(E_1')<\mu(E_0)$ since
$E_1'$ crosses $\fod$ from the left. This is a contradiction.

Otherwise, if $j>0$, then $E_j$ is a edge (or partial edge if $E_j=E_n$)
of a triangle $\Delta_w$ for which
$v_j$ is an incoming vertex. 
If $E_i'$ intersects $E_j$ in its interior, then since $E_j$ is the
intersection of two triangles, $E_i'$ intersects the interior of
one of these two triangles. 
By Theorem \ref{thm:bousseau generalization}, this is impossible. Thus
$E_i'$ must intersect $E_j$ at a vertex. Since $j$ was chosen to be
as small as possible, this vertex is $v_{j+1}$. Thus $E_i'$
must be a line segment contained in one of the rays through $v_{j+1}$
in $\foD^{\mathrm{stab}}$. Indeed, $E_i'$ cannot lie between two 
adjacent discrete rays passing through $v_{j+1}$ as this would violate
Theorem \ref{thm:bousseau generalization}. Hence, if $E_i'$ is not 
contained in a discrete ray, it must be contained in the dense cone
$C_{v_{j+1}}$, and hence in a dense ray of $\foD^{\mathrm{stab}}$.
It then
follows from Corollary~\ref{cor:bigger slopes} that the slope of $E_i'$
is greater than the slope of $\fod_{v_j}^+$. We may then use this
corollary repeatedly to see that the slope of $E_i'$ is greater than 
the slope of $\fod$, again a contradiction.
\end{proof}

\begin{theorem}[Vertex-freeness]\label{Thm:nonEmptiness} Let $\Diamond_v$ be a dense diamond based at $v$.
\begin{enumerate}
\item For $\sigma$ in the interior of $\Diamond_v$ and $\gamma$
a Chern character with $Z^{\sigma}(\gamma)\in i\BR_{>0}$, then
$\CM_{\gamma}^{\sigma}$ is non-empty only if $L_{\gamma}$ passes through
$v$. Further, there are no actual walls in the interior of $\Diamond_v$
on any ray passing through $v$.
\item The only rays meeting the roof of $\Diamond_v$ are those passing
through $v$.
\end{enumerate}
\end{theorem}

\begin{proof}
Let $\fod$, $\fod'$ be the discrete rays containing the
left and right roofs of $\Diamond_v$ respectively. For (1), suppose given
$\sigma=\sigma_1$, $\gamma$ as in the statement with $\CM_{\gamma}^{\sigma_1}$
non-empty and $L_{\gamma}$ not passing through $v$. Let
$\CE_1$ be an object in $\CM_{\gamma}^{\sigma}$. We will construct
a sequence of edges $E_1=E(\CE_1,\sigma_1), E_2,\ldots,E_n$. If the other
endpoint of $E_1$ lies outside of $\Diamond_v$, we stop. Otherwise,
let $L_1$ be the line through $v$ and $\sigma_2$. If $\CE_2,\CE_2'$
are the left and right generators of $\CE_1$ in the language of
\S\ref{subsec:wall crossing}, at least one of $m_{\CE_2},m_{\CE_2'}$
is not tangent to $L_1$ since $m_{\CE_1}$ is not tangent to $L_1$ and
is a positive linear combination
of $m_{\CE_2},m_{\CE_2'}$. Interchanging $\CE_2$ and $\CE_2'$ if necessary,
we can assume $m_{\CE_2}$ is not tangent to $L_1$. We then take
$E_2=E(\CE_2,\sigma_2)$. Repeating this process, we obtain $E_1,\ldots,E_n$
where $E_n$ has initial endpoint $\sigma_n$ inside $\Diamond_v$ and final
endpoint $\sigma_{n+1}$ outside of $\Diamond_v$. In particular, $E_n$
must pass through one of the four edges of $\Diamond_v$. If it passes through
the roof, i.e., through $\fod\cap\Diamond_v$ or  $\fod'\cap\Diamond_v$, 
then this violates
Lemma~\ref{lem:no cross}. If $E_n$ passes through the interior
of one of the lower edges, then it must enter the interior of a triangle,
violating Theorem~\ref{thm:bousseau generalization}. We thus obtain a
contradiction and the first statement of (1) holds. If an actual wall
occured in the interior of $\Diamond_v$, then we would now have an object
$\CE$ with $L_{\CE}$ passing through $v$, with destabilizing subobject
or quotient object $\CF$ having $L_{\CF}$ not passing through $v$. This
contradicts the first statement of (1).

For (2), after taking (1) into account, the only possibilities for
$\fod\in \foD^{\mathrm{stab}}$ meeting the roof of $\Diamond_v$ but not
passing through $v$ is
if $\fod$ meets the apex of $\Diamond_v$ but does not enter the interior.
However, this is also ruled out by Lemma~\ref{lem:no cross}.
\end{proof}

The following gives a more refined picture of the diamonds.

Fix a (non-initial) discrete ray $\fod\in\foD_{\mathrm{discrete}}$ with endpoint $v$. 
We let $\fod^{\circ}$ be the \emph{previous ray}, which we define to be
the first ray counterclockwise (resp.\ clockwise) from $\fod$ through $v$
if $\fod$ is a left-moving (resp.\ right-moving) ray. Note that $v$ may or may not
be the endpoint of $\fod^{\circ}$. We say a point $p\in\fod^{\circ}$ is 
\emph{past $v$} if $p\in v-\BR_{>0} m_{\fod^{\circ}}$. We have:

\begin{prop}
Every outer diamond whose base lies on
$\fod^{\circ}$ past $v$
lies between $\fod^{\circ}$ and $\fod$. 
\end{prop}
    
\begin{proof}
By symmetry, we can assume $\fod$ is a left-moving ray. Figure \ref{Fig: init}
shows one possible choice for $\fod$, with, in this case, $\fod^{\circ}$ being
the left-moving ray contained in $L_{\CO(1)}$.
   
Using Corollary \ref{cor: decreasingSlope}
along $\fod^{\circ}$, the angles of $\fod$, $\fod^1, \fod^2, \fod^3,\ldots$ 
with
$\fod^{\circ}$ are decreasing in the direction of $\fod^{\circ}$. Here, 
$\fod^1,\fod^2,\ldots$ denotes the
first discrete rays generated along $\fod^{\circ}$, i.e., $\fod_1^1,\fod_2^1,
\ldots$ in the notation of Definition \ref{def:jth discrete ray}.
These rays are the right roofs of the diamonds with base on $\fod^{\circ}$, and
hence these right roofs all lie below $\fod$. Thus in particular
all the diamonds with base past $v$ on $\fod^{\circ}$ lie between $\fod^{\circ}$
and $\fod$, as claimed.
\end{proof}

\begin{remark}
\label{rem:various diamonds}
We have the following picture of diamonds along any
discrete ray $\fod$. Such a ray will always contain half the roof of some
outer diamond $\Diamond_v^{\mathrm{out}}$. This diamond will lie
to the left (resp.\ right) of $\fod$ if $\fod$ is left-moving
(resp.\ right-moving).
On the other hand, before $\fod$
reaches the dense diamond $\Diamond_v$ based at $v$, there will be
a sequence of outer diamonds based at points on $\fod$. These will
all lie on the opposite side of $\fod$, i.e., on the right (resp.\ left) side of
$\fod$ if $\fod$ is left-moving (resp.\ right-moving). See Figure
\ref{Fig: effectiveIntervalETC}. We use the following language.
\end{remark}    

\begin{definition}[Effective interval, big diamond, and relatively big diamond]\label{Def:effectiveIntervalETC} Let $\fod$ be a discrete ray as above. 
We assign the following notions to $\fod$ (for which we refer to the above discussion and Figure \ref{Fig: effectiveIntervalETC}):
\begin{itemize}
    \item[(i)] The outer diamond that $\fod$ contains part of the roof
of is called \emph{big diamond} associated to $\fod$.
    \item[(ii)] The intersection of $\fod$ with the big diamond
is called the \emph{effective interval} of $\fod$.
    \item[(iii)] The first and largest diamond on the side of $\fod$ 
opposite the big diamond
is called \emph{relatively big diamond} associated with $\fod$.
   
\end{itemize}
See Figure \ref{Fig: effectiveIntervalETC}.
\end{definition}

Note that each initial half-line bounds two initial diamonds, but only one of them (the one that has infinitely many diamonds on its other side along the half-line) is its big diamond. For example, in Figure \ref{Fig: init}, the big diamond associated to the half-line corresponding to $\CO(-1)$ is the one based at $(0,1/2)$. The diamond based at $(-1/2,0)$ is another initial diamond bounded by this half-line, but it is not its associated big diamond (it is in fact the big diamond associated to the half-line corresponding to $\CO(1)$).

\begin{figure}[h]
    \centering
    \includegraphics[width=13cm]{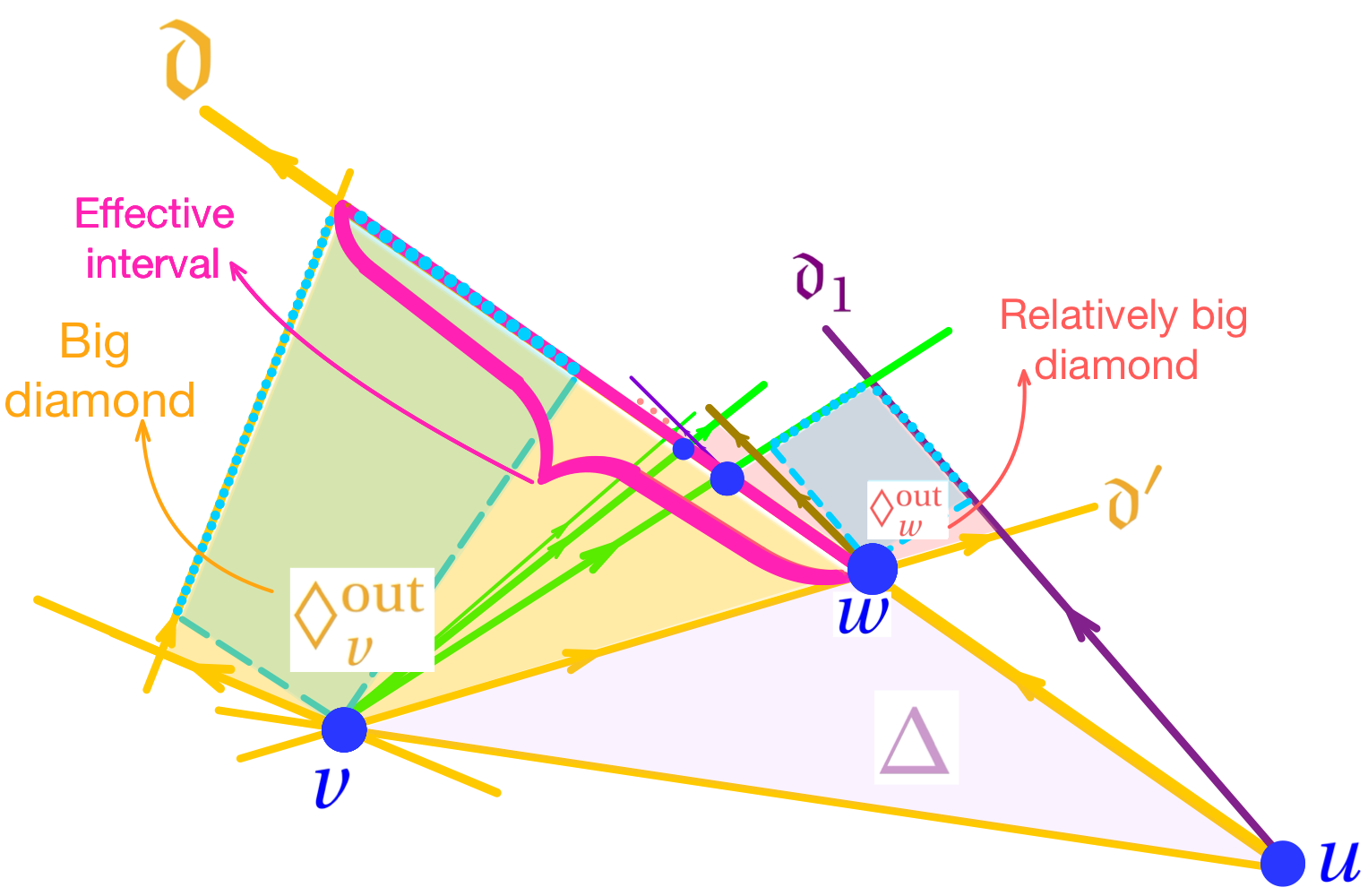}
    \caption{Effective interval, big diamond $\Diamond_v^{\mathrm{out}}$,  and relatively big diamond $\Diamond_w^{\mathrm{out}}$. 
    }
    \label{Fig: effectiveIntervalETC}
\end{figure}

\begin{notation}
    We denote the curve given by the union of the roofs of the dense diamonds by $C_{\mathrm{roofs}}$  (see Figure \ref{Fig:boundary}). 
\end{notation}

\begin{corollary}[Boundedness]\label{Cor: boundedness}
The union of all outer diamonds is bounded above by the union of roofs
of all super-diamonds. More precisely, the 
upper boundary of the union of all  outer diamonds is $C_{\mathrm{roofs}}$,
which lies below the union of roofs of all super-diamonds. 
\end{corollary}

\begin{proof} All the initial diamonds lie below the union of roofs
of the superdiamonds, and all other diamonds are contained in superdiamonds
by Theorem \ref{thm: Inclusion}, hence the first statement. The second
statement is then clear.

\end{proof}

  \begin{figure}[h]
  \subcaptionbox*{(1)}[0.45\linewidth]{%
    \includegraphics[width=\linewidth]{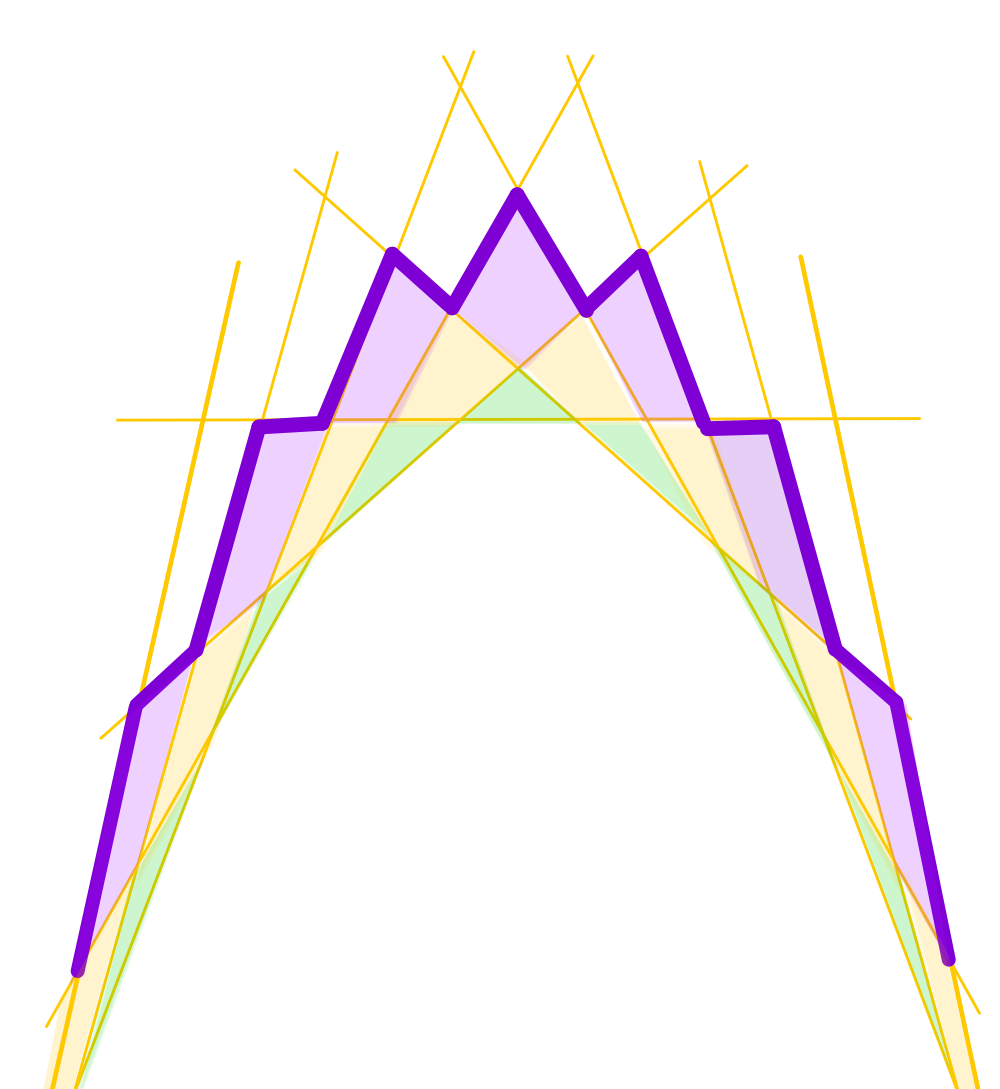}%
  }%
  \hskip3.1ex
  \subcaptionbox*{(2)}[.45\linewidth]{%
    \includegraphics[width=\linewidth]{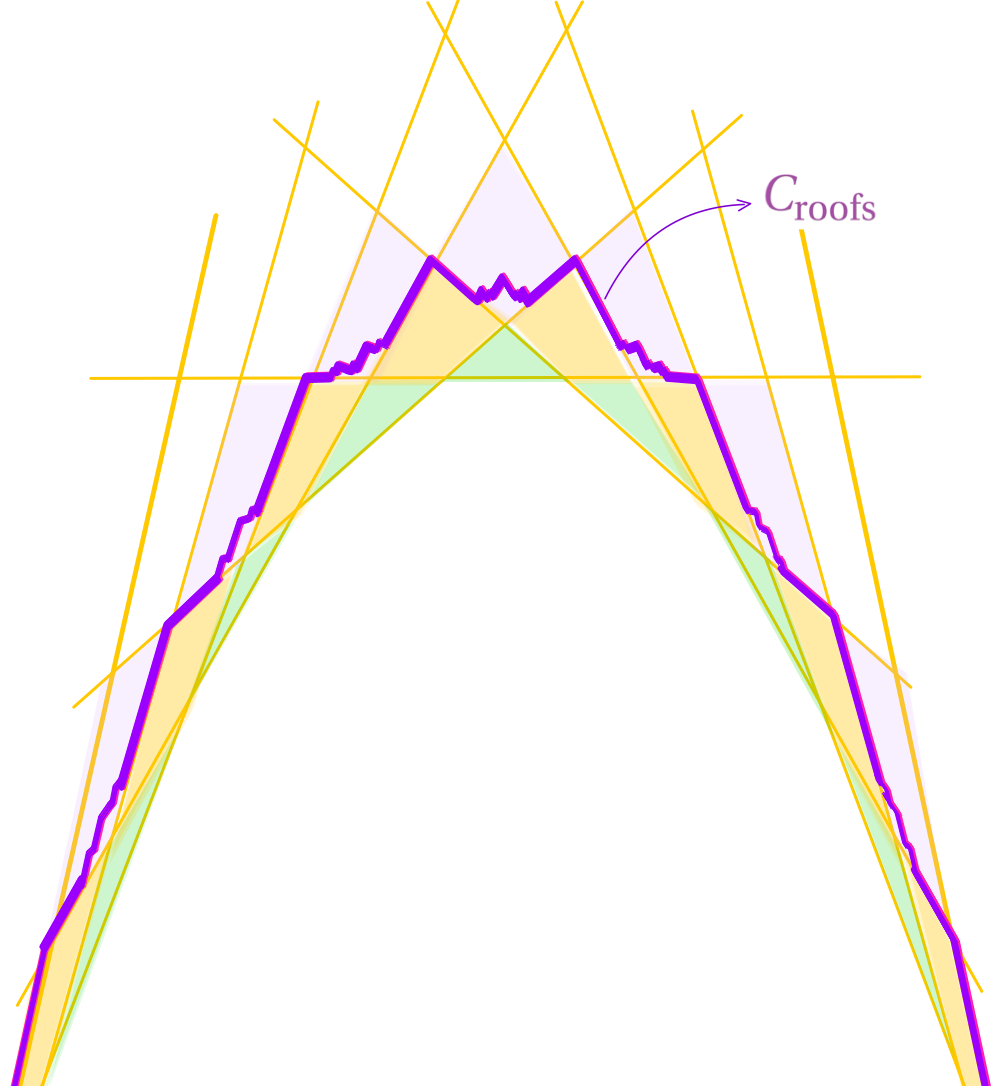}%
  }
    \caption{(1) The roofs of the super-diamonds provide an upper boundary for the space of all the outer diamonds. (2) A finer version fractal-like of the upper boundary, $C_{\mathrm{roofs}}$, is a curve consists of the roofs of all the dense diamonds. 
    }
    \label{Fig:boundary}
\end{figure}

\begin{definition}[Bounded region]\label{Def:bddRegion} We define the
\emph{bounded region}, $R_{\mathrm{bdd}}$, to be $R_{\Delta}\cup R_{\Diamond}$.
Note this is the union of all outer diamonds, and
that the upper boundary of this region is the union of the roofs of all the dense diamonds, $C_{\mathrm{roofs}}$.
\end{definition}

\begin{corollary}
\label{Thm:Equivalence}
The scattering diagrams $\foD^{\mathrm{stab}}$ and
$\foD^{\mathrm{stab}}_{\mathrm{in}}\cup\foD_{\mathrm{discrete}}
\cup\foD_{\mathrm{dense}}$ are equivalent in $R_{\mathrm{bdd}}$.
\end{corollary}

\begin{proof}
This now follows from Theorem \ref{Thm:nonEmptiness}, which shows the initial,
discrete and dense rays account for all rays in the bounded region.
\end{proof}   

\begin{theorem}
\label{thm:roof projection}
Denote by $\mathrm{pr}_1:M_{\BQ}\rightarrow \BQ$ the projection
to the $x$-axis. Then
\[
\pr_1|_{C_{\mathrm{roofs}}\cap M_{\BQ}}:
C_{\mathrm{roofs}}\cap M_{\BQ}\rightarrow\BQ
\]
is a bijection.  In particular, any vertical
line $x=q$ for $q\in\BQ$ intersects $C_{\mathrm{roofs}}$ at one point.
\end{theorem}

\begin{proof}
Denote by 
\[
S:=\{\mu(\CE)\,|\,\hbox{$\CE$ is an exceptional bundle}\}
\]
the set of slopes of exceptional bundles. 
\cite[Thm.~3.1]{veselov2025} states that $S$ coincides with the set of
Markov fractions. Also, the set of $x$-coordinates of the points
$V_{\CE}$ for $\CE$ ranging over exceptional bundles coincides with
$S-3/2$, the shift of $S$ to the left by $3/2$. 

For a given exceptional bundle $\CE$,
the left and right vertices of the dense diamond 
$\Diamond_{V_{\CE}}$ are the intersection
of a line $ax+by=c$ with $a,b,c\in \BQ$ and a line through $V_{\CE}\in M_{\BQ}$
of quadratic irational slope. Hence the left and right vertices have
quadratic irrational $x$-coordinates. 

For a given $\CE$, let $I_{\CE}\subseteq\BR$ be the projection of
the roof of the dense diamond $\Diamond_{V_{\CE}}$, excluding the
left and right vertices, so that this is an open interval with quadratic
irrational endpoints.
Note that $\mathrm{pr}_1(V_{\CE})$ is a rational element of $I_{\CE}$.

First we claim there
is no exceptional bundle $\CE'\not=\CE$ such that $\mu(\CE')-3/2
\in I_{\CE}$. Indeed, by Lemma~\ref{lem:slope inequalities},(8),
$\mu(\CE)\not=\mu(\CE')$. However, if $V_{\CE'}$ lies below the roof of
$\Diamond_{V_{\CE}}$, the vertical ray $\RR_{\ge 0}(0,1)+V_{\CE'}$ is
a ray in $\foD_{\mathrm{dense}}$ which intersects the
interior of $\Diamond_{V_{\CE}}$ but
neither contains a roof nor has endpoint $V_{\CE}$, violating 
Theorem~\ref{Thm:nonEmptiness}. Similarly, if $V_{\CE'}$ lies on or above
the roof of $\Diamond_{V_{\CE}}$, then $V_{\CE'}\in C_{V_{\CE}}$ and this
again contradicts Theorem~\ref{Thm:nonEmptiness}. Thus $I_{\CE}
\cap (S-3/2)=\{\mathrm{pr}_1(V_{\CE})\}$.

Second, from Remark~\ref{rem:various diamonds}, there is a 
sequence of diamonds
$\Diamond_{v_i}$, $i\ge 1$ whose base lies on the discrete ray containing
the left (resp.\ right) 
roof of $\Diamond_{V_{\CE}}$ and whose bases approach the left (resp.\ right)
vertex of $\Diamond_v$. Hence $I_{\CE}$ is the largest open interval 
containing $\mathrm{pr}_1(V_{\CE})$ and no other point of $S-3/2$.
By \cite[Lem.~4.16]{springborn}, we then have
\[
I_{\CE}=(\mu(\CE)-\delta_{\CE}-3/2,\mu(\CE)+\delta_{\CE}-3/2),
\]
where
\[
\delta_{\CE}=\frac{3}{2}-\sqrt{\frac{9}{4}-\frac{1}{r(\CE)^2}}.
\]
However, by \cite{drezet}, $\bigcup_{\CE} I_{\CE}$ contains all
rational numbers, hence proving surjectivity. Since $\mathrm{pr}_1$ is 
certainly injective on the roof of each diamond, and $I_{\CE}\cap I_{\CE'}=
\emptyset$ whenever $\CE\not=\CE'$ by the above disussion, we
obtain injectivity.
\end{proof}

\begin{remark}
In Theorems \ref{SufficientRank0Dene} and \ref{SufficientRank0DeneInitial}, 
we recover the description
of $I_{\CE}$ via an explicit calculation rather than from using the 
general theory of Markov fractions. However, this theory is quite beautiful
and it is worth drawing attention to it.

Also, we remark that \cite{CoskunHuizengaWoolf} showed that the complement of
$\bigcup_{\CE} I_{\CE}$ consists of the quadratic irrational endpoints
of the intervals along with a set of transcendental numbers. So it is a quite
subtle set.
\end{remark}

\begin{figure}[h]
    \centering
    \includegraphics[width=10.9cm]{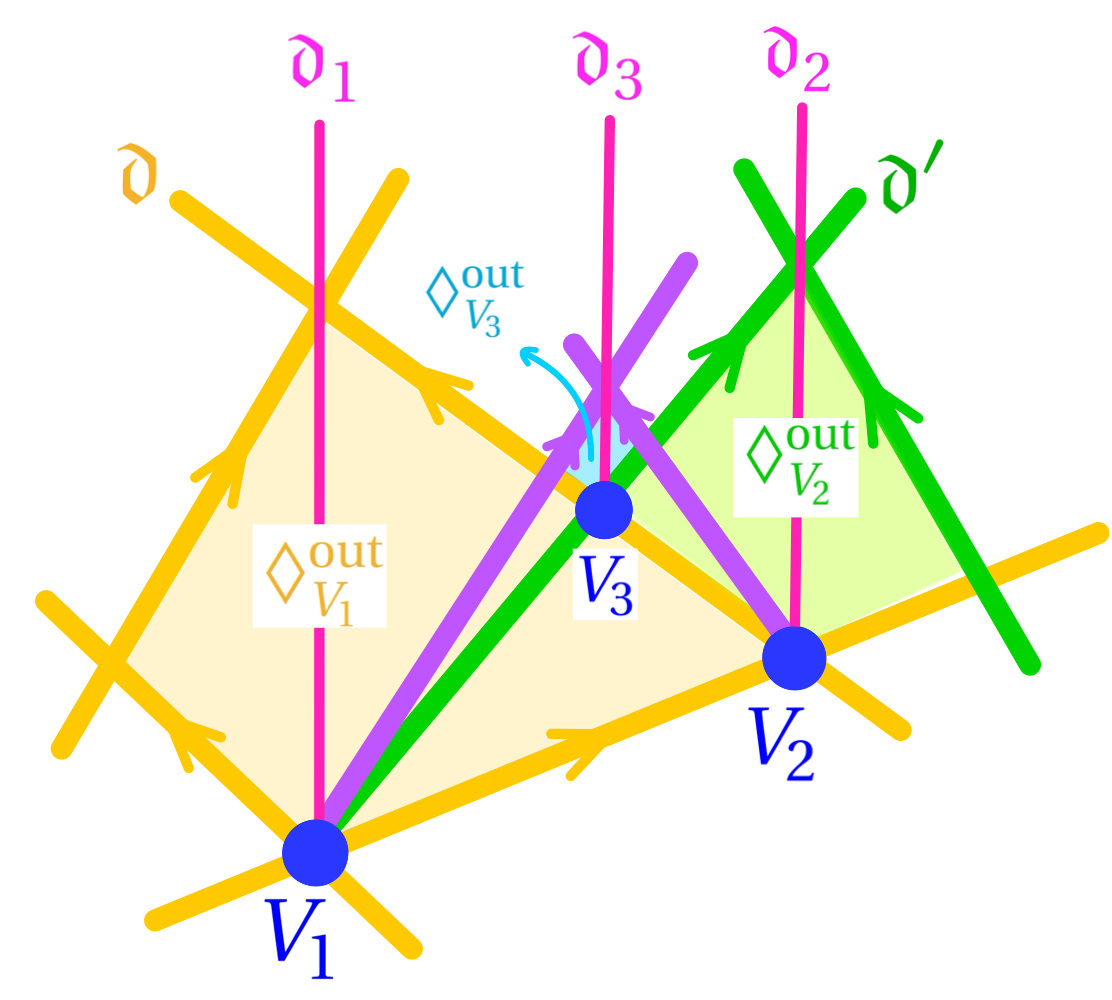}
    \caption{There is a diamond $\Diamond_{V_3}^{\mathrm{out}}$ between any two outer diamonds $\Diamond_{V_1}^{\mathrm{out}}$ and $\Diamond_{V_2}^{\mathrm{out}}$. 
}
    \label{Fig: rational}
\end{figure}

\section{Characterization of $R_{\mathrm{unbdd}}$ and  moduli non-emptiness}
\label{sec:unbounded}
In this section, we show that the behaviour in the unbounded region,
$R_{\mathrm{unbdd}}$, is extremely chaotic.
We do this by analyzing scattering at points of $C_{\mathrm{roofs}}$, showing
that there are cones based at each such point in which every ray of rational
slope appears as the support of a ray in $\foD^{\mathrm{stab}}$. We
have two different behaviours, depending on whether the point is
an \emph{apex point}, i.e., the intersection of the left and right roofs
of the diamond, or a degenerate diamond.

First, we prove several results about scattering in the setting
of the scattering diagrams introduced at the beginning of 
\S\ref{subsec:discrete scat structure}.

\begin{lemma}
\label{lem:dense scattering}
Let $m_1,m_2\in M$ be linearly independent and primitive, and let
\[
\foD_{\mathrm{in}}=\{\fod_1=(\BR m_1, (1+t_1 z^{-m_1})^3),
\fod_2=(\BR m_2, \prod_{k=1}^{\infty} (1+t_2^k z^{-km_2})^{a_k})\}
\]
for $a_k$ positive integers for all $k$. Then 
${\mathsf S}(\foD_{\mathrm{in}})$ contains a non-trivial ray with
support $\BR_{\ge 0}(am_1+bm_2)$ for all $a,b>0$ relatively prime with
$a/b \le 3 |m_1\wedge m_2|$. Further, ${\mathsf S}(\foD_{\mathrm{in}})$
contains no rays with support $\BR_{\ge 0}(am_1+bm_2)$ with $a/b> 3|m_1
\wedge m_2|$.
\end{lemma}

\begin{proof}
We may first use the deformation trick as carried out in
\cite{Graefnitz-Luo2023} by replacing the line $(\BR m_2, \prod_{k=1}^{\infty} (1+t_2^k z^{-km_2})^{a_k})$ with an infinite number of lines
$\fod_2^k:=(\BR m_2, (1+t_2^k z^{-km_2})^{a_k})$ for $k\ge 1$, and then perturbing these
lines to be be parallel to each other, see Figure \ref{Fig:d2k}.
We analyze the collision of $\fod_1$ with $\fod_2^k$.
Set $\foD^k_{\mathrm{in}}=\{\fod_1,\fod_2^k\}$.
Now let $D=|m_1\wedge m_2|$.
We first show that for any $a/b<3|m_1\wedge m_2|$, there is some
$k$ such that $\mathsf{S}(\foD^k_{\mathrm{in}})$ has a non-trivial ray
with support $\fod=\BR_{\ge 0}(a m_1+b m_2)$. 

\begin{figure}[h]
    \centering
   \includegraphics[width=12cm]{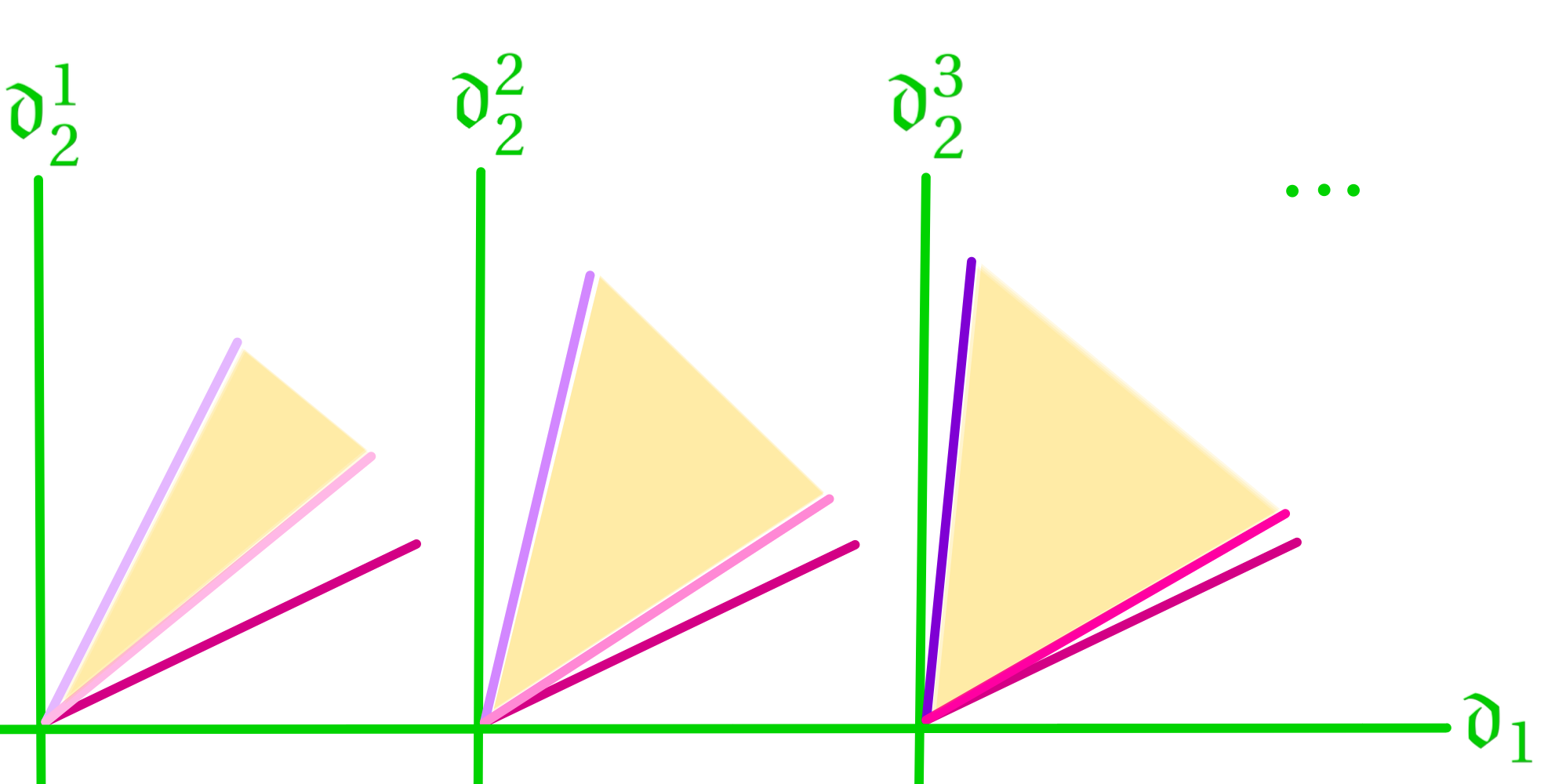}
   \caption{}
    \label{Fig:d2k}
\end{figure}

As in the proof of Lemma~\ref{lem:basic scattering2}, we may pass to a
sublattice $M'$ spanned by $m_1'=m_1$ and $m_2'=km_2$, at the cost
of replacing $\fod_1$ and $\fod_2^k$ with
$(\BR m_1,(1+t_1z^{-m_1'})^{3kD})$
and $(\BR m_2,(1+t_2^kz^{- m'_2})^{a_kkD})$ respectively. 
By \cite[Thm.~1.4]{Graefnitz-Luo2023} $\mathsf{S}(\foD^k_{\mathrm{in}})$
has non-trivial rays with support 
$\BR_{\ge 0}(am_1+bm_2)=\BR_{\ge 0}(am'_1+(b/k)m_2')$ 
for
\[
\frac{3k^2D^2a_k-\sqrt{3k^2D^2a_k(3k^2D^2a_k-4)}}{6kD}<
\frac{b}{ak}
<\frac{3k^2D^2a_k+\sqrt{3k^2D^2a_k(3k^2D^2a_k-4)}}{6kD}.
\]
After multiplying all three quantities by $k$, we make take $k\rightarrow
\infty$. The left-hand side converges to $(3D)^{-1}$ while the right
hand side converges to $\infty$. This shows the desired existence in the
given range. These rays then necessarily contribute non-trivially to
$\mathsf{S}(\foD_{\mathrm{in}})$ by positivity of the scattering
process (see \cite[Prop.~C.13]{GHKK} as applied in 
\cite[Prop.~2.9]{Graefnitz-Luo2023}). On one hand, by \cite[Thm.~1.4]{Graefnitz-Luo2023}, $\mathsf{S}(\foD^k_{\mathrm{in}})$
contains a non-trivial ray with support
$\BR_{\ge 0}(3kD m_1' + m_2')=\BR_{\ge 0}(3Dm_1+m_2)$ for any $k$.
Again, by positivity these rays contribute. On the other hand,
\cite[Prop.~3.13]{Graefnitz-Luo2023} states that this is the smallest
slope of a ray appearing in $\mathsf{S}(\foD^k_{\mathrm{in}})$, i.e., 
there is no ray with support $\BR_{\ge 0}(a m_1+bm_2)$ with $a/b>
3D$. 

This implies that $\mathsf{\foD_{\mathrm{in}}}$ can't contain any
non-trivial rays with support $\BR_{\ge 0}(a m_1+bm_2)$ with $a/b>
3D$. Indeed, in the deformed scattering diagram, all collisions
except the initial collisions of $\fod_1$ and $\fod_2^k$ for various $k$
must involve rays with support $\BR_{\ge 0}(a m_1+bm_2)$ with $a/b\le 3D$.
Thus they cannot produce new rays with $a/b>3D$.
\end{proof}

\begin{lemma}[Denseness at degenerate diamonds of $C_{\mathrm{roofs}}$]\label{Lem:DensnessNonApex}
Let $\sigma\in M_{\BR}$ coincide with a degenerate diamond. Then 
$\foD^{\mathrm{stab}}$ contains precisely two rays 
$\fod_{\mathrm{discrete}},\fod_{\mathrm{dense}}$ which contain $\sigma$
in their interior. Here $\fod_{\mathrm{discrete}}\in
\foD_{\mathrm{discrete}}$ and 
$\fod_{\mathrm{dense}}\in\foD_{\mathrm{dense}}$. Letting 
$m_{\mathrm{discrete}}$ and $m_{\mathrm{dense}}$ be the opposite vectors of these two rays and $D_{\sigma}=|m_{\mathrm{discrete}}\wedge m_{\mathrm{dense}}|$, 
we define the \emph{dense cone at $\sigma$} to be the 
cone with vertex $\sigma$ given by 
\[
C_{\sigma}:=
\{\sigma - (am_{\mathrm{discrete}}+bm_{\mathrm{dense}})\,|\,
a,b\in\BR_{>0}, a/b \le 3D_{\sigma}\}.
\]
Then there is a non-trivial ray $(\fod,H_{\fod})\in\foD^{\mathrm{stab}}$
with $\fod$ having endpoint $\sigma$ if and only if $\fod$ lies
in $C_{\sigma}$.
\end{lemma}

\begin{proof}
The description of the two rays of $\foD^{\mathrm{stab}}$ follows
immediately from the definition of degenerate diamond and Theorem
\ref{Thm:nonEmptiness}.  
Given this, we look at the local
scattering diagram $\foD^{\mathrm{stab}}_{\sigma}$. By definition of
degenerate diamond, $\sigma$ is contained in the roof of a dense
diamond $\Diamond_v$ contained in an outer diamond $\Diamond_v^{\mathrm{out}}$, with discrete generating rays $\fod_1,\fod_2$ with opposite
vectors $m_1,m_2$. Then we may write the opposite vector of
$\fod_{\mathrm{dense}}$ as $a_1 m_1+a_2m_2$ for some relatively prime positive
integers $a_1,a_2$.

If $H^{\sigma}_{\mathrm{dense}}$ is the function
attached to the tangent line of $\fod_{\mathrm{dense}}$ in the
local scattering diagram $\foD^{\mathrm{stab}}$, then the corresponding
function $f_{\mathrm{dense}}$ as derived from $H^{\sigma}_{\mathrm{dense}}$
as in Remark~\ref{rem:H auto} must take the form
\[
f_{\mathrm{dense}}=\prod_{k\ge 1} (1+t^{\alpha k} z^{a_1 k m_1+a_2 k m_2})^{c_k}
\]
for some positive integers $c_k$ and some real number $\alpha>0$. 
This form arises from the description
of rays in $\foD'_{\mathrm{dense}}$ in Lemma~\ref{lem:basic scattering}
and the proof of Lemma~\ref{lem:basic scattering2}. We note that
the $c_k$ may not agree with the $c_{ka_1,ka_2}$ of 
Lemma~\ref{lem:basic scattering} because the change of lattice trick
might change these exponents. On the other hand,
\[
f_{\mathrm{discrete}}=(1+t^{\beta} z^{m_{\mathrm{discrete}}})^3
\]
for some real number $\beta>0$.
The result now follows from Lemma \ref{lem:dense scattering}.
\end{proof}

\begin{corollary}[Denseness at apex points of $C_{\mathrm{roofs}}$]\label{cor: DensenessApex} Let $\sigma\in M_{\BR}$ be the apex of a diamond
$\Diamond$. Then 
$\foD^{\mathrm{stab}}$ contains precisely

three rays 
$\fod_{\mathrm{discrete}},\fod'_{\mathrm{discrete}},\fod_{\mathrm{dense}}$ 
which contain $\sigma$
in their interior. Here 
$\fod_{\mathrm{discrete}},\fod'_{\mathrm{discrete}}
\in \foD_{\mathrm{discrete}}$ are the discrete rays forming the left and
right roofs of $\Diamond$, and $\fod_{\mathrm{dense}}$ is the vertical
dense ray generated at the base of $\Diamond$. Letting 
$m=(r,-d)$ and
$m'=(-r',d')$ be the direction vectors of 
$\fod_{\mathrm{discrete}},\fod'_{\mathrm{discrete}}$,
with $r,r'>0$, 
define the \emph{dense cone at $\sigma$} to be the cone with vertex $\sigma$
given by 
\[
C_{\sigma}:=
\{\sigma+(a m+ (0,b))\,|\, a,b\in\BR_{>0},  a/b\le 3r\}
\cup
\{\sigma+(a m'+ (0,b))\,|\, a,b \in \BR_{>0}, a/b\le 3r'\}.
\]
Then there is a non-trivial ray $(\fod,H_{\fod})\in\foD^{\mathrm{stab}}$
with $\fod$ having endpoint $\sigma$ if and only if $\fod$ lies
in $C_{\sigma}$.
\end{corollary}

\begin{proof}
Let $v$ be the base of $\Diamond$.
Of course by definition, $\sigma$ is the intersection of the two 
discrete rays forming the roof of $\Diamond$. Further, by Lemma
\ref{lem:prince}, the vertical ray with endpoint $v$ is contained
in $C_v$, and hence supports a dense ray $\fod_{\mathrm{dense}}$. 
By \cref{Thm:nonEmptiness}, these are the only possible rays of $\foD^{\mathrm{stab}}$ 
containing $\sigma$ in their interior.

To analyze the scattering at $\sigma$, as in Lemma \ref{Lem:DensnessNonApex}, we may
pass to a scattering diagram of the sort considered in Lemma 
\ref{lem:dense scattering}, this time with 
\begin{align*}
\foD_{\mathrm{in}}= {} &
\left\{
\left(\BR m, (1+t^{\alpha}z^{-m})^3\right),
\left(\BR m', (1+t^{\beta}z^{-m'})^3\right),
\left(\BR (0,1), \prod_{k\ge 1} (1+t^{k\gamma}z^{(0,-k)})^{a_k}\right)\right\}\\
= {} & \{\fod,\fod',\fod''\}.
\end{align*}
We may then use the deformation trick to shift $\fod''$ as depicted in
Figure \ref{Fig:3d}. Note that this produces three initial collisions of these lines.
By Lemma \ref{lem:dense scattering}, the collision of $\fod$ and $\fod''$
produces only rays with direction vector positively 
proportional to $a m+(0,1)$ for $a\le 3 
|m\wedge (0,1)|=3r$. 

Similarly, the collision
of $\fod'$ and $\fod''$ produces only rays with direction vector
positively proportional to $a m'+(0,1)$ for
$a\le 3 |m'\wedge (0,1)|=3r'$. On the other hand,
using the change of lattice trick and applying 
\cite[Prop.~3.13]{Graefnitz-Luo2023}, we see that the collision of
$\fod$ and $\fod'$ produces rays with support
$\BR_{\ge 0}(am+bm')$
with $(3D_{\sigma})^{-1}\le a/b \le 3D_{\sigma}$, where $D_{\sigma}
=|m\wedge m'|$.

\begin{figure}[h]
    \centering
   \includegraphics[width=11cm]{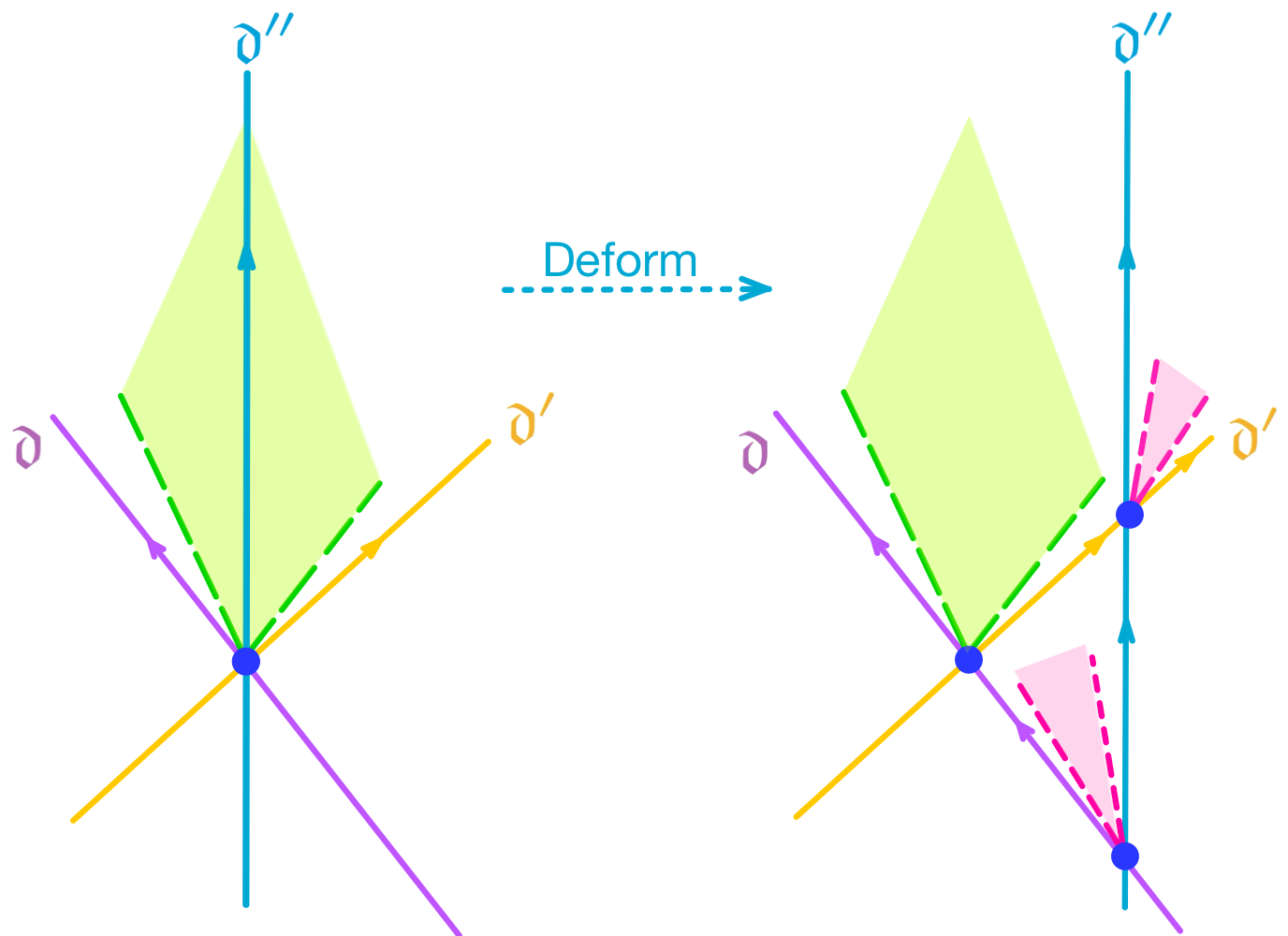}
   \caption{}
    \label{Fig:3d}
\end{figure}

We first show that any ray produced by the scattering of $\fod,\fod'$
must be contained in $C_{\sigma}$. For this, it is sufficient to show
that the rays $\sigma+\BR_{\ge 0}(3D_{\sigma}m+m')$ and
$\sigma + \BR_{\ge 0}(m+3D_{\sigma}m')$ are contained in the cone $C_{\sigma}$.
By acting via some power of $T$, we can without loss of generality assume
that the $x$-coordinate of $\sigma$ lies in the interval $[-1/2,1/2]$, and
then $d<0, d'>0$.
By symmetry, let us show that the former ray is contained in $C_{\sigma}$.
We may write $D_{\sigma}=rd'-r'd>0$ and 
$3D_{\sigma}m+m'=(3D_{\sigma}r-r', -3D_{\sigma}d+d')$. Noting that
$3D_{\sigma}r-r'>0$ by the assumption that $d<0,d'>0$, we need to show the 
slope of this vector is larger than the slope of $3rm+(0,1)=(3r^2,-3rd+1)$,
i.e.,
\begin{equation}
\label{eq:needed ineq}
\frac{-3D_{\sigma}d+d'}{3D_{\sigma}r-r'}> \frac{-3rd+1}{3r^2}.
\end{equation}
Note both numerators and denominators are positive. We calculate
\[
3r^2(-3D_{\sigma}d+d')-(3D_{\sigma}r-r')(-3rd+1)
= (3d' r -3d r')r-3D_{\sigma}r + r'=r'>0,
\]
hence \eqref{eq:needed ineq}.

It is now clear from the deformed scattering diagram Figure \ref{Fig:3d} that
all further scattering only produces rays whose asymptotic direction lie
in the asymptotic directions of the cone $C_{\sigma}$.

\end{proof}

\begin{corollary}[Unbounded region characterization]\label{Cor:unboundedCharacterization} 
Every rational point in $R_{\mathrm{unbdd}}$ is contained in the
interior of a countable number of non-trivial rays of $\text{\:}\foD^{\mathrm{stab}}$.
See Figure \ref{Fig: UnboundedCharacterization}.
\end{corollary}

\begin{proof}

Let $\sigma=(x,y)\in R_{\mathrm{unbdd}}$ be a point with rational coordinates.
By Theorem~\ref{thm:roof projection}, this point will lie above a unique dense diamond $\Diamond_v$, with base 
$v$,
apex $v_A=(x_A,y_A)$ and left and right vertices $v_L=(x_L,y_L)$ and 
$v_R=(x_R,y_R)$.
Without loss of generality, we can assume that $x_L< x \le x_A$. Note
we have strict inequality on the left as $x_L$ is necessarily irrational.

Let $\sigma'=(x',y')$ 
be a rational point in the left roof of $\Diamond_v$ with $x<x'< x_A$
and such that $\sigma$ lies to the right of the 
ray $\fod'$ emanating from $v$ through $\sigma'$. There will be
an interval of points on the left roof with this property, hence 
a countable number of such possible $\sigma'$. Now for a given $\sigma'$,
consider the cone $C_{\sigma'}$ of Lemma~\ref{Lem:DensnessNonApex}.
We claim that $\sigma\in C_{\sigma'}$. Since $\sigma$ lies to the
left of $\sigma'$, it is enough to show that the cone
$\sigma'+\BR_{\ge 0} m'+\BR_{\ge 0}(0,1)$ is contained in $C_{\sigma'}$,
where $m'$ is the direction vector of $\fod'$.
Using the notation of Lemma~\ref{Lem:DensnessNonApex} with
$\fod_{\mathrm{discrete}}$ being the left roof and $\fod_{\mathrm{dense}}$
being $\fod'$, let $(a_1,a_2)$ the the direction vector for
$\fod_{\mathrm{discrete}}$ and $(b_1,b_2)$ be the direction vector for
$\fod_{\mathrm{dense}}$. Then $a_1>0$, $b_1<0$, and after twisting
so that the base of the diamond has $x$-coordinate in $[-1/2,1/2]$, we
can assume $a_2,b_2>0$.
The right edge of $C_{\sigma'}$ has direction vector 
\[
3(a_1b_2-a_2b_1)(a_1,a_2)+(b_1,b_2),
\]
which then clearly has positive $x$-coordinate. Thus the ray
$v+\BR_{\ge 0}(0,1)$ is contained in $C_{\sigma'}$. As $v+\BR_{\ge 0}m'$
is contained in $C_{\sigma'}$ by definition, we thus see that
$\sigma'+\BR_{\ge 0} m'+\BR_{\ge 0}(0,1)$ is contained in $C_{\sigma'}$
as claimed.

We thus see that for every $\sigma'$ as above, there is a ray in
$\foD^{\mathrm{stab}}$ with endpoint $\sigma'$ and passing through
$\sigma$. This proves the result.
\end{proof}

 \begin{figure}[h]
 \subcaptionbox*{}[0.94\linewidth]{%
    \includegraphics[width=\linewidth]{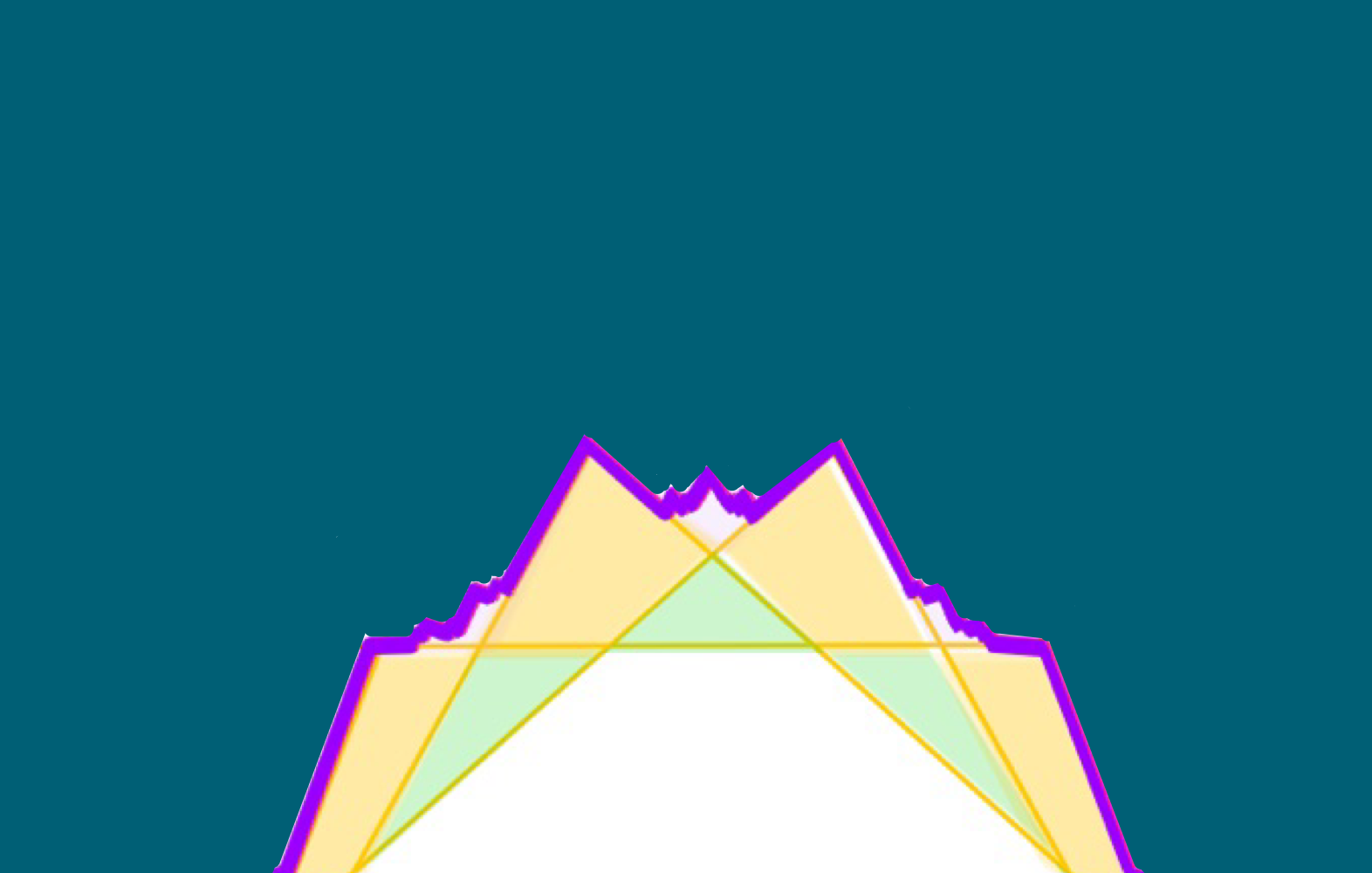}%
    }%
  \caption{The rational points in the unbounded region are all vertices.}
  \label{Fig: UnboundedCharacterization}
\end{figure}

\begin{corollary}
 [Classification of the rays generated in the unbounded region]\label{BoundedGenerating}Let $\gamma=(r,d,e)$ be the Chern character of an object
in $\mathrm{D}^b(\BP^2)$, such that $L_{\gamma}$ is not vertical and does not contain a discrete ray. 
Let $\sigma_0\in L_{\gamma}\cap C_{\mathrm{roofs}}$ be a point so that
$\sigma_0-\BR_{> 0}m_{\gamma}\subseteq R_{\mathrm{unbdd}}$.
Then the first generating point $\sigma$ of $\gamma$ lies in 
$R_{\mathrm{unbdd}}$ if and only if $\sigma_0-\BR_{>0}m_{\gamma}$ is
not contained in the dense cone $C_{\sigma_0}$ at 
$\sigma_0$.
\end{corollary}

\begin{proof}
Let $\sigma$ be the first generating point for $\gamma$, and let
$(\fod,H_{\fod})\in \foD^{\mathrm{stab}}$ be a ray with endpoint $\sigma$.
By assumption, $\fod$ is not a discrete ray.
By Theorem~\ref{Thm:nonEmptiness},
we only have the three following possible cases (see Figure \ref{Fig: allRaysGenInBdd}):

    \textbf{Case I.} The intersection of $L_\gamma$ with $C_{\mathrm{roofs}}$ is a point $p$ such that $p-\BR_{\ge 0}m_{\gamma}$ 
is contained in the dense cone $C_{p}$ at $p$. In this case,
from Lemma \ref{Lem:DensnessNonApex}, $p=\sigma$ is in fact the first generating point of $\fod$.  This also holds when $p$ is an apex (for which we use Lemma \ref{cor: DensenessApex}). 

    \textbf{Case II.} The intersection of $L_\gamma$ with $C_{\mathrm{roofs}}$ is a point $p$ so that $L_\gamma$ meets the base point $v$ of the corresponding  diamond (the dense diamond whose roof contains $p$). In this case $v$ is in fact the first generating point of $\fod$.

  \textbf{Case III.} If Cases I or II don't occur, then $\fod$ is in fact generated at some point $\sigma$ in $R_{\mathrm{unbdd}}$.

          \begin{figure}[H]
  \subcaptionbox*{Case I ($p$ non-apex)}[0.48\linewidth]{%
    \includegraphics[width=\linewidth]{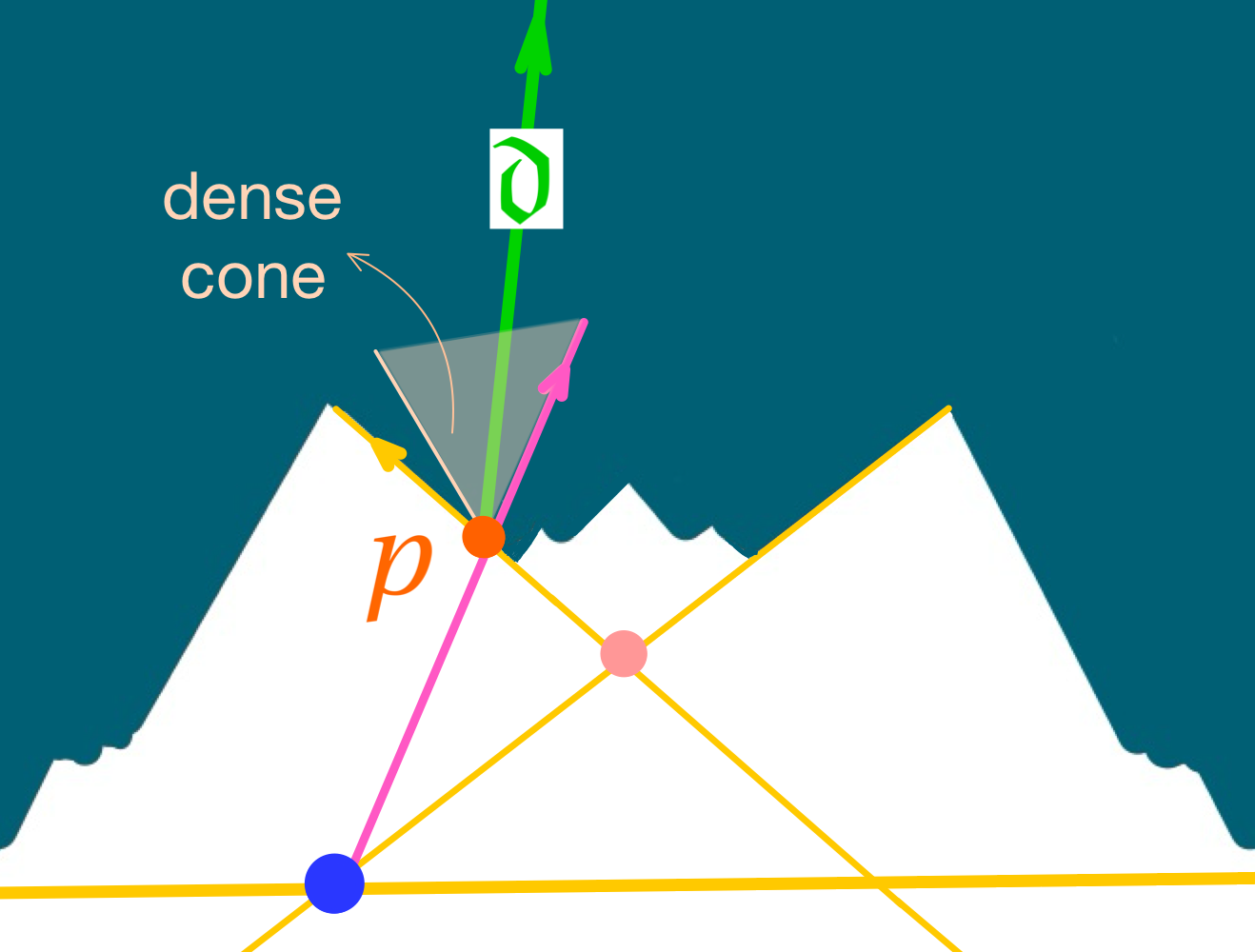}%
  }
  \hskip2.8ex
  \subcaptionbox*{Case I ($p$ apex)}[.433\linewidth]{%
    \includegraphics[width=\linewidth]{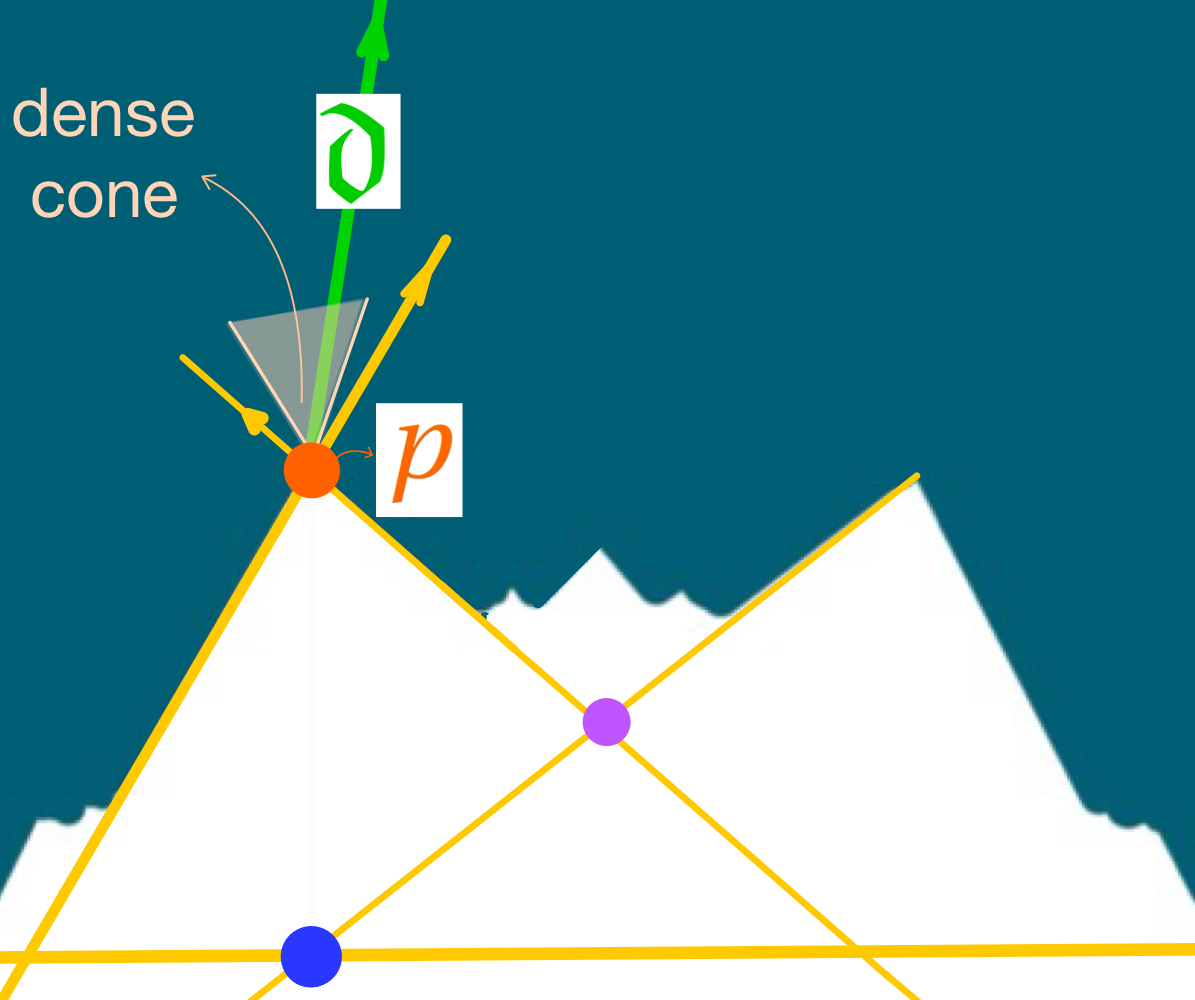}%
    }
  \hskip1.8ex
  \subcaptionbox*{Case II}[.453\linewidth]{%
    \includegraphics[width=\linewidth]{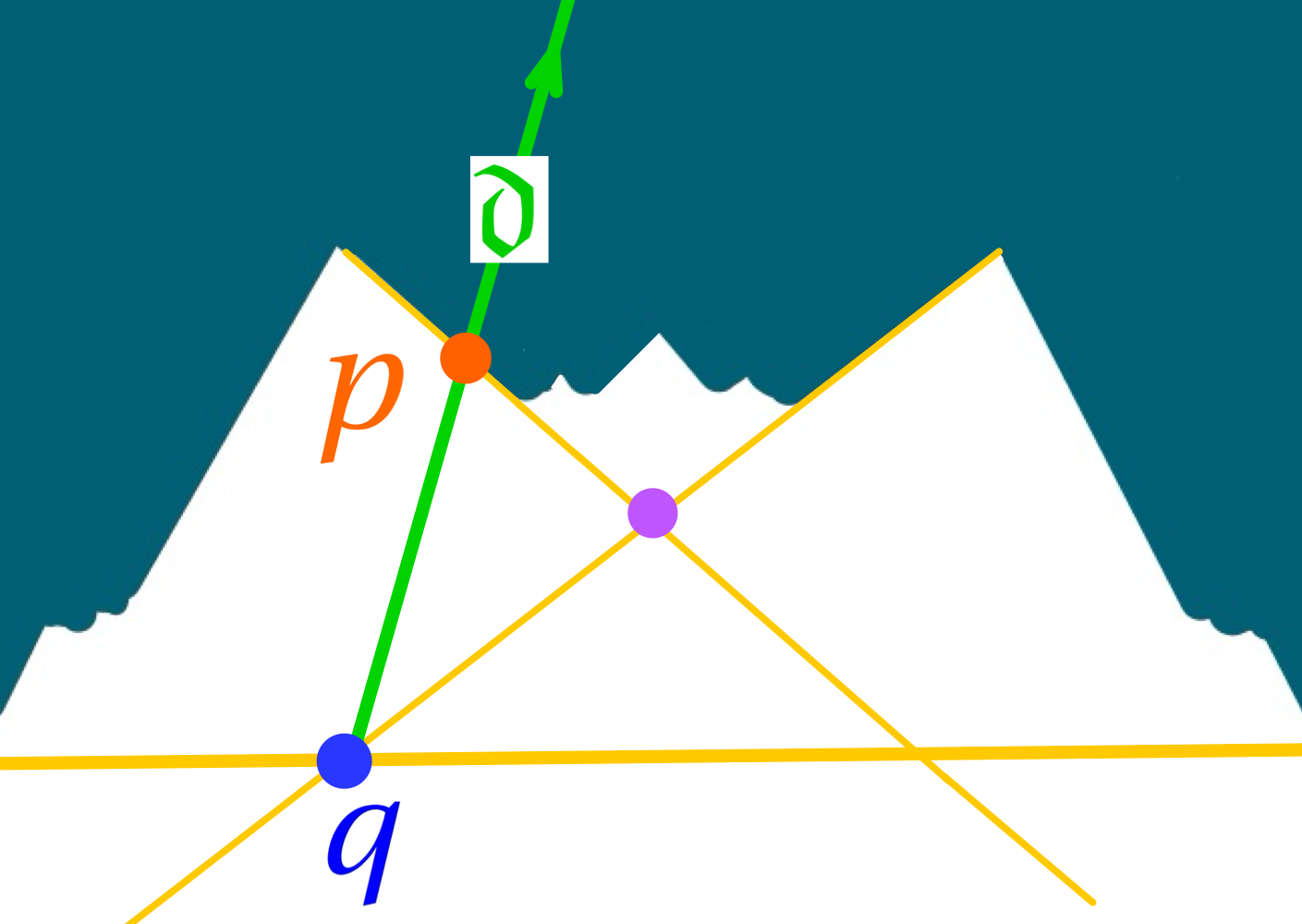}%
  }\hskip2.8ex
  \subcaptionbox*{Case III ($p$ non-apex)}[.453
  \linewidth]{%
    \includegraphics[width=\linewidth]{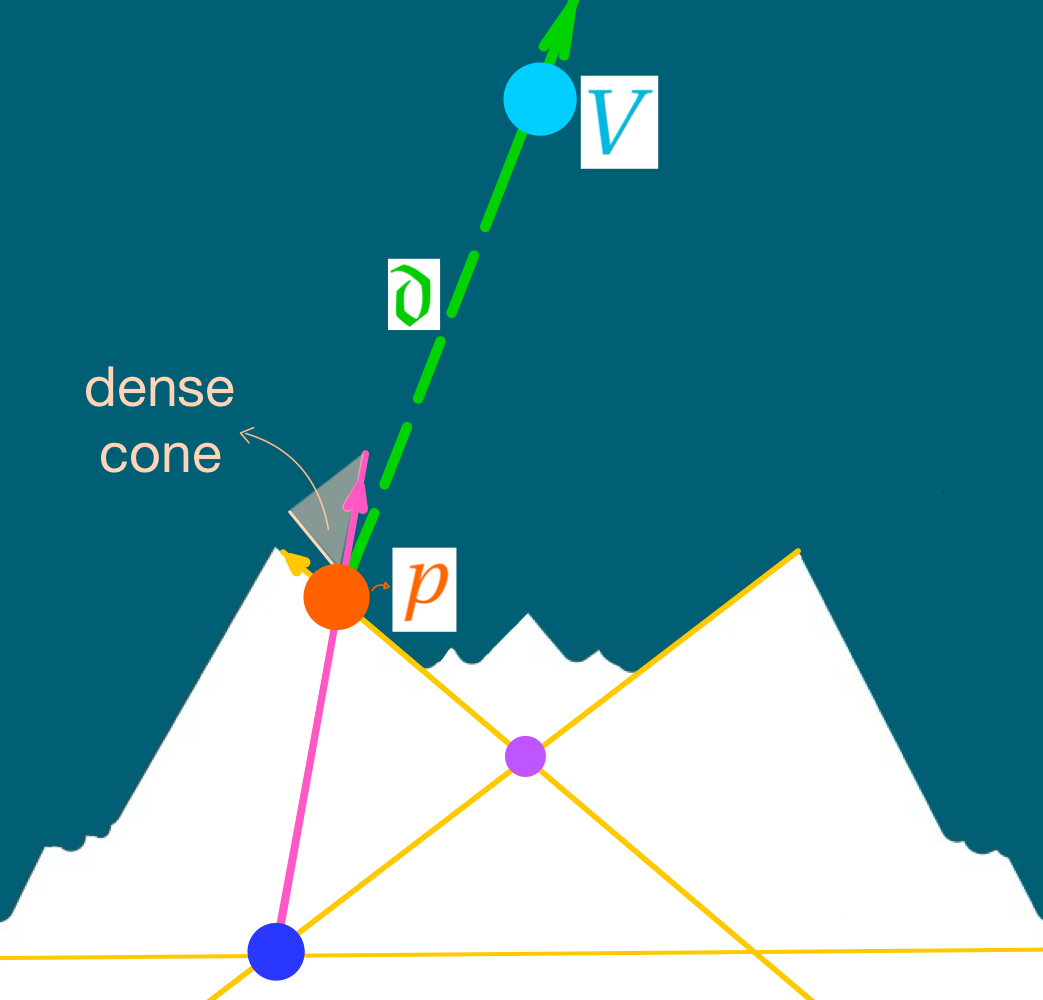}%
  }\hskip2.8ex
  \subcaptionbox*{Case III ($p$ apex)}[.453
  \linewidth]{%
    \includegraphics[width=\linewidth]{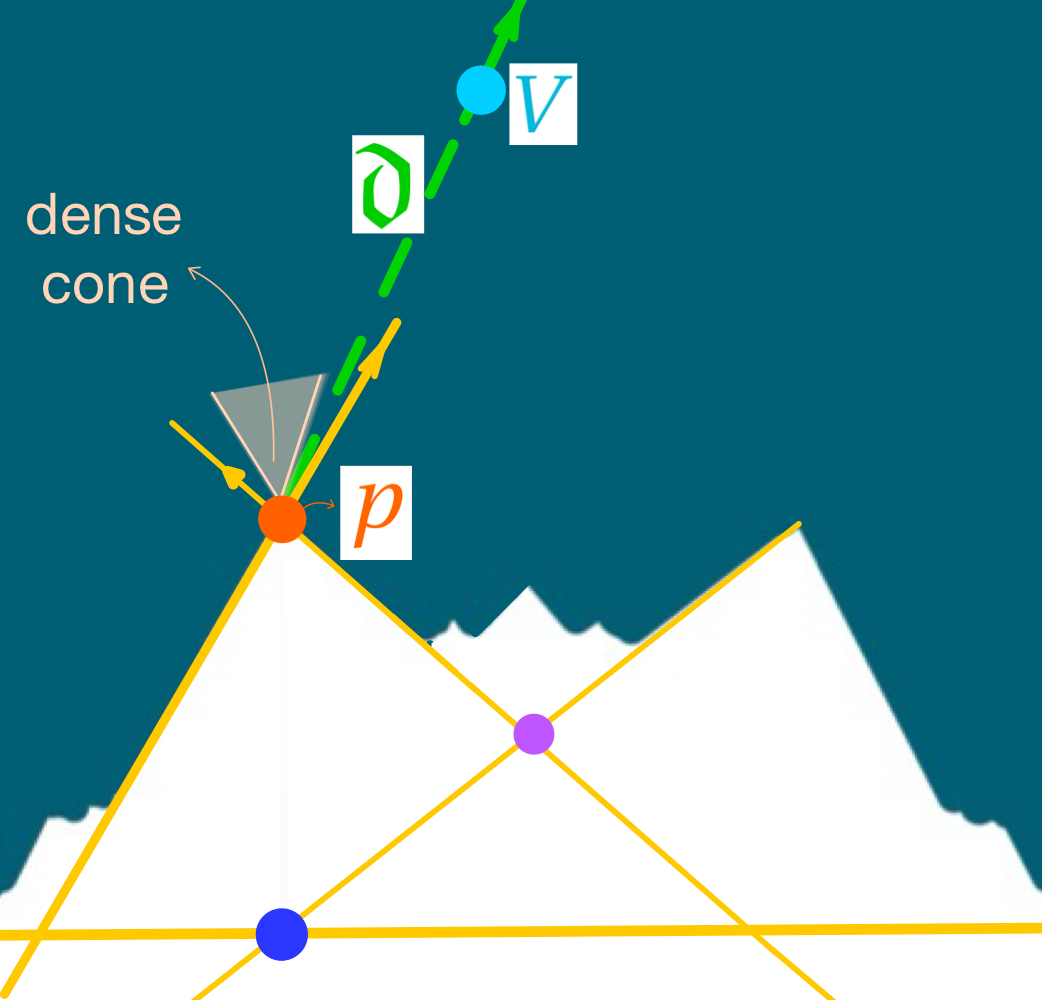}%
  }
    \caption{In Case I ($p$ non-apex) or Case I ($p$ apex), the first generating point of $\fod$ is $p$.  In Case II, the first generating point of $\fod$ the base point $q$, and in Case III ($p$ non-apex) or Case III ($p$ apex), the first generating point of $\fod$ is a point $V$ in the unbounded region.}
    
    \label{Fig: allRaysGenInBdd}
\end{figure}

\end{proof}

\begin{example}\label{Ex:ConceretBoundedGenerating}
    We give some explicit examples for the cases in the proof of Corollary \ref{BoundedGenerating} and Figure \ref{Fig: allRaysGenInBdd}.

    For  Case I ($p$ non-apex), we may take $\gamma=(5,-30,10)$
    
  For Case I ($p$ apex), we may take $\gamma=(2,-7,11/2)$. Note that the right boundary of the dense cone at $p=(-1/2,1)$ is given by $3y-7x=13/2$. Hence, $L_{\gamma}$ is  contained in the dense cone. 

   For  Case II, we may take $\gamma=(5,-20,10)$.

  For Case III ($p$ non-apex), we may take $\gamma=(5,-12,10)$.

  For Case III ($p$ apex), we may take $\gamma=(4,-9,17/2)$. Again since the right boundary of the dense cone at $p=(-1/2,1)$ is given by $3y-7x=13/2$, $L_{\gamma}$ is not contained in the dense cone.

\end{example}

\begin{remark}
    Note that if, in  Corollary \ref{BoundedGenerating}, the first generating
point lies in $R_{\unbdd}$, it is still possible for a moduli space
$\CM^{\sigma}_{m\gamma}$ to be non-empty before the first generating point.
This is because it is possible that the moduli space has zero Euler
characteristic in intersection cohomology but nevertheless be non-empty.
The analysis of the scattering at points of roofs of diamonds is purely
numerical.

Conversely, once $\sigma$ is past the first generating point, then
$\CM^{\sigma}_{m\gamma}$ is necessarily non-empty for some $m$.
\end{remark}

\section{An application: First wall-crossing for the moduli spaces of one-dimensional rank-zero objects}\label{section:rankZero}

In this section, we will show how the scattering diagram allows us to
easily determine the location of first (actual) walls for one-dimensional rank $0$
objects. Recall the $K$-theory lattice $\Gamma$ as in \eqref{eq:Gamma def}, 
which is identified via the Chern character with a sublattice of 
$\BZ\times\BZ\times\BZ/2$ as in \eqref{eq:chern image}.
In particular, the Chern characters of one-dimensional rank 0 
objects are of the form $(0,u,v)$ with either $u$ even and $v$ an integer,
or $u$ is odd and $v$ is of the form 
$(2k+1)/2$. We will focus on primitive Chern characters. The reason
for this 
is that we eventually want to get a correspondence between the rational 
vertical rays 
and the structure sheaves of schemes which are unions of $C\cup Z$ where
$C$ is a plane curve and $Z$ is a finite length scheme, with $C$ and $Z$
having as small a degree as possible.
We call an object with a primitive Chern character a \textit{primitive object}.

\begin{lemma}[Diamonds/one-dimensional rank $0$ objects correspondence]\label{lem: rank0DimondsCorrespondence}
The one-dimensional rank-zero objects of $\D^b(\BP^2)$ correspond to vertical lines.
In particular, there is a one-to-one correspondence between (outer or degenerate diamonds) diamonds and primitive Chern characters of one-dimensional 
rank-zero objects. 
\end{lemma}

\begin{proof} 
The first sentence is evident as a one-dimensional rank-zero object with Chern character $(0,u,v)$, for $u\neq 0$, corresponds to the line $x=-v/u$.
For the second statement, by Theorem~\ref{thm:roof projection}, the
line $x=-v/u$ intersects $C_{\mathrm{roofs}}$ at one point. If this point
is the apex of an outer diamond, we associate this rank-zero object
with this outer diamond. Otherwise the intersection is a degenerate
diamond in the sense of Definition~\ref{Def: positiveZeroDiamond}.
This gives the desired correspondence.
\end{proof}

\begin{remark}
It is worth noting that rank $0$ objects play a special role here. For example,
at a vertex $V_{\CE}$ for $\CE$ an exceptional bundle, one may check
that for any fixed $r\in \BZ^+$, there are infinitely many rays with endpoint
$V_{\CE}$ in $\foD_{\mathrm{dense}}$ corresponding to rank $r$ objects.
\end{remark}
   
  \begin{definition}[Discrete and dense rational number/line] 
  We call a rational number $q$, a \emph{discrete rational number}, if $x=q$ contains the vertical diagonal of an outer diamond. 
  If $x=q$ contains a degenerate diamond, then we call $q$ a \emph{dense rational number}.  Alternatively, we call $x=q$ a \emph{dense rational line}.
   \end{definition}

First, we need some lemmas, beginning with:

\begin{lemma}[A necessary condition for being generators of an outer diamond]\label{Lem: NecCondition}
    Let $\fod_1$ and $\fod_2$ be respectively the left and right generators of an outer diamond $\Diamond$ with respective primitive Chern characters $(r_1,d_1,e_1)$ and $(r_2,d_2,e_2)$ with $r_1,r_2>0$.
Then  
    \begin{align}
        D=\frac{2}{3}(r_1r_2-d_1d_2+r_1e_2+r_2e_1),
    \end{align}
    where $D=r_2d_1-r_1d_2>0$ is the  determinant at $\fod_1\cap\fod_2$.
    \end{lemma}
    \begin{proof}
First note that indeed $D$ is positive: the direction vectors for
$\fod_2$ and $\fod_1$ are $(r_2,-d_2)$ and $(-r_1,d_1)$ respectively, and
these form an oriented basis, as is visible in Figure~\ref{Fig: 1/D2}.
Let $\Delta_w$ be the triangle whose incoming vertex is the base
of $\Diamond$, corresponding to an strong exceptional triple $(\CE_0,\CE_1,\CE_2)$. 
Then recall from Remark~\ref{rmk:left right exceptionals} and 
Figure~\ref{Fig: RL} that $\fod_1$ corresponds to $\CE_0$ and
$\fod_2$ corresponds to $\CE_2(-3)[1]=\CS_2$. Let $\chi_1,\chi_2$
be the Euler characteristics of $\CE_0$, $\CE_2(-3)$ respectively.
Then by \eqref{eq: Kronrcker}, $0=\chi(\CE_0,\CS_2)=\chi(\CE_0,\CE_2(-3)[1])
=-\chi(\CE_0,\CE_2(-3))$, so
the pairing \eqref{eqn: PairingGamma} must vanish:
        \begin{align*}
            0=&\left((r_1,d_1,\chi_1),(r_2,d_2,\chi_2)\right) = -3d_1r_2-r_1r_2-d_1d_2+r_1\chi_2+\chi_1 r_2\\=&-3d_1r_2-r_1r_2-d_1d_2+r_1(e_2+r_2+\frac{3}{2}d_2)+r_2(e_1+r_1+\frac{3}{2}d_1)\\
            =&-\frac{3}{2}d_1r_2-d_1d_2+r_1e_2+\frac{3}{2}r_1d_2+r_2e_1+r_1r_2.
        \end{align*}
       Multiplying this by $\frac{2}{3}$ gives the result. 
    \end{proof}

\begin{lemma}[Characterization of the Chern character of the vertical diagonals of outer/dense diamonds]\label{lem: verticelPositive}
    Let $\gamma=(0,u,v)$ or $(0,-u,-v)$, for $u\in \BZ^{> 0}, v\in\frac{\BZ}{2}$, be primitive in $\Gamma$. Then $L_\gamma$ contains the vertical diagonal of an outer/dense diamond if and only if 
    \begin{align}\label{eq:(r,d,e)}
        \begin{cases}
            r=u,\\
            d=v+\frac{3}{2}u,\\
            e=\frac{4+5u^2+4v^2+12uv}{8u}=\frac{1+v^2}{2u}+\frac{5}{8}u+\frac{3}{2}v,
        \end{cases}
    \end{align}
    where $(r,d,e)\in\Gamma$ is the Chern character of the exceptional
bundle $\CE$ such that $V_{\CE}$ is the base of the diamond.
\end{lemma}

\begin{proof}
First suppose \eqref{eq:(r,d,e)} holds. Then as $L_{\gamma}$ is the line
$x=\frac{v}{u}$, and $\frac{v}{u}=\frac{d}{r}-\frac{3}{2}$ by 
\eqref{eq:(r,d,e)}, we see $L_{\gamma}$ passes through $V_{\CE}$ by
\eqref{eq: VertexPE} and hence contains the vertical diagonal of the diamond
based at $V_{\CE}$.

For the opposite direction, first, we show that $r=u$. Since $\chi=v+\frac{3}{2}u$, we can write $\gamma=(0,u,\chi-\frac{3}{2}u)$. Therefore, we obtain 
\begin{align}\label{dR=Xu}
    \frac{d}{r}=\frac{\chi}{u}.
\end{align}
Now, the primitivity of $\gamma=(0,u,v)$ in $\Gamma$
implies the primitivity of $(0,u,\chi)$, which in turn implies that $\gcd(u,\chi)=1$. On the other hand,
$\gcd(d,r)=1$ by Lemma~\ref{lem:slope inequalities},(6). Hence, from \eqref{dR=Xu}, we must have $r=u$ and $d=\chi$. Also, the latter immediately implies  $d=v+\frac{3}{2}u$.

       Finally, for $e=\frac{4+5u^2+4v^2+12uv}{8u}$, it is enough to use Lemma \ref{lem:slope inequalities}(6) and substitute $d,r$:
       \[e=\frac{d^2-r^2+1}{2r}=\frac{v^2+3vu+\frac{9u^2}{4}-u^2+1}{2u}=\frac{4v^2+12uv+5u^2+4}{8u}.\]

\end{proof}

Next, we have a useful statement on the length of the vertical diagonals of diamonds.

\begin{lemma}[Length of the vertical diagonal]\label{Lem: lengthDiagonal}The vertical diagonal corresponding to the primitive Chern character $(0,u,v)$ of a diamond based at a vertex $V_0=(x_0,y_0)$ dual to a ray (in the sense of
Definition~\ref{Def:vertexRayDuality}) $\fod$ of Chern character $(r,d,e)$ has length $\frac{1}{r^2}$, which by Lemma \ref{Lem: degree=rank} and and \eqref{eq:(r,d,e)}, is the same as $\frac{1}{D^2}$ and $\frac{1}{u^2}$, where $D$ is the determinant at the base, and $u$ is the first (primitive) Chern character associated to the vertical diagonal.
\end{lemma}
\begin{proof}

Let $V'=(x',y')$ be the intersection of $x=x_0$ and $\fod$. 
 Hence, we have $x'=x_0$ and $y'=\frac{e-dx_0}{r}=\frac{e-d\frac{v}{u}}{r}$. Then, using \eqref{eq: VertexPE} and \eqref{eq:(r,d,e)}, we get

\begin{align*}
y'-y_0=   2\frac{e}{r}-\frac{d}{r}\frac{v}{u}-\frac{3}{2}\frac{d}{r}+1=
   2\frac{e}{r}-\left(\frac{d}{r}\right)^2+1=\frac{1}{r^2}.
\end{align*}
The last equality comes from Lemma \ref{lem:slope inequalities}(6). Therefore, we have the result.  See Figure \ref{Fig: 1/D2}.

\begin{figure}[h]
    \centering
    \includegraphics[width=12cm]{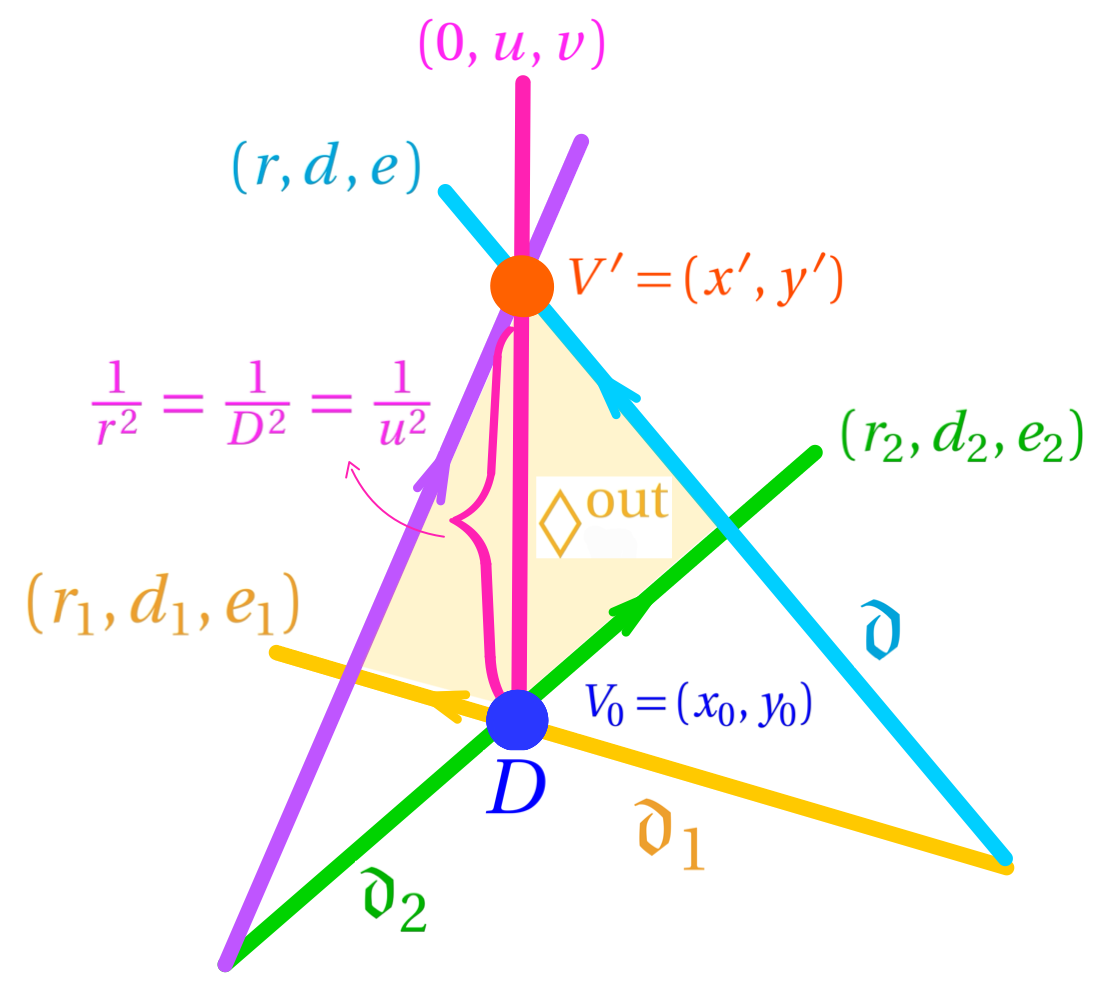}
    \caption{The length of the vertical diagonal of the diamond $\Diamond^{\mathrm{out}}$ based at the vertex $V_0$ dual to the ray $\fod$ of rank $r$ is $1/r^2=1/D^2=1/u^2$, where $D$ is the determinant at the base, and $u$ is the first (primitive) Chern character associated to the vertical diagonal.}
    \label{Fig: 1/D2}
\end{figure}
    
\end{proof}

Now, we can prove one of the main results of this section. First, we need some definitions.

\begin{definition}[Equivalent objects] We call objects $\CE,\CF$ in $\mathrm{D}^b(\BP^2)$ with the same Chern character \textit{equivalent objects} and will be denoted by $\CE\sim\CF$.       
\end{definition}

\begin{definition}[Dual objects] Let $\CE$ be an object in the derived category of Chern character $(r,d,e)$. The object $\CE^{\vee}={\bf R}\Hom(\CE,
\CO_{\BP^2})$ with Chern character $(r,-d,e)$ is called the \emph{dual object} of $\CE$.       
\end{definition}

\begin{definition}

We use the notation $\CO_{C_d\cup n}$ for the equivalence class of
objects containing the structure sheaf of a degree $d$ curve in $\BP^2$
union $n$ points. We use the convention that $\CO_{C_d\cup n}$ does not
exist if $n<0$.
\end{definition}

\begin{lemma}
\label{lem:Cherns}
The Chern character of $\CO_{C_d\cup n}$ is
$(0,d,n-d^2/2)$.
\end{lemma}

\begin{proof}A straightforward argument (either using a corresponding  short exact sequence, or using the definition of arithmetic genus) implies the claim.

\end{proof}

\begin{theorem}[First actual wall for moduli spaces
of discrete one-dimensional rank-zero objects]\label{Thm: FirstWallRank0}
    Let $\gamma=(0,u,v)$ or $(0,-u,-v)$, for $u\in \BZ^{> 0}, v\in\frac{\BZ}{2}$, be primitive, such that $L_{\gamma}$ contains the vertical diagonal of an outer/dense diamond. Then the associated primitive object to $\gamma$ is $\CO_{C_u\cup\left(u^2/2+v\right)}$ or $\CO^{\vee}_{C_u\cup\left(u^2/2-v\right)}$ in the two cases for $\gamma$.
    Also, the first actual wall for the moduli space of one-dimensional rank-zero objects in $\mathrm{D}^b(\BP^2)$ of Chern character (at least a multiple of)

$\gamma$ is given by the vertex $(\frac{v}{u},\frac{5}{8}-\frac{1+v^2}{2u^2})$.

\end{theorem}

\begin{proof}

By Lemma~\ref{lem:Cherns}, the Chern character of
$\CO_{C_u\cup (u^2/2+v)}$ is $(0,u,v)$. Similarly, the Chern character
of $\CO_{C_u\cup (u^2/2-v)}$ is $(0,u,-v)$, and hence the Chern character
of the dual object is $(0,-u,-v)$.

From \eqref{eq: VertexPE} and  \eqref{eq:(r,d,e)}, if
the base of the diamond is $V_{\CE}$ for $\CE$ an exceptional bundle
with Chern character $(r,d,e)$,  the coordinates $(x_0,y_0)$ of the first generating point are given by
 \begin{align*}
     \begin{cases}
         x_0=\frac{d}{r}-\frac{3}{2}=\frac{v}{u},\\
         y_0=\frac{3d}{2r}-\frac{e}{r}-1=\frac{3}{2}(\frac{3}{2}+\frac{v}{u})-1-\frac{3}{2}\frac{v}{u}-\frac{5}{8}-\frac{1}{2u^2}-\frac{v^2}{2u^2}=\frac{5}{8}-\frac{1+v^2}{2u^2},
     \end{cases}
 \end{align*}
 as claimed. {The non-emptiness of the moduli space comes from Lemma \ref{Lem:DensnessNonApex}.}
\end{proof}

We have the following corollaries for special cases.

\begin{corollary}[Rank-zero moduli spaces generated at the initial diamonds]Let $i$ be an odd integer.
\begin{itemize}
    \item[(1)] The upward vertical ray
contained in the line 
$x=i/2$ with endpoint $(\frac{i}{2},\frac{1-i^2}{8})$, corresponds to an initial diamond and the primitive object $\CO_{L \cup \frac{i+1}{2}}\sim\CO_{L}(\frac{i+1}{2})$, where $L$ is a line in the plane.

    \item[(2)] For $\CO_{L \cup \frac{1+i}{2}}\sim\CO_{L}((i+1)/2)$, the left generator is $\CO(\frac{i+1}{2})$ and the right generator is $\CO(\frac{i-1}{2})[1]$.
   
\end{itemize}
    
\end{corollary}

\begin{proof}
 The generators of the initial diamonds are given by the lines corresponding to $\CO(\frac{i-1}{2})$ and $\CO(\frac{i+1}{2})$:
    \begin{align*}
        \begin{cases}
            y+\frac{i-1}{2}x=\frac{(\frac{i-1}{2})^2}{2},\\
            y+\frac{i+1}{2}x=\frac{(\frac{i+1}{2})^2}{2},
        \end{cases}
    \end{align*}
    which gives $x=\frac{i}{2}$ and $y=\frac{-i^2+1}{8}$.

    Twisting $\CO(-1) \to \CO \to \CO_L$ by $\frac{i+1}{2}$, we obtain the triangle:
    \[\CO\left(\frac{i-1}{2}\right) \to\CO\left(\frac{i+1}{2}\right) \to 
\CO_L\left(\frac{i+1}{2}\right),\]
    which concludes the statement about the primitive Chern character $(0,1,\frac{i}{2})$ for $\CO_L(\frac{i+1}{2})$, and the statement about the left and right generators.
\end{proof}

\begin{corollary}[Rank-zero moduli spaces generated at the center of super-diamonds]
    Let $i\in \BZ^{\geq 0}$. 

    \begin{itemize}
        \item[(1)] Consider the diamond $\Diamond_v$ based at $v=(i,\frac{1-i^2}{2})$. The
        upward vertical ray contained in the line 
$x=i$ generated at $v$ has corresponding primitive object $\CO_{C_2 \cup 2(1+i)}\sim\CO_{C_2}(i+1)$ or $\CO^{\vee}_{C_2 \cup 2(1-i)}$.  The latter only makes sense when $i=0,1$. 

  \item[(2)] 
  Consider the diamond $\Diamond_v$ based at $v=(-i,\frac{1-i^2}{2})$. The
        upward vertical ray contained in the line 
$x=-i$ generated at $v$ has corresponding primitive object $\CO^{\vee}_{C_2 \cup 2(1+i)}$ or $\CO_{C_2 \cup 2(1-i)}$.  Again, the latter only makes sense when $i=0,1$.

    \item[(3)] For $\CO_{C_2 \cup 2(1+i)}\sim\CO_{C_2}(i+1)$, the left generator is $\CO(i+1)$ and the right generator is $\CO(i-1)[1]$. For $\CO^{\vee}_{C_2 \cup 2(1-i)}$, the left generator is $\CO(-i+1)[-1]$ and the right generator is $\CO(-i-1)$.
    \end{itemize}
\end{corollary}
\begin{proof} We prove (1); then, (2) is similar. 

The lower edges of the central diamonds in the super-diamonds are given by the lines corresponding to $\CO(i-1)$ and $\CO(i+1)$:

    \begin{align*}
        \begin{cases}
            y+(i-1)x=\frac{(i-1)^2}{2},\\
           y+(i+1)x=\frac{(i+1)^2}{2},
        \end{cases}
    \end{align*}
    which gives $x=i$ and $y=\frac{1-i^2}{2}$.

For (3) and the second part of (1) and (2), we note that from 
    \[\CI_{C_2}(i+1)\to\CO(i+1)\to\CO_{C_2}(i+1),\]
    we have the exact triangle
    \[\CO(i+1)\to\CO_{C_2}(i+1)\to\CO(i-1)[1].\]
    Dualizing this, we get
\[\CO(-i-1)\leftarrow\CO^{\vee}_{C_2}(-i-1)\leftarrow\CO(-i+1)[-1].\]

Hence, (1) and (2) and the statement about the left and right generators in (3).

\end{proof}

\begin{lemma}
  [Ranks of the lower sides of outer diamonds are coprime]\label{Lem:ranksCoprime}
   Let $(r',d',e')$ and $(r'',d'',e'')$ be the primitive Chern characters corresponding the rays $\fod'$ and $\fod''$ that generate a diamond $\Diamond^{\mathrm{out}}$
with $r',r''>0$. Then we have
        \begin{align}\label{eq:ranksCoprime}
       \gcd(r',r'')=1.
    \end{align}
\end{lemma}

\begin{proof}
From Lemma~\ref{Lem: degree=rank} and Figure~\ref{Fig: RL}, we have $r'=D'$ and $r''=D''$, where $D^{(i)}$ for $i=1,2$ is the determinant at the vertex dual to the ray corresponding to the object with Chern character $(r^{(i)},d^{(i)},e^{(i)})$. But $D',D''$ are part of a Markov triple, and by \cite[Theorem B]{YingZhang2007} or \cite{Frobenius-1913}, we know that the elements of any Markov triple are pairwise coprime. Hence, $r'=D'$ and $r''=D''$ are coprime, as claimed.

\end{proof}

Starting from two generating rays of an arbitrary outer diamond, we can determine the one-dimensional rank-zero object.
\begin{theorem}
   Let $\gamma'=(r',d',e')$ and $\gamma''=(r'',d'',e'')$ be the Chern 
characters of the exceptional bundles corresponding to the rays 
$\fod'$ and $\fod''$ that generate a diamond $\Diamond^{\mathrm{out}}$, with $r',r''>0$.  If $D={r''d'-r'd''}>0$ is the determinant at $\fod'\cap\fod''$, then the vertical diagonal of $\Diamond^{\mathrm{out}}$ corresponds to $\CO_{C_{D}\cup \left({\frac{D^2}{2}}+r''e'-r'e''\right)}$ or $\CO^{\vee}_{C_{D}\cup \left({\frac{D^2}{2}}-r''e'+r'e''\right)}$. 

\end{theorem}
\begin{proof}
  From Lemma \ref{Lem:ranksCoprime}, we have $\gcd(r',r'')=1$. 

We calculate
    \begin{align*}
     r''\gamma'-r'\gamma''
        =(0,{r''d'-r'd''},{r''e'-r'e''})=(0,D,r''e'-r'e''),
    \end{align*}
which is the Chern character of $\CO_{C_{D}\cup\left({\frac{D^2}{2}}+r''e'-r'e''\right)}$ 
by Lemma~\ref{lem:Cherns}.
   
   Similarly, if we consider $r'\gamma''-r''\gamma'$ instead, we end up with 
$(0,-D,-r''e'+r'e'')$, which is the Chern character for 
$\CO^{\vee}_{C_{D}\cup \left({\frac{D^2}{2}}-r''e'+r'e''\right)}$. 
   
   This implies the result in both cases. 
\end{proof}

\begin{theorem}[Discrete rational numbers determine outer diamonds]\label{thm:VerticalPositiveDiamond}
    The reduced rational numbers associated to discrete vertical rays (or the discrete reduced rational numbers) determine the Chern characters of exceptional
bundles associated to the generating rays of the
corresponding outer diamond.
More precisely, 
let $(0,u,v)$ be the primitive Chern character corresponding
to a vertical line $x=q$ for $q$ a discrete rational number, $q\not\in
\BZ/2$.
Then the Chern characters $(r',d',e')$ and $(r'',d'',e'')$ 
of the exceptional bundles associated to the generators
of the corresponding diamond are given by
    \begin{align}\label{eq:rr'dd'ee'}
    \begin{cases}
        r'=&D',\\
        r''=&D'',\\
        d'=& (\frac{v}{u}+\frac{3}{2})D'-\frac{D''}{u},\\
        d''=& \frac{D'}{u}+(\frac{v}{u}-\frac{3}{2})D'',\\
        e'=&D'\left({5u^2+12uv+4v^2-4\over 8u^2}\right) - D''{v\over u^2},\\
         e''=&D''\left({5u^2-12uv+4v^2-4\over 8u^2}\right)+D'{v\over u^2},
    \end{cases}
\end{align}
where $D$ is the determinant at the base of the diamond, and $(D',D'',D)$ is 
the Markov triple of determinants of vertices of the
base triangle of the diamond, ordered $D>D''\geq D'$,
(see Figure \ref{Fig: denseRational}).
\end{theorem}

\begin{proof}Let $(0, u, v)$ be the associated Chern character of the vertical ray, which has the first generating point at $(\frac{v}{u},\frac{5}{8}-\frac{1+v^2}{2u^2})$ (see Theorem \ref{Thm: FirstWallRank0}). Also, if $V_{\CE}$ is the
base of the diamond, let $(r,d,e)$ be the Chern character of $\CE$. 
Via Figure \ref{Fig: RL}, we see that $L_{\CE}$ is the right roof of the
diamond, and $L_{\CE(-3)}$ is the left roof of the diamond.
Using \eqref{eq:(r,d,e)}, one sees the Chern character of $\CE(-3)$ is
$(r,d-3u,e-3v)$.

Let $(D',D'',D)$ be the 
Markov triple 
corresponding to the base triangle $\Delta$ of $\Diamond^{\mathrm{out}}$, 
ordered so that $D>D''\geq D'$. In particular, if $\Delta$ is not an
initial triangle, then by
Proposition~\ref{prop:det order}, $D$ is the determinant
of the incoming vertex of $\Delta$, 
the hybrid vertex
of $\Delta$ has determinant $D''$, the outgoing vertex has determinant
$D'$, and $D''>D'$. 

If, on the other hand, $\Delta$ is an initial triangle,
obviously $D'=D''=1$ and $D=2$ is still the determinant of the incoming vertex.

Without loss of generality, we assume that the vertex with determinant $D''$ is to the right of $V_{\CE}$, and the vertex with determinant $D'$ is to the left of it, as depicted in Figure \ref{Fig: denseRational}.  
Let $(r^{(i)},d^{(i)},e^{(i)})$ be the primitive Chern characters of the generators of the
diamond with $r',r''>0$.

Recall from Lemma~\ref{Lem: degree=rank} that the determinant of a vertex can
be computed in any triangle containing it.
Hence, in Figure \ref{Fig: denseRational}, we can calculate $D',D''$ as follows. First, we have $D'=rd''-dr''+3rr''$ and
    $D''=r'd-rd'$. From Lemma \ref{Lem: degree=rank}, we have $D^{(i)}=r^{(i)}$ for $0\leq i \leq 2$. Hence, we obtain
\begin{align}
  \label{eq:d''}  d''=& \frac{D'}{r}+\frac{d}{r}D''-3D'',\\
  \label{eq:d'} d'=& \frac{d}{r}D'-\frac{D''}{r}.
\end{align}

On the other hand, since for $i=1,2$, the rays corresponding to the objects with Chern characters $(r^{(i)},d^{(i)},e^{(i)})$ intersect at $V_0=(x_0,y_0)=(\frac{v}{u},\frac{5}{8}-\frac{1+v^2}{2u^2})$ from Theorem~\ref{Thm: FirstWallRank0}, and using \eqref{eq:d'} and \eqref{eq:d''}, we have 
\begin{align}
 \label{eq:e'}   e'=&D'\left({5u^2+12uv+4v^2-4\over 8u^2}\right) - D''{v\over u^2},\\
    \label{eq:e''}   e''=&D''\left({5u^2-12uv+4v^2-4\over 8u^2}\right)+D'{v\over u^2}.
\end{align}

Then, combining  \eqref{eq:d'},\eqref{eq:d''},\eqref{eq:e'}, \eqref{eq:e''} with Lemma \ref{Lem: degree=rank}
and \eqref{eq:(r,d,e)}, we obtain the result.

\end{proof}

\begin{remark}
The Frobenius unicity conjecture \cite{Frobenius-1913} states that the
Markov triple $(D',D'',D)$ is determined by its largest element.
Thus in Theorem \ref{thm:VerticalPositiveDiamond}, one expects that
the only input needed is $u$ and $v$ as $D=u$.
\end{remark}

We assume the notation in the proof of Theorem \ref{thm:VerticalPositiveDiamond} and Figure \ref{Fig: denseRational}, and obtain the following important corollary, which describes the left and right generators  at the first generating point of discrete vertical rays.

\begin{corollary}[Explicit left/right generators for discrete vertical rays] Let
$\gamma=(0,u,v)$ be a primitive Chern character of a one-dimensional rank-zero object such that $L_{\gamma}$ contains the vertical diagonal of a outer/dense diamond. Then the Chern character of the left and right generators at the first generating point for the moduli space of objects with Chern character $(0,u,v)$ are
\begin{align*}
    \left(D', (\frac{v}{u}+\frac{3}{2})D'-\frac{D''}{u},D'\left({5u^2+12uv+4v^2-4\over 8u^2}\right) - D''{v\over u^2}\right),\\ \left(D'', \frac{D'}{u}+(\frac{v}{u}-\frac{3}{2})D'',D''\left({5u^2-12uv+4v^2-4\over 8u^2}\right)+D'{v\over u^2}\right),
\end{align*}
where $D',D''$ are as in the statement of Theorem~\ref{thm:VerticalPositiveDiamond}.
\end{corollary}

\begin{proof}
    This comes from \eqref{eq:rr'dd'ee'} and Figure \ref{Fig: denseRational}.
\end{proof}

\begin{theorem}
[The first generating point for dense one-dimensional rank-zero moduli spaces (non-initial base triangle)]
\label{SufficientRank0Dene}
Let $\Diamond^{\mathrm{out}}$ be an outer diamond based at a non-initial triangle with base $V_{\CE}$ for
an exceptional bundle $\CE$ with Chern character $(r,d,e)$, so that $L_{\CE}$
contains the right roof of $\Diamond^{\mathrm{out}}$.
Let the primitive Chern character of the vertical diagonal of $\Diamond^{\mathrm{out}}$ be $(0,u,v)$, where $u=r$ and $v=d-\frac{3}{2}r$ as in
\eqref{eq:(r,d,e)}. Also, we fix the notation for determinants as above. Then,
the $x$-coordinates $q$ of the degenerate diamonds contained in the right
roof of $\Diamond^{\mathrm{out}}$ are in the range 
\begin{align}\label{eq:denseRationalRange}
    \frac{v}{u}<q<\frac{v}{u}+\Big(\frac{3}{2}-\frac{\sqrt{9r^2-4}}{2r}\Big).
\end{align}

In particular, the first generating point for the Chern character
$\gamma=(0,u',v')$ with $q=v'/u'$ in this range is a degenerate diamond
on the right roof of $\Diamond^{\mathrm{out}}$ and 
the object with Chern character $(r,d,e)$ corresponding to $\fod$ is the left generator at this point.
\end{theorem}

\begin{proof}
From Lemma \ref{lem: verticelPositive}, we have $u=r$. Also from Definition \ref{Def: V_E}, we get $v=d-\frac{3}{2}r$.

Recall from Definition~\ref{Def:dense cone} that $r_{\infty}(D)^{-1}=\frac{3D-\sqrt{9D^2-4}}{2}$ is the slope of the right boundary of the dense diamond $\Diamond$ based at a vertex of determinant $D$, with respect to the basis $m_1,m_2$
of Definition~\ref{Def:dense cone}. For short-hand, we omit the $D$ and write $r_{\infty}^{-1}$. Also, recall that the left and the right lower sides of $\Diamond^{\mathrm{out}}$ correspond to objects of Chern characters $(r',d',e')$ and $(r'',d'',e'')$, respectively. See Figure \ref{Fig: denseRational}. Without loss of generality, we consider the ray $\fod$ corresponding to the right roof of a diamond, the direction vector of the right boundary of the dense diamond is given by
\begin{align*}
          r_{\infty}^{-1}\begin{pmatrix}
        -r'\\d'
    \end{pmatrix}+\begin{pmatrix}
        r''\\-d''
    \end{pmatrix}=\begin{pmatrix}
        -r_{\infty}^{-1}D'+D'' \\ r_{\infty}^{-1}(\frac{dD'-D''}{r})-\frac{D'}{r}-(\frac{d}{r}-3)D'')
    \end{pmatrix},
\end{align*}
using \eqref{eq:d''},\eqref{eq:d'}.
Therefore, the ray is given by
\begin{align*}
    \left(r_{\infty}^{-1}D'-D''\right)y+&\left( r_{\infty}^{-1}\left(\frac{dD'-D''}{r}\right)-\frac{D'}{r}-\left(\frac{d}{r}-3\right)D''\right)x\\=& \left( r_{\infty}^{-1}D'-D''\right)\left(\frac{3d}{2r}-1-\frac{e}{r}\right)+\left( r_{\infty}^{-1}\left(\frac{dD'-D''}{r}\right)-\frac{D'}{r}-\left(\frac{d}{r}-3\right)D''\right)\left(\frac{d}{r}-\frac{3}{2}\right).
\end{align*}

Intersecting this with the equation of $\fod$, $ry+dx=e$, applying
Lemma \ref{lem:slope inequalities}(6), and then using $u=r$, $v=d-\frac{3}{2}r$, and Lemma \ref{lem:slope inequalities}(6), we get the following $x$-coordinate
\begin{align*}
    x= \frac{-r_\infty^{-1}(vD''+D')-vD'+(1+3uv)D''}{-r_\infty^{-1}uD''-uD'+3u^2D''},
\end{align*}
which determines all the dense rational rays on $\fod$, i.e., any rational vertical ray $x=q$, with
\begin{align*}
\frac{v}{u}<q< & \frac{-r_\infty^{-1}(vD''+D')-vD'+(1+3uv)D''}{-r_\infty^{-1}uD''-uD'+3u^2D''}\\
= {} & \frac{v}{u}+ \frac{-r_{\infty}(D)^{-1}D'+D''}{D\big(-r_{\infty}(D)^{-1}D''+3DD''-D'\big)}
\end{align*}
has its first generating point on the right roof of the dense diamond $\Diamond$. 
\medskip

\emph{Claim.} We have  $\frac{-r_{\infty}(D)^{-1}D'+D''}{D\big(-r_{\infty}(D)^{-1}D''+3DD''-D'\big)}=\frac{3}{2}-\frac{\sqrt{9r^2-4}}{2r}$.

\begin{proof}
    We use the notation $R:=r_{\infty}(D)^{-1}=\frac{3D-\sqrt{9D^2-4}}{2}$. Then,we have
    \begin{align*}
        R^2=\left(\frac{3D-\sqrt{9D^2-4}}{2}\right)^2=\frac{9D^2-2-3D\sqrt{9D^2-4}}{2}=3DR-1.
    \end{align*}
From this, it follows that
\[
-R^2D''+3RDD''-RD'=-RD'+D''.
\]
However, this is equivalent to 
\[
\frac{R}{D}=\frac{-RD'+D''}{D(-RD''+3DD''-D')}
\]
In turn, this is equivalent to
\[
\frac{-r_{\infty}(D)^{-1}D'+D''}{D\big(-r_{\infty}(D)^{-1}D''+3DD''-D'\big)}=\frac{3}{2}-\frac{\sqrt{9r^2-4}}{2r},
\]
which is the claim.
\end{proof}

The last statement about the left generator is also a straightforward implication of this.
The fact that this degenerate diamond is a first generating point
comes from Lemma \ref{Lem:DensnessNonApex}.

\begin{figure}[h]
    \centering
    \includegraphics[width=12.2cm]{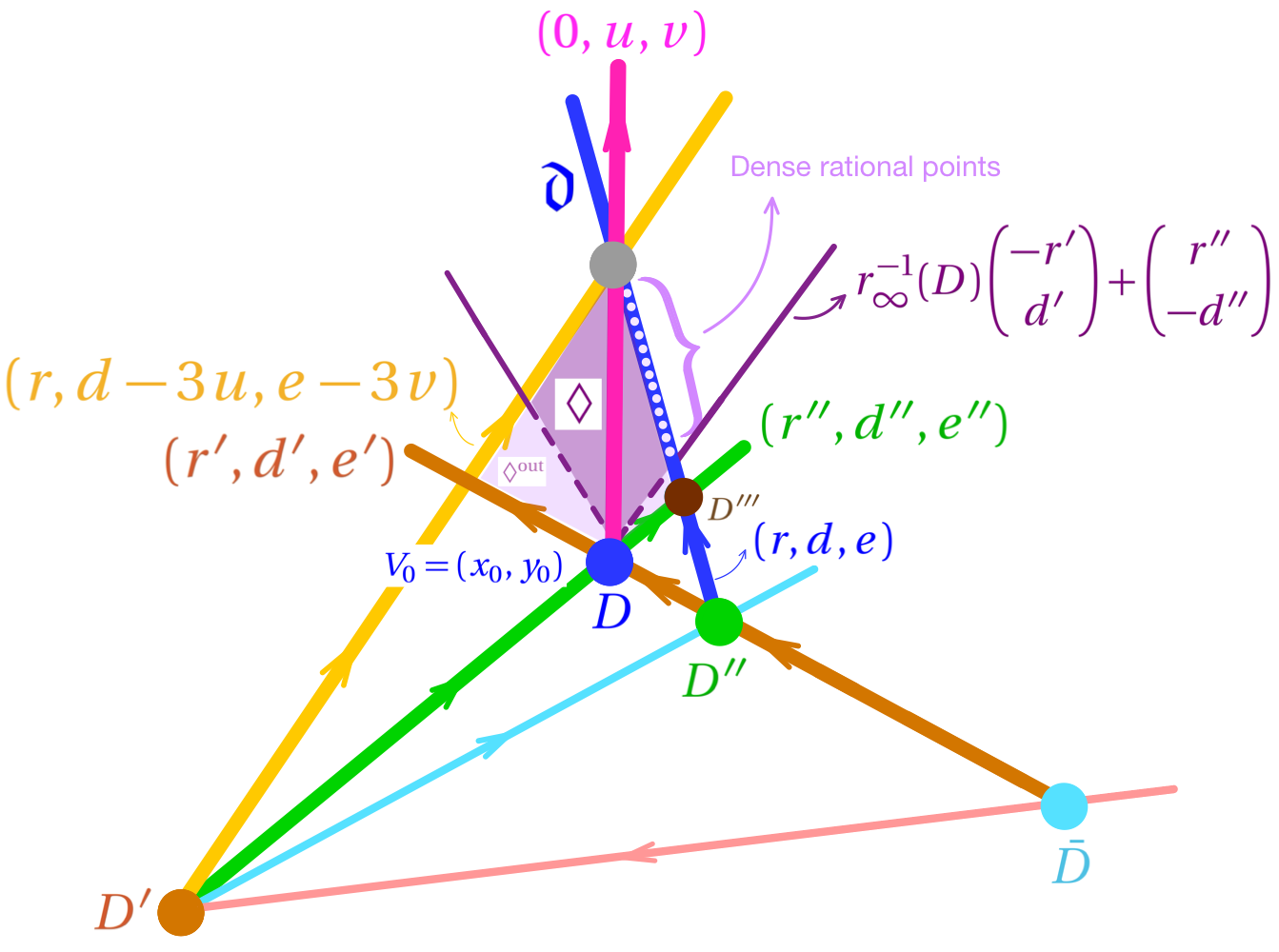}
    \caption{The Chern character of the objects corresponding to lower rays can be determined in terms of $u,v,D^{(i)}$ (see \eqref{eq:rr'dd'ee'}). Moreover, we can determining all the rational dense points on a given ray $\fod$ (Theorem \ref{SufficientRank0Dene}).} 
    \label{Fig: denseRational}
\end{figure}
\end{proof}

By a direct calculation, we can obtain the following results for diamonds at the center of super-diamonds and the initial diamonds, respectively (see Figure \ref{Fig: denseRationalInitial}):

\begin{theorem}
    [The first generating point for dense one-dimensional rank-zero moduli spaces (initial base triangle)]\label{SufficientRank0DeneInitial}
Let $\Diamond^{\mathrm{out}}$ be an outer diamond based at an initial triangle with base  $(m,\frac{1-m^2}{2})$, the intersection point of $L_{\CO(m+1)}$ and $L_{\CO(m-1)}$.  Then,
the $x$-coordinates $q$ of the degenerate diamonds contained in the right
roof of $\Diamond^{\mathrm{out}}$ are in the following range 
\begin{align}\label{eq:denseRationalRangeSuperDiamond}
    m<q<m+\frac{3-2\sqrt{2}}{2}.
\end{align}
\end{theorem}

\begin{theorem}
    [The first generating point for dense one-dimensional rank-zero moduli spaces (initial diamonds)]\label{SufficientRank0DeneInitialDiamond}
Let $\Diamond^{\mathrm{init}}$ be an initial diamond with base  $(m+\frac{1}{2},\frac{-m-m^2}{2})$, the intersection point of $L_{\CO(m)}$ and $L_{\CO(m+1)}$.  Then,
the $x$-coordinates $q$ of the degenerate diamonds contained in the right
roof of $\Diamond^{\mathrm{init}}$ are in the following range 
\begin{align}\label{eq:denseRationalRangeSuperDiamond}
   m+\frac{1}{2}<q<m+\frac{1}{2}+\frac{3-\sqrt{5}}{2}.
\end{align}
\end{theorem}

\begin{figure}[h]
    \centering
    \includegraphics[width=13.3cm]{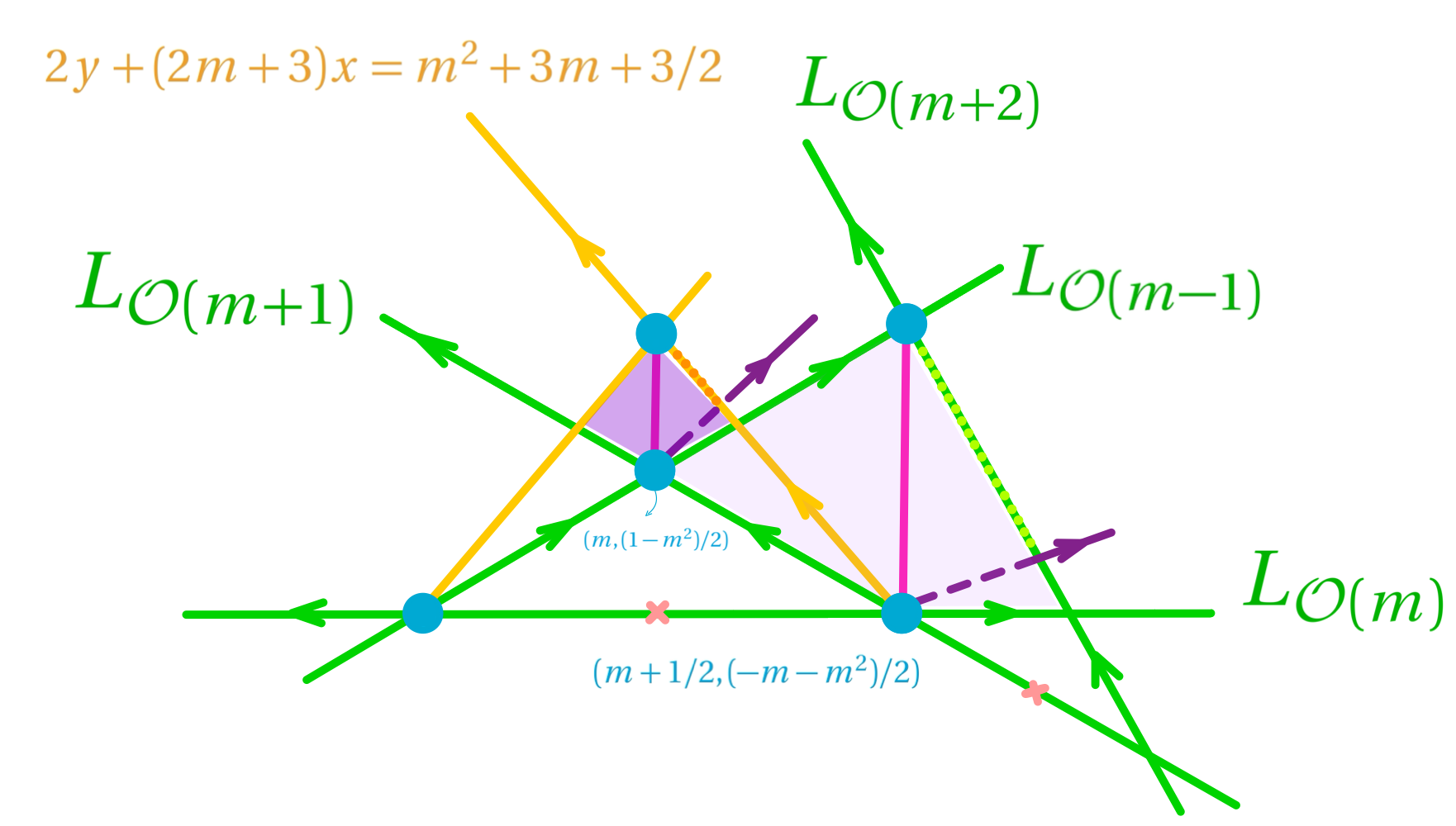}
    \caption{} 
    \label{Fig: denseRationalInitial}
\end{figure}

\begin{example} Considering the rays generated by $L_{\CO(1)}$ and $L_{\CO(2)}$ at the vertex $(3/2,-1)$, we can show that vertical diagonals of the outer diamonds in the super-diamond, $\Diamond^{\mathrm{sup}}$ along the line corresponding to $\CO$, are explicitly given by $x=\frac{3}{2}-\frac{F_{2k+1}}{F_{2k+3}}$. The primitive Chern character for the vertical diagonal $x=\frac{3}{2}-\frac{F_{2k+1}}{F_{2k+3}}$, is $(0, F_{2k+3},\frac{3}{2}F_{2k+3}-F_{2k+1})$. One can show that the primitive object for this is  $\CO_{C_{F_{2k+3}}\cup \#}(m)$, a one-dimensional rank-zero object, where $C_{F_{2k+3}}\cup \#$ is a curve of degree $F_{2k+3}$, together with $\#=\frac{F_{2k+3}(F_{2k+3}+3)}{2}-F_{2k+1}$ number of points. 
\end{example}

\begin{corollary}
\label{cor:interval}
Let $\CE$ be an exceptional bundle of rank $r$ and first Chern class
$d$, $u=r$, $v=d-\frac{3}{2}r$. Let $\Diamond$ be the (dense) diamond
based at $V_{\CE}$. Then the image of the union of the left and right roofs
of $\Diamond$ under the projection to the $x$-axis is
\[
\left[ \frac{u}{v}-\left(\frac{3}{2}-\frac{\sqrt{9r^2-4}}{2r}
\right),
\frac{u}{v}+\left(\frac{3}{2}-\frac{\sqrt{9r^2-4}}{2r}
\right)\right].
\]
\end{corollary}
\begin{proof}
    The right half is an immediate corollary of Theorem \ref{SufficientRank0Dene}, and the left half is obtained by symmetry.
\end{proof}

\section{An application: Connection to \cite{Li-17} and the
Le Potier curve} \label{Section: LePotier} 

\cite{Li-17} considers the geography of coherent sheaves on $\BP^2$ by
using what the author calls
the $\{1,\frac{\ch_1}{\ch_0},\frac{\ch_2}{\ch_0}\}$-plane. 
Omitting the $1$, we identify this plane with $\BR^2$, with coordinates
\[
X=\frac{\ch_1}{\ch_0},\quad Y=\frac{\ch_2}{\ch_0}.
\]
Any coherent
sheaf $\CE$ of non-zero rank gives a point 
\[
V_{\CE}^{\mathrm{Li}}:=(\ch_1(\CE)/\ch_0(\CE),
\ch_2(\CE)/\ch_0(\CE))
\]
in this plane. 

Via \eqref{eq: VertexPE}, there
is a clear relationship between the plane $M_{\BR}$ (with coordinates $x,y$)
and the $\{1,\frac{\ch_1}{\ch_0},\frac{\ch_2}{\ch_0}\}$-plane. 
We define an affine linear transformation 
\begin{align*}
\Upsilon:&\BR^2\rightarrow M_{\BR},\\ 
&(X,Y)\mapsto (X-3/2,3X/2-Y-1).
\end{align*}

Using $\chi=\ch_2+\frac{3}{2}\ch_1+\ch_0$ and \eqref{eq: VertexPE}, we see
that
\[
\Upsilon(V^{\mathrm{Li}}_{\CE})=V_{\CE}.
\]
The inverse map is
\[
\Upsilon^{-1}(x,y)=(x+3/2, -y+3x/2+5/4).
\]

\cite{Li-17} introduces the parabolas
\[
\overline\Delta_a=\left\{(X,Y)\,\bigg|\,\frac{1}{2}X^2-Y=a\right\}.
\]
Note that if 
$V^{\mathrm{Li}}_{\CE}\in\overline{\Delta}_0$,
then the Bogomolov-Gieseker inequality is an equality for $\CE$. If
$\CE$ is stable, then $V^{\mathrm{Li}}_{\CE}$
lies on or below $\overline{\Delta}_0$.

One calculates that
\[
\Upsilon(\overline\Delta_a)=\partial_b:=\{(x,y)\in M_{\BR}\,|\,
y=-x^2/2+b\}
\]
with $b=a+1/8$. 

In particular, $\Upsilon(\overline\Delta_0)$ is  the upward translation of the
boundary of the closure of $U$ by $1/8$. For any $\CE$ stable of
non-zero rank, the Bogomolov-Gieseker inequality now implies that 
$V_{\CE}$ lies on or above $\partial_{1/8}$. Further, this curve
contains $V_{\CO(n)}=(n-3/2,-(n-1)(n-2)/2)$ for all $n\in \BZ$,
being the hybrid vertices of the initial triangles. On the other hand, $V_{\CE}$
for $\CE$ an exceptional bundle of rank greater than one lies
above the curve $\partial_{1/8}$, as follows easily from 
Lemma~\ref{lem:slope inequalities},(6). In particular,
$\partial_{1/8}$ lies below all diamonds and contains the bases of
the initial diamonds, and hence $\partial_{1/8}$ provides the best
lower bound for the outer diamonds among all parabolas $\partial_b$.

Similarly, translating the boundary of the closure of $U$ by $9/8$ to
obtain $\partial_{9/8}=\Upsilon(\overline\Delta_1)$, we obtain a parabola
which meets each initial diamond in its apex. Further, the apex of every other
diamond lies strictly below $\partial_{9/8}$. Indeed, if the base of
a diamond $\Diamond_v$ is $V_{\CE}$ for some exceptional bundle $\CE$,
then the apex has coordinates $(d/r-3/2,(3d-\chi)/r+1/r^2)$ by 
Lemma \ref{Lem: lengthDiagonal}. Plugging this into the inequality 
$y\le -x^2/2+9/8$, writing $\chi$ in terms of the second Chern class
$e$  and simplifying, we obtain the inequality $-d^2+2er+2r^2-2\ge 0$.
Using Lemma \ref{lem:slope inequalities}(6), this reduces to $r^2\ge 1$.
Thus the inequality always holds and is an equality only when $V_{\CE}$
is the base of an initial diamond.
See Figure \ref{fig: LPcurve}.

\cite[Def.~1.6]{Li-17} defines the Le Potier curve in the
$\{1,\frac{\ch_1}{\ch_0},\frac{\ch_2}{\ch_0}\}$-plane, which we 
paraphrase as follows. For an exceptional bundle $\CE$, set
\[
V^+_{\CE}:=V^{\mathrm{Li}}_{\CE}-(0,1/\ch_0(\CE)^2).
\]
If $(\CE_0,\CE_1,\CE_2)$ is a strong exceptional triple, define
$V^{l}_{\CE_1}$ (resp.\ $V^{r}_{\CE_1}$) to be the intersection of the line 
segment with endpoints $V^{\mathrm{Li}}_{\CE_0},V^+_{\CE_1}$ 
(resp.\ $V^{\mathrm{Li}}_{\CE_2},V^+_{\CE_1}$)
with the parabola $\overline\Delta_{1/2}$. The Le Potier curve $C_{LP}$ is then
the union of all line segments $[V^l_{\CE_1},V^+_{\CE_1}]$
joining $V^l_{\CE_1}$ to $V^+_{\CE_1}$
and $[V^r_{\CE_1},V^+_{\CE_1}]$ joining
$V^r_{\CE_1}$ to $V^+_{\CE_1}$, 
running over all strong exceptional triples.
By construction, this curve lies between $\overline\Delta_{1/2}$ and
$\overline\Delta_1$. As shown in \cite{Drezet-LePotier}, it has the fundamental
property that if $\gamma=(r,d,e)\in\Gamma$ satisfies $(d/r,e/r)$ lying on or
below $C_{LP}$, then there is a slope semistable coherent sheaf of 
Chern character $\gamma$. Conversely, a slope semistable coherent sheaf
$\CE$ of positive rank is either exceptional, or $V^{\mathrm{Li}}_{\CE}$
lies below $C_{LP}$.

\begin{prop}
$C_{LP}$ corresponds to $C_{\mathrm{roofs}}$, i.e., 
\[
\Upsilon(C_{LP})=C_{\mathrm{roofs}}.
\]
\end{prop}

\begin{proof}
Let $(\CE_0,\CE_1,\CE_2)$ be a strong exceptional triple, and let $\Diamond$
be the
dense diamond with base $V_{\CE_1}$. Let $v$ be the apex of this
diamond. By Lemma~\ref{Lem: lengthDiagonal}, $v=V_{\CE_1}+(0,1/r_1^2)$,
where $r_1$ is the rank of $\CE_1$. Let $L,R$ be the line segments
which are the left and right roofs
of $\Diamond$.
By the definition of $C_{\mathrm{roofs}}$, it is sufficient to show that
\[
\Upsilon([V^l_{\CE_1},V^+_{\CE_1}])=L,
\quad
\Upsilon([V^r_{\CE_1},V^+_{\CE_1}])=R.
\]
To show this, first note that 
\[
\Upsilon(V^+_{\CE_1})=\Upsilon(V^{\mathrm{Li}}_{\CE_1}-(0,1/r_1^2))=
V_{\CE_1}+(0,1/r_1^2)=v.
\]
Second, we will show that $\Upsilon(V^r_{\CE_1})$ is the right-hand
vertex of $\Diamond$. From Figure \ref{Fig: RL}, we see that the
line through $v$ and $V_{\CE_2}$ is $L_{\CE_1}$, and hence this
is the image of the line through $V^+_{\CE_1}$ and $V^{\mathrm{Li}}_{\CE_2}$
under $\Upsilon$. Thus we need to check that the intersection point
of $\partial_{5/8}$ with $L_{\CE_1}$ is the right vertex of $\Diamond$.
The $x$-coordinate of this vertex is 
$\frac{d_1}{r_1}-\frac{\sqrt{9r_1^2-4}}{2r_1}$ by Corollary~\ref{cor:interval}.
So it is sufficient to compare this with the $x$-coordinate of 
$L_{\CE_1}\cap \partial_{5/8}$. Solving the equations
$r_1y+d_1x=e_1, y=-x^2/2+5/8$ and using Lemma~\ref{lem:slope inequalities}, we get
\begin{equation}
\label{eq:two intersection points}
x=\frac{d_1}{r_1}\pm \frac{\sqrt{9r_1^2-4}}{2r_1},
\end{equation}
and indeed there are two intersection points. However, one checks that
\[
1 < \frac{\sqrt{5}}{2} \le \frac{\sqrt{9r_1^2-4}}{2r_1} < \frac{3}{2}.
\]
Since the difference in $x$-coordinates of $V_{\CE_1}$ and $V_{\CE_2}$ is
less than $1$ by Lemma~\ref{lem:slope inequalities},(3), the intersection
point with the larger value of $x$ in fact lies to the right of
$V_{\CE_2}$. Hence the desired intersection point has $x$-coordinate
given by the minus sign in \eqref{eq:two intersection points}. This is the
desired value.

The argument for the left-hand intersection point is similar.
\end{proof}

To summarize, to go from the $\{1,\frac{\ch_1}{\ch_0},\frac{\ch_2}{\ch_0}\}$-plane to the stability scattering diagram, we decompose $\Upsilon$ as follows:  We first flip the $\{1,\frac{\ch_1}{\ch_0},\frac{\ch_2}{\ch_0}\}$-plane about the $X$-axis, and then shift everything up by $1/8$. At this stage $\overline{\Delta}_0$ goes to $\partial_{1/8}$ and $\overline{\Delta}_1$ goes to $\partial_{9/8}$.Then, using the following transformation only on the Le Potier curve,
\begin{align}
    (\alpha,\beta)\mapsto (\alpha-3/2,3/2\alpha+\beta-9/8),
\end{align}
the Le Potier curve $C_{LP}$ will go to the boundary of the bounded region of the scattering diagram, $C_{\mathrm{roofs}}$. Also, each exceptional bundle $\CE$ in the $\{1,\frac{\ch_1}{\ch_0},\frac{\ch_2}{\ch_0}\}$-plane goes to the dual object of the corresponding line/ray to $\CE$ in the scattering diagram.

See Figure \ref{fig: LPcurve}.

    \begin{figure}[h]
 \centering
    \includegraphics[width=14.9cm]{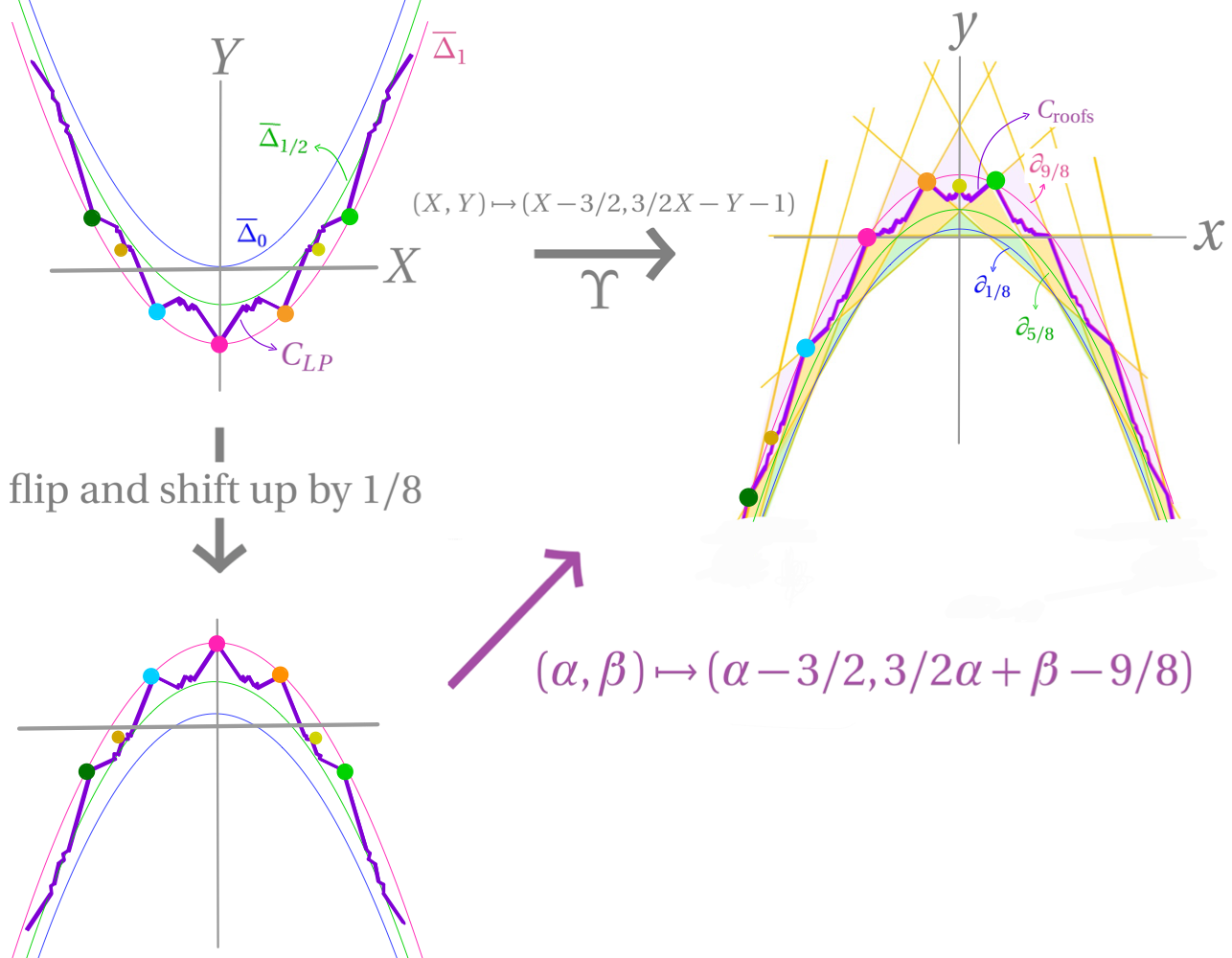}
  \caption{From the $\{1,\frac{\ch_1}{\ch_0},\frac{\ch_2}{\ch_0}\}$-plane to the stability scattering diagram: Under this transformation, the fractal Le Potier curve, $C_{LP}$, goes to the fractal upper boundary of the bounded region in the scattering diagram, $C_{\mathrm{roofs}}$. Also, $\overline\Delta_{0}$ goes to $\partial_{1/8}$, $\overline\Delta_{1/2}$ goes to $\partial_{5/8}$, and $\overline\Delta_{1}$ goes to $\partial_{9/8}$. Further, the points in the $\{1,\frac{\ch_1}{\ch_0},\frac{\ch_2}{\ch_0}\}$-plane transform to the dual vertices in the scattering diagram: e.g., in this picture, $\CT_{\BP^2}$ corresponding to the point $(3/2,3/4)$ in the $(X,Y)$ coordinate transforms to $V_{\CT_{\BP^2}}$ corresponding to the point $(0,1/2)$ in the $(x,y)$ coordinate.}
  \label{fig: LPcurve}
\end{figure}

\begin{corollary}\label{cor:vertex location}
Let $(x_0,y_0)\in U$ be a point with rational coordinates. Then 
$(x_0,y_0)$ is the intersection of two non-trivial rays of 
$\foD^{\mathrm{stab}}$ with distinct tangent lines if and only if
$(x_0,y_0)=\left(\frac{\ch_1(\CE)}{\ch_0(\CE)}-\frac{3}{2}, \frac{3d-\chi(\CE)}
{\ch_0(\CE)}\right)$ for 
some slope-stable coherent sheaf of positive rank $\CE$.
\end{corollary}

\begin{proof}
We have seen from Theorem~\ref{thm:bousseau generalization}
and Theorem~\ref{Thm:nonEmptiness}
that the points of $R_{\bdd}$ below $C_{\mathrm{roofs}}$
which are intersections of two rays are precise points of the
form $V_{\CE}$ for $\CE$ an exceptional bundle. Further,
every rational point on $C_{\mathrm{roofs}}$ or above is the intersection
of two rays by Corollary~\ref{Cor:unboundedCharacterization}.
On the other hand,
by \cite[\S4.4]{Drezet-LePotier} (see \cite[Thm.~1.8]{LZ2} for the statement
as needed here), a rational point $\sigma$ of $U$ satisfies
$\sigma=V_{\CE}$ for a slope-stable coherent sheaf $\CE$ if and only if
$\sigma=V_{\CE}$ for $\CE$ an exceptional bundle or $\sigma$ lies on
$C_{\mathrm{roofs}}$ or above. This proves the result.
\end{proof}

\bibliographystyle{amsplain-nodash} 
\bibliography{bib}
\end{document}